\documentclass[ks]{imsart}

\RequirePackage{amsthm,amsmath,amsfonts,amssymb}
\usepackage[margin=1.5in]{geometry}
\usepackage{mathrsfs}
\usepackage{algorithm}
\RequirePackage[authoryear]{natbib}
\RequirePackage[colorlinks,citecolor=blue,urlcolor=blue]{hyperref}
\RequirePackage{graphicx}
\usepackage{placeins}
\usepackage{xcolor}

\startlocaldefs
\theoremstyle{plain}

\newtheorem{theorem}{Theorem}[section]
\newtheorem{lemma}[theorem]{Lemma}
\newtheorem{proposition}[theorem]{Proposition}
\newtheorem{corollary}[theorem]{Corollary}

\makeatletter
\let\c@algorithm\c@theorem
\makeatother

\theoremstyle{definition}
\newtheorem{definition}[theorem]{Definition}

\theoremstyle{remark}

\usepackage{bbm}
\usepackage[inline]{enumitem}
\usepackage{subcaption}
\usepackage{multirow}
\usepackage{rotating}
\usepackage{graphicx}
\usepackage{tikz}
\usepackage{pgfplots}

\pgfplotsset{compat=1.18}
\usepackage{xfp} 
\usetikzlibrary{arrows.meta,positioning,shapes.geometric,calc,intersections}

\usepackage[acronym]{glossaries}

\newacronym{ids}{\textit{IDS}}{\textit{intrusion detection system}}
\newacronym{it}{\textit{IT}}{\textit{information technology}}
\newacronym{roc}{\textit{ROC}}{\textit{receiver operating characteristic}}
\newacronym{cbi}{CBI}{\textit{conservative Bayesian inference}}
\newacronym{mdl}{\textit{MDL}}{\textit{Machine/Deep Learning}}
\newacronym{svm}{\textit{SVMs}}{\textit{support vector machines}}
\newacronym{ann}{\textit{ANNs}}{\textit{artificial neural networks}}
\newacronym{dnn}{\textit{DNNs}}{\textit{deep neural networks}}
\newacronym{pfd}{\textit{pfd}}{\textit{probability of failure on demand}}
\newacronym{iid}{i.i.d.}{\textit{independent identically distributed}}
\newacronym{pk}{\textit{PK}}{\textit{prior knowledge}}
\newacronym{mct}{MCT}{\textit{monotone convergence theorem}}
\newacronym{dct}{DCT}{\textit{dominated convergence theorem}}
\newacronym{cdf}{CDF}{\textit{cumulative distribution function}}
\newacronym{fdds}{\textit{fdds}}{\textit{finite-dimensional distributions}}
\newacronym{wrt}{\textit{w.r.t.}}{\textit{with respect to}}
\newacronym{wlog}{\textit{w.l.o.g.}}{\textit{without loss of generality}}
\newacronym{ivt}{IVT}{\textit{Intermediate Value Theorem}}
\newacronym{rhs}{\textit{r.h.s.}}{\textit{right hand side}}
\newacronym{lhs}{\textit{l.h.s.}}{\textit{left hand side}}
\newacronym{fp}{\textit{FP}}{false positive}
\newacronym{fn}{\textit{FN}}{false negative}
\newacronym{av}{\textit{AVs}}{autonomous vehicles}
\newacronym{pcm}{\textit{pcm}}{\textit{probability of a crash event per mile}}
\newacronym{pfm}{\textit{pfm}}{\textit{probability of a fatality event per mile}}
\newacronym{dpm}{dpm}{\textit{disengagements per mile}}
\newacronym{srgm}{\textit{SRGM}}{\textit{Software Reliability Growth Models}}
\newacronym{cots}{\textit{COTS}}{\textit{Commercial off-the-shelf}}

\endlocaldefs

\begin{document}

\begin{frontmatter}
\title{Fixed-Point Characterisations of Extremal Distributions under Partial Distributional Constraints}
\runtitle{Fixed-Point Characterisations of Extremal Priors}

\begin{aug}
\author[A]{\fnms{Kizito}~\snm{Salako}\ead[label=e1]{k.o.salako@citystgeorges.ac.uk}\orcid{0000-0003-0394-7833}},
\author[A]{\fnms{Rabiu}~\snm{Tsoho Muhammad}\ead[label=e2]{Rabiu-Tsoho.Muhammad@citystgeorges.ac.uk}\orcid{0009-0002-0700-6344}}
\address[A]{The Centre for Software Reliability, Department of Computer Science, City St. George's, University of London, Northampton Square, EC1V 0HB, The United Kingdom\\\printead{e1,e2}}

\runauthor{K. Salako and R. Tsoho Muhammad}
\end{aug}

\begin{abstract}
We present a methodological framework for solving robust inference problems with partially specified distributions over measurable subsets of a parameter space. Partial specifications define a set of admissible distributions. The goal is to determine extremal values (over these admissible distributions) for statistical quantities, where these quantities---these objective functions---are ratios of expectations of analytic functions, continuous functions, piecewise continuous functions, and uniform limits of piecewise continuous functions. We show that extremal values are approached by sequences of admissible distributions, whose limiting extremal distributions are characterised by fixed-point conditions on their support locations. This characterises \emph{where} extremal distributions place probability mass and yields a practical computational framework for solving the corresponding optimisation problems. We establish convergence and asymptotic properties of the resulting extremal distributions and extremal objective function values. This work extends robust inference methods (e.g. robust Bayesian inference) by combining extremal-distribution reduction, fixed-point characterisation, and approximation-based analysis within a unified framework.

\end{abstract}



\end{frontmatter}
\newpage
\tableofcontents
\newpage
\listoffigures
\newpage

\section{Introduction}
The following constrained optimisation problem arises often in robust inference settings. For $d\in\mathbb N$, let $K:=[0,1]^d$ and let $K_1,\ldots,K_n$ be a partition of $K$ into Borel-measurable sets.
Let $\mathcal D:=\Bigl\{\mathbb P\in \mathcal{P}(K): \mathbb P(K_i)=p_i,\ i=1,\ldots,n\Bigr\}$, 
where $\mathcal P(K)$ denotes the set of all Borel probability measures on $K$, $p_i> 0$ $(i=1,\ldots,n)$, and $\sum_{i=1}^n p_i=1$.
Let $f,g:K\to [0,\infty)$ be $\mathcal D$-integrable; i.e. $f,g\in\mathcal L(\mathbb P)$ for all $\mathbb P\in\mathcal D$. Further assume $\inf_{{\mathbf x}\in K_j} g({\mathbf x})>0$ for some $j\in\{1,\ldots,n\}$. Then, the optimisation problem of interest is
\begin{equation}
\inf_{\mathbb P \in\mathcal D}\frac{\mathbb E_{\mathbb P }[f({\mathbf X})]}{\mathbb E_{\mathbb P }[g({\mathbf X})]}.
\label{eqn_genCBIprob}
\end{equation}
Problem~\eqref{eqn_genCBIprob} is solved when the infimum and a related extremal distribution of ${\mathbf X}$ are deduced. It is natural to consider \eqref{eqn_genCBIprob} within a Bayesian context, where the prior distribution of $\textbf{X}$ is partially specified. 
%
Indeed, generalisations of \eqref{eqn_genCBIprob} have been extensively studied in the robust Bayesian inference literature (see \cite{berger1990_SensitivityToPrior}, \cite{moreno1991robust}, \cite{1991_Lavine}, \cite{Berger_1994_RobustnessInBidinesionalModels}), and certain nonlinear fractional program forms of~\eqref{eqn_genCBIprob} can be solved using Dinkelbach-type fixed--point iteration (see \cite{Dinkelbach1967},  \cite{Schaible1976Dinkelbach}). 
Beyond robust Bayes, \eqref{eqn_genCBIprob} and its extensions also give generalised Boole--Fréchet bounds, including bounds under marginal-density constraints (where these bounds cannot be obtained directly via Hailperin's classical linear-programming formulation; see \cite{Hailperin_1965}). Several special cases of \eqref{eqn_genCBIprob} (and its generalisations in Section~\ref{sec_methods}) have been solved and applied in various contexts; for example, see \cite{LaurenceWang2005Basket}, \cite{bishop_toward_2011}, \cite{GiacominiKitagawa2021RobustBayesian}, \cite{strigini_software_2013}, 
\cite{Ruger1978}, \cite{zhao_assessing_2019}, \cite{EmbrechtsPuccettiRueschendorf2013}, \cite{littlewood_reliability_2020}, \cite{salako_conservative_2021}, \cite{Whitt1976BivariateDistributions}, \cite{SalakoZhao_TSE_2023,kizito_QRE_2024}). 

This paper makes three main methodological contributions to solving constrained inference problems with partially specified distributions: \textbf{i)}~it derives fixed-point characterisations of extremal priors for \eqref{eqn_genCBIprob}, showing that the optimisation can be reduced to discrete extremal priors, or limits of feasible priors, whose support points satisfy explicit Dinkelbach-type fixed-point conditions. This strengthens earlier extremal-prior reduction results by identifying \emph{where} extremal priors place mass; \textbf{ii)}~it develops these structural results into a practical \emph{computational framework}, based on discrete reformulations, fixed-point iteration, and approximation arguments, for solving \eqref{eqn_genCBIprob} and several nontrivial generalisations (e.g. overlapping measurable-set constraints and multiple marginal-density constraints); and \textbf{iii)}~it proves convergence/asymptotic properties of the resulting extremal values and priors, extending the framework beyond analytic objectives to continuous, piecewise continuous, and uniform-limit classes of functions.

The paper is organised as follows. Section~\ref{sec_methods} develops the theoretical and computational framework for \eqref{eqn_genCBIprob} and its generalisations. In particular, Subsection~\ref{subsec_discreteformulationsandGeneralisations} establishes equivalent finite and discrete extremal-prior formulations, including extensions to overlapping measurable-set constraints and marginal-density constraints. Subsection~\ref{subsec_solutions} gives the main fixed-point characterisations and algorithms: analytic objectives are treated directly through Dinkelbach-type conditions, while continuous, piecewise continuous, and uniform-limit objective classes are handled through approximation and convergence arguments. Section~\ref{subsec_examples} then illustrates the framework with applications to one- and two-dimensional extremal-prior problems, Markov chain likelihoods with mixed prior constraints, Bernoulli reliability models, generalised Boole--Fr\'echet bounds, interval-identified econometric inference, robust value-at-risk aggregation, model-free basket-option bounds, and inference using Markov chains with random transition probabilities. Section~\ref{sec_discussion} discusses the interpretation of the resulting extremal priors, their discrete or degenerate nature, their role in conservative Bayesian inference, their convergence to ordinary Bayesian inference under refined partial specifications, and the attractor--repeller structure that governs extremal support locations.

\section{Methodology}
\label{sec_methods}
This section establishes a general framework for solving \eqref{eqn_genCBIprob} and its extensions. The key steps are: reduce the original optimisation to an extremal discrete problem, characterise the support of the resulting extremal priors via fixed-point conditions, then extend this to broader prior constraints and objective-function classes. 

The following are the structural foundation and computational tools for the framework. Four theorems give solutions of  \eqref{eqn_genCBIprob} under four nested classes of $f, g$ functions: \textbf{i)}~Theorem~\ref{thm:global_fp_tuple} for analytic functions; \textbf{ii)}~Theorem~\ref{thm_gensol} for continuous functions; \textbf{iii)}~Theorem~\ref{thm_refinedgensol} for bounded piecewise continuous functions with a measurable set of discontinuities; \textbf{iv)}~Theorem~\ref{thm:regulated_extension} for uniform limits of bounded piecewise continuous functions. Theorems~\ref{thm_CBIwithfails_sol}--
\ref{thm_VaRExample}, \ref{thm:canonical-strike-recursive}, \ref{thm_Gen2stateMarkovChain} exemplify the general theorems' solutions. 

The following definition is useful in what follows.
\begin{definition}[\emph{Dinkelbach difference}] $h_{\varphi,f,g}({\mathbf x}):= f({\mathbf x})-\varphi g({\mathbf x})$ for $\varphi\geqslant 0$, ${\mathbf x}\in K$
\label{def:Dinkelbachdifference}
\end{definition}

\subsection{Equivalent Discrete Formulations of \eqref{eqn_genCBIprob} and Its Generalisations}
\label{subsec_discreteformulationsandGeneralisations}
Propositions~\ref{prop_cbi_transform}, \ref{prop_genprob1_v1}, \ref{prop_cont_marginal_approx}, \ref{cor_refined_strip_oscillations}, \ref{prop_gen_multiple_marginals_piecewise_refined} simplify and extend \eqref{eqn_genCBIprob}. These extend Theorem~1 of \cite{moreno1991robust} (\emph{cf.} Theorem~3 of \cite{BetroRuggeriMeczarski1994-JSPI}).

%
%
%
%
%
%
%
%
%
\begin{figure}[ht!]
\centering

\subfloat[$\tilde K_1\subseteq K$]{%
\centering
\setlength{\fboxsep}{6pt}%
\colorbox{gray!35}{%
\begin{tikzpicture}[x=4.5cm,y=4.5cm]
  \clip (0,0) rectangle (1,1);
  \fill[white] (0,0) rectangle (1,1);

  \fill[gray!60]
    (0,0) -- (0.72,0)
    .. controls (0.78,0.12) and (0.70,0.24) ..
    (0.74,0.36)
    .. controls (0.79,0.50) and (0.66,0.60) ..
    (0.69,0.73)
    .. controls (0.72,0.86) and (0.62,0.94) ..
    (0.60,1.00)
    -- (0,1) -- cycle;

  \draw[white,line width=1.1pt] (0,0) rectangle (1,1);
  \draw[white,line width=1.1pt]
    (0.72,0)
    .. controls (0.78,0.12) and (0.70,0.24) ..
    (0.74,0.36)
    .. controls (0.79,0.50) and (0.66,0.60) ..
    (0.69,0.73)
    .. controls (0.72,0.86) and (0.62,0.94) ..
    (0.60,1.00);

  \node[font=\Large\bfseries, anchor=south east] at (0.98,0.98) {$K$};
  \node[font=\large] at (0.30,0.52) {$\boldsymbol{\tilde K_1}$};
\end{tikzpicture}}%
\label{fig:overlap_atoms_partition_a}}
\hspace{1cm}
\subfloat[$\tilde K_2\subseteq K$]{%
\centering
\setlength{\fboxsep}{6pt}%
\colorbox{gray!35}{%
\begin{tikzpicture}[x=4.5cm,y=4.5cm]
  \clip (0,0) rectangle (1,1);
  \fill[white] (0,0) rectangle (1,1);

  \fill[gray!60]
    (0,0) -- (1,0) -- (1,0.50)
    .. controls (0.88,0.42) and (0.72,0.58) ..
    (0.56,0.52)
    .. controls (0.38,0.45) and (0.20,0.67) ..
    (0,0.60)
    -- cycle;

  \draw[white,line width=1.1pt] (0,0) rectangle (1,1);
  \draw[white,line width=1.1pt]
    (0,0.60)
    .. controls (0.20,0.67) and (0.38,0.45) ..
    (0.56,0.52)
    .. controls (0.72,0.58) and (0.88,0.42) ..
    (1,0.50);

  \node[font=\Large\bfseries, anchor=south east] at (0.98,0.98) {$K$};
  \node[font=\large] at (0.50,0.22) {$\boldsymbol{\tilde K_2}$};
\end{tikzpicture}}%
\label{fig:overlap_atoms_partition_b}}

\vspace{1ex}

\subfloat[$\tilde K_3\subseteq K$]{%
\centering
\setlength{\fboxsep}{6pt}%
\colorbox{gray!35}{%
\begin{tikzpicture}[x=4.5cm,y=4.5cm]
  \clip (0,0) rectangle (1,1);
  \fill[white] (0,0) rectangle (1,1);

  \fill[gray!60]
    (0.30,0.62)
    .. controls (0.34,0.82) and (0.48,0.98) ..
    (0.56,1.00)
    -- (1.00,1.00)
    -- (1.00,0.34)
    .. controls (0.95,0.28) and (0.84,0.26) ..
    (0.74,0.30)
    .. controls (0.58,0.36) and (0.40,0.48) ..
    (0.30,0.62) -- cycle;

  \draw[white,line width=1.1pt] (0,0) rectangle (1,1);
  \draw[white,line width=1.1pt]
    (0.30,0.62)
    .. controls (0.34,0.82) and (0.48,0.98) ..
    (0.56,1.00)
    -- (1.00,1.00)
    -- (1.00,0.34)
    .. controls (0.95,0.28) and (0.84,0.26) ..
    (0.74,0.30)
    .. controls (0.58,0.36) and (0.40,0.48) ..
    (0.30,0.62) -- cycle;

  \node[font=\Large\bfseries, anchor=south east] at (0.98,0.98) {$K$};
  \node[font=\large] at (0.71,0.67) {$\boldsymbol{\tilde K_3}$};
\end{tikzpicture}}%
\label{fig:overlap_atoms_partition_c}}
\hspace{1cm}
\subfloat[Induced partition $\{A_I:I\in\mathcal I\}$]{%
\centering
\setlength{\fboxsep}{6pt}%
\colorbox{gray!35}{%
\begin{tikzpicture}[x=4.5cm,y=4.5cm]
  \clip (0,0) rectangle (1,1);
  \fill[gray!60] (0,0) rectangle (1,1);

  \draw[white,line width=1.1pt]
    (0.72,0)
    .. controls (0.78,0.12) and (0.70,0.24) ..
    (0.74,0.36)
    .. controls (0.79,0.50) and (0.66,0.60) ..
    (0.69,0.73)
    .. controls (0.72,0.86) and (0.62,0.94) ..
    (0.60,1.00);

  \draw[white,line width=1.1pt]
    (0,0.60)
    .. controls (0.20,0.67) and (0.38,0.45) ..
    (0.56,0.52)
    .. controls (0.72,0.58) and (0.88,0.42) ..
    (1,0.50);

  \draw[white,line width=1.1pt]
    (0.30,0.62)
    .. controls (0.34,0.82) and (0.48,0.98) ..
    (0.56,1.00)
    -- (1.00,1.00)
    -- (1.00,0.34)
    .. controls (0.95,0.28) and (0.84,0.26) ..
    (0.74,0.30)
    .. controls (0.58,0.36) and (0.40,0.48) ..
    (0.30,0.62) -- cycle;

  \draw[white,line width=1.1pt] (0,0) rectangle (1,1);

  \node[font=\Large\bfseries, anchor=south east] at (0.98,0.98) {$K$};

  \node[font=\scriptsize] at (0.16,0.82) {$A_{\{1\}}$};
  \node[font=\scriptsize] at (0.89,0.14) {$A_{\{2\}}$};
  \node[font=\scriptsize] at (0.85,0.76) {$A_{\{3\}}$};
  \node[font=\scriptsize] at (0.31,0.28) {$A_{\{1,2\}}$};
  \node[font=\scriptsize] at (0.53,0.72) {$A_{\{1,3\}}$};
  \node[font=\scriptsize] at (0.88,0.4) {$A_{\{2,3\}}$};
  \node[font=\scriptsize] at (0.64,0.44) {$A_{\{1,2,3\}}$};
\end{tikzpicture}}%
\label{fig:overlap_atoms_partition_d}}

\caption[Atomisation induced by overlapping measurable-set constraints]{Illustration of Proposition~\ref{prop_genprob1_v1} on $K=[0,1]^2$. Panels (a)–(c) show three overlapping measurable sets $\tilde K_1,\tilde K_2,\tilde K_3\subseteq K$, chosen so that $\tilde K_1\cup\tilde K_2\cup\tilde K_3=K$. Panel (d) shows the nonempty atoms $A_I=\Bigl(\bigcap_{i\in I}\tilde K_i\Bigr)\cap\Bigl(\bigcap_{i\notin I}\tilde K_i^c\Bigr)$ which partition $K$. Boundary segments are assigned according to a fixed Borel, or half-open, convention.}
\label{fig:overlap_atoms_partition}
\end{figure}
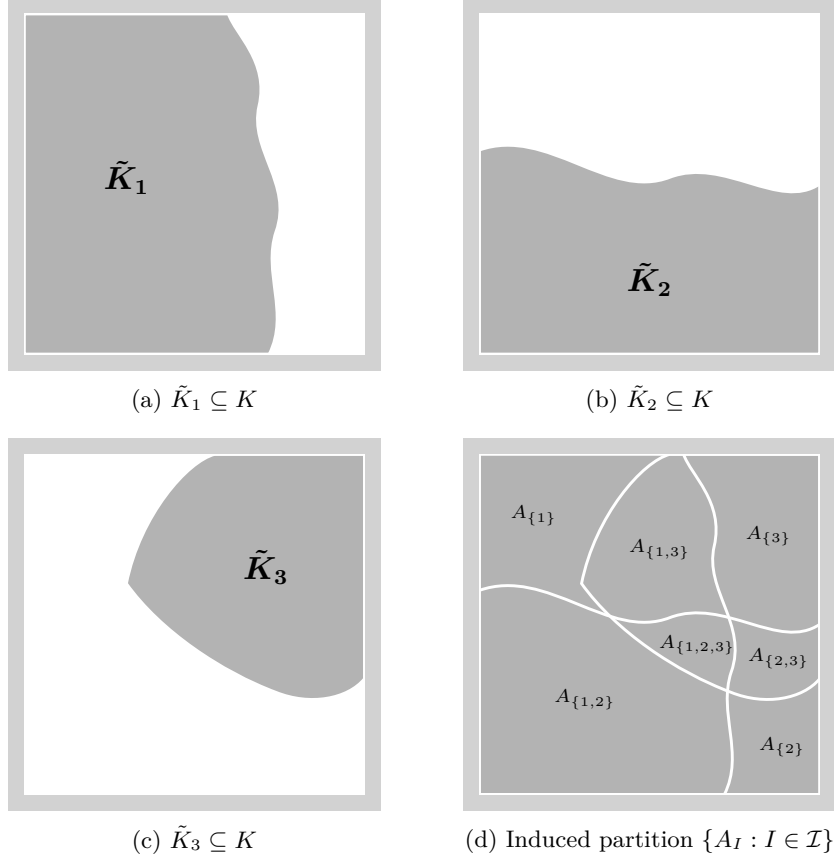
Proposition~\ref{prop_cbi_transform} considers the following extension of Problem~\eqref{eqn_genCBIprob}. 
Let
$\boldsymbol\psi=(\psi_1,\ldots,\psi_q):K\to\mathbb R^q$ be Borel-measurable, and let
$\boldsymbol m=(m_1,\ldots,m_q)\in\mathbb R^q$. Define
\[
\mathcal D_{p,\boldsymbol m}
:=
\Bigl\{
\mathbb P\in\mathcal P(K):
\mathbb P(K_i)=p_i,\ i=1,\ldots,n,\quad
\mathbb E_{\mathbb P}[\psi_j({\mathbf X})]=m_j,\ j=1,\ldots,q
\Bigr\}.
\]
Let $f,g:K\to[0,\infty)$ be $\mathcal D_{p,\boldsymbol m}$-integrable; that is,
$f,g\in\mathcal L(\mathbb P)$ for every $\mathbb P\in\mathcal D_{p,\boldsymbol m}$.
Further assume $\inf_{{\mathbf x}\in K_j} g({\mathbf x})>0$ for some $j\in\{1,\ldots,n\}$;
so $0<
\mathbb E_{\mathbb P}[g({\mathbf X})]$ for all $\mathbb P\in\mathcal D_{p,\boldsymbol m}$. For $r\in\mathbb N$, define
\[
\mathcal D_{p,\boldsymbol m}^{(r)}
:=
\Bigl\{
\mathbb P\in\mathcal D_{p,\boldsymbol m}:
\mathbb P=\sum_{\ell=1}^r w_\ell\delta_{{\mathbf x}_\ell},\
{\mathbf x}_\ell\in K,\ w_\ell\geqslant 0,\
\sum_{\ell=1}^r w_\ell=1
\Bigr\},
\]
where zero-weight terms in the representation are ignored.

\begin{proposition}[finite-support reduction of \eqref{eqn_genCBIprob}]\label{prop_cbi_transform}
Assume $\mathcal D_{p,\boldsymbol m}\neq\varnothing$. Then
\begin{equation}
\inf_{\mathbb P\in\mathcal D_{p,\boldsymbol m}}
\frac{\mathbb E_{\mathbb P}[f({\mathbf X})]}
{\mathbb E_{\mathbb P}[g({\mathbf X})]}
=
\inf_{\mathbb P\in\mathcal D_{p,\boldsymbol m}^{(n+q)}}
\frac{\mathbb E_{\mathbb P}[f({\mathbf X})]}
{\mathbb E_{\mathbb P}[g({\mathbf X})]}.
\label{eqn_weak_opt_equiv_momentgen_disc}
\end{equation}
Indeed, for every $\mathbb P\in\mathcal D_{p,\boldsymbol m}$ there exist $\mathbb Q_0\in\mathcal D_{p,\boldsymbol m}^{(n+q+1)}$ and $\mathbb Q\in\mathcal D_{p,\boldsymbol m}^{(n+q)}$ such that
\begin{equation}
\frac{\mathbb E_{\mathbb P}[f({\mathbf X})]}{\mathbb E_{\mathbb P}[g({\mathbf X})]}=\frac{\mathbb E_{\mathbb Q_0}[f({\mathbf X})]}
{\mathbb E_{\mathbb Q_0}[g({\mathbf X})]}\geqslant\frac{\mathbb E_{\mathbb Q}[f({\mathbf X})]}
{\mathbb E_{\mathbb Q}[g({\mathbf X})]}.
\label{eqn_feasiblepriorequivanddomination}
\end{equation}
In particular, by considering no moment constraints (so $q=0$), \eqref{eqn_weak_opt_equiv_momentgen_disc} implies Problem~\eqref{eqn_genCBIprob} has an equivalent discrete formulation. That is, 
\begin{equation}
\inf_{\mathbb P \in \mathcal D}
\frac{\mathbb E_{\mathbb P }[f({\mathbf X})]}{\mathbb E_{\mathbb P }[g({\mathbf X})]}=\inf_{K_1\times,\ldots,\times K_n}
\frac{\sum_{i=1}^n f({\mathbf x}_i)p_i}{\sum_{i=1}^n g({\mathbf x}_i)p_i}
\label{eqn_weak_opt_equiv_disc}
\end{equation}
Furthermore, for continuous $f,g$,
\begin{equation*}
\inf_{K_1\times,\ldots,\times K_n}
\frac{\sum_{i=1}^n f({\mathbf x}_i)p_i}{\sum_{i=1}^n g({\mathbf x}_i)p_i}
=
\min_{\overline K_1\times,\ldots,\times \overline K_n}
\frac{\sum_{i=1}^n f({\mathbf x}_i)p_i}{\sum_{i=1}^n g({\mathbf x}_i)p_i}.
\end{equation*}

\end{proposition}

\begin{proof} See Appendix~A, Supplementary Material (\cite{Salako_Muhammad_2025_suppmat}).
\end{proof}

\noindent\textbf{Remarks: }If $q=0$, the support bound reduces to $n$ from $n+q$. Since $p_i>0$ for every
$i=1,\ldots,n$, any $\mathbb P\in \mathcal D_{p,\boldsymbol m}^{(n)}$ must place exactly one support point in each partition element $K_i$. The $q$ moment constraints add at most $q$
additional support points in total.

Proposition~\ref{prop_cbi_transform} is a finite-support reduction for fractional objectives under fixed partition probabilities and finitely many moment equalities. Its proof uses a finite-dimensional barycentric representation followed by a rank-pruning argument. When no moment constraints are present, it recovers the finite-partition reduction of Moreno et al.~\cite[Theorem~1]{moreno1991robust}. More generally, it is a structured equality-constraint instance of the generalized moment framework of Betr\`o et al.~\cite[Theorem~3]{BetroRuggeriMeczarski1994-JSPI}, itself based on Winkler's extreme-point theorem \cite[Theorem 2.1]{Winkler1988ExtremePointsMomentSets}. 
The present formulation differs by: \begin{enumerate*}[label=\textbf{\roman*)}] \item exploiting the affine dependence among the $n$ partition indicators, while incorporating the $q$ moment constraints, to obtain the improved $n+q$ support bound; \item proving the stronger pointwise objective-value equivalence/domination statement \eqref{eqn_feasiblepriorequivanddomination}. \end{enumerate*}

The two parts of \eqref{eqn_feasiblepriorequivanddomination} play different roles. The equality involving $\mathbb Q_0$ is a finite-moment representation statement: after fixing $\mathbb P$, one may add one further moment constraint encoding the objective value attained by $\mathbb P$. This gives the $(n+q+1)$-atomic representative $\mathbb Q_0$ with the same objective value, by Richter's theorem \cite[Satz~4]{Richter1957Parameterfreie}. However, this equality-preserving argument alone does not yield the sharper $(n+q)$-atomic domination statement involving $\mathbb Q$: preserving the additional objective-value constraint costs one atom, whereas omitting it leaves the direction of the objective change uncontrolled. The proposition's proof resolves this obstruction by a rank-pruning step: it preserves only the original partition and moment constraints while pruning one atom in a direction that does not increase the objective value, thereby producing $\mathbb Q$.

Proposition~\ref{prop_genprob1_v1} shows that the partition in \eqref{eqn_genCBIprob} is not as restrictive as it first appears. It is valid when finite constraints ``overlap''.

\begin{proposition}[Overlapping measurable-set constraints]
\label{prop_genprob1_v1}
Let $K:=[0,1]^d$ and let $\tilde{K}_1,\allowbreak \ldots,\allowbreak \tilde{K}_n$ be Borel-measurable subsets of $K$
such that $\bigcup_{i=1}^n \tilde K_i = K$.
Let $\widetilde{\mathcal D}
:=
\Bigl\{
\mathbb P\in \mathcal P(K): \mathbb P(\tilde K_i)=\tilde p_i,\ i=1,\ldots,n
\Bigr\}$ where $\tilde p_i>0$ $(i=1,\ldots,n)$, and assume $\widetilde{\mathcal D}\neq\varnothing$.
Let $f,g:K\to [0,\infty)$ be $\widetilde{\mathcal D}$-integrable and $\inf_{{\mathbf x}\in \tilde K_j} g({\mathbf x})>0$ for some
$j$.

For each nonempty $I\subseteq\{1,\ldots,n\}$, define the atom $A_I
:=
\Bigl(\bigcap_{i\in I}\tilde K_i\Bigr)
\cap
\Bigl(\bigcap_{i\notin I}\tilde K_i^c\Bigr)$, and let $\mathcal I:=\{I\subseteq\{1,\ldots,n\}: I\neq\varnothing,\ A_I\neq\varnothing\}$.
Then $\{A_I: I\in\mathcal I\}$ is a partition of $K$, and
\begin{equation}
\inf_{\mathbb P\in \widetilde{\mathcal D}}
\frac{\mathbb E_{\mathbb P }[f({\mathbf X})]}{\mathbb E_{\mathbb P}[g({\mathbf X})]}=\inf_{\substack{p_I\geqslant 0\text{ for all }I\in\mathcal I,\\
\sum_{I\in\mathcal I} p_I = 1,\\
\sum_{I\in\mathcal I:\, i\in I} p_I = \tilde p_i,\ i=1,\ldots,n}}
\!\!\!\!
\inf_{\substack{{\mathbf x}_I\in A_I\text{ for }I\in\mathcal I\\\text{such that }p_I>0}}
\frac{\sum_{I\in\mathcal I} f({\mathbf x}_I)\, p_I}
     {\sum_{I\in\mathcal I} g({\mathbf x}_I)\, p_I}.
\label{eq:overlap_refinement}
\end{equation}


In particular, any extremal prior may be taken to be discrete, of the form
\[
\mathbb P^*=\sum_{I\in\mathcal I} p_I^*\,\delta_{{\mathbf x}_I^*},
\qquad {\mathbf x}_I^*\in \overline{A}_I,
\]
for some feasible $\{p_I^*\}_{I\in\mathcal I}$ with at most $n+1$ non-zero $p_I^*$.
If, moreover, $\tilde K_1,\ldots,\tilde K_n$ are pairwise disjoint, then
$\mathcal I=\bigl\{\{1\},\ldots,\{n\}\bigr\}$, $A_{\{i\}}=\tilde K_i$, $p_{\{i\}}=\tilde p_i$,
and \eqref{eq:overlap_refinement} becomes \eqref{eqn_weak_opt_equiv_disc}.
\end{proposition}
\begin{proof}
See Appendix~B, Supplementary Material (\cite{Salako_Muhammad_2025_suppmat}).
\end{proof}

Propositions~\ref{prop_cont_marginal_approx}, \ref{cor_refined_strip_oscillations}, \ref{prop_gen_multiple_marginals_piecewise_refined} allow for marginal-density constraints on the prior.
\begin{proposition}[Marginal-density extension of \eqref{eqn_genCBIprob}]
\label{prop_cont_marginal_approx}
Let $K:=U\times V=[0,1]^r\times [0,1]^s$ for $r,s\in\mathbb N$, where $r+s=d$.
Let $\rho:U\to[0,\infty)$ be Borel-measurable with $\int_U \rho({\mathbf u})\,d{\mathbf u}=1$, and define
\[
\mathcal D:=
\Bigl\{
\mathbb P\in\mathcal P(K):
\mathbb P(A\times V)=\int_A \rho({\mathbf u})\,d{\mathbf u}
\text{ for every Borel-measurable }A\subseteq U
\Bigr\}.
\]
Let $f,g:K\to[0,\infty)$ be bounded Borel functions, and assume $c:=\inf_{{\mathbf x}\in K} g({\mathbf x})>0$. For each $m\in\mathbb N$, let $B_1^{(m)},\ldots,B_{n_m}^{(m)}$ be a partition of $U$
into Borel-measurable sets, and write $K_i^{(m)}:=B_i^{(m)}\times V$, $p_i^{(m)}:=\int_{B_i^{(m)}} \rho({\mathbf u})\,d{\mathbf u}$, for $i=1,\ldots,n_m$.
Assume that the partitions are dominant for $(f,g)$ in the sense that
\[
\omega_f(m):=
\max_{1\leqslant i\leqslant n_m}
\sup\Bigl\{
|f({\mathbf u},{\mathbf v})-f({\mathbf u}',{\mathbf v})|:
{\mathbf u},{\mathbf u}'\in B_i^{(m)},\ ({\mathbf u},{\mathbf v}),({\mathbf u}',{\mathbf v})\in K
\Bigr\}\to 0,
\]
\[
\omega_g(m):=
\max_{1\leqslant i\leqslant n_m}
\sup\Bigl\{
|g({\mathbf u},{\mathbf v})-g({\mathbf u}',{\mathbf v})|:
{\mathbf u},{\mathbf u}'\in B_i^{(m)},\ ({\mathbf u},{\mathbf v}),({\mathbf u}',{\mathbf v})\in K
\Bigr\}\to 0
\]
as $m\to\infty$.
Lastly, define
\[
 {{\mathcal D}^{(m)}}:=
\Bigl\{
\mathbb P\in\mathcal P(K):
\mathbb P\bigl(K_i^{(m)}\bigr)=p_i^{(m)}\text{ for }i=1,\ldots,n_m
\Bigr\}.
\]
Then, each problem $\phi_m^*:=\inf\limits_{ {{\mathcal D}^{(m)}}}\frac{\mathbb E[f({\mathbf X})]}{\mathbb E[g({\mathbf X})]}$ has an equivalent form that is solved by Proposition~\ref{prop_cbi_transform}. Moreover, if $\phi^*:=\inf\limits_{\mathcal D}\frac{\mathbb E[f({\mathbf X})]}{\mathbb E[g({\mathbf X})]}$, then $\phi_m^*\to \phi^*$ as $m\to\infty$. More precisely, if $M_f\allowbreak:=\allowbreak\sup\limits_{{\mathbf x}\in K} f({\mathbf x})$ and $\eta_m:=
\frac{2\,\omega_f(m)}{c}
+
\frac{2M_f\,\omega_g(m)}{c^2}$, then $\phi^*-\eta_m\leqslant \phi_m^*\leqslant \phi^*+\eta_m$ for every $m\in\mathbb N$.
\end{proposition}
\begin{proof} See Appendix~C, Supplementary Material (\cite{Salako_Muhammad_2025_suppmat}).
\end{proof}
Proposition~\ref{cor_refined_strip_oscillations} is valid when finite and marginal-density constraints ``overlap''.
\begin{proposition}[Overlapping finite and marginal-density constraints]
\label{cor_refined_strip_oscillations}
Let $K := U \times V = [0,1]^r \times [0,1]^s$, let
$\rho : U \to [0,\infty)$ be Borel-measurable with $\int_U \rho(\textnormal{\textbf{u}})\,d\textnormal{\textbf{u}} = 1$, and let
$F_1,\ldots,F_L$ be a fixed finite Borel partition of $K$. Let
$q_1,\ldots,q_L \geqslant 0$ satisfy $\sum_{\ell=1}^L q_\ell = 1$, and define
\[
{\mathcal D_{F,q}}
:=
\left\{
\mathbb P \in \mathcal P(K) :
\begin{array}{l}
\mathbb P(A \times V) = \int_A \rho(\textnormal{\textbf{u}})\,d\textnormal{\textbf{u}} \ \text{for every Borel } A \subseteq U,\\[3pt] 
\mathbb P(F_\ell)=q_\ell,\ \ell=1,\ldots,L
\end{array}
\right\}.
\]
Assume ${\mathcal D_{F,q}}\neq\varnothing$. Let $f,g : K \to [0,\infty)$ be bounded Borel functions, and assume $c := \inf_{\textnormal{\textbf{x}} \in K} g(\textnormal{\textbf{x}}) > 0$. For each $m \in \mathbb N$, let
$B_1^{(m)},\ldots,B_{n_m}^{(m)}$ be a finite Borel partition of $U$. Define $K_i^{(m)} := B_i^{(m)} \times V\,$ and 
$\,p_i^{(m)} := \int_{B_i^{(m)}} \rho(\textnormal{\textbf{u}})\,d\textnormal{\textbf{u}},\,$ for $\,i=1,\ldots,n_m$. Let
\[
{\mathcal D_{F,q}}^{(m)}
:=
\Bigl\{
\mathbb P \in \mathcal P(K) :
\mathbb P\bigl(K_i^{(m)}\bigr)=p_i^{(m)},\ i=1,\ldots,n_m,\ 
\mathbb P(F_\ell)=q_\ell,\ \ell=1,\ldots,L
\Bigr\}.
\]
For each $m$, let $R_1^{(m)},\ldots,R_{L_m}^{(m)}$ be a finite Borel partition of $K$ refining both
$\{K_i^{(m)}\}_{i=1}^{n_m}$ and $\{F_\ell\}_{\ell=1}^L$; that is, each
$R_\tau^{(m)}$ is contained in some $K_i^{(m)}$ and in some $F_\ell$.
Define
\[
\widehat{\omega}_f(m) := \max_{1 \leqslant \tau \leqslant L_m}
\sup_{\textnormal{\textbf{x}},\textnormal{\textbf{y}} \in R_\tau^{(m)}} |f(\textnormal{\textbf{x}})-f(\textnormal{\textbf{y}})|,
\qquad
\widehat{\omega}_g(m) := \max_{1 \leqslant \tau \leqslant L_m}
\sup_{\textnormal{\textbf{x}},\textnormal{\textbf{y}} \in R_\tau^{(m)}} |g(\textnormal{\textbf{x}})-g(\textnormal{\textbf{y}})|.
\]
Assume $\widehat{\omega}_f(m) \to 0$ and $\widehat{\omega}_g(m) \to 0$
as $m \to \infty$. For each $m$, define the refined consistency polytope
\begin{align*}
{\mathcal W_{F,q,m}}
:=\left\{
\boldsymbol{\alpha}=(\alpha_\tau)_{\tau=1}^{L_m} \in [0,\infty)^{L_m} \;:\;\begin{array}{l}
\sum_{\tau:\,R_\tau^{(m)} \subseteq K_i^{(m)}} \alpha_\tau = p_i^{(m)},\ i=1,\ldots,n_m,\\[2pt]
\sum_{\tau:\,R_\tau^{(m)} \subseteq F_\ell} \alpha_\tau = q_\ell,\ \ell=1,\ldots,L
\end{array}
\right\}.
\end{align*}
So, ${\mathcal W_{F,q,m}}\neq\varnothing$ for every $m \in \mathbb N$. For each $m$ and each $\tau \in \{1,\ldots,L_m\}$, let $i(\tau)$ denote the unique
index such that $R_\tau^{(m)} \subseteq K_{i(\tau)}^{(m)}$. The vertical section $R_{\tau,\textnormal{\textbf{u}}}^{(m)} := \{\, \textnormal{\textbf{v}} \in V : (\textnormal{\textbf{u}},\textnormal{\textbf{v}}) \in R_\tau^{(m)} \,\}$ is defined for $\textnormal{\textbf{u}} \in B_{i(\tau)}^{(m)}$. Let $\mathcal T_m^+
:=
\Bigl\{
\tau \in \{1,\ldots,L_m\} :
\exists \boldsymbol{\alpha} \in {\mathcal W_{F,q,m}}\ \text{with}\ \alpha_\tau>0
\Bigr\}$ be the set of active refined atoms and assume that, for each $m \in \mathbb N$ and $\tau \in \mathcal T_m^+$,
there exists measurable probability kernel $\textnormal{\textbf{u}} \mapsto \nu_{\tau,\textnormal{\textbf{u}}}^{(m)} \in \mathcal P(V)$ (where $\textnormal{\textbf{u}} \in B_{i(\tau)}^{(m)}$),
such that $\nu_{\tau,\textnormal{\textbf{u}}}^{(m)}\!\bigl(R_{\tau,\textnormal{\textbf{u}}}^{(m)}\bigr)=1$ for $\rho$-a.e. $\textnormal{\textbf{u}} \in B_{i(\tau)}^{(m)}$. Finally, let $V_{F,q,m}(\boldsymbol{\alpha})
:=
\inf\limits_{\textnormal{\textbf{x}}_\tau \in R_\tau^{(m)} \text{ for } \alpha_\tau>0}
\frac{\sum_{\tau=1}^{L_m} \alpha_\tau f(\textnormal{\textbf{x}}_\tau)}
{\sum_{\tau=1}^{L_m} \alpha_\tau g(\textnormal{\textbf{x}}_\tau)}$ for $\boldsymbol{\alpha} \in {\mathcal W_{F,q,m}}$, let $\phi_{F,q}^*
:=
\inf\limits_{\mathbb P \in {\mathcal D_{F,q}}} \frac{\mathbb E_{\mathbb P}[f(\textnormal{\textbf{X}})]}{\mathbb E_{\mathbb P}[g(\textnormal{\textbf{X}})]}$, and let 
$\phi_{F,q,m}
:=
\inf\limits_{\boldsymbol{\alpha} \in {\mathcal W_{F,q,m}}} V_{F,q,m}(\boldsymbol{\alpha})$. 

\noindent Then:
\begin{enumerate}
\item
For each $m$, $I_m^+ := \{\, i \in \{1,\ldots,n_m\} : p_i^{(m)} > 0 \,\}$ and $J^+ := \{\, \ell \in \{1,\ldots,L\} : q_\ell > 0 \,\}$, every $\mathbb P \in {\mathcal D_{F,q}}^{(m)}$ satisfies
\[
\mathbb P\bigl(K_i^{(m)}\bigr)=0 \ \text{for all } i\notin I_m^+,
\qquad
\mathbb P(F_\ell)=0 \ \text{for all } \ell\notin J^+,
\]
and ${\mathcal D_{F,q}}^{(m)}
=
\Bigl\{
\mathbb P \in \mathcal P(K) :
\mathbb P\bigl(K_i^{(m)}\bigr)=p_i^{(m)},\ i\in I_m^+,\ 
\mathbb P(F_\ell)=q_\ell,\ \ell\in J^+
\Bigr\}$. Moreover,
\[
\phi_{F,q,m}
=
\inf_{\mathbb P \in {\mathcal D_{F,q}}^{(m)}} \frac{\mathbb E_{\mathbb P}[f(\textnormal{\textbf{X}})]}{\mathbb E_{\mathbb P}[g(\textnormal{\textbf{X}})]}.
\]

\item
For $M_f := \sup\limits_{\mathbf{x} \in K} f(\textnormal{\textbf{x}})$ and $\eta_{F,q,m}
:=
\frac{\widehat{\omega}_f(m)}{c}
+
\frac{M_f\,\widehat{\omega}_g(m)}{c^2}$,
$\phi_{F,q}^* - \eta_{F,q,m}
\leqslant
\phi_{F,q,m}
\leqslant
\phi_{F,q}^*$ holds for every $m \in \mathbb N$. Hence, $\phi_{F,q,m} \to \phi_{F,q}^*
\text{ as } m \to \infty$.
\end{enumerate}
\end{proposition}

\begin{proof}
See Appendix~D, Supplementary Material (\cite{Salako_Muhammad_2025_suppmat}).
\end{proof}
\begin{figure}[!ht]
\centering

\subfloat[Coarse partition at level $m$\label{fig:multiple_marginal_partition_a}]{%
\setlength{\fboxsep}{6pt}%
\colorbox{gray!35}{%
\begin{tikzpicture}[x=4.1cm,y=4.1cm]

  \def\gap{0.08}

  \fill[gray!35] (-0.26,-0.26) rectangle (1,1);

  \fill[white] (0,0) rectangle (1,1);

  \fill[gray!60]
    (0,-\gap)
    -- plot[smooth] coordinates {
      (0.00,-\gap-0.02) (0.08,-\gap-0.06) (0.18,-\gap-0.15)
      (0.30,-\gap-0.05) (0.42,-\gap-0.03) (0.54,-\gap-0.11)
      (0.68,-\gap-0.16) (0.82,-\gap-0.05) (1.00,-\gap-0.08)
    }
    -- (1,-\gap) -- cycle;

  \fill[gray!60]
    (-\gap,0)
    -- plot[smooth] coordinates {
      (-\gap-0.01,0.00) (-\gap-0.03,0.10) (-\gap-0.06,0.22)
      (-\gap-0.10,0.36) (-\gap-0.15,0.52) (-\gap-0.13,0.68)
      (-\gap-0.08,0.84) (-\gap-0.02,1.00)
    }
    -- (-\gap,1) -- cycle;

  \foreach \x in {0.26,0.60,0.82}{
    \draw[black,line width=0.7pt] (\x,0) -- (\x,1);
  }
  \foreach \y in {0.22,0.54,0.80}{
    \draw[black,line width=0.7pt] (0,\y) -- (1,\y);
  }

  \begin{scope}
    \clip
      (0,-\gap)
      -- plot[smooth] coordinates {
        (0.00,-\gap-0.02) (0.08,-\gap-0.06) (0.18,-\gap-0.15)
        (0.30,-\gap-0.05) (0.42,-\gap-0.03) (0.54,-\gap-0.11)
        (0.68,-\gap-0.16) (0.82,-\gap-0.05) (1.00,-\gap-0.08)
      }
      -- (1,-\gap) -- cycle;
    \foreach \x in {0.26,0.60,0.82}{
      \draw[black,line width=0.7pt] (\x,-0.30) -- (\x,0);
    }
  \end{scope}

  \begin{scope}
    \clip
      (-\gap,0)
      -- plot[smooth] coordinates {
        (-\gap-0.01,0.00) (-\gap-0.03,0.10) (-\gap-0.06,0.22)
        (-\gap-0.10,0.36) (-\gap-0.15,0.52) (-\gap-0.13,0.68)
        (-\gap-0.08,0.84) (-\gap-0.02,1.00)
      }
      -- (-\gap,1) -- cycle;
    \foreach \y in {0.22,0.54,0.80}{
      \draw[black,line width=0.7pt] (-0.30,\y) -- (0,\y);
    }
  \end{scope}

  \node[font=\large] at (0.50,1.06) {$\boldsymbol{K=U_1\times U_2}$};

  \node[font=\normalsize] at (0.50,-0.28) {$\rho_1(\mathbf u_1)$};
  \node[font=\normalsize, rotate=90] at (-0.28,0.50) {$\rho_2(\mathbf u_2)$};

  \node[font=\small] at (0.13,-0.02) {$B^{(m)}_{1,1}$};
  \node[font=\small] at (0.43,-0.02) {$B^{(m)}_{1,2}$};
  \node[font=\small] at (0.71,-0.02) {$B^{(m)}_{1,3}$};
  \node[font=\small] at (0.91,-0.02) {$B^{(m)}_{1,4}$};

  \node[font=\small, rotate=90] at (0.02,0.11) {$B^{(m)}_{2,1}$};
  \node[font=\small, rotate=90] at (0.02,0.38) {$B^{(m)}_{2,2}$};
  \node[font=\small, rotate=90] at (0.02,0.67) {$B^{(m)}_{2,3}$};
  \node[font=\small, rotate=90] at (0.02,0.90) {$B^{(m)}_{2,4}$};

\end{tikzpicture}}}
\hfill
\subfloat[Refined partition at level $m'>m$\label{fig:multiple_marginal_partition_b}]{%
\setlength{\fboxsep}{6pt}%
\colorbox{gray!35}{%
\begin{tikzpicture}[x=4.1cm,y=4.1cm]

  \def\gap{0.08}

  \fill[gray!35] (-0.26,-0.26) rectangle (1,1);

  \fill[white] (0,0) rectangle (1,1);

  \fill[gray!60]
    (0,-\gap)
    -- plot[smooth] coordinates {
      (0.00,-\gap-0.02) (0.08,-\gap-0.06) (0.18,-\gap-0.15)
      (0.30,-\gap-0.05) (0.42,-\gap-0.03) (0.54,-\gap-0.11)
      (0.68,-\gap-0.16) (0.82,-\gap-0.05) (1.00,-\gap-0.08)
    }
    -- (1,-\gap) -- cycle;

  \fill[gray!60]
    (-\gap,0)
    -- plot[smooth] coordinates {
      (-\gap-0.01,0.00) (-\gap-0.03,0.10) (-\gap-0.06,0.22)
      (-\gap-0.10,0.36) (-\gap-0.15,0.52) (-\gap-0.13,0.68)
      (-\gap-0.08,0.84) (-\gap-0.02,1.00)
    }
    -- (-\gap,1) -- cycle;

  \foreach \x in {0.26,0.60,0.82}{
    \draw[black,line width=0.85pt] (\x,0) -- (\x,1);
  }
  \foreach \y in {0.22,0.54,0.80}{
    \draw[black,line width=0.85pt] (0,\y) -- (1,\y);
  }

  \foreach \x in {0.12,0.40,0.50,0.71,0.91}{
    \draw[black,line width=0.45pt] (\x,0) -- (\x,1);
  }
  \foreach \y in {0.10,0.34,0.44,0.67,0.91}{
    \draw[black,line width=0.45pt] (0,\y) -- (1,\y);
  }

  \begin{scope}
    \clip
      (0,-\gap)
      -- plot[smooth] coordinates {
        (0.00,-\gap-0.02) (0.08,-\gap-0.06) (0.18,-\gap-0.15)
        (0.30,-\gap-0.05) (0.42,-\gap-0.03) (0.54,-\gap-0.11)
        (0.68,-\gap-0.16) (0.82,-\gap-0.05) (1.00,-\gap-0.08)
      }
      -- (1,-\gap) -- cycle;
    \foreach \x in {0.26,0.60,0.82}{
      \draw[black,line width=0.85pt] (\x,-0.30) -- (\x,0);
    }
    \foreach \x in {0.12,0.40,0.50,0.71,0.91}{
      \draw[black,line width=0.45pt] (\x,-0.30) -- (\x,0);
    }
  \end{scope}

  \begin{scope}
    \clip
      (-\gap,0)
      -- plot[smooth] coordinates {
        (-\gap-0.01,0.00) (-\gap-0.03,0.10) (-\gap-0.06,0.22)
        (-\gap-0.10,0.36) (-\gap-0.15,0.52) (-\gap-0.13,0.68)
        (-\gap-0.08,0.84) (-\gap-0.02,1.00)
      }
      -- (-\gap,1) -- cycle;
    \foreach \y in {0.22,0.54,0.80}{
      \draw[black,line width=0.85pt] (-0.30,\y) -- (0,\y);
    }
    \foreach \y in {0.10,0.34,0.44,0.67,0.91}{
      \draw[black,line width=0.45pt] (-0.30,\y) -- (0,\y);
    }
  \end{scope}

  \node[font=\large] at (0.50,1.06) {$\boldsymbol{K=U_1\times U_2}$};

  \node[font=\normalsize] at (0.50,-0.28) {$\rho_1(\mathbf u_1)$};
  \node[font=\normalsize, rotate=90] at (-0.28,0.50) {$\rho_2(\mathbf u_2)$};

\end{tikzpicture}}}

\caption[Product-cell refinement under prescribed marginal densities]{Illustration of Proposition~\ref{prop_gen_multiple_marginals_piecewise_refined}
for the case $K=U_1\times U_2$. Each panel shows the product domain together
with continuous marginal densities $\rho_1$ and $\rho_2$. The black grid lines
indicate partitions $\{B^{(m)}_{1,i}\}$ of $U_1$ and $\{B^{(m)}_{2,i}\}$ of
$U_2$, which induce the finitely constrained approximation problems
$ {{\mathcal D}^{(m)}}$ and the corresponding product-cell atomisation. In panel (b),
the refined grid illustrates a later approximation level $m'>m$, with smaller
cell diameters and hence a finer product-cell approximation to the original
multiple-marginal problem.}
\label{fig:multiple_marginal_approximation}
\end{figure}
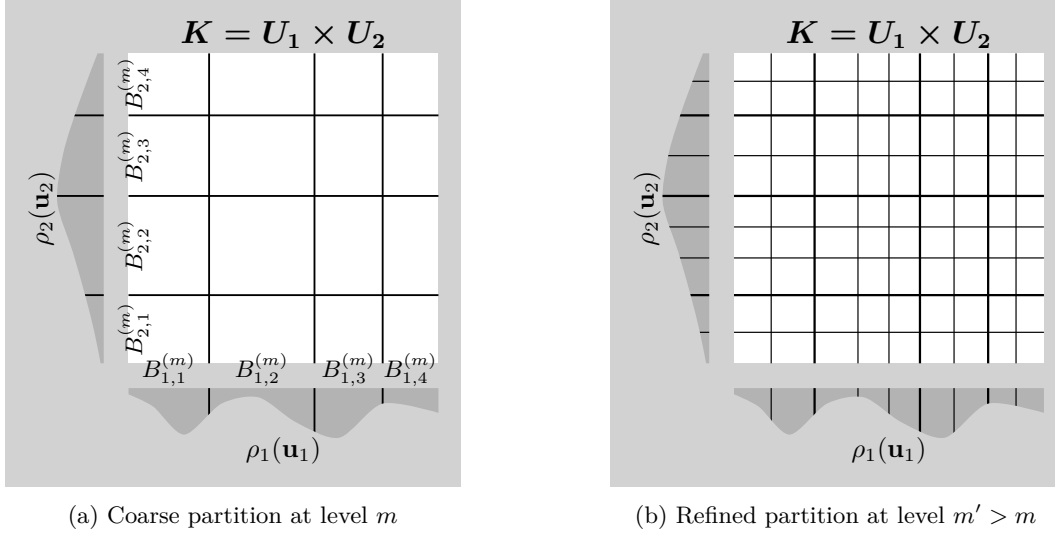
\begin{proposition}[Multiple marginal-densities]
\label{prop_gen_multiple_marginals_piecewise_refined}
Let $K:=U_1\times\cdots\times U_M
=[0,1]^{d_1}\times\cdots\times[0,1]^{d_M}$, where $d_1,\dots,d_M\in\mathbb N$ and $d_1+\cdots+d_M=d$. For each $j=1,\ldots,M$, let $\rho_j:U_j\to[0,\infty)$ be Borel-measurable with $\int_{U_j}\rho_j({\mathbf u}_j)\,d{\mathbf u}_j=1$, define $\mu_j\in\mathcal P(U_j)$ by $\mu_j(A):=\int_A \rho_j({\mathbf u}_j)\,d{\mathbf u}_j$ for all $A\in\mathcal B(U_j)$,
and let
\begin{align*}
&\mathcal D:=\\[3pt]
&\Bigl\{
\mathbb P\in\mathcal P(K):
\mathbb P(U_1\times\cdots\times A_j\times\cdots\times U_M)=\mu_j(A_j)
\text{ for all }A_j\in\mathcal B(U_j),\ j=1,\ldots,M
\Bigr\}.
\end{align*}
Let $f,g:K\to[0,\infty)$ be bounded Borel functions, and assume $c:=\inf_{\mathbf x\in K}g(\mathbf x)>0$. For each $m\in\mathbb N$ and each $j=1,\dots,M$, let $B^{(m)}_{j,1},\dots,B^{(m)}_{j,n_j(m)}$ be a finite Borel partition of $U_j$. For each $j=1,\dots,M$ and $r=1,\dots,n_j(m)$, define $K^{(m)}_{j,r}:=
U_1\times\cdots\times B^{(m)}_{j,r}\times\cdots\times U_M,\,$ and $\,p^{(m)}_{j,r}:=\mu_j\bigl(B^{(m)}_{j,r}\bigr)$,
and let
\[
{\mathcal D^{(m)}}:=
\Bigl\{
\mathbb P\in\mathcal P(K):
\mathbb P\bigl(K^{(m)}_{j,r}\bigr)=p^{(m)}_{j,r}
\text{ for all }j=1,\dots,M,\ r=1,\dots,n_j(m)
\Bigr\}.
\]
Define product cells $\,C^{(m)}_{\mathbf i}:=
B^{(m)}_{1,i_1}\times\cdots\times B^{(m)}_{M,i_M}\,$ for 
$\,\mathbf i=(i_1,\dots,i_M)\in
\prod_{j=1}^M\{1,\dots,n_j(m)\}$,
and relabel them as $C^{(m)}_1,\dots,C^{(m)}_{N_m}$ for notational convenience. Assume the partitions are ``dominant'' for $(f,g)$ in the sense that, as $m\to\infty$,
\begin{align*}
\omega_f(m)&:=
\max_{1\leqslant \nu\leqslant N_m}\sup_{\textnormal{\textbf{x}},\textnormal{\textbf{y}}\in C^{(m)}_\nu}|f(\textnormal{\textbf{x}})-f(\textnormal{\textbf{y}})|\to 0,\\[4pt]
\omega_g(m)&:=
\max_{1\leqslant \nu\leqslant N_m}\sup_{\textnormal{\textbf{x}},\textnormal{\textbf{y}}\in C^{(m)}_\nu}|g(\textnormal{\textbf{x}})-g(\textnormal{\textbf{y}})|\to 0\,.
\end{align*}
Define the strip-consistency polytope
\[
{\mathcal W}_m:=
\left\{
\textnormal{\textbf{w}}=(w_\nu)_{\nu=1}^{N_m}\in[0,\infty)^{N_m}:\!
\sum_{\nu:\,C^{(m)}_\nu\subseteq K^{(m)}_{j,r}} w_\nu
=
p^{(m)}_{j,r}
\ \text{for all }\;\begin{array}{l}j=1,\dots,M,\\[3pt]r=1,\dots,n_j(m)\end{array}
\right\}.
\]
For $\textnormal{\textbf{w}}\in {\mathcal W}_m$, define $V_m(\textnormal{\textbf{w}}):=
\inf\limits_{\substack{\textnormal{\textbf{x}}_\nu\in C^{(m)}_\nu\\ w_\nu>0}}
\frac{\sum_{\nu=1}^{N_m} w_\nu f(\textnormal{\textbf{x}}_\nu)}
{\sum_{\nu=1}^{N_m} w_\nu g(\textnormal{\textbf{x}}_\nu)}$ where cells with $w_\nu=0$ are omitted from the inner optimisation. Finally, let $\,\phi_m:=
\inf\limits_{\textnormal{\textbf{w}}\in {\mathcal W}_m}V_m(\textnormal{\textbf{w}})$ and 
let $\,\phi^*:=\inf\limits_{\mathbb P\in {\mathcal D}}\frac{\mathbb E_{\mathbb P}[f(\textnormal{\textbf{X}})]}{\mathbb E_{\mathbb P}[g(\textnormal{\textbf{X}})]}$.

\noindent Then:

\begin{enumerate}
\item The $m$-th approximating problem is $\phi_m
=
\inf\limits_{\mathbb P\in {{\mathcal D}^{(m)}}}\frac{\mathbb E_{\mathbb P}[f(\textnormal{\textbf{X}})]}{\mathbb E_{\mathbb P}[g(\textnormal{\textbf{X}})]}$.

\item For $\,M_f:=\sup\limits_{\textnormal{\textbf{x}}\in K}f(\textnormal{\textbf{x}})\,$ and $\,\eta_m:=\frac{\omega_f(m)}{c}
+\frac{M_f\,\omega_g(m)}{c^2},\,$
$\,\phi^*-\eta_m\leqslant \phi_m\leqslant \phi^*+\eta_m\,$ holds for all $m\in\mathbb N$. That is, $\phi_m\to \phi^*$ as $m\to\infty$
\end{enumerate}
\end{proposition}

\begin{proof} See Appendix~E, Supplementary Material (\cite{Salako_Muhammad_2025_suppmat}).
\end{proof}
\noindent\textbf{Remark: }
\cite{LavineWassermanWolpert1991-SpecifiedMarginals} prove convergence of product-cell discretisations for fixed-marginal upper-bounds on expectations of uniformly continuous functions. This overlaps with a special case of Proposition~\ref{prop_gen_multiple_marginals_piecewise_refined}. However, Proposition~\ref{prop_gen_multiple_marginals_piecewise_refined} generalises this overlap to approximating fractional objectives with approximation error that depends on the oscillations of the numerator and denominator of the objective function. Lavine et al. require the functions involved to be uniformly continuous, but they allow more general prescribed marginal measures than the density marginals considered here.

Corollaries~\ref{cor_finite_marginal_approx} and \ref{cor_fgvanishingoscillations} restrict Proposition~\ref{prop_gen_multiple_marginals_piecewise_refined}, respectively, to continuous functions and piecewise-defined functions whose oscillations vanish asymptotically over recursively refined partitions of $K$. 
\begin{corollary}[Continuous $f$ and $g$]
\label{cor_finite_marginal_approx}
Assume the setting of Proposition~\ref{prop_gen_multiple_marginals_piecewise_refined} without the partition dominance assumption. Assume $\delta_m:=
\max_{1\leqslant j\leqslant M}\max_{1\leqslant r\leqslant n_j(m)}
\operatorname{diam}\bigl(B^{(m)}_{j,r}\bigr)\to 0$ as $m\to\infty$. Suppose $f,g:K\to[0,\infty)$ are continuous on $K$. Then, for each $m\in\mathbb N$, 
\[
\phi_m:=\inf_{\mathbb P\in{\mathcal D^{(m)}}}\frac{\mathbb E_{\mathbb P}[f(\textnormal{\textbf{X}})]}{\mathbb E_{\mathbb P}[g(\textnormal{\textbf{X}})]}
\]
is an instance of Proposition~\ref{prop_genprob1_v1}, and hence admits the
equivalent finite-dimensional atomised formulation described there. In the
present product-partition setting, this atomisation coincides with the canonical
product-cell formulation of
Proposition~\ref{prop_gen_multiple_marginals_piecewise_refined}. Moreover, if
\[
\phi^*:=\inf_{\mathbb P\in {\mathcal D}}\frac{\mathbb E_{\mathbb P}[f(\textnormal{\textbf{X}})]}{\mathbb E_{\mathbb P}[g(\textnormal{\textbf{X}})]},
\]
then
\[
\phi_m\to \phi^*
\text{ as }m\to\infty.
\]
\end{corollary}

\begin{proof}

For each fixed $m\in\mathbb N$, the class $\mathcal D^{(m)}$ is specified by finitely many
strip constraints $K_{j,r}^{(m)}$. Discarding any zero-mass strip
constraints, this is an overlapping finite-constraint problem of the
type treated by Proposition~\ref{prop_genprob1_v1}; in the present product-partition
setting, the induced atoms are the product cells
$C_{\mathbf i}^{(m)}$, and this atomisation coincides with the canonical
product-cell formulation of Proposition~\ref{prop_gen_multiple_marginals_piecewise_refined}.

Since $K$ is compact and $f,g$ are continuous on $K$, both functions are
uniformly continuous. Because $\delta_m\to0$, the diameters of the product cells
$C_{\mathbf i}^{(m)}$ tend to $0$ uniformly in $\mathbf i$; for example, under
the Euclidean product metric, $\operatorname{diam}\bigl(C_{\mathbf i}^{(m)}\bigr)\leqslant \sqrt{M}\,\delta_m$. Hence, the dominance hypothesis of
Proposition~\ref{prop_gen_multiple_marginals_piecewise_refined} holds
automatically; $\phi_m\to\phi^*$ follows immediately.
\end{proof}

\begin{corollary}[Piecewise defined $f$, $g$ with asymptotically vanishing oscillations]
\label{cor_fgvanishingoscillations}
Assume the setting of Proposition~\ref{prop_gen_multiple_marginals_piecewise_refined}. Assume, moreover, that the marginal partition
elements are chosen so that $\overline{B^{(m)}_{j,r}}$ is connected for every $j,r,m$. Let $E_1,\ldots,E_L$ be a fixed finite Borel
partition of $K$. For each $m\in\mathbb{N}$ and each product cell $C^{(m)}_\nu$, refine
$C^{(m)}_\nu$ by $\{E_\ell\}_{\ell=1}^L$, and relabel the nonempty refined atoms as $R^{(m)}_1,\ldots,R^{(m)}_{L_m}$. Thus, each $R^{(m)}_\tau$ is of the form $C^{(m)}_\nu\cap E_\ell$ for some $\nu,\ell$, and $K=\bigsqcup_{\tau=1}^{L_m} R^{(m)}_\tau$.

Define
\[
\widetilde{\omega}_f(m):=
\max_{1\leqslant \tau\leqslant L_m}\sup_{\textnormal{\textbf{x}},\textnormal{\textbf{y}}\in \overline{R^{(m)}_\tau}} |f(\textnormal{\textbf{x}})-f(\textnormal{\textbf{y}})|,\qquad
\widetilde{\omega}_g(m):=
\max_{1\leqslant \tau\leqslant L_m}\sup_{\textnormal{\textbf{x}},\textnormal{\textbf{y}}\in \overline{R^{(m)}_\tau}} |g(\textnormal{\textbf{x}})-g(\textnormal{\textbf{y}})|.
\]
Assume $\widetilde{\omega}_f(m)\to 0$ and $\widetilde{\omega}_g(m)\to 0$ as $m\to\infty$. Define the refined consistency polytope
\[
\widetilde{\mathcal W}_m:=
\left\{
\boldsymbol{\alpha}=(\alpha_\tau)_{\tau=1}^{L_m}\in [0,\infty)^{L_m}:\!\!
\sum_{\tau:\,R^{(m)}_\tau\subseteq K^{(m)}_{j,r}}\alpha_\tau=p^{(m)}_{j,r}
\text{ for all }\;\begin{array}{l}j=1,\ldots,M,\\[3pt]r=1,\ldots,n_j(m)\end{array}
\right\}.
\]
Finally, let $\widetilde{V}_m(\boldsymbol{\alpha}):=
\inf\limits_{{\textnormal{\textbf{x}}}_\tau\in R^{(m)}_\tau\ \text{for }\alpha_\tau>0}
\frac{\sum_{\tau=1}^{L_m}\alpha_\tau f({\textnormal{\textbf{x}}}_\tau)}
{\sum_{\tau=1}^{L_m}\alpha_\tau g({\textnormal{\textbf{x}}}_\tau)}\,$ for $\,\boldsymbol{\alpha}\in \widetilde{\mathcal W}_m,\,$
let $\,\widetilde{\phi}_m:=\inf\limits_{\boldsymbol{\alpha}\in \widetilde{\mathcal W}_m}\widetilde{V}_m(\boldsymbol{\alpha}),\,$ and let $\,\phi^*:=\inf\limits_{{\mathbb P}\in{\mathcal D}}\frac{\mathbb E_{\mathbb P}[f({\textnormal{\textbf{X}}})]}{\mathbb E_{\mathbb P}[g({\textnormal{\textbf{X}}})]}$.

\noindent Then:
\begin{enumerate}
\item The $m$-th approximating problem is $\widetilde{\phi}_m=\inf\limits_{\mathbb P\in {{\mathcal D}^{(m)}}}\frac{\mathbb E_{\mathbb P}[f({\textnormal{\textbf{X}}})]}{\mathbb E_{\mathbb P}[g({\textnormal{\textbf{X}}})]}$.
\item For $\,M_f:=\sup\limits_{{\textnormal{\textbf{x}}}\in K} f({\textnormal{\textbf{x}}})\,$ and 
$\,\widetilde{\eta}_m:=
\frac{L\,\widetilde{\omega}_f(m)}{c}
+\frac{L\,M_f\,\widetilde{\omega}_g(m)}{c^2},\,$ $\phi^*-\widetilde{\eta}_m\leqslant \widetilde{\phi}_m\leqslant \phi^*+\widetilde{\eta}_m$ holds for all $m\in\mathbb N$. That is, $\widetilde{\phi}_m\to \phi^*$ as $m\to\infty$.
\end{enumerate}
\end{corollary}

\begin{proof}
See Appendix~F, Supplementary Material (\cite{Salako_Muhammad_2025_suppmat}).
\end{proof}

\subsection{Solutions to Problem~\eqref{eqn_genCBIprob} and Its Generalisations}
\label{subsec_solutions}
The following Theorems~\ref{thm:global_fp_tuple}, \ref{thm_gensol}, \ref{thm_refinedgensol}, \ref{thm:regulated_extension} present solutions and solution methods for \eqref{eqn_genCBIprob} and its generalisations (Propositions~\ref{prop_genprob1_v1}, \ref{prop_cont_marginal_approx}, \ref{cor_refined_strip_oscillations}, \ref{prop_gen_multiple_marginals_piecewise_refined}). 

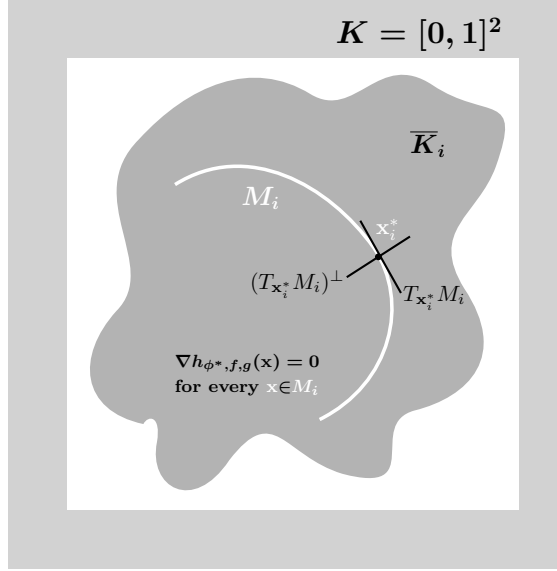
\begin{figure}[htbp!]
\centering
\resizebox{0.53\linewidth}{!}{%
\setlength{\fboxsep}{6pt}%
\colorbox{gray!35}{%
\begin{tikzpicture}[x=7.2cm,y=7.2cm,>=Latex]

  \fill[gray!35] (-0.10,-0.10) rectangle (1.08,1.08);
  \fill[white] (0,0) rectangle (1,1);

  \fill[gray!60]
    (0.17,0.19)
    .. controls (0.03,0.22) and (0.02,0.34) ..
    (0.11,0.44)
    .. controls (0.21,0.54) and (0.03,0.69) ..
    (0.16,0.82)
    .. controls (0.27,0.94) and (0.41,1.00) ..
    (0.54,0.92)
    .. controls (0.64,0.86) and (0.69,1.02) ..
    (0.81,0.94)
    .. controls (0.93,0.86) and (1.01,0.93) ..
    (0.95,0.76)
    .. controls (0.91,0.65) and (0.80,0.62) ..
    (0.88,0.49)
    .. controls (0.98,0.34) and (0.89,0.25) ..
    (0.77,0.22)
    .. controls (0.66,0.19) and (0.80,0.03) ..
    (0.61,0.08)
    .. controls (0.48,0.12) and (0.46,0.24) ..
    (0.36,0.11)
    .. controls (0.28,-0.01) and (0.18,0.05) ..
    (0.20,0.15)
    .. controls (0.21,0.21) and (0.18,0.21) ..
    (0.17,0.19) -- cycle;

  \draw[white,line width=1.6pt]
    (0.24,0.72)
    .. controls (0.40,0.82) and (0.58,0.72) ..
    (0.68,0.58)
    .. controls (0.76,0.46) and (0.72,0.28) ..
    (0.56,0.20);

  \fill[black] (0.69,0.56) circle (0.0075);

  \draw[black,line width=0.95pt] (0.65,0.64) -- (0.74,0.48);

  \draw[black,line width=0.95pt] (0.62,0.515) -- (0.76,0.605);

  \node[font=\Large\bfseries, anchor=south east] at (0.98,1.0) {$\boldsymbol{K=[0,1]^2}$};

  \node[font=\large] at (0.8,0.81) {$\boldsymbol{\overline K_i}$};

  \node[font=\large, text=white] at (0.43,0.69) {$\boldsymbol{M_i}$};

  \node[font=\normalsize, anchor=west, text=white] at (0.67,0.62) {$\mathbf x_i^*$};

  \node[font=\normalsize, anchor=west] at (0.73,0.47) {$T_{\mathbf x_i^*}M_i$};

  \node[font=\normalsize, anchor=east] at (0.63,0.5) {$(T_{\mathbf x_i^*}M_i)^\perp$};

  \node[align=left, text=white] at (0.4,0.3) {%
    \textcolor{black}{$\boldsymbol{\nabla h_{\phi^*,f,g}(\mathbf x)=0}$}\\[2pt]
    \textcolor{black}{\textbf{for every }}$\boldsymbol{\mathbf x\textcolor{black}{\in} M_i}$
  };

\end{tikzpicture}%
}}%
\caption[Embedded manifold of fixed-point minimisers]{Geometric illustration of the embedded-manifold case in Theorem~\ref{thm:global_fp_tuple} for $K=[0,1]^2$. A point $\mathbf x_i^*\in M_i\subset \operatorname{int}(\overline K_i)$ lies on a $C^2$ embedded submanifold of minimisers of $h_{\phi^*,f,g}$. Along $M_i$, the gradient vanishes, and tangent directions belong to $\ker \nabla^2 h_{\phi^*,f,g}(\mathbf x)$; in the Morse--Bott non-degenerate case, the kernel coincides with the tangent space and the Hessian is positive definite in normal directions.}
\label{fig:manifold_theorem26}
\end{figure}

\begin{theorem}[Analytic $f,g$]\label{thm:global_fp_tuple}
In \eqref{eqn_genCBIprob}, let $K_1,\ldots,K_n$ be non-empty sets, let $\overline{K}_1,\ldots,\overline{K}_n$ be their closures, and suppose $f,\,g$ are analytic on an open neighbourhood of $K$. Then a fixed-point tuple $(\phi^*,{\mathbf x}_1^*,\ldots,{\mathbf x}_n^*)$ solves \eqref{eqn_weak_opt_equiv_disc},
where $\phi^*$ is the infimum of the discrete objective in \eqref{eqn_weak_opt_equiv_disc} and $({\mathbf x}_1^*,\ldots,{\mathbf x}_n^*)\in \overline{K}_1\times\cdots\times \overline{K}_n$: specifically, the tuple satisfies
\begin{enumerate}
\item[\textnormal{(E1)}]$
{\mathbf x}_i^* \in \underset{{\mathbf x}\in \overline{K}_i}{\mathrm{argmin}}\;h_{\phi^*,f,g}(\mathbf x)$, $i=1,\ldots,n$;\hfill
\textnormal{(E2)} $\phi^*=
\sum_{i=1}^n f({\mathbf x}_i^*)p_i/\sum_{i=1}^n g({\mathbf x}_i^*)p_i$.
\end{enumerate}
These imply $\phi^*$ is obtained from the discrete objective in \eqref{eqn_weak_opt_equiv_disc} by using an extremal prior distribution $\mathbb P$ with support in $\{{\mathbf x}_1^*,\ldots,{\mathbf x}_n^*\}$:
\begin{equation}
\mathbb P=\sum_{i=1}^n p_i\,\delta_{{\mathbf x}_i^*},
\qquad {\mathbf x}_i^*\in \overline{K}_i.
\label{eqn_analyticextremalprior}
\end{equation}
In general, the support of $\mathbb P$ could intersect some $K_i^c\cap\partial\overline{K}_i$, so $\mathbb P$ is the limit of feasible priors. More precisely:
\begin{enumerate}
\item if ${\mathbf x}_i^*\in \operatorname{int}(\overline{K}_i)$, then $\nabla h_{\phi^*,f,g}({\mathbf x}_i^*)=0$ and $\nabla^2 h_{\phi^*,f,g}({\mathbf x}_i^*)\succeq 0$;
\item if $\overline{K}_i$ is convex and ${\mathbf x}_i^*\in \partial \overline{K}_i$, then $\nabla h_{\phi^*,f,g}({\mathbf x}_i^*)(\mathbf y-{\mathbf x}_i^*) \geqslant 0$ for all $\mathbf y\in \overline{K}_i$;
equivalently, $-\nabla h_{\phi^*,f,g}({\mathbf x}_i^*)\in N_{\overline{K}_i}({\mathbf x}_i^*)$ where $N_{\overline{K}_i}({\mathbf x}_i^*)$ is the outward-pointing normal cone to $\overline{K}_i$ at ${\mathbf x}_i^*$;
\item if, for some neighbourhood $U_i$ of ${\mathbf x}_i^*$, the set
\[
M_i:=\{{\mathbf x}\in U_i\cap \overline{K}_i:\ h_{\phi^*,f,g}({\mathbf x})=h_{\phi^*,f,g}({\mathbf x}_i^*)\}
\]
is a $C^2$ embedded submanifold and $M_i\subset\mathrm{int}(\overline{K}_i)$ then, for every $\mathbf x\in M_i$,
\[
\nabla h_{\phi^*,f,g}(\mathbf x)=0,
\qquad
T_{\mathbf x}M_i \subseteq \ker \nabla^2 h_{\phi^*,f,g}(\mathbf x).
\]
In the Morse--Bott non-degenerate case, $\ker \nabla^2 h_{\phi^*,f,g}({\mathbf x})=T_{\mathbf x}M_i$, 
and $\nabla^2 h_{\phi^*,f,g}({\mathbf x})$ is positive definite on the normal space of $M_i$ at ${\mathbf x}$.
\end{enumerate}
\end{theorem}
\begin{proof}
See Appendix~G, Supplementary Material (\cite{Salako_Muhammad_2025_suppmat}).
\end{proof}
\begin{figure}[ht!]
\centering
\resizebox{0.5\linewidth}{!}{%
\setlength{\fboxsep}{6pt}%
\colorbox{gray!35}{%
\begin{tikzpicture}[x=7.0cm,y=7.0cm,>=Latex]

  \fill[gray!35] (-0.08,-0.08) rectangle (1.06,1.06);

  \path[fill=white]
    (0,0) -- (1,0) -- (1,1) -- (0,1) -- cycle;

  \fill[gray!60]
    (0,0) -- (0.33,0) -- (0.29,0.26) -- (0,0.31) -- cycle;

  \fill[gray!30]
    (0.33,0) -- (0.71,0) -- (0.67,0.22) -- (0.51,0.34) -- (0.29,0.26) -- cycle;

  \fill[gray!50]
    (0.71,0) -- (1,0) -- (1,0.30) -- (0.84,0.36) -- (0.67,0.22) -- cycle;

  \fill[gray!40]
    (0,0.31) -- (0.29,0.26) -- (0.35,0.57) -- (0,0.63) -- cycle;

  \fill[gray!60]
    (0.29,0.26) -- (0.51,0.34) -- (0.63,0.56) -- (0.47,0.72) -- (0.35,0.57) -- cycle;

  \fill[gray!40]
    (0.67,0.22) -- (0.84,0.36) -- (1,0.30) -- (1,0.62) -- (0.78,0.66) -- (0.63,0.56) -- (0.51,0.34) -- cycle;

  \fill[gray!30]
    (0,0.63) -- (0.35,0.57) -- (0.47,0.72) -- (0.40,1) -- (0,1) -- cycle;

  \fill[gray!40]
    (0.47,0.72)
    .. controls (0.56,0.84) and (0.67,0.88) ..
    (0.80,0.74)
    -- (0.78,0.66) -- (0.63,0.56) -- cycle;

  \fill[gray!50]
    (0.40,1) -- (1,1) -- (1,0.62) -- (0.78,0.66)
    -- (0.80,0.74)
    .. controls (0.67,0.88) and (0.56,0.84) ..
    (0.47,0.72) -- cycle;

  \draw[white,line width=1.2pt] (0,0) -- (0.33,0) -- (0.29,0.26) -- (0,0.31);
  \draw[white,line width=1.2pt] (0.33,0) -- (0.71,0) -- (0.67,0.22) -- (0.51,0.34) -- (0.29,0.26);
  \draw[white,line width=1.2pt] (0.71,0) -- (1,0) -- (1,0.30) -- (0.84,0.36) -- (0.67,0.22);

  \draw[white,line width=1.2pt] (0,0.31) -- (0.29,0.26) -- (0.35,0.57) -- (0,0.63);
  \draw[white,line width=1.2pt] (0.29,0.26) -- (0.51,0.34) -- (0.63,0.56) -- (0.47,0.72) -- (0.35,0.57);
  \draw[white,line width=1.2pt] (0.51,0.34) -- (0.67,0.22) -- (0.84,0.36) -- (1,0.30);
  \draw[white,line width=1.2pt] (0.63,0.56) -- (0.78,0.66) -- (1,0.62);

  \draw[white,line width=1.2pt] (0,0.63) -- (0.35,0.57) -- (0.47,0.72) -- (0.40,1);
  \draw[white,line width=1.2pt]
    (0.47,0.72)
    .. controls (0.56,0.84) and (0.67,0.88) ..
    (0.80,0.74);
  \draw[white,line width=1.2pt] (0.78,0.66) -- (0.80,0.74);
  \draw[white,line width=1.2pt] (0,0) rectangle (1,1);

  \node[font=\Large\bfseries, anchor=south east] at (0.98,1) {$\boldsymbol{K=[0,1]^2}$};

  \node[font=\normalsize] at (0.14,0.2) {$\boldsymbol{K_1}$};
  \node[font=\normalsize] at (0.50,0.13) {$\boldsymbol{K_2}$};
  \node[font=\normalsize] at (0.88,0.1) {$\boldsymbol{K_3}$};

  \node[font=\normalsize] at (0.19,0.48) {$\boldsymbol{K_4}$};
  \node[font=\normalsize] at (0.43,0.50) {$\boldsymbol{K_5}$};
  \node[font=\normalsize] at (0.81,0.49) {$\boldsymbol{K_6}$};

  \node[font=\normalsize] at (0.16,0.73) {$\boldsymbol{K_7}$};
  \node[font=\normalsize] at (0.67,0.73) {$\boldsymbol{K_8}$};
  \node[font=\normalsize] at (0.74,0.91) {$\boldsymbol{K_9}$};


  \fill[black] (0.15,0.12) circle (0.008);
  \node[font=\normalsize, anchor=west] at (0.15,0.12) {$\mathbf x_1^*$};

  \fill[black] (0.42,0.31) circle (0.008);
  \node[font=\normalsize, anchor=west] at (0.44,0.32) {$\mathbf x_2^*=\mathbf x_5^*$};

  \fill[black] (0.92,0.33) circle (0.008);
  \node[font=\normalsize, anchor=west] at (0.77,0.28) {$\mathbf x_3^*=\mathbf x_6^*$};

  \fill[black] (0.00,0.47) circle (0.008);
  \node[font=\normalsize, anchor=west] at (0,0.47) {$\mathbf x_4^*$};

  \fill[black] (0.14,0.84) circle (0.008);
  \node[font=\normalsize, anchor=west] at (0.14,0.85) {$\mathbf x_7^*$};

  \fill[black] (0.64,0.83) circle (0.008);
  \node[font=\normalsize, anchor=west] at (0.66,0.83) {$\mathbf x_8^*=\mathbf x_9^*$};

  \node[font=\small, align=center] at (0.18,0.05) {interior \\ point};
  \node[font=\small, align=center] at (0.10,0.40) {boundary \\ point};
  \node[font=\small, align=center] at (0.88,0.2) {shared \\boundary point};

\end{tikzpicture}%
}}%
\caption[Discrete extremal prior on a partitioned domain]{Schematic illustration of an extremal prior in the analytic case of Theorem~\ref{thm:global_fp_tuple}. The partitioned domain $K=[0,1]^2$ is divided into irregular cells $K_1,\ldots,K_9$, and each $K_i$ cell has an associated minimizer location $\mathbf x_i^*\in \overline K_i$. The extremal prior $\mathbb P=\sum_{i=1}^9 p_i\,\delta_{\mathbf x_i^*}$ is therefore discrete, while some minimizer locations may coincide on pairwise shared cell boundaries, as indicated by equalities such as $\mathbf x_2^*=\mathbf x_5^*$ and $\mathbf x_8^*=\mathbf x_9^*$. Minimizers on shared boundaries in a partition make clear that extremal priors may violate the Problem~\eqref{eqn_genCBIprob} constraints: extremal priors are weak limits of feasible priors.}
\label{fig:extremal_prior_partitioned_domain}
\end{figure}
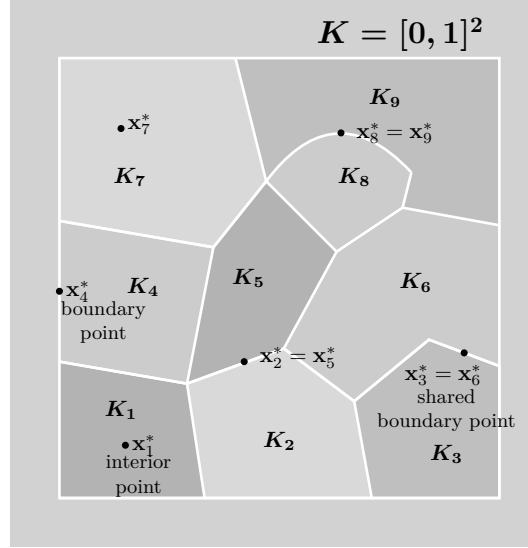
\begin{algorithm}[ht!]
\caption{Dinkelbach-type iteration for the fixed points of Theorem~\ref{thm:global_fp_tuple}}
\label{alg:dinkelbach}
\begin{center}
\fbox{%
\begin{minipage}{0.6\linewidth}
Choose $\widehat{\phi}_0>0$ and tolerance $\varepsilon>0$.
For $t=0,1,2,\ldots$ repeat:
\begin{enumerate}
\item For $i=1,\ldots,n$ compute $\widehat{\mathbf x}^{(t)}_i
\in
\underset{\mathbf x\in \overline{K}_i}{\mathrm{argmin}}\;
h_{\widehat{\phi}_t,f,g}(\mathbf x)$.
\item Stop if $
\left|
\sum_{i=1}^n p_i
h_{\widehat{\phi}_t,f,g}(\widehat{\mathbf x}^{(t)}_i)
\right|
\leqslant \varepsilon$. 
\item Update $\widehat{\phi}_{t+1}
=
\frac{\sum_{i=1}^n p_i f(\widehat{\mathbf x}^{(t)}_i)}
     {\sum_{i=1}^n p_i g(\widehat{\mathbf x}^{(t)}_i)}$.
\end{enumerate}
Return $\widehat{\phi}_\varepsilon:=\widehat{\phi}_{t}$ and
$\widehat{\mathbf x}_{i,\varepsilon}:=\widehat{\mathbf x}^{(t)}_i$.
\end{minipage}}
\end{center}
\end{algorithm}

\noindent\textbf{Remarks: }Algorithm~\ref{alg:dinkelbach} is a Dinkelbach-type iteration for Theorem~\ref{thm:global_fp_tuple} solutions: $\widehat{\phi}_t$ converges to $\phi^*$, and any accumulation point of the sequence $\{(\widehat{\phi}_t,\widehat{\mathbf x}^{(t)}_1,\ldots,\widehat{\mathbf x}^{(t)}_n)\}_{t\in\mathbb Z^{\geqslant0}}$ satisfies the fixed-point conditions (E1) and (E2). For a proof of the convergence, see Appendix~H of the Supplementary Material (\cite{Salako_Muhammad_2025_suppmat}). 
Dinkelbach-type iteration is a classical method for solving nonlinear
fractional programming problems. Likewise, finite linear-fractional programmes
admit the Charnes--Cooper transformation
(see \cite{CharnesCooper1962LinearFractionalFunctionals}). In the robust Bayesian
setting, \cite{Cozman1999PosteriorBounds} explicitly identified
Lavine-type posterior-bound bracketing with Dinkelbach-type fractional
programming, and the White--Snow reformulation with the Charnes--Cooper
transformation. Thus, when the support points are fixed and the
admissible probabilities are described by finite linear constraints, this programme falls within the Cozman/fractional-programming setting.

More generally, the support points of extremal priors are unknown and
must themselves be determined. Following the finite-support reduction in Proposition~\ref{prop_cbi_transform}, Theorem~\ref{thm:global_fp_tuple} identifies the extremal support configurations as those satisfying the fixed-point conditions (E1)--(E2). The subsequent extensions give
approximation and convergence guarantees for broader function classes,
overlapping-set constraints, and marginal-density settings.

\begin{theorem}[Continuous $f,\,g$]
\label{thm_gensol}
In \eqref{eqn_genCBIprob} let $f,g:K\to [0,\infty)$ be continuous functions. So, there exist sequences $\{f_m\}_{m\in\mathbb N}$, $\{g_m\}_{m\in\mathbb N}$ of non-negative functions such that each $f_m$ and $g_m$ are analytic
on an open neighbourhood of $K$, $\|f_m-f\|_{\infty}\to 0$, 
and $\|g_m-g\|_{\infty}\to 0$, as $m\to\infty$. Restrict to all sufficiently large $m$ (to ensure $\inf_{{\mathbf x}\in K_j} g_m({\mathbf x})>0$ for some $j$). For $\mathbf x=({\mathbf x}_1,\dots,{\mathbf x}_n)\in K_1\times\cdots\times K_n$, define
\[
\phi_m({\mathbf x}_1,\ldots,{\mathbf x}_n):=
\frac{\sum_{i=1}^n f_m({\mathbf x}_i)p_i}{\sum_{i=1}^n g_m({\mathbf x}_i)p_i},
\qquad
\tilde{\phi}({\mathbf x}_1,\ldots,{\mathbf x}_n):=
\frac{\sum_{i=1}^n f({\mathbf x}_i)p_i}{\sum_{i=1}^n g({\mathbf x}_i)p_i}.
\]
Let $\phi_m^*:=\inf\limits_{{\mathbf x}_1\in K_1,\dots,{\mathbf x}_n\in K_n}\phi_m({\mathbf x}_1,\dots,{\mathbf x}_n)\,$ and let 
$\,\tilde{\phi}^*:=\inf\limits_{{\mathbf x}_1\in K_1,\dots,{\mathbf x}_n\in K_n}\tilde{\phi}({\mathbf x}_1,\dots,{\mathbf x}_n)$. Then $\phi_m^*\to \tilde{\phi}^*$ as $m\to\infty$. Moreover, if ${\mathbf x}^{(m)}=({\mathbf x}^{(m)}_1,\dots,{\mathbf x}^{(m)}_n)\in
\overline{K}_1\times\cdots\times\overline{K}_n$
is any minimiser of $\phi_m$, then every accumulation point
$\tilde{{\mathbf x}}=(\tilde{{\mathbf x}}_1,\dots,\tilde{{\mathbf x}}_n)$ of the sequence
$\{{\mathbf x}^{(m)}\}$ minimises $\tilde{\phi}$. 

Each ${\mathbf x}^{(m)}$ determines an extremal prior $\mathbb P_m$,
\begin{equation}
\mathbb P_m=\sum_{i=1}^n p_i\,\delta_{{\mathbf x}_i^{(m)}},
\qquad {\mathbf x}_i^{(m)}\in\overline{K}_i.
\label{eqn_seqofpriors}
\end{equation}
Along any convergent subsequence ${\mathbf x}^{(m_k)}\to\tilde{{\mathbf x}}$, the
corresponding priors $\mathbb P_{m_k}$ converge weakly to an
extremal prior $\mathbb P$ for the problem with objective $\tilde{\phi}$, supported on $\tilde{{\mathbf x}}_1,\dots,\tilde{{\mathbf x}}_n$. Limiting extremal priors may depend on the subsequence. The limiting fixed-point tuple $(\tilde{\phi}^*,\tilde{{\mathbf x}})$ solves \eqref{eqn_weak_opt_equiv_disc}: it satisfies 
\begin{align*}
\textnormal{(E1)}\;\tilde{{\mathbf x}}_i\in\underset{{\mathbf x}\in\overline{K}_i}{\mathrm{argmin}}\;h_{\tilde \phi^*,f,g}(\mathbf x),\;i=1,\ldots,n;\hspace{0.5cm} \textnormal{(E2)}\;\tilde{\phi}^*=
\tilde{\phi}(\tilde{{\mathbf x}}_1,\ldots,\tilde{{\mathbf x}}_n),
\end{align*}
with limiting extremal prior
\begin{equation}
\mathbb P=\sum_{i=1}^n p_i\,\delta_{\tilde{{\mathbf x}}_i},
\qquad \tilde{{\mathbf x}}_i\in\overline{K}_i.
\label{eqn_limitofseqofpriors}
\end{equation}
\end{theorem}
\begin{proof}
See Appendix~I, Supplementary Material (\cite{Salako_Muhammad_2025_suppmat}).
\end{proof}

\noindent\textbf{Remarks: }Together, Theorem~\ref{thm:global_fp_tuple} and Algorithm~\ref{alg:dinkelbach} are the basis of a general solution pipeline for Problem~\eqref{eqn_genCBIprob}; see Figure~\ref{fig_CBIMethodologyPipeline}. The Dinkelbach/\emph{Karush–Kuhn–Tucker} (KKT) analysis involved in Theorem~\ref{thm:global_fp_tuple} yields the fixed-point characterisation of the extremal solutions, identifying ``\emph{where}'' extremal priors assign probability mass in their domain---this determines explicit extremal distributions for obtaining the bounds studied in \cite{1991_Lavine} and \cite{moreno1991robust}. Consequently, Problem~\eqref{eqn_genCBIprob} reduces to the finite-dimensional (E1)--(E2) system, which can be solved by Dinkelbach iteration (\emph{cf.} (E$2$) with Example~1 of \cite{WassermanLavineWolpert1993}). The resulting extremal configurations can often be reused to derive other bounds for related problems through symmetry or limiting arguments. Seminal works for KKT analysis are \cite{Karush1939,KuhnTucker1951}, while \cite{NocedalWright2006} contains a modern treatment.

In Theorem~\ref{thm_gensol}, the class $C(K)$ of continuous functions on $K$ is closed under uniform limits of continuous functions (including uniformly convergent series of continuous functions), and the objective in \eqref{eqn_weak_opt_equiv_disc} depends continuously on $f$ and $g$ under uniform convergence (given $\inf_{K_i}g>0$ for some $i$). Hence, analytic approximations give a general solution method for continuous $f,\,g$, in \eqref{eqn_weak_opt_equiv_disc}, even when $f,\,g$ are defined as uniformly converging series. 

Proposition~\ref{prop:semicont_refinement} is useful in extending the results beyond $C(K)$.

\begin{proposition}[Refinement of \eqref{eqn_weak_opt_equiv_disc} for bounded piecewise continuous $f$, $g$ with a measurable set of discontinuities]
\label{prop:semicont_refinement}
Let $K_1,\dots,K_n$ be the partition of $K$ in \eqref{eqn_genCBIprob}. 
Assume each $K_i$ is partitioned into Borel-measurable sets $K_{i,\ell}$ ($\ell=1,\ldots,r_i$) on which $f$ and $g$ are continuous, and a Borel-measurable set $D_i$ on which $f$ or $g$ are discontinuous; so $K_i=\Bigl(\bigsqcup_{\ell=1}^{r_i} K_{i,\ell}\Bigr)\sqcup D_i$. Then, \eqref{eqn_genCBIprob} is equivalent to
\begin{equation}
\inf_{\substack{p_{i,\ell}\geqslant0,\;q_{i}\geqslant0\\
\sum_{\ell=1}^{r_i}p_{i,\ell}+q_{i}=p_i}}
\;
\inf_{{\mathbf x}_{i,\ell}\in K_{i,\ell},\,\mathbf d_i\in D_i}
\frac{
\sum_{i=1}^n\left(
\sum_{\ell=1}^{r_i} f({\mathbf x}_{i,\ell})\,p_{i,\ell}
+
f(\mathbf d_i)\,q_{i}
\right)
}{
\sum_{i=1}^n\left(
\sum_{\ell=1}^{r_i} g({\mathbf x}_{i,\ell})\,p_{i,\ell}
+
g(\mathbf d_i)\,q_{i}
\right)
}.
\label{eqn_v2_morenoetalCBIprob}
\end{equation}
\end{proposition}

\begin{proof}
See Appendix~J, Supplementary Material (\cite{Salako_Muhammad_2025_suppmat}).
\end{proof}


\begin{theorem}[Bounded piecewise continuous $f$, $g$ with a measurable set of discontinuities]
\label{thm_refinedgensol}
Let $f,g$ in \eqref{eqn_v2_morenoetalCBIprob} be such that, for each $i,\ell$, their restrictions to $K_{i,\ell}$ admit continuous extensions $\tilde f_{i,\ell},\tilde g_{i,\ell}$ to $\overline K_{i,\ell}$. Also, for each $i$, their restrictions to $D_i$ admit continuous extensions $\tilde f_i,\tilde g_i$ to $\overline D_i$.

Then, the infimum in \eqref{eqn_v2_morenoetalCBIprob}, $\phi^*$, is obtained in the limit by a sequence of feasible priors in \eqref{eqn_v2_morenoetalCBIprob} whose weak limit is an extremal prior $\mathbb P$,
\begin{equation}
\mathbb P=\sum_{i=1}^n p_i\,\delta_{{\mathbf x}_i^*},
\qquad {\mathbf x}_i^*\in\overline{K}_i,
\label{eqn_semicontprior}
\end{equation}
with the support of $\mathbb P$ forming part of a fixed-point tuple
$(\phi^*,{\mathbf x}_1^*,\dots,{\mathbf x}_n^*)$ that satisfies
\begin{enumerate}
\item[\text{(E1)}]${\mathbf x}_i^* \in M_i(\phi^*)$,
\item[\text{(E2)}]$\phi^*=\textstyle \sum_{i=1}^{n}\overset{\hspace{0.2em}*}{f}({\mathbf x}^*_i)p_i/\sum_{i=1}^{n}\overset{\hspace{0.2em}*}{g}({\mathbf x}^*_i)p_i$,
\end{enumerate}
where 
\begin{align*}
&(\overset{\hspace{0.2em}*}{f}({\mathbf x}^*_i),\,\overset{\hspace{0.2em}*}{g}({\mathbf x}^*_i))\\[5pt]
={}&\begin{cases}
(\tilde f_{i,\ell}({\mathbf x}_i^*),\tilde g_{i,\ell}({\mathbf x}_i^*)),
& \text{for some }\ell\text{ with }{\mathbf x}_i^*\in \underset{{\mathbf x}\in\overline K_{i,\ell}}{\mathrm{argmin}}\;h_{\phi^*,\tilde f_{i,\ell},\tilde g_{i,\ell}}({\mathbf x})\subseteq M_{i,0}(\phi^*)
\\[4pt]
( \tilde f_i({\mathbf x}_i^*),\tilde g_i({\mathbf x}_i^*)),
& {\mathbf x}_i^*\in M_{i,1}(\phi^*)
   \setminus M_{i,0}(\phi^*),   
\end{cases}
\end{align*}
\begin{gather*}
M_{i,0}(\varphi)
:={}
\!\!\bigcup_{\ell=1}^{r_i}
\underset{{\mathbf x}\in\overline K_{i,\ell}}{\mathrm{argmin}}\;h_{\varphi,\tilde f_{i,\ell},\tilde g_{i,\ell}}({\mathbf x}),\;M_{i,1}(\varphi)
:={}
\underset{{\mathbf x}\in\overline D_{i}}{\mathrm{argmin}}\;h_{\varphi,\tilde f_i,\tilde g_i}({\mathbf x}),\\
M_i(\varphi):=
M_{i,0}(\varphi)\bigcup M_{i,1}(\varphi)\,,
\end{gather*}
with the convention $M_{i,1}(\varphi)=\varnothing$ when $D_i=\varnothing$.

Indeed, the fixed-point tuple must be chosen so that ${\mathbf x}_i^*$ is in the set 
\[
\mathcal C_i(\phi^*)
:=
\Big(
\bigcup_{\ell=1}^{r_i}\mathcal C_{\overline{K}_{i,\ell}}(\phi^*)
\Big)
\cup
\mathcal C_{\overline{D}_{i}}(\phi^*) 
\]
for $i=1,\ldots,n$, where $\mathcal C_{\overline{K}_{i,\ell}}(\phi^*)$ and $\mathcal C_{\overline{D}_{i}}(\phi^*)$ consist of all minimisers of $h_{\phi^*,\tilde{f}_{i,\ell},\tilde{g}_{i,\ell}}$ on $\overline{K}_{i,\ell}$ and $h_{\phi^*,\tilde{f}_{i},\tilde{g}_{i}}$ on $\overline{D}_i$, respectively, 
that are subsequential limits of global minimisers of Dinkelbach differences formed from analytic approximating sequences that converge uniformly to the $\tilde f_{i,\ell}$, $\tilde g_{i,\ell}$, $\tilde f_{i}$, $\tilde g_{i}$ extensions. That is, $M_i(\phi^*)=\mathcal C_i(\phi^*)$. Theorems~\ref{thm:global_fp_tuple} and \ref{thm_gensol} give global minimisers. Each $\mathbf{x}^*\in\mathcal C_i(\phi^*)$ is a boundary or interior point of either $\overline{D}_i$ or some $\overline{K}_{i,\ell}$.
\end{theorem}
\begin{proof} See Appendix~K, Supplementary Material (\cite{Salako_Muhammad_2025_suppmat}).
\end{proof}
\noindent\textbf{Remarks: }In the continuous special case $r_i=1$, $K_{i,1}=K_i$ and $D_i=\varnothing$ for all $i=1,\ldots,n$, Theorem~\ref{thm_refinedgensol} reduces to the fixed-point characterisation of Theorem~\ref{thm_gensol}. More generally, Theorem~\ref{thm_refinedgensol} shows that analytic approximation is ``exhaustive'': it identifies the extremal
support points, and there are no extremal support candidates outside the class of extremal support points obtained from analytic approximations. 

Under the assumptions of Proposition~\ref{prop:semicont_refinement} and Theorem~\ref{thm_refinedgensol}, Algorithm~\ref{alg:dinkelbach_piecewise} and its convergence proof are analogous (with slight modifications) to Algorithm~\ref{alg:dinkelbach}.

\begin{algorithm}[ht!]
\caption{Dinkelbach-type iteration for Theorem~\ref{thm_refinedgensol}, Proposition~\ref{prop:semicont_refinement}}
\label{alg:dinkelbach_piecewise}
\begin{center}
\fbox{%
\begin{minipage}{0.98\linewidth}
Choose $\widehat{\phi}_0>0$ and tolerance $\varepsilon>0$.
For $t=0,1,2,\ldots$ repeat:
\begin{enumerate}
\item For $i=1,\ldots,n$ and $\ell=1,\ldots,r_i$, compute $\widehat{\mathbf x}^{(t)}_{i,\ell}
\in
\underset{\mathbf x\in \overline{K}_{i,\ell}}{\mathrm{argmin}}\;
h_{\widehat{\phi}_t,\widetilde f_{i,\ell},\widetilde g_{i,\ell}}(\mathbf x)$. If $D_i\neq\varnothing$, also compute $\widehat{\mathbf d}^{(t)}_{i}
\in
\underset{\mathbf d\in \overline{D}_i}{\mathrm{argmin}}\;
h_{\widehat{\phi}_t,\widetilde f_i,\widetilde g_i}(\mathbf d)$.

\item For $i=1,\ldots,n$, compute refined masses
$\widehat p^{(t)}_{i,\ell}\geqslant0$ and $\widehat q^{(t)}_i\geqslant0$ satisfying $\sum_{\ell=1}^{r_i}\widehat p^{(t)}_{i,\ell}
+\widehat q^{(t)}_i
=
p_i$ and attaining
\[
\underset{
\substack{
p_{i,\ell}\geqslant0,\ q_i\geqslant0\\
\sum_{\ell=1}^{r_i}p_{i,\ell}+q_i=p_i
}}
{\mathrm{argmin}}\;
\left\{
\sum_{\ell=1}^{r_i}
p_{i,\ell}
h_{\widehat{\phi}_t,\widetilde f_{i,\ell},\widetilde g_{i,\ell}}
(\widehat{\mathbf x}^{(t)}_{i,\ell})
+
q_i
h_{\widehat{\phi}_t,\widetilde f_i,\widetilde g_i}
(\widehat{\mathbf d}^{(t)}_i)
\right\}.
\]
Terms involving $D_i$ are omitted when $D_i=\varnothing$, in which case set
$\widehat q^{(t)}_i=0$.

\item Stop if
\[
\left|
\sum_{i=1}^n
\left\{
\sum_{\ell=1}^{r_i}
\widehat p^{(t)}_{i,\ell}
h_{\widehat{\phi}_t,\widetilde f_{i,\ell},\widetilde g_{i,\ell}}
(\widehat{\mathbf x}^{(t)}_{i,\ell})
+
\widehat q^{(t)}_i
h_{\widehat{\phi}_t,\widetilde f_i,\widetilde g_i}
(\widehat{\mathbf d}^{(t)}_i)
\right\}
\right|
\leqslant \varepsilon .
\]

\item Update
\[
\widehat{\phi}_{t+1}
=
\frac{
\sum_{i=1}^n
\left\{
\sum_{\ell=1}^{r_i}
\widehat p^{(t)}_{i,\ell}
\widetilde f_{i,\ell}(\widehat{\mathbf x}^{(t)}_{i,\ell})
+
\widehat q^{(t)}_i
\widetilde f_i(\widehat{\mathbf d}^{(t)}_i)
\right\}
}{
\sum_{i=1}^n
\left\{
\sum_{\ell=1}^{r_i}
\widehat p^{(t)}_{i,\ell}
\widetilde g_{i,\ell}(\widehat{\mathbf x}^{(t)}_{i,\ell})
+
\widehat q^{(t)}_i
\widetilde g_i(\widehat{\mathbf d}^{(t)}_i)
\right\}
}.
\]
\end{enumerate}
Return $\widehat{\phi}_\varepsilon:=\widehat{\phi}_{t}$ and the approximately extremal prior $\widehat {\mathbb P}_\varepsilon
:=
\sum_{i=1}^n
\left\{
\sum_{\ell=1}^{r_i}
\widehat p^{(t)}_{i,\ell}
\delta_{\widehat{\mathbf x}^{(t)}_{i,\ell}}
+
\widehat q^{(t)}_i
\delta_{\widehat{\mathbf d}^{(t)}_i}
\right\}$.
\end{minipage}}
\end{center}
\end{algorithm}

\begin{theorem}[Uniform-closure generalisation of Theorem~\ref{thm_refinedgensol}]
\label{thm:regulated_extension} 
Let $\mathcal R_{2.13}$ be the class of bounded functions having the piecewise extension properties in Theorem~\ref{thm_refinedgensol}, and let $\overline{\mathcal R}_{2.13}^{\|\cdot\|_\infty}$ be its closure under uniform limits: that is, for each $f\in \overline{\mathcal R}_{2.13}^{\|\cdot\|_\infty}$ there exists a sequence $\{f_m\}\subset \mathcal R_{2.13}$ such that $\|f_m-f\|_{\infty}\to 0$ as $m\to\infty$. Suppose $f,g\in \overline{\mathcal R}_{2.13}^{\|\cdot\|_\infty}$ and $\inf_{{\mathbf x}\in K_j} g({\mathbf x})>0$ for some $j\in\{1,\dots,n\}$. 

Then, the infimum $\phi^*$ of \eqref{eqn_weak_opt_equiv_disc} is obtained in the limit by a sequence of feasible priors in \eqref{eqn_v2_morenoetalCBIprob} whose weak limit is some extremal prior
\[
\mathbb P^*=\sum_{i=1}^n p_i\,\delta_{{\mathbf x}_i^*},
\qquad {\mathbf x}_i^*\in \overline K_i .
\]
More precisely, let $\{f_m\},\{g_m\}\subset\mathcal R_{2.13}$ be such that $\|f_m-f\|_\infty\to 0$, $\|g_m-g\|_\infty\to 0$. Restrict to all sufficiently large $m$ (to ensure $\inf_{{\mathbf x}\in K_j} g_m({\mathbf x})>0$) and denote the infimum of \eqref{eqn_weak_opt_equiv_disc} for $f_m,g_m$ as $\phi_m^*$. 
Then:
\begin{enumerate}
\item[\textnormal{(i)}] $\phi_m^*\to \phi^*$;
\item[\textnormal{(ii)}] Theorem~\ref{thm_refinedgensol} applies to the pair
$(f_m,g_m)$ after passing to a common refinement of their individual partitions, and
there is an extremal prior
\[
\mathbb P_m=\sum_{i=1}^n p_i\,\delta_{{\mathbf x}_i^{(m)}},
\qquad {\mathbf x}_i^{(m)}\in \overline K_i ;
\]
\item[\textnormal{(iii)}] the sequence of tuples $\{({\mathbf x}_1^{(m)},\dots,{\mathbf x}_n^{(m)})\}$ has a
convergent subsequence,
\[
({\mathbf x}_1^{(m_k)},\dots,{\mathbf x}_n^{(m_k)})
\to
({\mathbf x}_1^*,\dots,{\mathbf x}_n^*)
\in \overline K_1\times\cdots\times \overline K_n.\] The limiting tuple $({\mathbf x}_1^*,\dots,{\mathbf x}_n^*)$ is the support of some extremal prior $\mathbb P^*$. The sequence of extremal priors $\{\mathbb P_{m_k}\}$ weakly converges to $\mathbb P^*$;
\item[\textnormal{(iv)}] there exists a sequence of feasible priors $\{\mathbb Q_k\}\subset \mathcal D$
such that the sequence weakly converges  to the prior $\mathbb P^*$ in \textnormal{(iii)}, and
\[
\frac{\mathbb E_{\mathbb Q_k}[f(\mathbf X)]}{\mathbb E_{\mathbb Q_k}[g(\mathbf  X)]}\to \phi^*.
\]
\end{enumerate}
\end{theorem}
\begin{proof} See Appendix~L, Supplementary Material (\cite{Salako_Muhammad_2025_suppmat}).
\end{proof}

\begin{figure}[ht!]
\centering

\setlength{\fboxsep}{3pt}%
\setlength{\fboxrule}{1.5pt}%

\fcolorbox{black}{gray!50}{%
  \resizebox{0.95\linewidth}{!}{%
    \begin{tikzpicture}[
      font=\small,
      node distance=8mm and 12mm,
      >=Latex,
      box/.style={
        draw=none,
        fill=white,
        rounded corners=2pt,
        align=center,
        inner xsep=6pt,
        inner ysep=5pt,
        minimum height=9mm,
        text width=39mm
      },
      proc/.style={
        draw=none,
        fill=white,
        rounded corners=2pt,
        align=center,
        inner xsep=6pt,
        inner ysep=5pt,
        minimum height=9mm,
        text width=43mm
      },
      wideproc/.style={
        draw=none,
        fill=white,
        rounded corners=2pt,
        align=center,
        inner xsep=6pt,
        inner ysep=5pt,
        minimum height=9mm,
        text width=70mm
      },
      arrow/.style={->, line width=0.6pt},
      dashedarrow/.style={->, dashed, line width=0.6pt},
      bidashed/.style={<->, dashed, line width=0.6pt},
      note/.style={
        draw=none,
        fill=white,
        rounded corners=2pt,
        align=left,
        inner xsep=7pt,
        inner ysep=6pt,
        text width=0.93\linewidth
      }
    ]

\node[box] (p1) {Problem~\eqref{eqn_genCBIprob}\\[1pt]
$\displaystyle \inf_{\mathbb P\in\mathcal D}\ \frac{\mathbb E_{\mathbb P}[f(\mathbf X)]}{\mathbb E_{\mathbb P}[g(\mathbf X)]}$\\
subject to prior constraints};

\node[proc, right=of p1] (s1) {Reduce to a discrete\\extremal-prior problem\\(Proposition~\ref{prop_cbi_transform})};

\node[box, right=of s1] (p2) {Equivalent finite / discrete formulation,\\[1pt] or sequences of finite approximating formulations (e.g. Problems~\eqref{eqn_weak_opt_equiv_disc}, \eqref{eq:overlap_refinement}, and consistency-polytope formulations)};

\draw[arrow] (p1) -- (s1);
\draw[arrow] (s1) -- (p2);

\node[wideproc, above=10mm of p2] (g1) {More general prior specifications:\\overlapping measurable-set constraints\\and multiple marginal-density constraints\\(Propositions~\ref{prop_genprob1_v1}, \ref{prop_cont_marginal_approx}, \ref{cor_refined_strip_oscillations}, \ref{prop_gen_multiple_marginals_piecewise_refined})};

\draw[dashedarrow] (p1.north) to[bend left=15] (g1.west);
\draw[dashedarrow] (g1.south) -- (p2.north);

\node[proc, below=11mm of p2] (s2) {For analytic $f$ and $g$, characterise\\extremal-prior support via\\fixed-point conditions (E1)--(E2)\\(Theorem~\ref{thm:global_fp_tuple})};

\node[proc, left=35mm of s2] (s3) {Numerically solve the\\resulting fixed-point system\\using (E1)--(E2) and\\Algorithms~\ref{alg:dinkelbach},\;\ref{alg:dinkelbach_piecewise}, as applicable};

\draw[arrow] (p2.south) -- (s2.north);
\draw[arrow] (s2) -- (s3);

\node[wideproc, below=11mm of s2] (g2) {Extend the analytic framework by approximation:\\continuous functions (Theorem~\ref{thm_gensol}),\\piecewise continuous functions (Proposition~\ref{prop:semicont_refinement}+Theorem~\ref{thm_refinedgensol}),\\and uniform-limit classes (Theorem~\ref{thm:regulated_extension})};

\draw[bidashed] (s2.south) -- (g2.north);
\draw[bidashed] (g2.north west) to[bend right=40] (s3.south east);

\node[wideproc, below=11mm of s3, xshift=0mm] (out) {Outputs: extremal values, extremal priors\\(or limits of feasible priors), support locations,\\and convergence / asymptotic properties};

\draw[arrow] (s3.south) -- (out.north);
\draw[dashedarrow] (g2.west) to[bend right=-60] (out.south);

\end{tikzpicture}
}}
\caption[Fixed-point solution pipeline for robust inference]{Methodological pipeline for solving Problem~\eqref{eqn_genCBIprob} generalisations via discrete extremal-prior reduction, fixed-point support characterisation, and Dinkelbach-type iteration.}
\label{fig_CBIMethodologyPipeline}
\end{figure}
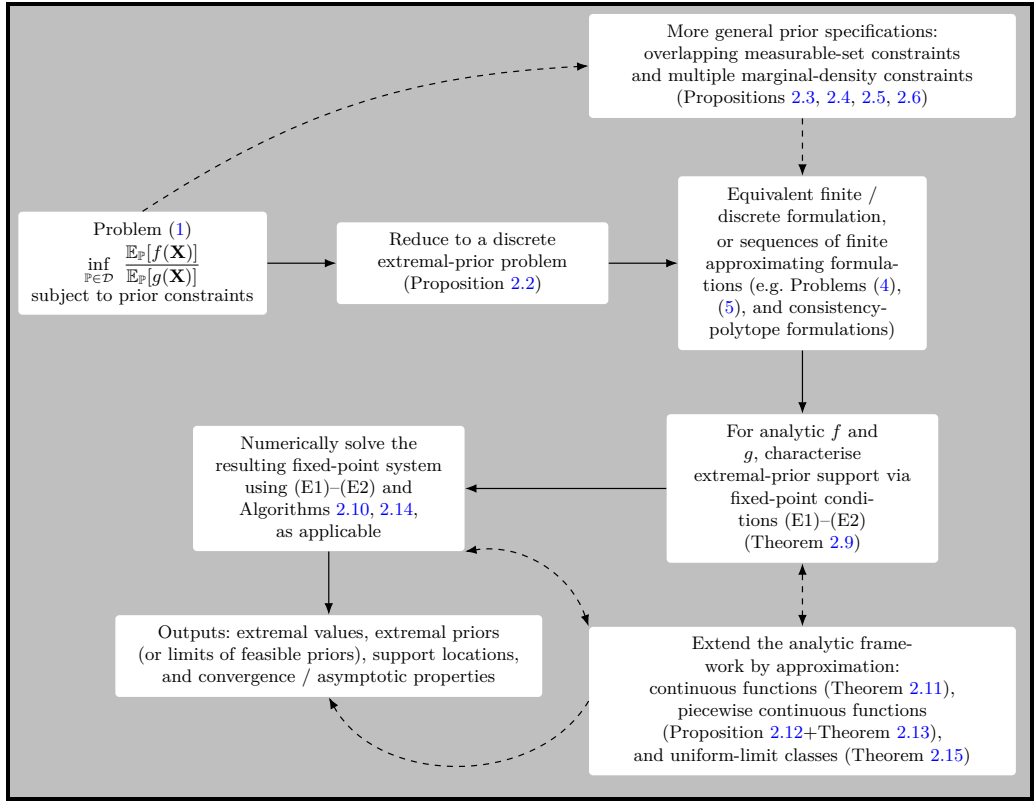

\noindent\textbf{Remarks: }$\mathcal S_b^+(K)\subseteq\overline{\mathcal R}_{2.13}^{\|\cdot\|_\infty}$, where $\mathcal S_b^+(K)$ is the set of  all simple non-negative Borel functions on $K$. Since $\overline{\mathcal R}_{2.13}^{\|\cdot\|_\infty}$ is a uniform closure, it contains the uniform closure of its subsets, so it contains $B^+_b(K)={\overline{\mathcal S_b^+(K)}}^{\|\cdot\|_\infty}\subseteq \overline{\mathcal R}_{2.13}^{\|\cdot\|_\infty}$, where $B^+_b(K)$ is the set of all non-negative bounded Borel functions on $K$. Conversely, the functions in $\mathcal R_{2.13}$ are non-negative, bounded, and Borel measurable. So, $\mathcal R_{2.13}\subseteq B^+_b(K)$. The uniform closure of such functions preserves these properties, and $B^+_b(K)$ is also closed in the uniform norm so contains the uniform closure of its subsets, hence $\overline{\mathcal R}_{2.13}^{\|\cdot\|_\infty}\subseteq B^+_b(K)$. Thus, $\overline{\mathcal R}_{2.13}^{\|\cdot\|_\infty}= B^+_b(K)$.

Since $f,g\in B^+_b(K)$ one can uniformly approximate $f$ and $g$ using canonical adjacent simple functions. For example, with $\delta_m\downarrow0$, construct $f_m^-:=\delta_m\left\lfloor\frac{f}{\delta_m}\right\rfloor$, $f_m^+:=\delta_m\left\lceil\frac{f}{\delta_m}\right\rceil$,
and similarly for $g_m^\pm$. Alternatively, let $\Pi_m=\{C_{m,1},\ldots,C_{m,N_m}\}$ be a nested sequence of finite Borel partitions of $K$ such that $\operatorname{mesh}(\Pi_m)
:=
\max_{1\leqslant r\leqslant N_m}\operatorname{diam}(C_{m,r})
\longrightarrow 0$. If $f$ and $g$ are uniformly continuous on $K$, define, for $\mathbf{x}\in C_{m,r}$, $\,f_m^-(\mathbf{x}):=\inf_{\mathbf{y}\in C_{m,r}} f(\mathbf{y})$, $\,f_m^+(\mathbf{x}):=\sup_{\mathbf{y}\in C_{m,r}} f(\mathbf{y})$, 
and similarly for $g_m^\pm$. 

With either of these example approximations, $\,f_m^-\leqslant f\leqslant f_m^+\,$ and $\,g_m^-\leqslant g\leqslant g_m^+$
with uniformly vanishing gaps; i.e. $\,\|f_m^+-f_m^-\|_\infty\longrightarrow0$, $\,\|g_m^+-g_m^-\|_\infty\longrightarrow0$. For every admissible $\mathbb P$ and all sufficiently large $m$, $g_m^\pm$ are close enough to $g$ to inherit $g$'s positivity, and
\begin{equation}
\frac{\mathbb E_{\mathbb P}[f_m^-]}{\mathbb E_{\mathbb P}[g_m^+]}
\leqslant
\frac{\mathbb E_{\mathbb P}[f]}{\mathbb E_{\mathbb P}[g]}
\leqslant
\frac{\mathbb E_{\mathbb P}[f_m^+]}{\mathbb E_{\mathbb P}[g_m^-]}.
\label{eqn_simplefunctionsandwich}
\end{equation}

Hence, the extremal objective is certified between two adjacent fixed-point problems:
\begin{equation}
\inf_{\mathbb P\in\mathcal D}\frac{\mathbb E_{\mathbb P}[f_m^-]}{\mathbb E_{\mathbb P}[g_m^+]}
\leqslant
\inf_{\mathbb P\in\mathcal D}\frac{\mathbb E_{\mathbb P}[f]}{\mathbb E_{\mathbb P}[g]}
\leqslant
\inf_{\mathbb P\in\mathcal D}\frac{\mathbb E_{\mathbb P}[f_m^+]}{\mathbb E_{\mathbb P}[g_m^-]}.
\label{eqn_extremisesimplefunctionsandwich}
\end{equation}

Each pair $(f_m^-,g_m^+)$ and $(f_m^+,g_m^-)$ admits simple-function numerator and denominator representations on partitions consisting of finitely many Borel level sets. The common refinement of the numerator and denominator partitions brings each pair within the scope of Theorem~\ref{thm_refinedgensol}. For all sufficiently large $m$, these pairs have benefits:
\begin{itemize}
    \item \textbf{Systematic construction:}
    ``canonical'' approximating sequences for $f$ and $g$.

    \item \textbf{Certified and convergent bounds:}
    each approximation stage gives a lower and upper bound on the
    extremal value, and these bounds converge to the extremal value.

    \item \textbf{Stagewise support characterisation:}
    at stage $m$, each bounding fixed-point problem has an extremal distribution whose support points provide approximate extremal support locations for the target problem.

    \item \textbf{Spatial localisation under domain refinement:}
    when the approximations are based on partitions with vanishing mesh (such as our example construction using $\Pi_m$ with uniformly continuous $f$, $g$),
    every stage-$m$ support point lies in a partition cell whose
    diameter tends to zero. Along any convergent subsequence of the
    stagewise support points, the corresponding cells localise the limiting extremal support point.
    
    \item \textbf{Generality of value quantisation:}
    when approximations are based on value quantisation (e.g. our first example construction using $\delta_m$) this guarantees uniform approximation for every $f,g\in B_b^+(K)$. The level-set cells need not have vanishing spatial diameter and, therefore, need not provide comparable spatial localisation compared with ``vanishing mesh'' approximations.
\end{itemize}




In summary, Theorem~\ref{thm:regulated_extension} is a computational route for solving \eqref{eqn_weak_opt_equiv_disc} for a general class of functions: solve increasingly accurate approximating problems, use their extremal objective forms guaranteed by Theorem~\ref{thm_refinedgensol}, recover arbitrarily accurate approximations to $\phi^*$, construct asymptotically optimal feasible priors, and identify extremal weak-limit priors for the Theorem~\ref{thm:regulated_extension} objective. In particular, \textbf{i)}~to approximate $\phi^*$, solve the approximating problems from Theorem~\ref{thm_refinedgensol} and monitor the convergence of the extremal values $\phi_m^*\to\phi^*$; \textbf{ii)}~to approximate an extremal prior, one should examine the corresponding extremal support tuples for clustering, convergent subsequences, or approximate stabilisation.
\section{Applications}
\label{subsec_examples}
The following examples illustrate this paper's methodological pipeline (see Figure~\ref{fig_CBIMethodologyPipeline}):
\subsection{A 1-Dimensional Extremal-Prior Problem With Two Support-Points}$\varepsilon, b, \theta, n$ are constants such that $\varepsilon,b,\theta\in(0,1)$ and $n\in{\mathbb N}\cup\{0\}$. For $F:[0,1]\to\mathbb R$, consider
\begin{gather}
\inf_{0\leqslant x \leqslant \varepsilon\leqslant y \leqslant 1}\frac{\theta(1-x)^{n}F(x)\,+\,(1-\theta)(1-y)^{n}F(y)}{(1-x)^{n}\theta + (1-y)^{n}(1-\theta)}\,.
\label{eqn_2intervalGenForm}    
\end{gather}
So, $f(z):=(1-z)^n F(z)$ and $g(z):=(1-z)^n$. Figures~\ref{figa:simpleCBI_step-univariate_approx-power-and-trunc}, \ref{figa:simpleCBI_univariate_infima-convergence}, \ref{figa:simpleCBI_bivariate_infima-convergence}, and Figures~\ref{figb:simpleCBI_step-univariate_approx-power-and-trunc}, \ref{figb:simpleCBI_univariate_infima-convergence}, \ref{figb:simpleCBI_bivariate_infima-convergence} illustrate Theorem~\ref{thm:regulated_extension}. Note that when $F$ is analytic, Theorem~\ref{thm:global_fp_tuple} applies to \eqref{eqn_2intervalGenForm} without ``$m$'' approximations. More examples illustrate Theorem~\ref{thm:regulated_extension} in Figures~\ref{figa:nontrivial-approximation-illustrations}, \ref{figa:nontrivial-convergence-illustrations}, \ref{figb:nontrivial-approximation-illustrations}, \ref{figb:nontrivial-convergence-illustrations}.

\newlength{\subfigwidthA}
\newlength{\subfigwidthB}
\setlength{\subfigwidthA}{0.4\textwidth}
\setlength{\subfigwidthB}{0.4\textwidth}

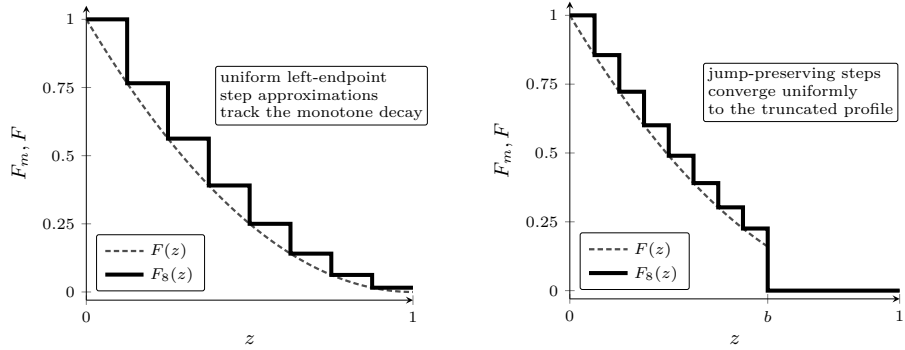
\begin{figure}[htbp!]
\centering

\subfloat[Approximators to $F(z)=(1-z)^2$ on $0\leqslant z\leqslant 1$. With the uniform partition $I_{j,m}=[(j-1)/m,j/m)$ for $j<m$ and $I_{m,m}=\{z:(m-1)/m\leqslant z\leqslant 1\}$, define $F_m(z)=\sum_{j=1}^m (1-(j-1)/m)^2\,\mathbf 1_{I_{j,m}}(z)$.%
\label{figa:simpleCBI_step-univariate_approx-power-and-trunc}%
]{%
\resizebox{\subfigwidthA}{!}{%
\begin{tikzpicture}
\begin{axis}[
    width=0.48\textwidth,
    height=6.2cm,
    xmin=0, xmax=1,
    ymin=-0.03, ymax=1.05,
    axis lines=left,
    xtick={0,1},
    ytick={0,0.25,0.5,0.75,1},
    xlabel={$z$},
    ylabel={$F_m,F$},
    every axis label/.style={font=\small},
    tick label style={font=\scriptsize},
    clip=false,
    legend style={
        draw=black,
        fill=white,
        rounded corners=1pt,
        font=\scriptsize,
        at={(0.03,0.03)},
        anchor=south west
    },
    legend cell align=left
]

\addplot[
    smooth,
    domain=0:1,
    samples=200,
    line width=1.0pt,
    color=gray!60!black,
    dash pattern=on 2.4pt off 1.4pt,
    forget plot
] {(1-x)^2};

\addplot[
    const plot,
    line width=1.8pt,
    color=black,
    forget plot
] coordinates {
    (0.000000,1.000000)
    (0.125000,0.765625)
    (0.250000,0.562500)
    (0.375000,0.390625)
    (0.500000,0.250000)
    (0.625000,0.140625)
    (0.750000,0.062500)
    (0.875000,0.015625)
    (1.000000,0.015625)
};

\addlegendimage{no markers, color=gray!60!black, dash pattern=on 2.4pt off 1.4pt, line width=1.0pt}
\addlegendentry{$F(z)$}
\addlegendimage{no markers, color=black, solid, line width=1.8pt}
\addlegendentry{$F_8(z)$}

\node[
    draw=black, fill=white, rounded corners=1pt,
    inner sep=2pt, font=\scriptsize, align=left
] at (axis cs:0.72,0.72)
{uniform left-endpoint\\step approximations\\track the monotone decay};
\end{axis}
\end{tikzpicture}%
}%
}%
\hspace{0.05\textwidth}
\subfloat[Approximators to $F(z)=(1-z)^2\mathbf 1_{z\leqslant b}$. Partition $0\leqslant z \leqslant b$ into $I_{j,m}=\{z:(j-1)b/m\leqslant z < jb/m\}$ for $j<m$ and $I_{m,m}=\{z:(m-1)b/m\leqslant z\leqslant b\}$, and set $F_m(z)=\sum_{j=1}^m (1-(j-1)b/m)^2\,\mathbf 1_{I_{j,m}}(z)$ on $0\leqslant z\leqslant b$, with $F_m(z)=0$ on $b<z\leqslant 1$.%
\label{figb:simpleCBI_step-univariate_approx-power-and-trunc}%
]{%
\resizebox{\subfigwidthB}{!}{%
\begin{tikzpicture}
\begin{axis}[
    width=0.48\textwidth,
    height=6.2cm,
    xmin=0, xmax=1,
    ymin=-0.03, ymax=1.05,
    axis lines=left,
    xtick={0,0.6,1},
    xticklabels={$0$,$b$,$1$},
    ytick={0,0.25,0.5,0.75,1},
    xlabel={$z$},
    ylabel={$F_m,F$},
    every axis label/.style={font=\small},
    tick label style={font=\scriptsize},
    clip=false,
    legend style={
        draw=black,
        fill=white,
        rounded corners=1pt,
        font=\scriptsize,
        at={(0.03,0.03)},
        anchor=south west
    },
    legend cell align=left
]

\addplot[
    smooth,
    domain=0:0.6,
    samples=200,
    line width=1.0pt,
    color=gray!60!black,
    dash pattern=on 2.4pt off 1.4pt,
    forget plot
] {(1-x)^2};

\addplot[
    line width=1.0pt,
    color=gray!60!black,
    dash pattern=on 2.4pt off 1.4pt,
    forget plot
] coordinates {
    (0.600000,0.160000)
    (0.600000,0.000000)
};

\addplot[
    line width=1.0pt,
    color=gray!60!black,
    dash pattern=on 2.4pt off 1.4pt,
    forget plot
] coordinates {
    (0.600000,0.000000)
    (1.000000,0.000000)
};

\addplot[
    const plot,
    line width=1.8pt,
    color=black,
    forget plot
] coordinates {
    (0.000000,1.000000)
    (0.075000,0.855625)
    (0.150000,0.722500)
    (0.225000,0.600625)
    (0.300000,0.490000)
    (0.375000,0.390625)
    (0.450000,0.302500)
    (0.525000,0.225625)
    (0.600000,0.225625)
    (0.600000,0.000000)
    (1.000000,0.000000)
};

\addlegendimage{no markers, color=gray!60!black, dash pattern=on 2.4pt off 1.4pt, line width=1.0pt}
\addlegendentry{$F(z)$}
\addlegendimage{no markers, color=black, solid, line width=1.8pt}
\addlegendentry{$F_8(z)$}

\node[
    draw=black, fill=white, rounded corners=1pt,
    inner sep=2pt, font=\scriptsize, align=left
] at (axis cs:0.7,0.72)
{jump-preserving steps\\converge uniformly\\to the truncated profile};
\end{axis}
\end{tikzpicture}%
}%
}%
\caption[Step approximations to continuous and truncated F functions]{Approximators for two $F$ functions for \eqref{eqn_2intervalGenForm}.  
The black curves approximate the limiting dashed $F$ functions. The indicators $\mathbf 1_{\mathtt S}$ equal 1 when boolean $\mathtt S$ is true, and equal 0 otherwise.}
\label{fig:simpleCBI_step-univariate_approx-power-and-trunc}
\end{figure}
%
%
%
%
%
%
%
%
%
%
%
%
%
%
%
%
%

\setlength{\subfigwidthA}{0.44\textwidth}
\setlength{\subfigwidthB}{0.44\textwidth}

\begin{figure}[htbp!]
\centering

\subfloat[%
Converging $\phi_m^*$ from \eqref{eqn_2intervalGenForm}, for the approximants to $F(z)=(1-z)^2$ in Figure~\ref{figa:simpleCBI_step-univariate_approx-power-and-trunc}.%
\label{figa:simpleCBI_univariate_infima-convergence}%
]{%
\resizebox{\subfigwidthA}{!}{%
\begin{tikzpicture}[baseline=(current bounding box.north)]
\begin{axis}[
    width=0.47\textwidth,
    height=6.2cm,
    xmin=1, xmax=40,
    ymin=0.39, ymax=1.02,
    axis lines=left,
    xtick={1,5,10,15,20,25,30,35,40},
    ytick={0.4,0.5,0.6,0.7,0.8,0.9,1.0},
    xlabel={$m$},
    ylabel={infimum value},
    every axis label/.style={font=\small},
    tick style={black},
    grid=both,
    grid style={gray!20},
    legend style={
        draw=black,
        fill=white,
        font=\scriptsize,
        at={(0.97,0.97)},
        anchor=north east
    },
    clip=false
]

\addplot[
    color=black,
    line width=1.4pt,
    mark=*,
    mark size=1.4pt
] coordinates {
    (1,1.0000000000)
    (2,0.9492758251)
    (3,0.4440256637)
    (4,0.5413649271)
    (5,0.5532184323)
    (6,0.4312937324)
    (7,0.4705056644)
    (8,0.4932023905)
    (9,0.4210577590)
    (10,0.4497085578)
    (11,0.4734885756)
    (12,0.4180796486)
    (13,0.4424227854)
    (14,0.4644098664)
    (15,0.4182722256)
    (16,0.4395823909)
    (17,0.4572774835)
    (18,0.4195647836)
    (19,0.4363520603)
    (20,0.4031291143)
    (21,0.4189590408)
    (22,0.4326804930)
    (23,0.4043499582)
    (24,0.4180796486)
    (25,0.4306627409)
    (26,0.4054756790)
    (27,0.4179523431)
    (28,0.4296207141)
    (29,0.4067773257)
    (30,0.4182722256)
    (31,0.4291331694)
    (32,0.4081644253)
    (33,0.4185572943)
    (34,0.4275668815)
    (35,0.4090763088)
    (36,0.4180796486)
    (37,0.4265744608)
    (38,0.4094782465)
    (39,0.4179328573)
    (40,0.4016555642)
};
\addlegendentry{$\phi_m^*$}

\addplot[
    color=gray!60!black,
    dash pattern=on 2.4pt off 1.4pt,
    line width=1.0pt
] coordinates {
    (1,0.4016555642) (40,0.4016555642)
};
\addlegendentry{$\phi^*$}

\end{axis}
\end{tikzpicture}%
}%
}%
\hspace{0.05\textwidth}
\subfloat[%
Converging $\phi_m^*$ from \eqref{eqn_2intervalGenForm}, for the approximants to
$F(z)=(1-z)^2\mathbf{1}_{z\leqslant 0.6}$ in Figure~\ref{figb:simpleCBI_step-univariate_approx-power-and-trunc}. %
\label{figb:simpleCBI_univariate_infima-convergence}%
]{%
\resizebox{\subfigwidthB}{!}{%
\begin{tikzpicture}[baseline=(current bounding box.north)]
\begin{axis}[
    width=0.47\textwidth,
    height=6.2cm,
    xmin=1, xmax=40,
    ymin=0.39, ymax=1.02,
    axis lines=left,
    xtick={1,5,10,15,20,25,30,35,40},
    ytick={0.4,0.5,0.6,0.7,0.8,0.9,1.0},
    xlabel={$m$},
    ylabel={infimum value},
    every axis label/.style={font=\small},
    tick style={black},
    grid=both,
    grid style={gray!20},
    legend style={
        draw=black,
        fill=white,
        font=\scriptsize,
        at={(0.97,0.97)},
        anchor=north east
    },
    clip=false
]

\addplot[
    color=black,
    line width=1.4pt,
    mark=*,
    mark size=1.4pt
] coordinates {
    (1,1.0000000000)
    (2,0.4900000000)
    (3,0.5532184323)
    (4,0.4603089844)
    (5,0.5000987455)
    (6,0.4497085578)
    (7,0.4090763088)
    (8,0.4499333821)
    (9,0.4182722256)
    (10,0.4527638708)
    (11,0.4260095185)
    (12,0.4031291143)
    (13,0.4282652909)
    (14,0.4090763088)
    (15,0.4306627409)
    (16,0.4138945196)
    (17,0.4333668465)
    (18,0.4182722256)
    (19,0.4044798183)
    (20,0.4216274444)
    (21,0.4090763088)
    (22,0.4237929687)
    (23,0.4123457524)
    (24,0.4016555642)
    (25,0.4153959115)
    (26,0.4054561526)
    (27,0.4182722256)
    (28,0.4089414267)
    (29,0.4202469363)
    (30,0.4115516666)
    (31,0.4033011069)
    (32,0.4138945196)
    (33,0.4061107243)
    (34,0.4161302460)
    (35,0.4087415268)
    (36,0.4016893878)
    (37,0.4110683137)
    (38,0.4043495050)
    (39,0.4129713022)
    (40,0.4065705959)
};
\addlegendentry{$\phi_m^*$}

\addplot[
    color=gray!60!black,
    dash pattern=on 2.4pt off 1.4pt,
    line width=1.0pt
] coordinates {
    (1,0.4065705959) (40,0.4065705959)
};
\addlegendentry{$\phi^*$}

\end{axis}
\end{tikzpicture}%
}%
}%
\caption[1-dimensional convergence of approximating infima]{Converging infima from \eqref{eqn_2intervalGenForm} using the approximations in Figure~\ref{fig:simpleCBI_step-univariate_approx-power-and-trunc}: here, $n=10$, $\theta=0.5$, $\varepsilon=0.35$. This illustrates a 1D application of Theorem~\ref{thm:regulated_extension}. 
The black curves plot the approximating infima $\phi_m^*$. The grey dashed line is the last computed $\phi^*_m$ value, used as a numerical estimate of the limiting infimum $\phi^*\approx 0.401621$. For these parameter values, $\phi^*$ is the same in both figures: the continuous-problem minimizer satisfies $y^*\approx 0.42148<0.6=b$, so the truncation at $b$ is inactive at the optimizer.}
\label{fig:simpleCBI_univariate_infima-convergence}
\end{figure}
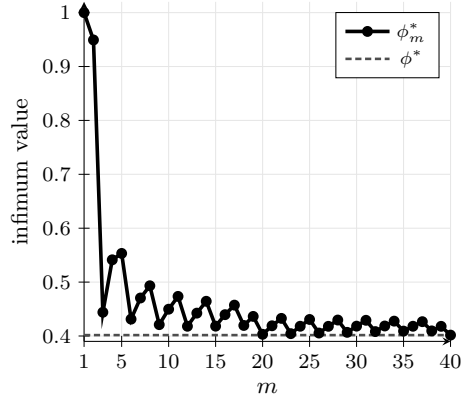
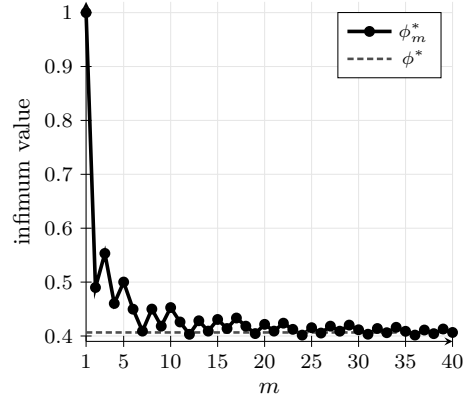
%
%
%
%
%
%
%

%
%
%
%
%
%
%
%
%
\setlength{\subfigwidthA}{0.45\textwidth}
\setlength{\subfigwidthB}{0.45\textwidth}

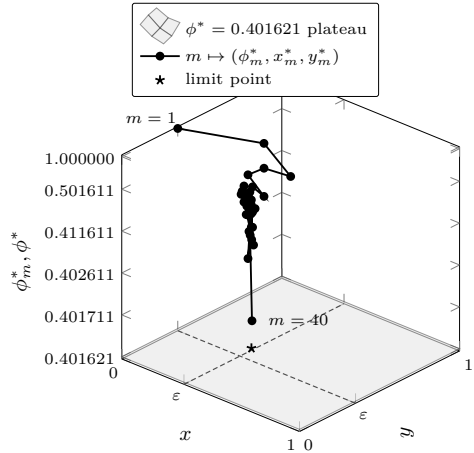
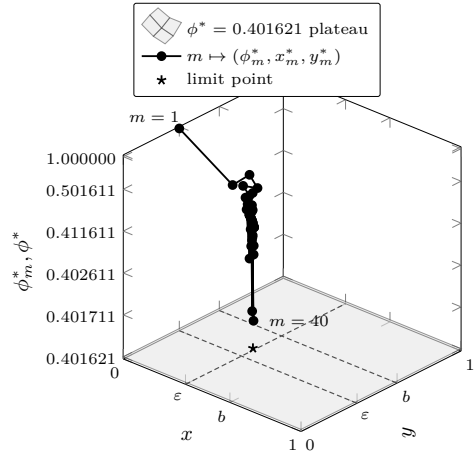
\begin{figure}[htbp!]
\centering

\subfloat[Converging infima and minimiser locations from $F_m$ approximants in Figure~\ref{figa:simpleCBI_step-univariate_approx-power-and-trunc}.%
\label{figa:simpleCBI_bivariate_infima-convergence}%
]{%
\resizebox{\subfigwidthA}{!}{%
\begin{tikzpicture}
\begin{axis}[
    width=0.48\textwidth,
    height=7.0cm,
    view={42}{28},
    xmin=0, xmax=1,
    ymin=0, ymax=1,
    zmin=-5.05, zmax=-0.18,
    axis lines=box,
    xlabel={$x$},
    ylabel={$y$},
    zlabel={$\phi_m^*,\phi^*$},
    every axis label/.style={font=\small},
    tick label style={font=\scriptsize},
    xtick={0,0.35,1},
    xticklabels={$0$,$\varepsilon$,$1$},
    ytick={0,0.35,1},
    yticklabels={$0$,$\varepsilon$,$1$},
    ztick={-5,-4,-3,-2,-1,-0.223016},
    zticklabels={$0.401621$,$0.401711$,$0.402611$,$0.411611$,$0.501611$,$1.000000$},
    legend style={
        draw=black,
        fill=white,
        rounded corners=1pt,
        font=\scriptsize,
        at={(0.03,1.2)},
        anchor=north west
    },
    legend cell align=left,
    clip=false,
]
\addplot3[
    surf,
    shader=flat,
    draw=gray!60!black,
    fill=gray!20,
    opacity=0.60,
    domain=0:1,
    y domain=0:1,
    samples=2,
    samples y=2
] {-5.000000};
\addlegendentry{$\phi^*=0.401621$ plateau}

\addplot3[
    color=black,
    line width=0.8pt,
    mark=*,
    mark size=1.5pt
] coordinates {
    (0.000000,0.350000,-0.223016) (0.350000,0.500000,-0.261485) (0.350000,0.666667,-1.372485) (0.350000,0.500000,-0.854636) (0.350000,0.400000,-0.819280) (0.350000,0.500000,-1.527497) (0.350000,0.428571,-1.161815) (0.350000,0.375000,-1.038146) (0.350000,0.444444,-1.711155) (0.350000,0.400000,-1.317878) (0.350000,0.363636,-1.143407) (0.350000,0.416667,-1.783344) (0.350000,0.384615,-1.389215) (0.350000,0.357143,-1.202049) (0.350000,0.400000,-1.778295) (0.350000,0.375000,-1.420544) (0.350000,0.411765,-1.254407) (0.350000,0.388889,-1.745846) (0.350000,0.421053,-1.459158) (0.350000,0.400000,-2.818719) (0.350000,0.428571,-1.760752) (0.350000,0.409091,-1.507667) (0.350000,0.434783,-2.562428) (0.350000,0.416667,-1.783344) (0.350000,0.400000,-1.536829) (0.350000,0.423077,-2.412896) (0.350000,0.407407,-1.786714) (0.350000,0.392857,-1.552693) (0.350000,0.413793,-2.286825) (0.350000,0.400000,-1.778295) (0.350000,0.419355,-1.560319) (0.350000,0.406250,-2.183537) (0.350000,0.424242,-1.770927) (0.350000,0.411765,-1.585766) (0.350000,0.428571,-2.126957) (0.350000,0.416667,-1.783344) (0.350000,0.405405,-1.602697) (0.350000,0.421053,-2.104182) (0.350000,0.410256,-1.787233) (0.350000,0.425000,-4.351823)
};
\addlegendentry{$m\mapsto(\phi^*_m,x_m^*,y_m^*)$}

\addplot3[
    only marks,
    mark=star,
    mark options={scale=1.05, draw=black, fill=white},
    line width=0.8pt
] coordinates {
    (0.350000,0.421481,-5.000000)
};
\addlegendentry{limit point}

\addplot3[dash pattern=on 2.2pt off 1.3pt, color=gray!60!black, line width=0.5pt]
coordinates {(0.35,0,-5.05) (0.35,1,-5.05)};
\addplot3[dash pattern=on 2.2pt off 1.3pt, color=gray!60!black, line width=0.5pt]
coordinates {(0,0.35,-5.05) (1,0.35,-5.05)};

\node[font=\scriptsize, fill=white, inner sep=1pt] at (axis cs:-0.15,0.350000,-0.223016) {$m=1$};
\node[font=\scriptsize, inner sep=1pt] at (axis cs:0.49000,0.56000,-4.351823) {$m=40$};
\end{axis}
\end{tikzpicture}%
}%
}%
\hspace{0.05\textwidth}
\subfloat[Converging infima and minimiser locations from $F_m$ approximants in Figure~\ref{figb:simpleCBI_step-univariate_approx-power-and-trunc}.%
\label{figb:simpleCBI_bivariate_infima-convergence}%
]{%
\resizebox{\subfigwidthB}{!}{%
\begin{tikzpicture}
\begin{axis}[
    width=0.48\textwidth,
    height=7.0cm,
    view={42}{28},
    xmin=0, xmax=1,
    ymin=0, ymax=1,
    zmin=-5.05, zmax=-0.18,
    axis lines=box,
    xlabel={$x$},
    ylabel={$y$},
    zlabel={$\phi_m^*,\phi^*$},
    every axis label/.style={font=\small},
    tick label style={font=\scriptsize},
    xtick={0,0.35,0.6,1},
    xticklabels={$0$,$\varepsilon$,$b$,$1$},
    ytick={0,0.35,0.6,1},
    yticklabels={$0$,$\varepsilon$,$b$,$1$},
    ztick={-5,-4,-3,-2,-1,-0.223016},
    zticklabels={$0.401621$,$0.401711$,$0.402611$,$0.411611$,$0.501611$,$1.000000$},
    legend style={
        draw=black,
        fill=white,
        rounded corners=1pt,
        font=\scriptsize,
        at={(0.03,1.2)},
        anchor=north west
    },
    legend cell align=left,
    clip=false,
]
\addplot3[
    surf,
    shader=flat,
    draw=gray!60!black,
    fill=gray!20,
    opacity=0.60,
    domain=0:1,
    y domain=0:1,
    samples=2,
    samples y=2
] {-5.000000};
\addlegendentry{$\phi^*=0.401621$ plateau}

\addplot3[
    color=black,
    line width=0.8pt,
    mark=*,
    mark size=1.7pt
] coordinates {
    (0.000000,0.350000,-0.223016) (0.300000,0.350000,-1.053602) (0.350000,0.400000,-0.819280) (0.350000,0.450000,-1.231377) (0.350000,0.360000,-1.006618) (0.350000,0.400000,-1.317878) (0.350000,0.428571,-2.126957) (0.350000,0.375000,-1.315852) (0.350000,0.400000,-1.778295) (0.350000,0.420000,-1.291131) (0.350000,0.381818,-1.612638) (0.350000,0.400000,-2.818719) (0.350000,0.415385,-1.574234) (0.350000,0.428571,-2.126957) (0.350000,0.400000,-1.536829) (0.350000,0.412500,-1.910680) (0.350000,0.388235,-1.498177) (0.350000,0.400000,-1.778295) (0.350000,0.410526,-2.542310) (0.350000,0.420000,-1.698615) (0.350000,0.428571,-2.126957) (0.350000,0.409091,-1.654002) (0.350000,0.417391,-1.969211) (0.350000,0.425000,-4.351823) (0.350000,0.408000,-1.860599) (0.350000,0.415385,-2.415096) (0.350000,0.400000,-1.778295) (0.350000,0.407143,-2.134876) (0.350000,0.413793,-1.729651) (0.350000,0.420000,-2.002588) (0.350000,0.425806,-2.772107) (0.350000,0.412500,-1.910680) (0.350000,0.418182,-2.346822) (0.350000,0.405882,-1.838058) (0.350000,0.411429,-2.146883) (0.350000,0.416667,-4.106211) (0.350000,0.421622,-2.024236) (0.350000,0.426316,-2.562500) (0.350000,0.415385,-1.944613) (0.350000,0.420000,-2.304561)
};
\addlegendentry{$m\mapsto(\phi^*_m,x_m^*,y_m^*)$}

\addplot3[
    only marks,
    mark=star,
    mark options={scale=1.05, draw=black, fill=white},
    line width=0.8pt
] coordinates {
    (0.350000,0.421481,-5.000000)
};
\addlegendentry{limit point}

\addplot3[dash pattern=on 2.2pt off 1.3pt, color=gray!60!black, line width=0.5pt]
coordinates {(0.35,0,-5.05) (0.35,1,-5.05)};
\addplot3[dash pattern=on 2.2pt off 1.3pt, color=gray!60!black, line width=0.5pt]
coordinates {(0,0.35,-5.05) (1,0.35,-5.05)};
\addplot3[dash pattern=on 2.2pt off 1.3pt, color=gray!60!black, line width=0.5pt]
coordinates {(0,0.6,-5.05) (1,0.6,-5.05)};

\node[font=\scriptsize, fill=white, inner sep=1pt] at (axis cs:-0.15,0.36,-0.22) {$m=1$};
\node[font=\scriptsize, inner sep=1pt] at (axis cs:0.520000,0.52,-4.24561) {$m=40$};
\end{axis}
\end{tikzpicture}%
}%
}%
\caption[2-dimensional convergence of infima and fixed-point triples]{Convergence of infima from approximations when \eqref{eqn_2intervalGenForm} is treated as a 2D application of Theorem~\ref{thm:regulated_extension} over the subset $[0,\varepsilon]\times[\varepsilon,1]$ of the unit square. This 2D reformulation uses the bivariate functions $f_m(x,y)=\theta(1-x)^nF_m(x)+(1-\theta)(1-y)^nF_m(y)$, $g(x,y)=\theta(1-x)^n+(1-\theta)(1-y)^n$,
so the plotted infima $\phi^*_m$ are $\inf f_m(x,y)/g(x,y)$ over
$0\leqslant x\leqslant \varepsilon\leqslant y\leqslant 1$. These optimisations are equivalent to those in Figure~\ref{fig:simpleCBI_univariate_infima-convergence}. The black points are the sequences of fixed-point triples from each $m$-th approximating problem.
}
\label{fig:simpleCBI_bivariate_infima-convergence}
\end{figure}
%
%
%
%
%

%
%
%

\setlength{\subfigwidthA}{0.44\textwidth}
\setlength{\subfigwidthB}{0.44\textwidth}

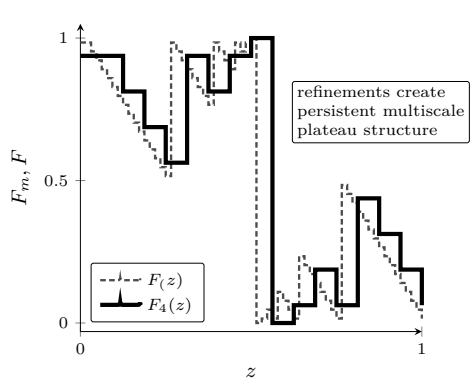
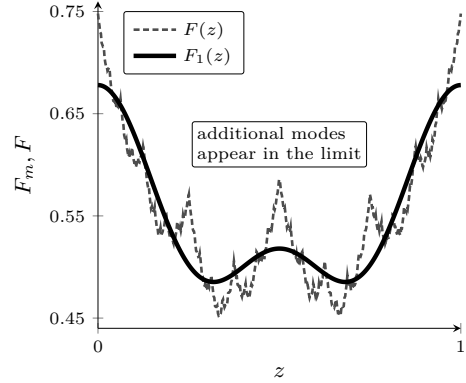
\begin{figure}[htbp!]
\centering

\subfloat[Start from $F_0\equiv \tfrac12$. At stage $m$, each interval is split in half. If the parent height is $1$, the children are $(1-2^{-m},\,1)$; if it is $0$, they are $(0,\,2^{-m})$; otherwise they are $\bigl(\min\{1,h+2^{-m}\},\,\max\{0,h-2^{-m}\}\bigr)$. Shown are $F_4$ and a finer approximation representing limit $F=\lim_{m\to\infty}F_m$.%
\label{figa:nontrivial-approximation-illustrations}%
]{%
\resizebox{\subfigwidthA}{!}{%
\begin{tikzpicture}
\begin{axis}[
    width=0.48\textwidth,
    height=6.2cm,
    xmin=0, xmax=1,
    ymin=-0.03, ymax=1.05,
    axis lines=left,
    xtick={0,1},
    ytick={0,0.5,1},
    xlabel={$z$},
    ylabel={$F_m,F$},
    every axis label/.style={font=\small},
    tick label style={font=\scriptsize},
    clip=false,
    legend style={
        draw=black,
        fill=white,
        rounded corners=1pt,
        font=\scriptsize,
        at={(0.03,0.03)},
        anchor=south west
    },
    legend cell align=left
]

\addplot[
    const plot,
    line width=1.0pt,
    color=gray!60!black,
    dash pattern=on 2.4pt off 1.4pt
] coordinates {
    (0.000000,0.984375)
    (0.015625,0.984375)
    (0.031250,0.953125)
    (0.046875,0.921875)
    (0.062500,0.890625)
    (0.078125,0.859375)
    (0.093750,0.828125)
    (0.109375,0.796875)
    (0.125000,0.765625)
    (0.140625,0.734375)
    (0.156250,0.703125)
    (0.171875,0.671875)
    (0.187500,0.640625)
    (0.203125,0.609375)
    (0.218750,0.578125)
    (0.234375,0.546875)
    (0.250000,0.515625)
    (0.265625,0.984375)
    (0.281250,0.953125)
    (0.296875,0.921875)
    (0.312500,0.890625)
    (0.328125,0.859375)
    (0.343750,0.828125)
    (0.359375,0.796875)
    (0.375000,0.765625)
    (0.390625,0.984375)
    (0.406250,0.953125)
    (0.421875,0.921875)
    (0.437500,0.890625)
    (0.453125,0.984375)
    (0.468750,0.953125)
    (0.484375,0.984375)
    (0.500000,1.000000)
    (0.515625,0.000000)
    (0.531250,0.015625)
    (0.546875,0.046875)
    (0.562500,0.015625)
    (0.578125,0.109375)
    (0.593750,0.078125)
    (0.609375,0.046875)
    (0.625000,0.015625)
    (0.640625,0.234375)
    (0.656250,0.203125)
    (0.671875,0.171875)
    (0.687500,0.140625)
    (0.703125,0.109375)
    (0.718750,0.078125)
    (0.734375,0.046875)
    (0.750000,0.015625)
    (0.765625,0.484375)
    (0.781250,0.453125)
    (0.796875,0.421875)
    (0.812500,0.390625)
    (0.828125,0.359375)
    (0.843750,0.328125)
    (0.859375,0.296875)
    (0.875000,0.265625)
    (0.890625,0.234375)
    (0.906250,0.203125)
    (0.921875,0.171875)
    (0.937500,0.140625)
    (0.953125,0.109375)
    (0.968750,0.078125)
    (0.984375,0.046875)
    (1.000000,0.015625)
};
\addlegendentry{$F_(z)$}

\addplot[
    const plot,
    line width=1.8pt,
    color=black
] coordinates {
    (0.000000,0.937500)
    (0.062500,0.937500)
    (0.125000,0.812500)
    (0.187500,0.687500)
    (0.250000,0.562500)
    (0.312500,0.937500)
    (0.375000,0.812500)
    (0.437500,0.937500)
    (0.500000,1.000000)
    (0.562500,0.000000)
    (0.625000,0.062500)
    (0.687500,0.187500)
    (0.750000,0.062500)
    (0.812500,0.437500)
    (0.875000,0.312500)
    (0.937500,0.187500)
    (1.000000,0.062500)
};
\addlegendentry{$F_4(z)$}

\node[
    draw=black, fill=white, rounded corners=1pt,
    inner sep=2pt, font=\scriptsize, align=left
] at (axis cs:0.88,0.74)
{refinements create\\persistent multiscale\\plateau structure};
\end{axis}
\end{tikzpicture}%
}%
}%
\hspace{0.05\textwidth}
\subfloat[The approximants are
$F_m(z)=0.55+\sum_{j=0}^m 0.08\cdot 0.6^j\cos(2^{j+1}\pi z)$. Shown are $F_1$ and a finer approximation representing limit $F=\lim_{m\to\infty}F_m$.%
\label{figb:nontrivial-approximation-illustrations}%
]{%
\resizebox{\subfigwidthB}{!}{%
\begin{tikzpicture}
\begin{axis}[
    width=0.48\textwidth,
    height=6.2cm,
    xmin=0, xmax=1,
    ymin=0.44, ymax=0.76,
    axis lines=left,
    xtick={0,1},
    ytick={0.45,0.55,0.65,0.75},
    xlabel={$z$},
    ylabel={$F_m,F$},
    every axis label/.style={font=\small},
    tick label style={font=\scriptsize},
    clip=false,
    samples=500,
    domain=0:1,
    legend style={
        draw=black,
        fill=white,
        rounded corners=1pt,
        font=\scriptsize,
        at={(0.07,0.97)},
        anchor=north west
    },
    legend cell align=left
]

\addplot[
    color=gray!60!black,
    line width=1.05pt,
    dash pattern=on 2.4pt off 1.4pt
] expression {
    0.55
    + 0.08*cos(deg(2*pi*x))
    + 0.048*cos(deg(4*pi*x))
    + 0.0288*cos(deg(8*pi*x))
    + 0.01728*cos(deg(16*pi*x))
    + 0.010368*cos(deg(32*pi*x))
    + 0.0062208*cos(deg(64*pi*x))
    + 0.00373248*cos(deg(128*pi*x))
    + 0.002239488*cos(deg(256*pi*x))
    + 0.0013436928*cos(deg(512*pi*x))
};
\addlegendentry{$F(z)$}

\addplot[
    color=black,
    line width=1.8pt,
    samples=300,
    domain=0:1
] {0.55
    + 0.08*cos(deg(2*pi*x))
    + 0.048*cos(deg(4*pi*x))
};
\addlegendentry{$F_1(z)$}

\node[
    draw=black, fill=white, rounded corners=1pt,
    inner sep=2pt, font=\scriptsize, align=left
] at (axis cs:0.50,0.62)
{additional modes\\appear in the limit};
\end{axis}
\end{tikzpicture}%
}%
}%

\caption[Uniform-limit constructions from piecewise-continuous and analytic approximants]{Two uniform-limit constructions: Figure~\ref{figa:nontrivial-approximation-illustrations} shows piecewise continuous approximations (Theorem~\ref{thm:regulated_extension}). Figure~\ref{figb:nontrivial-approximation-illustrations} shows analytic approximants (Theorem~\ref{thm_gensol}).}
\label{fig:nontrivial-approximation-illustrations}
\end{figure}
%
%
%
%
%
%
%

\setlength{\subfigwidthA}{0.45\textwidth}
\setlength{\subfigwidthB}{0.45\textwidth}

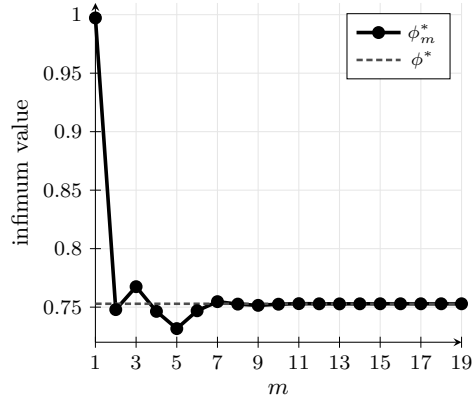
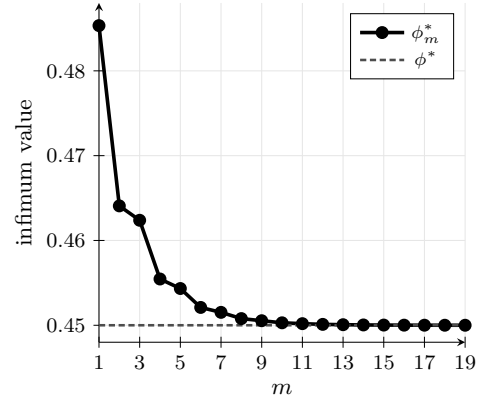
\begin{figure}[htbp!]
\centering

\subfloat[%
Numerical convergence of $\phi_m^*$ for the approximants in Figure~\ref{figa:nontrivial-approximation-illustrations}. The corresponding objective is $\phi_m^*=\inf\limits_{0\leqslant x\leqslant 0.1\leqslant y\leqslant 1}\frac{0.5(1-x)^{10}F_m(x)+0.5(1-y)^{10}F_m(y)}{0.5(1-x)^{10}+0.5(1-y)^{10}}$.%
\label{figa:nontrivial-convergence-illustrations}%
]{%
\resizebox{\subfigwidthA}{!}{%
\begin{tikzpicture}[baseline=(current bounding box.north)]
\begin{axis}[
    width=0.47\textwidth,
    height=6.2cm,
    xmin=1, xmax=19,
    ymin=0.72, ymax=1.01,
    axis lines=left,
    xtick={1,3,5,7,9,11,13,15,17,19},
    ytick={0.75,0.80,0.85,0.90,0.95,1.00},
    xlabel={$m$},
    ylabel={infimum value},
    every axis label/.style={font=\small},
    tick style={black},
    grid=both,
    grid style={gray!20},
    legend style={
        draw=black,
        fill=white,
        font=\scriptsize,
        at={(0.97,0.97)},
        anchor=north east
    },
    clip=false
]

\addplot[
    color=black,
    line width=1.4pt,
    mark=*,
    mark size=1.8pt
] coordinates {
    (1,0.9972070684)
    (2,0.7479053013)
    (3,0.7674912754)
    (4,0.7463788519)
    (5,0.7316591390)
    (6,0.7468971375)
    (7,0.7546941345)
    (8,0.7525883112)
    (9,0.7515276480)
    (10,0.7525042105)
    (11,0.7529924918)
    (12,0.7528599089)
    (13,0.7527935912)
    (14,0.7528546263)
    (15,0.7528851417)
    (16,0.7528768484)
    (17,0.7528727016)
    (18,0.7528765162)
    (19,0.7528784236)
};
\addlegendentry{$\phi_m^*$}

\addplot[
    color=gray!60!black,
    dash pattern=on 2.4pt off 1.4pt,
    line width=1.0pt
] coordinates {
    (1,0.7528784236) (19,0.7528784236)
};
\addlegendentry{$\phi^*$}

\end{axis}
\end{tikzpicture}%
}%
}%
\hspace{0.05\textwidth}
\subfloat[%
Numerical convergence of $\phi_m^*$ for the approximants in Figure~\ref{figb:nontrivial-approximation-illustrations}.
The corresponding objective is
$\phi_m^*=\inf\limits_{0\leqslant x\leqslant 0.35\leqslant y\leqslant 1}
\frac{0.5(1-x)^{10}F_m(x)+0.5(1-y)^{10}F_m(y)}
{0.5(1-x)^{10}+0.5(1-y)^{10}}$.%
\label{figb:nontrivial-convergence-illustrations}%
]{%
\resizebox{\subfigwidthB}{!}{%
\begin{tikzpicture}[baseline=(current bounding box.north)]
\begin{axis}[
    width=0.47\textwidth,
    height=6.2cm,
    xmin=1, xmax=19,
    ymin=0.448, ymax=0.488,
    axis lines=left,
    xtick={1,3,5,7,9,11,13,15,17,19},
    ytick={0.45,0.46,0.47,0.48},
    xlabel={$m$},
    ylabel={infimum value},
    every axis label/.style={font=\small},
    tick style={black},
    grid=both,
    grid style={gray!20},
    legend style={
        draw=black,
        fill=white,
        font=\scriptsize,
        at={(0.97,0.97)},
        anchor=north east
    },
    clip=false
]

\addplot[
    color=black,
    line width=1.4pt,
    mark=*,
    mark size=1.8pt
] coordinates {
    (1,0.4853333333)
    (2,0.4640768142)
    (3,0.4623806458)
    (4,0.4554562582)
    (5,0.4543342545)
    (6,0.4521037516)
    (7,0.4515229729)
    (8,0.4507885507)
    (9,0.4505378308)
    (10,0.4502912045)
    (11,0.4501908267)
    (12,0.4501065910)
    (13,0.4500679687)
    (14,0.4500387996)
    (15,0.4500242813)
    (16,0.4500140776)
    (17,0.4500086935)
    (18,0.4500050924)
    (19,0.4500031198)
};
\addlegendentry{$\phi_m^*$}

\addplot[
    color=gray!60!black,
    dash pattern=on 2.4pt off 1.4pt,
    line width=1.0pt
] coordinates {
    (1,0.4500031198) (19,0.4500031198)
};
\addlegendentry{$\phi^*$}

\end{axis}
\end{tikzpicture}%
}%
}%
\caption[Infimum convergence under uniform-limit approximations]{Numerical convergence of \eqref{eqn_2intervalGenForm} infima using the sequence of approximations illustrated in Figure~\ref{fig:nontrivial-approximation-illustrations}; here, $n=10$ and $\theta=0.5$. In Figure~\ref{figa:nontrivial-convergence-illustrations}, $\varepsilon=0.1$. In Figure~\ref{figb:nontrivial-convergence-illustrations}, $\varepsilon=0.35$. The black curves plot the computed approximating infima $\phi_m^*$. The dashed line is the last computed value---i.e. an approximation of the exact limit.}
\label{fig:nontrivial-convergence-illustrations}
\end{figure}

\FloatBarrier
\subsection{A Two-Parameter Markov Chain Likelihood under Mixed Constraints}Define $\mathcal R:=\{(x,\lambda)\mid 0\leqslant x <1,\;\max \left\{0,2-1/x\right\} \leqslant\lambda\leqslant1\}$ and define $g(x,\lambda ) {}:={} (1-x)\left(1-x(1-\lambda)/(1-x)\right)^{999}$. Consider 
\begin{equation}
\inf_{\text{all feasible }\mathbb P\text{ over }\mathcal R}\frac{\mathbb E[f(X,\Lambda)]}{\mathbb E[g(X,\Lambda)]}\,,
\label{eqn_xKlotzlklhdFn_maintxt}
\end{equation}
where $f(x,\lambda)=g(x,\lambda){\bf 1}_{0\leqslant x\leqslant 10^{-4}}$, $\mathbb P(\Lambda>X)=\mathbb P(\Lambda<X)=0.05$, and $X\allowbreak\sim\allowbreak\operatorname{Beta}(1,\allowbreak 10000)$, for all feasible $\mathbb P$. The likelihood $g(x,\lambda)$ is induced by Klotz’s 2-state discrete-time Markov-chain model (see \cite{klotz_statistical_1973}). Solving \eqref{eqn_xKlotzlklhdFn_maintxt} proceeds by applying Proposition~\ref{cor_refined_strip_oscillations}, Proposition~\ref{prop:semicont_refinement} and Theorem~\ref{thm:regulated_extension}; see Figure~\ref{fig:combined-klotz-and-convergence}.

%
%

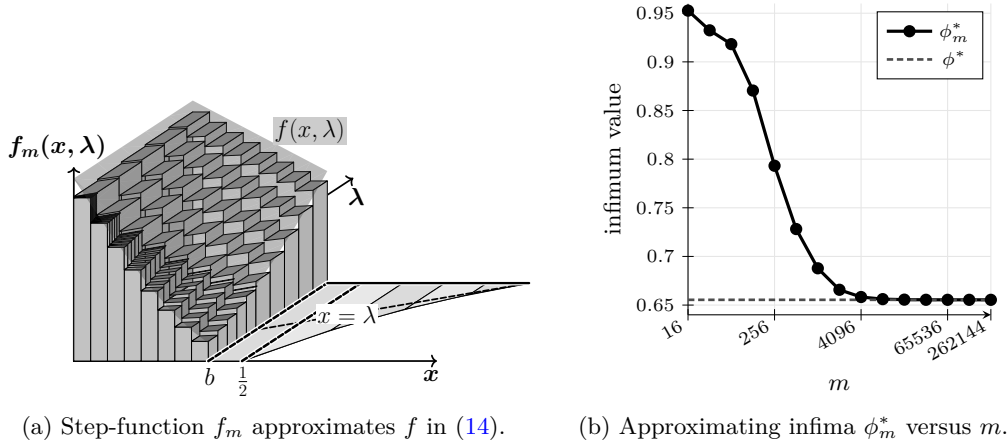
\begin{figure}[h!]
\centering

\subfloat[Step-function $f_m$ approximates $f$ in \eqref{eqn_xKlotzlklhdFn_maintxt}.\label{figa:combined-klotz-and-convergence}]{
\begin{minipage}[t]{0.5\textwidth}
\centering
\resizebox{\textwidth}{!}{%
\begin{tikzpicture}[
    x={(8.2cm,0cm)},
    y={(2.9cm,1.9cm)},
    z={(0cm,4.3cm)},
    line join=round,
    line cap=round
]

\def\bvis{0.40}
\def\M{8}
\def\K{5}
\def\diagw{0.010}
\def\sheetlift{0.035}
\def\sheetfillopacity{0.34}
\def\sheetfilltone{black!78}
\def\axisstub{0.06}

\def\labelfont{\fontsize{17}{19}\selectfont}
\def\titlefont{\fontsize{17}{19}\selectfont}

\pgfmathdeclarefunction{KlotzL}{2}{%
  \pgfmathparse{(1-#1)*(1-((1-#2)*#1)/(1-#1))^2}%
}

\draw[->,very thick] (0,0,0) -- (1.08,0,0);
\draw[->,very thick] (0,0,0) -- (0,2.35,0);
\draw[->,very thick] (0,0,0) -- (0,0,1.10);

\draw[very thick] (\bvis,1,0)--(1,1,0);
\draw[very thick] (0,0,0)--(0.50,0,0);
\draw[very thick] (0,0,0)--(0,1,0);
\draw[very thick] plot[smooth] coordinates {
  (0.50,0.0000,0)
  (0.62,0.3871,0)
  (0.74,0.6486,0)
  (0.86,0.8372,0)
  (1.00,1.0000,0)
};

\filldraw[fill=black!14,draw=black]
  (0.4000,0.0000,0.0000) -- (0.5000,0.0000,0.0000) --
  (0.5000,1.0000,0.0000) -- (0.4000,1.0000,0.0000) -- cycle;

\filldraw[fill=black!9,draw=black]
  (0.5000,0.0000,0.0000) -- (0.6250,0.4000,0.0000) --
  (0.6250,1.0000,0.0000) -- (0.5000,1.0000,0.0000) -- cycle;

\filldraw[fill=black!14,draw=black]
  (0.6250,0.4000,0.0000) -- (0.7500,0.6667,0.0000) --
  (0.7500,1.0000,0.0000) -- (0.6250,1.0000,0.0000) -- cycle;

\filldraw[fill=black!9,draw=black]
  (0.7500,0.6667,0.0000) -- (0.8750,0.8571,0.0000) --
  (0.8750,1.0000,0.0000) -- (0.7500,1.0000,0.0000) -- cycle;

\filldraw[fill=black!14,draw=black]
  (0.8750,0.8571,0.0000) -- (1.0000,1.0000,0.0000) --
  (1.0000,1.0000,0.0000) -- (0.8750,1.0000,0.0000) -- cycle;

\filldraw[fill=white,draw=none]
  (0,\axisstub,0.000) --
  (0,1.115,0.000) --
  (0,1.115,1.02) --
  (0,\axisstub,1.02) -- cycle;

\filldraw[fill=white,draw=none]
  (0.000,1,0.000) --
  (\bvis,1,0.000) --
  (\bvis,1,1.02) --
  (0.000,1,1.02) -- cycle;

\draw[very thick] (0,0,0) -- (0,\axisstub,0);

\draw[black, line width=1.2pt, dash pattern=on 4pt off 2pt]
  (0,0,0) -- (1,1,0);

\foreach \k in {5,4,3,2,1} {
  \pgfmathsetmacro{\xL}{0}
  \pgfmathsetmacro{\xR}{\bvis/\M}
  \pgfmathsetmacro{\xM}{0.5*\xR}
  \pgfmathsetmacro{\lamML}{\xM + (\k-1)*(1-\xM)/\K}
  \pgfmathsetmacro{\lamMR}{\xM + \k*(1-\xM)/\K}
  \pgfmathsetmacro{\lamM}{0.5*(\lamML+\lamMR)}
  \pgfmathsetmacro{\zH}{KlotzL(\xM,\lamM)}
  \pgfmathsetmacro{\yL}{(\k-1)/\K}
  \pgfmathsetmacro{\yR}{\k/\K}
  \filldraw[fill=black!22,draw=black]
    (0,\yL,0) -- (0,\yR,0) -- (0,\yR,\zH) -- (0,\yL,\zH) -- cycle;
}

\foreach \i in {8,7,6,5,4,3,2,1} {
  \pgfmathsetmacro{\xL}{(\i-1)*\bvis/\M}
  \pgfmathsetmacro{\xR}{\i*\bvis/\M}
  \pgfmathsetmacro{\xM}{0.5*(\xL+\xR)}
  \pgfmathsetmacro{\lamM}{0.5*(\xM/\K)}
  \pgfmathsetmacro{\zH}{KlotzL(\xM,\lamM)}
  \filldraw[fill=black!24,draw=black]
    (\xL,0,0) -- (\xR,0,0) -- (\xR,0,\zH) -- (\xL,0,\zH) -- cycle;
}

\foreach \i in {8,7,6,5,4,3,2} {
  \pgfmathsetmacro{\xL}{(\i-1)*\bvis/\M}
  \pgfmathsetmacro{\xR}{\i*\bvis/\M}
  \pgfmathsetmacro{\xPrevL}{(\i-2)*\bvis/\M}
  \pgfmathsetmacro{\xPrevR}{(\i-1)*\bvis/\M}
  \pgfmathsetmacro{\xPrevM}{0.5*(\xPrevL+\xPrevR)}
  \pgfmathsetmacro{\xM}{0.5*(\xL+\xR)}

  \foreach \k in {1,2,3,4,5} {
    \pgfmathsetmacro{\yL}{(\k-1)*\xL/\K}
    \pgfmathsetmacro{\yR}{\k*\xL/\K}
    \pgfmathsetmacro{\lamPrevL}{(\k-1)*\xPrevM/\K}
    \pgfmathsetmacro{\lamPrevR}{\k*\xPrevM/\K}
    \pgfmathsetmacro{\lamPrevM}{0.5*(\lamPrevL+\lamPrevR)}
    \pgfmathsetmacro{\lamCurL}{(\k-1)*\xM/\K}
    \pgfmathsetmacro{\lamCurR}{\k*\xM/\K}
    \pgfmathsetmacro{\lamCurM}{0.5*(\lamCurL+\lamCurR)}
    \pgfmathsetmacro{\zPrev}{KlotzL(\xPrevM,\lamPrevM)}
    \pgfmathsetmacro{\zCur}{KlotzL(\xM,\lamCurM)}
    \filldraw[fill=black!28,draw=black]
      (\xL,\yL,\zCur) -- (\xL,\yR,\zCur) --
      (\xL,\yR,\zPrev) -- (\xL,\yL,\zPrev) -- cycle;
  }

  \foreach \k in {1,2,3,4,5} {
    \pgfmathsetmacro{\yL}{\xL + (\k-1)*(1-\xL)/\K}
    \pgfmathsetmacro{\yR}{\xL + \k*(1-\xL)/\K}
    \pgfmathsetmacro{\lamPrevL}{\xPrevM + (\k-1)*(1-\xPrevM)/\K}
    \pgfmathsetmacro{\lamPrevR}{\xPrevM + \k*(1-\xPrevM)/\K}
    \pgfmathsetmacro{\lamPrevM}{0.5*(\lamPrevL+\lamPrevR)}
    \pgfmathsetmacro{\lamCurL}{\xM + (\k-1)*(1-\xM)/\K}
    \pgfmathsetmacro{\lamCurR}{\xM + \k*(1-\xM)/\K}
    \pgfmathsetmacro{\lamCurM}{0.5*(\lamCurL+\lamCurR)}
    \pgfmathsetmacro{\zPrev}{KlotzL(\xPrevM,\lamPrevM)}
    \pgfmathsetmacro{\zCur}{KlotzL(\xM,\lamCurM)}
    \filldraw[fill=black!22,draw=black]
      (\xL,\yL,\zCur) -- (\xL,\yR,\zCur) --
      (\xL,\yR,\zPrev) -- (\xL,\yL,\zPrev) -- cycle;
  }
}

\foreach \k in {5,4,3,2,1} {
  \foreach \i in {8,7,6,5,4,3,2,1} {
    \pgfmathsetmacro{\xL}{(\i-1)*\bvis/\M}
    \pgfmathsetmacro{\xR}{\i*\bvis/\M}
    \pgfmathsetmacro{\xM}{0.5*(\xL+\xR)}
    \pgfmathsetmacro{\lamL}{(\k-1)*\xM/\K}
    \pgfmathsetmacro{\lamR}{\k*\xM/\K}
    \pgfmathsetmacro{\lamM}{0.5*(\lamL+\lamR)}
    \pgfmathsetmacro{\zH}{KlotzL(\xM,\lamM)}
    \filldraw[fill=black!46,draw=black]
      (\xL,{(\k-1)*\xL/\K},\zH) --
      (\xR,{(\k-1)*\xR/\K},\zH) --
      (\xR,{\k*\xR/\K},\zH) --
      (\xL,{\k*\xL/\K},\zH) -- cycle;
  }
}

\foreach \k in {5,4,3,2,1} {
  \foreach \i in {8,7,6,5,4,3,2,1} {
    \pgfmathsetmacro{\xL}{(\i-1)*\bvis/\M}
    \pgfmathsetmacro{\xR}{\i*\bvis/\M}
    \pgfmathsetmacro{\xM}{0.5*(\xL+\xR)}
    \pgfmathsetmacro{\lamL}{\xM + (\k-1)*(1-\xM)/\K}
    \pgfmathsetmacro{\lamR}{\xM + \k*(1-\xM)/\K}
    \pgfmathsetmacro{\lamM}{0.5*(\lamL+\lamR)}
    \pgfmathsetmacro{\zH}{KlotzL(\xM,\lamM)}
    \filldraw[fill=black!34,draw=black]
      (\xL,{\xL + (\k-1)*(1-\xL)/\K},\zH) --
      (\xR,{\xR + (\k-1)*(1-\xR)/\K},\zH) --
      (\xR,{\xR + \k*(1-\xR)/\K},\zH) --
      (\xL,{\xL + \k*(1-\xL)/\K},\zH) -- cycle;
  }
}

\foreach \k in {1,2,3,4,5} {
  \pgfmathsetmacro{\xLastL}{(7)*\bvis/\M}
  \pgfmathsetmacro{\xLastR}{(8)*\bvis/\M}
  \pgfmathsetmacro{\xLastM}{0.5*(\xLastL+\xLastR)}
  \pgfmathsetmacro{\yL}{(\k-1)*\bvis/\K}
  \pgfmathsetmacro{\yR}{\k*\bvis/\K}
  \pgfmathsetmacro{\lamL}{(\k-1)*\xLastM/\K}
  \pgfmathsetmacro{\lamR}{\k*\xLastM/\K}
  \pgfmathsetmacro{\lamM}{0.5*(\lamL+\lamR)}
  \pgfmathsetmacro{\zH}{KlotzL(\xLastM,\lamM)}
  \filldraw[fill=black!30,draw=black]
    (\bvis,\yL,0) -- (\bvis,\yR,0) --
    (\bvis,\yR,\zH) -- (\bvis,\yL,\zH) -- cycle;
}

\foreach \k in {1,2,3,4,5} {
  \pgfmathsetmacro{\xLastL}{(7)*\bvis/\M}
  \pgfmathsetmacro{\xLastR}{(8)*\bvis/\M}
  \pgfmathsetmacro{\xLastM}{0.5*(\xLastL+\xLastR)}
  \pgfmathsetmacro{\yL}{\bvis + (\k-1)*(1-\bvis)/\K}
  \pgfmathsetmacro{\yR}{\bvis + \k*(1-\bvis)/\K}
  \pgfmathsetmacro{\lamL}{\xLastM + (\k-1)*(1-\xLastM)/\K}
  \pgfmathsetmacro{\lamR}{\xLastM + \k*(1-\xLastM)/\K}
  \pgfmathsetmacro{\lamM}{0.5*(\lamL+\lamR)}
  \pgfmathsetmacro{\zH}{KlotzL(\xLastM,\lamM)}
  \filldraw[fill=black!30,draw=black]
    (\bvis,\yL,0) -- (\bvis,\yR,0) --
    (\bvis,\yR,\zH) -- (\bvis,\yL,\zH) -- cycle;
}

\foreach \i in {1,2,3,4,5,6,7,8} {
  \pgfmathsetmacro{\xL}{(\i-1)*\bvis/\M}
  \pgfmathsetmacro{\xR}{\i*\bvis/\M}
  \pgfmathsetmacro{\xM}{0.5*(\xL+\xR)}
  \pgfmathsetmacro{\zH}{KlotzL(\xM,\xM)}
  \filldraw[fill=white,draw=black]
    (\xL,\xL,\zH) --
    (\xR,\xR,\zH) --
    (\xR,{\xR+\diagw},\zH) --
    (\xL,{\xL+\diagw},\zH) -- cycle;
}

\foreach \i in {1,2,3,4,5,6,7,8} {
  \pgfmathsetmacro{\xL}{(\i-1)*\bvis/\M}
  \pgfmathsetmacro{\xR}{\i*\bvis/\M}
  \foreach \k in {1,2,3,4,5} {
    \pgfmathsetmacro{\yLL}{(\k-1)*\xL/\K}
    \pgfmathsetmacro{\yLR}{\k*\xL/\K}
    \pgfmathsetmacro{\yRL}{(\k-1)*\xR/\K}
    \pgfmathsetmacro{\yRR}{\k*\xR/\K}
    \pgfmathsetmacro{\zLL}{KlotzL(\xL,\yLL)}
    \pgfmathsetmacro{\zLR}{KlotzL(\xL,\yLR)}
    \pgfmathsetmacro{\zRL}{KlotzL(\xR,\yRL)}
    \pgfmathsetmacro{\zRR}{KlotzL(\xR,\yRR)}

    \fill[fill=\sheetfilltone,fill opacity=\sheetfillopacity]
      (\xL,\yLL,{\zLL+\sheetlift}) --
      (\xR,\yRL,{\zRL+\sheetlift}) --
      (\xR,\yRR,{\zRR+\sheetlift}) -- cycle;
    \fill[fill=\sheetfilltone,fill opacity=\sheetfillopacity]
      (\xL,\yLL,{\zLL+\sheetlift}) --
      (\xR,\yRR,{\zRR+\sheetlift}) --
      (\xL,\yLR,{\zLR+\sheetlift}) -- cycle;
  }
}

\foreach \i in {1,2,3,4,5,6,7,8} {
  \pgfmathsetmacro{\xL}{(\i-1)*\bvis/\M}
  \pgfmathsetmacro{\xR}{\i*\bvis/\M}
  \foreach \k in {1,2,3,4,5} {
    \pgfmathsetmacro{\yLL}{\xL + (\k-1)*(1-\xL)/\K}
    \pgfmathsetmacro{\yLR}{\xL + \k*(1-\xL)/\K}
    \pgfmathsetmacro{\yRL}{\xR + (\k-1)*(1-\xR)/\K}
    \pgfmathsetmacro{\yRR}{\xR + \k*(1-\xR)/\K}
    \pgfmathsetmacro{\zLL}{KlotzL(\xL,\yLL)}
    \pgfmathsetmacro{\zLR}{KlotzL(\xL,\yLR)}
    \pgfmathsetmacro{\zRL}{KlotzL(\xR,\yRL)}
    \pgfmathsetmacro{\zRR}{KlotzL(\xR,\yRR)}

    \fill[fill=\sheetfilltone,fill opacity=\sheetfillopacity]
      (\xL,\yLL,{\zLL+\sheetlift}) --
      (\xR,\yRL,{\zRL+\sheetlift}) --
      (\xR,\yRR,{\zRR+\sheetlift}) -- cycle;
    \fill[fill=\sheetfilltone,fill opacity=\sheetfillopacity]
      (\xL,\yLL,{\zLL+\sheetlift}) --
      (\xR,\yRR,{\zRR+\sheetlift}) --
      (\xL,\yLR,{\zLR+\sheetlift}) -- cycle;
  }
}

\draw[white, line width=3.6pt] (0.50,0,0) -- (0.50,1,0);
\draw[black, line width=1.8pt, dash pattern=on 4pt off 2pt]
  (0.50,0,0) -- (0.50,1,0);

\draw[white, line width=3.6pt] (\bvis,0,0) -- (\bvis,1,0);
\draw[black, line width=1.8pt, dash pattern=on 4pt off 2pt]
  (\bvis,0,0) -- (\bvis,1,0);

\draw[white, line width=3.2pt] (\bvis,1,0) -- (1,1,0);
\draw[black, line width=1.6pt] (\bvis,1,0) -- (1,1,0);

\node[font=\labelfont,fill=white,inner sep=1.0pt] at (0.60,0.60,0) {$x=\lambda$};
\node[font=\labelfont,inner sep=1.0pt,anchor=south,rotate=0] at (\bvis+0.12,-0.33,0) {$b$};
\node[font=\labelfont,inner sep=1.0pt,anchor=south,rotate=0] at (0.63,-0.36,-0.05) {$\tfrac12$};

\node[font=\labelfont,anchor=east]  at (1.10,0,-0.08) {$\boldsymbol{x}$};
\node[font=\labelfont,anchor=west]  at (0,2.27,-0.07) {$\boldsymbol{\lambda}$};
\node[font=\labelfont,anchor=south] at (-0.05,0,1.12) {$\boldsymbol{f_m(x,\lambda)}$};

\node[font=\titlefont, fill=black!20, inner sep=1.2pt]
  at (0.33,1.05,0.84) {$f(x,\lambda)$};

\end{tikzpicture}%
}
\end{minipage}
}
\hspace{0.03\textwidth}
\subfloat[Approximating infima $\phi_m^*$ versus $m$.\label{figb:combined-klotz-and-convergence}]{
\begin{minipage}[t]{0.4\textwidth}
\centering
\begin{tikzpicture}[baseline=(current bounding box.north)]
\begin{axis}[
    width=\textwidth,
    height=5.7cm,
    xmode=log,
    log basis x=2,
    xmin=16, xmax=262144,
    ymin=0.64, ymax=0.96,
    axis lines=left,
    xtick={16,256,4096,65536,262144},
    xticklabels={$16$,$256$,$4096$,$65536$,$262144$},
    xticklabel style={font=\scriptsize, rotate=30, anchor=east},
    ytick={0.65,0.70,0.75,0.80,0.85,0.90,0.95},
    xlabel={$m$},
    ylabel={infimum value},
    every axis label/.style={font=\small},
    tick label style={font=\scriptsize},
    tick style={black},
    grid=both,
    grid style={gray!20},
    minor grid style={gray!12},
    legend style={
        draw=black,
        fill=white,
        font=\scriptsize,
        at={(0.975,0.975)},
        anchor=north east
    },
    clip=false
]

\addplot[
    color=black,
    line width=1.2pt,
    mark=*,
    mark size=1.7pt
] coordinates {
    (16,0.9526836069)
    (32,0.9324237079)
    (64,0.9182284044)
    (128,0.8705052605)
    (256,0.7932081640)
    (512,0.7281508267)
    (1024,0.6878128481)
    (2048,0.6657235328)
    (4096,0.6582828889)
    (8192,0.6561086986)
    (16384,0.6555871730)
    (32768,0.6554425819)
    (65536,0.6554122345)
    (131072,0.6554043143)
    (262144,0.6554021690)
};
\addlegendentry{$\phi_m^*$}

\addplot[
    color=gray!60!black,
    dash pattern=on 2.4pt off 1.4pt,
    line width=1.0pt
] coordinates {
    (16,0.6554021690) (262144,0.6554021690)
};
\addlegendentry{$\phi^*$}

\end{axis}
\end{tikzpicture}
\end{minipage}
}

\caption[Step-prism approximation and convergence for a Markov chain likelihood problem]{
For \eqref{eqn_xKlotzlklhdFn_maintxt}, an illustration of a step-function approximation $f_m(x,\lambda)$ to $f(x,\lambda)$ and numerical illustration of convergence of infima.
In Figure~\ref{figa:combined-klotz-and-convergence}, over the $x$--$\lambda$ region $\mathcal R$ with $b=10^{-4}$ (see assumptions stated for \eqref{eqn_xKlotzlklhdFn_maintxt}), the nonzero region $x\leqslant b$
is partitioned into eight $x$-strips. Within each strip, the $\lambda<x$ and $\lambda>x$ regions are each subdivided into five
$\lambda$-bands. The approximating $f_m(x,\lambda)$ is defined as centroid-sampled step prisms, with the steps over the $x=\lambda$ diagonal
depicted by a thin step ribbon.
The translucent gray sheet at the top is the graph of $f(x,\lambda)$.
In Figure~\ref{figb:combined-klotz-and-convergence}, the computed approximating infima
$\phi_m^*$ are shown for \eqref{eqn_xKlotzlklhdFn_maintxt}. The horizontal axis is dyadic in $m=2^k$, and the dashed line marks the finest
computed infimum estimate $\phi^*\approx 0.6554021690$.}
\label{fig:combined-klotz-and-convergence}
\end{figure}
\FloatBarrier
\subsection{Robust Reliability Inference under Interval-Mass Constraints} The following problem reproduced from \cite{salako2025conservative} illustrates Theorem~\ref{thm:global_fp_tuple}'s fixed-point characterisation of extremal solutions. Consider an i.i.d. Bernoulli process with unknown parameter $X$ for the probability of success, and the set $\mathcal D$ of all feasible prior distributions of $X$ over the unit interval. For $0=y_0<y_1<\ldots<y_n=1$, $X$ may lie in each of $n$ intervals $K_1:=[y_0,y_1]$, $K_2:=(y_1,y_2]$, \ldots, $K_n:=(y_{n-1},y_n]$, with probabilities $\mathbb P\in\mathcal D$ that obey $\mathbb P (X\in K_i)=p_{i}$ for $i=1,\ldots,n$ for some unknown feasible $\mathbb P\in \mathcal D$. For $r$ successes and $m$, $k$ failures ($r>0,\,m\geqslant 1,\,k>0$),  Problem~\eqref{eqn_weak_opt_equiv_disc} becomes\begin{align}
\underset{x_1\in K_1,\ldots,x_n\in K_n}{\inf}\frac{\sum_{i=1}^{n}x_i^r(1-x_i)^{m+k} p_i}{\sum_{i=1}^{n}x_i^r(1-x_i)^{k} p_i}\,.
\label{eqn_morenoetalBer}
\end{align}
\begin{theorem}
\label{thm_CBIwithfails_sol}
For $n\geqslant 2$ a unique fixed point triplet ($\phi^\ast$, $y_\ast$, $y_{**}$) solves \eqref{eqn_morenoetalBer}---for some $1\leqslant j_1\leqslant j_2\leqslant n$ such that $\ y_{**}\leqslant\frac{r}{r+m+k}<\frac{r}{r+k}\leqslant y_\ast$ (with strict inequalities when $\phi^*>0$), $\ y_{**}\in \overline{K}_{j_1}$, and $\ y_\ast\in \overline{K}_{j_2}$, the tuple satisfies
\begin{gather*}
(1-y_{**})^m\left(\frac{r-y_{**}(m+k+r)}{r-y_{**}(k+r)}\right)=\phi^\ast = (1-y_\ast)^m\left(\frac{r-y_\ast(m+k+r)}{r-y_\ast(k+r)}\right),\\ \phi^\ast=\min\{\phi^\ast_1,\phi^\ast_2\}\,,
\end{gather*}
where
\begin{align*}
\phi^\ast_1&=\left\{\begin{array}{ll}\dfrac{\sum_{i=2}^{j_1}f(y_{i-1})p_i+\sum_{i=j_1+1}^{j_2-1}f(y_{i})p_i+f({y_\ast})p_{j_2}+\sum_{i=j_2+1}^{n}f(y_{i-1})p_i}{\sum_{i=2}^{j_1}g(y_{i-1})p_i+\sum_{i=j_1+1}^{j_2-1}g(y_{i})p_i+g({y_\ast})p_{j_2}+\sum_{i=j_2+1}^{n}g(y_{i-1})p_i}\,,\hfill j_1<j_2\\
\dfrac{\sum_{i=2}^{n}f(y_{i-1})p_i}{\sum_{i=2}^{n}g(y_{i-1})p_i}\,,\hfill j_1=j_2\end{array}\right.\\
\phi^\ast_2&=\left\{\begin{array}{ll}\dfrac{\sum_{i=2}^{j_1-1}f(y_{i-1})p_i+\sum_{i=j_1}^{j_2-1}f(y_{i})p_i+f({y_\ast})p_{j_2}+\sum_{i=j_2+1}^{n}f(y_{i-1})p_i}{\sum_{i=2}^{j_1-1}g(y_{i-1})p_i+\sum_{i=j_1}^{j_2-1}g(y_{i})p_i+g({y_\ast})p_{j_2}+\sum_{i=j_2+1}^{n}g(y_{i-1})p_i}\,,\hfill j_1<j_2\\
\dfrac{\sum_{i=2}^{j_2-1}f(y_{i-1})p_i+f(y_\ast)p_{j_2}+\sum_{i=j_2+1}^{n}f(y_{i-1})p_i}{\sum_{i=2}^{j_2-1}g(y_{i-1})p_i+g(y_\ast)p_{j_2}+\sum_{i=j_2+1}^{n}g(y_{i-1})p_i}\,,\hfill j_1=j_2\,.\end{array}\right.
\end{align*} 
For $n\geqslant 1$ the infimum $\phi^*$ is attained by the prior distribution $\mathbb P$,
\begin{align}
\label{eqn_priorCBIsolBer}
    \mathbb P(X=x)=\left\{\begin{array}{cl}
    p_1,  & \text{if } x=\mathbbm{1}_{n=1},\\ 
     p_{i}, & \left\{\begin{array}{l}\text{if } x=y_{i-1} \text{ and } \{2\leqslant i<j_1\}\cup\{j_2<i\leqslant n\},\\\text{if } x=y_{i} \text{ and } \{j_{1}<i<j_{2}\},\end{array}\right.\\
        p_{j_1}, & \text{if } x=y_{j_{1}-1}\mathbbm{1}_{{\phi}^*=\phi_1^*}+\min\{y_{j_1},\,y_*\}\mathbbm{1}_{{\phi^*}=\phi_2^*\ne\phi_1^*},\\
           p_{j_2}, & \text{if } x=\left(y_{j_{1}-1}\mathbbm{1}_{j_1=j_2}+y_{*}\mathbbm{1}_{j_1<j_2}\right)\mathbbm{1}_{{\phi}^*=\phi_1^*}+y_{*}\mathbbm{1}_{{\phi^*}=\phi_2^*\ne\phi_1^*},\\
           0, & \text{otherwise.}
    \end{array}\right.
\end{align}
where $\mathbb P$ is the limit of feasible priors in Problem~\eqref{eqn_morenoetalBer}. Note that $\phi^* = 0$ when $n=1,2$: the corresponding weak-limit extremal priors are $\mathbb P(X=1)=p_1=1$ for $n=1$, and $\mathbb P(X=0)=p_1$, $\mathbb P(X=1)=p_2=1-p_1$ for $n=2$.
\end{theorem}
\begin{proof} Appendix~M, Supplementary Material (\cite{Salako_Muhammad_2025_suppmat}). 
\end{proof}
\noindent\textbf{Remark: } Special cases of Theorem~\ref{thm_CBIwithfails_sol} support conservative software reliability assessments (e.g. see \cite{bishop_toward_2011}, \cite{strigini_software_2013}).
%
%
%
%
\begin{figure}[htbp!]
\centering
\scalebox{0.9}{
   \begin{tikzpicture}
  \draw[->] (0,0) -- (14,0);
\draw[darkgray, line width=2.5pt] (0,-0.005) -- (0,0.90) node[below=5pt] {\space};
\draw[darkgray, line width=2.5pt] (6.5,-0.005) -- (6.5,0.5) node[below=5pt] {\space};
\draw[darkgray, line width=2.5pt] (2,-0.005) -- (2,0.75) node[below=5pt] {\space};
\draw[darkgray, line width=2.5pt] (3.5,-0.005) -- (3.5,0.9) node[below=5pt] {\space};
\draw[darkgray, line width=2.5pt] (12.5,-0.005) -- (12.5,0.3) node[below=5pt] {\space};
\draw[darkgray, line width=2.5pt] (8.7,-0.005) -- (8.7,0.61) node[below=5pt] {\space};
\draw[darkgray, line width=2.5pt] (9.7,-0.005) -- (9.7,0.99) node[below=5pt] {\space};
\draw[darkgray, line width=2.5pt] (10.5,-0.005) -- (10.5,0.74) node[below=5pt] {\space};
   \draw (0,0) -- (0,-0.15) node[below=5pt] at (0,0) {$0$};
  \foreach \x/\label in {
    1/{\dots},
    2/{$y_{j_{1}-2}$},
    3.5/{$y_{{j_1}-1}$},
    5.5/{$y_{j_1}$},
    6.5/{$y_{j_1+1}$},
    7.5/{\dots},
    8.7/{$y_{{j_2}-1}$},
    10.5/{$y_{j_2}$},
    11.5/{\dots},
    12.5/{$y_{n-1}$},
    13.5/{$1$}
  }{
    \draw (\x,0) -- (\x,-0.15);
    \node[below=5pt] at (\x,0) {\label};
  }
  \node[below=5.4pt] at (4.6,0) {$y_{**}$};
    \node[below=2.5pt] at (8,0) {\space};
  \node[below=5.4pt] at (9.7,0) {$y_{*}$};
\node[above=22pt] at (0,0.1) {$p_1$};
\node[above=12pt] at (6.5,0.1) {$p_{j_1+1}$};
\node[above=15pt] at (2,0.1) {$p_{{j_1}-1}$};
 \node[above=22pt] at (3.5,0.1) {$p_{j_1}$};
  \node[above=12pt] at (8.7,0.1) {$p_{{j_2}-1}$};
  \node[above=23pt] at (9.7,0.1) {$p_{j_2}$};
   \node[above=15pt] at (10.5,0.1) {$p_{{j_2}+1}$}; \node[above=5pt] at (12.5,0.1) {$p_{n}$};
\end{tikzpicture}}
\caption[Limiting extremal prior for robust Bernoulli inference]{An illustration of the $j_1<j_2$, $\phi^*=\phi^*_1$ extremal-prior form in \eqref{eqn_priorCBIsolBer}. Reproduced from \cite{salako2025conservative}.}
\label{fig_worst_prior}
\end{figure}
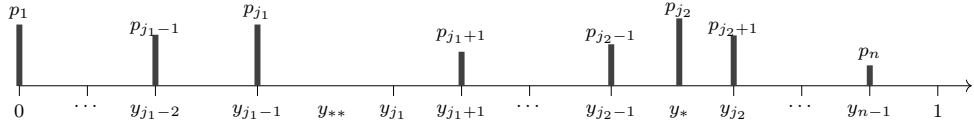

\subsection{Generalised $k$-out-of-$n$ Boole–Fr\'echet Bounds} Problem~\eqref{eqn_genCBIprob} and its extensions (via the propositions) significantly generalise Boole-Fr\'echet bounds. Indeed, the theorems imply that generalised Boole-Fr\'echet bounds, and their extremal distributions, admit fixed-point characterisations (\emph{cf.} Dynamic Programming fixed-point characterisations of Boole-Fr\'echet bounds in \cite{salako2025constructive}). For example, consider the following Hailperin/R{\"u}ger upper bound on the probability of at least $k$-out-of-$n$ events occurring (see \cite{Ruger1978}), here given in fixed-point form:\begin{theorem}
\label{thm_atleastkoutofnbounds}
Let $A_1,\dots,A_n$ be marginal events on a common probability space, with $0<p_1<\cdots<p_n<1$, $\mathbb P(A_i)=p_i$ for $i=1,\dots,n$, and $k\in\{1,\dots,n\}$.
Define
\[
\Phi_k
:=
\sup\Bigl\{
\mathbb P\Bigl(\sum_{i=1}^n \mathbf 1_{A_i}\geqslant k\Bigr)
:\ \mathbb P(A_i)=p_i,\ i=1,\dots,n
\Bigr\}.
\] For each $j=1,\dots,k$, let $m_j:=n-k+j$ and $U_j:=\min\!\left\{\frac{\sum_{i=1}^{m_j}p_i}{j},\,1\right\}$.

Then, the solution is the fixed-point pair $(\Phi_k,j^*)$, alternatively $(\Phi_k,r^*)$,
\[
\Phi_k
=
U_{j^*}
=
\min\!\left\{\frac{\sum_{i=1}^{n-r^*}p_i}{k-r^*},\,1\right\},\qquad r^*=\sum_{i=1}^{k-1}{\boldsymbol 1}_{p_{n-k+1+i}\geqslant\Phi_k}=k-j^*.
\]
where 
\[
j^*:=\begin{cases}\min\{j\in\{1,\ldots,k-1\}:U_j\leqslant p_{n-k+1+j}\},\text{ if }\min\text{exists}\\k,\text{ otherwise}\end{cases}
\] 

Hence, $j^*$ determines a family of extremal priors with the following property: if an element in the support of the prior represents the occurrence of exactly $k$-out-of-$n$ marginal $A_i$ events, then it necessarily contains those $r^*=k-j^*$ marginal events with the largest $p_i$ probabilities; the remaining $j^*$ marginal events in the support element are drawn from the remaining $n-r^*$ marginal events with the smallest $p_i$ probabilities. Marginal events $A_{n-r^*+1}$, ..., $A_n$ ($r^*>0$) are ``dominating'' events.
\end{theorem}
\begin{proof} Appendix~N, Supplementary Material (\cite{Salako_Muhammad_2025_suppmat}). \end{proof}

\subsection{Robust Posterior Bounds for Interval-Identified Econometric Models}Problem~\eqref{eqn_genCBIprob} has applications in Economics. For example,
Theorem~\ref{thm:2d-gk-fixed-point-extension-corrected} treats scalar
interval-identified models and provides an analogue of Theorem~1 in
\cite{GiacominiKitagawa2021RobustBayesian}. Both results express robust
posterior bounds through the same underlying containment principle: the
lower posterior probability of an event is determined by those values of
the identified parameter for which the corresponding identified set is
entirely contained in that event. Theorem~
\ref{thm:2d-gk-fixed-point-extension-corrected} as a whole is not a
direct corollary of Giacomini and Kitagawa's result. Although its
robust-bound formula is a specialization of their containment result, its additional contribution is to recast
the containment principle in fixed-point and extremal-prior form, yielding
a sequence of finite-dimensional extremal problems with explicit
finite-support prior representations and convergent bounds for the robust
posterior value. The theorem is stated for the infimum; the corresponding supremum result is analogous.
\begin{figure}[htb!]
\centering
\resizebox{0.47\linewidth}{!}{%
\setlength{\fboxsep}{6pt}%
\colorbox{gray!35}{%
\begin{tikzpicture}[
    x=7.2cm,y=7.2cm,>=Latex,
    line join=round,line cap=round
]

  \fill[gray!35] (-0.11,-0.13) rectangle (1.13,1.10);
  \fill[white] (0,0) rectangle (1,1);

  \fill[gray!60,draw=white,line width=1.15pt]
    (0.00,0.18)
    -- (0.00,0.77)
    .. controls (0.06,0.88) and (0.10,0.90) ..
    (0.14,0.90)
    .. controls (0.24,0.96) and (0.32,0.94) ..
    (0.37,0.91)
    .. controls (0.48,0.88) and (0.53,0.82) ..
    (0.61,0.86)
    .. controls (0.71,0.91) and (0.79,0.91) ..
    (0.86,0.88)
    .. controls (0.93,0.89) and (0.98,0.88) ..
    (1.00,0.84)
    -- (1.00,0.22)
    .. controls (0.98,0.18) and (0.93,0.15) ..
    (0.86,0.16)
    .. controls (0.80,0.14) and (0.74,0.14) ..
    (0.69,0.15)
    .. controls (0.57,0.16) and (0.49,0.21) ..
    (0.41,0.13)
    .. controls (0.31,0.07) and (0.22,0.07) ..
    (0.14,0.10)
    .. controls (0.08,0.11) and (0.02,0.14) ..
    (0.00,0.18)
    -- cycle;

  %
  \fill[gray!78,draw=white,line width=1.05pt]
    (0.14,0.90)
    .. controls (0.10,0.90) and (0.06,0.88) ..
    (0.00,0.77)
    -- (0.00,0.18)
    .. controls (0.02,0.14) and (0.08,0.11) ..
    (0.14,0.10)
    .. controls (0.22,0.20) and (0.25,0.30) ..
    (0.31,0.34)
    .. controls (0.42,0.40) and (0.54,0.31) ..
    (0.65,0.34)
    .. controls (0.76,0.37) and (0.81,0.24) ..
    (0.86,0.16)
    .. controls (0.93,0.15) and (0.98,0.18) ..
    (1.00,0.22)
    -- (1.00,0.84)
    .. controls (0.98,0.88) and (0.93,0.89) ..
    (0.86,0.88)
    .. controls (0.78,0.80) and (0.73,0.73) ..
    (0.69,0.70)
    .. controls (0.58,0.82) and (0.46,0.86) ..
    (0.35,0.79)
    .. controls (0.25,0.73) and (0.21,0.79) ..
    (0.14,0.90)
    -- cycle;

  %
  \def\phiX{0.28}

  \def\KphiBot{0.08258}
  \def\CphiBot{0.31286}
  \def\CphiTop{0.76627}
  \def\KphiTop{0.93845}

  \draw[white,line width=1.35pt]
    (\phiX,\KphiTop) -- (\phiX,\CphiTop);

  \draw[
    white,
    line width=1.35pt,
    dash pattern=on 3pt off 2pt
  ]
    (\phiX,\CphiTop) -- (\phiX,\CphiBot);

  \draw[white,line width=1.35pt]
    (\phiX,\CphiBot) -- (\phiX,\KphiBot);

  \fill[black] (\phiX,0.28) circle (0.008);


  \draw[black,line width=0.9pt]
    (\phiX,-0.012) -- (\phiX,0.012);

  \node[font=\normalsize,anchor=north]
    at (\phiX,-0.01) {$\varphi$};

  \def\Scca{0.14}
  \def\Sccb{0.86}

  \draw[black,line width=1.0pt]
    (\Scca,-0.075) -- (\Sccb,-0.075);

  \draw[black,line width=1.0pt]
    (\Scca,-0.06) -- (\Scca,-0.09);

  \draw[black,line width=1.0pt]
    (\Sccb,-0.06) -- (\Sccb,-0.09);

  \node[font=\large]
    at (0.50,-0.12)
    {$S_C^c$};

  \draw[black,line width=1.0pt]
    (1.04,\CphiBot) -- (1.04,\CphiTop);

  \draw[black,line width=1.0pt]
    (1.025,\CphiBot) -- (1.055,\CphiBot);

  \draw[black,line width=1.0pt]
    (1.025,\CphiTop) -- (1.055,\CphiTop);

  \node[font=\large,anchor=west]
    at (1.07,0.53957)
    {$C_{\varphi}$};

  \node[font=\Large\bfseries,anchor=south east]
    at (0.98,1.00)
    {$\boldsymbol{U\times V=[0,1]^2}$};

  \node[font=\large]
    at (0.76,0.23)
    {$\boldsymbol{K}$};

  \node[font=\large]
    at (0.52,0.64)
    {$\boldsymbol{C}$};

  \node[font=\large,anchor=west]
    at (0.29,0.28)
    {$s(\varphi)$};

  \node[font=\large,align=left,anchor=west]
    at (-0.03,0.56)
    {$K_{\varphi}\cap C_\varphi^c$};

  \draw[black,line width=0.8pt]
    (0.10,0.52) -- (\phiX-0.008,0.24);

  \draw[black,line width=0.8pt]
    (0.10,0.61) -- (\phiX-0.008,0.84);

  \node[font=\large,anchor=west]
    at (1.03,-0.002)
    {$U$};

  \node[font=\large,anchor=south]
    at (-0.01,1.02)
    {$V$};

\end{tikzpicture}%
}}%
\caption[Fibrewise repair for interval-identified robust inference]{Fibrewise repair for Theorem~\ref{thm:2d-gk-fixed-point-extension-corrected}. The unit square is $U\times V=[0,1]^2$. The compact set $K$ is shown in grey and the event set $C\subseteq K$ in darker grey. Outside the displayed interval $S_C^c$, the fibres of $C$ coincide with the corresponding fibres of $K$; within $S_C^c=\{\varphi\in U:K_{\varphi}\nsubseteq C_{\varphi}\}$, the fibres of $C$ are strict subsets of the fibres of $K$. The illustrated fibre is taken at $\varphi\in S_C^c$, so that $K_{\varphi}\cap C_\varphi^c\neq\varnothing$; the solid white line segments at $\varphi$ indicate the repairable set $K_{\varphi}\cap C_\varphi^c$, while the dashed white segment indicates the $\varphi$-fibre portion inside $C$. A fibre-wise ``selector'' may then choose a repaired point $s(\varphi)\in K_{\varphi}\cap C_\varphi^c$. Although $K_\varphi$, $C_\varphi \subseteq V$, we illustrate these as the fibre over $\varphi$ (using the embedding $\eta\mapsto(\varphi,\eta)$) to emphasise their geometric implications. See the proof for details.}
\label{fig:fibrewise_repair_theorem218}
\end{figure}
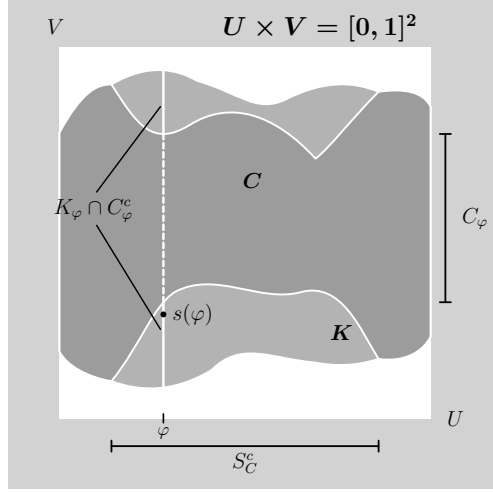
\begin{theorem}
\label{thm:2d-gk-fixed-point-extension-corrected}
Let $U=V=[0,1]$, let $\mu\in\mathcal P(U)$, and let $K\subseteq U\times V$ be compact with nonempty vertical sections $K_\varphi:=\{\eta\in V:(\varphi,\eta)\in K\}$  for $\varphi\in U$. Let $C\subseteq K$ be Borel, and let $L:U\to(0,\infty)$ be continuous with $0<c\leqslant L(\varphi)\leqslant M<\infty$. Define $f(\varphi,\eta):=L(\varphi)\mathbf 1_C(\varphi,\eta)$ and  $g(\varphi,\eta):=L(\varphi)$. Let $\mathcal D:= \bigl\{ \mathbb P\in\mathcal P(K): \mathbb P(A\times V)=\mu(A) \text{ for every Borel }A\subseteq U \bigr\}$, and define the lower robust posterior probability \[ I^-:= \inf_{\mathbb P\in\mathcal D}\frac{{\mathbb E}_{\mathbb P}[f]}{{\mathbb E}_{\mathbb P}[g]}. \] Denote $S_C:=\{\varphi\in U:K_\varphi\subseteq C_\varphi\}$ and $C_\varphi:=\{\eta\in V:(\varphi,\eta)\in C\}$, and assume $S_C$ is Borel. 
Then, \[ I^-= \frac{\int_{S_C}L(\varphi)\,\mu(d\varphi)} {\int_U L(\varphi)\,\mu(d\varphi)}. \] 
Moreover, $I^-$ is the limit of a monotonically increasing sequence of finite fixed-point extremal values obtained from $S_C$-adapted partitions of $K$. 
\end{theorem}

\begin{proof}
Appendix~O, Supplementary Material (\cite{Salako_Muhammad_2025_suppmat}).
\end{proof}

\noindent\textbf{Remark: }Although Giacomini and Kitagawa establish the existence of feasible priors that attain the robust bound---this can also be established by a modest modification of the fixed-point argument used in proving Theorem~\ref{thm:2d-gk-fixed-point-extension-corrected}---the formulation developed here \textbf{i)} broadens the solution description to include sequences of feasible priors whose objective values approach the bound, and \textbf{ii)} provides a constructive approximation route for solving the problem.  At each approximation
level, a finitely supported prior solves, or arbitrarily closely solves, a finite relaxation of the original problem. As partitions of the prior domain are refined, the corresponding
extremal objective values converge to the robust bound. A subtlety is that a weak limit of these approximating priors may itself fail to attain the robust bound. This is because the objective contains the indicator of the Borel event $C$, so the objective---viewed as a functional over priors---is not necessarily continuous under weak convergence. Thus, the framework provides a fixed-point characterisation that reveals the structural basis of the robust bound; it also describes finite approximations to the bound and the limiting process by which they approach it.

\subsection{Robust Value-at-Risk Aggregation under Prescribed Marginals}
\citet{EmbrechtsPuccettiRueschendorf2013} study robust
\emph{Value-at-Risk} (VaR) aggregation under prescribed one-dimensional and
higher-dimensional marginal information.
Theorem~\ref{thm_VaRExample} characterises extremal probabilities of Borel events over probability measures with uniform coordinate marginals. Corollary~\ref{Cor_VaRExample} applies these results to the
reduced robust VaR problems defined in
\citet[eqs.~(30a)--(30b)]{EmbrechtsPuccettiRueschendorf2013}, thereby extending their treatment with threshold-event
representations, ``exchange'' fixed-point characterisations of extremal
priors, and convergent finite approximations.  

\begin{theorem}[Extremal event probabilities under prescribed uniform marginals]
\label{thm_VaRExample}

Let $K=[0,1]^d$, let $\lambda$ denote Lebesgue probability measure on
$[0,1]$, and let
\begin{equation}
\mathcal D
:=
\left\{
\mathbb P\in\mathcal P(K):
(\operatorname{pr}_j)_\#\mathbb P=\lambda,
\quad j=1,\ldots,d
\right\}.
\label{eqn_VaRMathcalD}
\end{equation}
For Borel set $G\subseteq K$, define
\begin{align}
V(G)
&:=
\sup_{\mathbb P\in\mathcal D}\mathbb P(G).
\label{eqn_VaRVG} \\
I(G)
&:=
\inf_{\mathbb P\in\mathcal D}\mathbb P(G)
=
1-V(G^c).
\label{eqn_VaRIG}
\end{align}

For a finite positive measure $\alpha\in\mathcal M_+(K)$ on $K$, define its
marginal-completion fibre by
\begin{equation}
\mathcal C(\alpha)
:=
\left\{
\beta\in\mathcal M_+(K):
(\operatorname{pr}_j)_\#\beta
=
(\operatorname{pr}_j)_\#\alpha,
\quad j=1,\ldots,d
\right\}.
\label{eqn_betaforalpha}
\end{equation}
Every $\beta\in\mathcal C(\alpha)$ has $\beta(K)=\alpha(K)$. For $\mathbb P\in\mathcal D$, define the complete admissible exchange class $\mathfrak X(\mathbb P)
:=
\left\{
(\alpha,\beta):
0\leqslant\alpha\leqslant\mathbb P,
\ 
\beta\in\mathcal C(\alpha)
\right\}$,
and the associated exchange map $T_{\alpha,\beta}\mathbb P
:=
\mathbb P-\alpha+\beta$. Consider the extremal event-probability problems
\begin{equation}
\phi_G^+
:=
\sup_{\mathbb P\in\mathcal D}
\int_K\mathbf 1_G\,d\mathbb P
=
V(G),
\qquad
\phi_G^-
:=
\inf_{\mathbb P\in\mathcal D}
\int_K\mathbf 1_G\,d\mathbb P
=
I(G).
\label{eqn_event_extremal_problems}
\end{equation}
The following assertions hold.

\medskip
\noindent
\textbf{(a) Fixed-point characterisation of the general extremal
problem.}

For $\mathbb P\in\mathcal D$, $\phi\in\mathbb R$, and
$(\alpha,\beta)\in\mathfrak X(\mathbb P)$, define the exchange-space
extremal objective
\begin{equation}
h_{\phi,G}^{\mathbb P}(\alpha,\beta)
:=
\bigl(T_{\alpha,\beta}\mathbb P\bigr)(G)-\phi.
\label{eqn_exchange_extremal_objective}
\end{equation}
A fixed-point pair $(\mathbb P^\ast,\phi^\ast)\in\mathcal D\times\mathbb R$ solves
the maximizing problem in \eqref{eqn_event_extremal_problems} if and only if
it satisfies the fixed-point conditions
\begin{align}
(E1)_G &:\ 
(0,0)
\in
\operatorname*{arg\,max}_{(\alpha,\beta)\in\mathfrak X(\mathbb P^\ast)}
h_{\phi^\ast,G}^{\mathbb P^\ast}(\alpha,\beta),
\label{eqn_exchange_E1}
\\
(E2)_G &:\ 
\phi^\ast=\mathbb P^\ast(G).
\label{eqn_exchange_E2}
\end{align}
Equivalently, $\beta(G)-\alpha(G)\leqslant 0
\ 
\text{for every }
(\alpha,\beta)\in\mathfrak X(\mathbb P^\ast)$. The corresponding Dinkelbach difference is the
mass-exchange objective $\ \Delta_G(\alpha,\beta)
:=
\beta(G)-\alpha(G)$, which satisfies $\Delta_G(\alpha,\beta)
=
\bigl(T_{\alpha,\beta}\mathbb P\bigr)(G)-\mathbb P(G)$. The residual $\ \mathscr R_G(\mathbb P)
:=\sup_{(\alpha,\beta)\in\mathfrak X(\mathbb P)}\allowbreak\Delta_G(\alpha,\beta)$ over the exchange class for $\mathbb P$ is the global objective gap
\begin{equation}
\mathscr R_G(\mathbb P)
=
V(G)-\mathbb P(G).
\label{eqn_exchange_residual_gap}
\end{equation}
Consequently, $\mathscr R_G(\mathbb P)=0
\Longleftrightarrow
\mathbb P\in
\operatorname*{arg\,max}_{\mathbb Q\in\mathcal D}\mathbb Q(G),\,$ and $\,\mathbb P_r(G)\longrightarrow V(G)
\Longleftrightarrow
\mathscr R_G(\mathbb P_r)\longrightarrow0$. All corresponding statements for the minimizing problem follow by applying
the maximizing result to $G^c$ and using $I(G)=1-V(G^c)$.

\begin{figure}[htb!]
\centering
\resizebox{0.47\linewidth}{!}{%
\setlength{\fboxsep}{6pt}%
\colorbox{gray!35}{%
\begin{tikzpicture}[
    x=7.2cm,y=7.2cm,>=Latex,
    line join=round,line cap=round
]

  \fill[gray!35] (-0.11,-0.13) rectangle (1.13,1.10);
  \fill[white] (0,0) rectangle (1,1);

  \fill[gray!60,draw=white,line width=1.15pt]
    (0.12,0.30)
    .. controls (0.14,0.54) and (0.20,0.84) ..
    (0.39,0.90)
    .. controls (0.56,0.95) and (0.74,0.88) ..
    (0.80,0.73)
    .. controls (0.84,0.62) and (0.76,0.52) ..
    (0.73,0.42)
    .. controls (0.69,0.31) and (0.58,0.23) ..
    (0.46,0.24)
    .. controls (0.34,0.25) and (0.29,0.15) ..
    (0.19,0.12)
    .. controls (0.11,0.14) and (0.09,0.22) ..
    (0.12,0.30) -- cycle;

  \fill[gray!78,fill opacity=0.58,draw=white,line width=1.05pt]
    (0.25,0.45)
    .. controls (0.24,0.58) and (0.31,0.73) ..
    (0.43,0.77)
    .. controls (0.56,0.80) and (0.67,0.73) ..
    (0.70,0.61)
    .. controls (0.71,0.50) and (0.63,0.41) ..
    (0.51,0.39)
    .. controls (0.39,0.37) and (0.29,0.39) ..
    (0.25,0.45) -- cycle;

  \begin{scope}

    \foreach \i in {9,...,18}{
      \path[
        fill=gray!40,
        fill opacity=0.28,
        draw=black!55,
        line width=0.22pt
      ]
        ({\i/32},{12/32})
        rectangle
        ({(\i+1)/32},{13/32});
    }

    \foreach \i in {8,...,19}{
      \path[
        fill=gray!40,
        fill opacity=0.28,
        draw=black!55,
        line width=0.22pt
      ]
        ({\i/32},{13/32})
        rectangle
        ({(\i+1)/32},{14/32});
    }

    \foreach \i in {7,...,20}{
      \path[
        fill=gray!40,
        fill opacity=0.28,
        draw=black!55,
        line width=0.22pt
      ]
        ({\i/32},{14/32})
        rectangle
        ({(\i+1)/32},{15/32});
    }

    \foreach \i in {7,...,21}{
      \path[
        fill=gray!40,
        fill opacity=0.28,
        draw=black!55,
        line width=0.22pt
      ]
        ({\i/32},{15/32})
        rectangle
        ({(\i+1)/32},{16/32});
    }

    \foreach \i in {7,...,22}{
      \path[
        fill=gray!40,
        fill opacity=0.28,
        draw=black!55,
        line width=0.22pt
      ]
        ({\i/32},{16/32})
        rectangle
        ({(\i+1)/32},{17/32});
    }

    \foreach \i in {8,...,22}{
      \path[
        fill=gray!40,
        fill opacity=0.28,
        draw=black!55,
        line width=0.22pt
      ]
        ({\i/32},{17/32})
        rectangle
        ({(\i+1)/32},{18/32});
    }

    \foreach \i in {8,...,22}{
      \path[
        fill=gray!40,
        fill opacity=0.28,
        draw=black!55,
        line width=0.22pt
      ]
        ({\i/32},{18/32})
        rectangle
        ({(\i+1)/32},{19/32});
    }

    \foreach \i in {8,...,22}{
      \path[
        fill=gray!40,
        fill opacity=0.28,
        draw=black!55,
        line width=0.22pt
      ]
        ({\i/32},{19/32})
        rectangle
        ({(\i+1)/32},{20/32});
    }

    \foreach \i in {9,...,22}{
      \path[
        fill=gray!40,
        fill opacity=0.28,
        draw=black!55,
        line width=0.22pt
      ]
        ({\i/32},{20/32})
        rectangle
        ({(\i+1)/32},{21/32});
    }

    \foreach \i in {9,...,21}{
      \path[
        fill=gray!40,
        fill opacity=0.28,
        draw=black!55,
        line width=0.22pt
      ]
        ({\i/32},{21/32})
        rectangle
        ({(\i+1)/32},{22/32});
    }

    \foreach \i in {10,...,21}{
      \path[
        fill=gray!40,
        fill opacity=0.28,
        draw=black!55,
        line width=0.22pt
      ]
        ({\i/32},{22/32})
        rectangle
        ({(\i+1)/32},{23/32});
    }

    \foreach \i in {11,...,20}{
      \path[
        fill=gray!40,
        fill opacity=0.28,
        draw=black!55,
        line width=0.22pt
      ]
        ({\i/32},{23/32})
        rectangle
        ({(\i+1)/32},{24/32});
    }

    \foreach \i in {12,...,19}{
      \path[
        fill=gray!40,
        fill opacity=0.28,
        draw=black!55,
        line width=0.22pt
      ]
        ({\i/32},{24/32})
        rectangle
        ({(\i+1)/32},{25/32});
    }

  \end{scope}

  \coordinate (p1) at (0.33,0.62); 
  \coordinate (p2) at (0.84,0.62); 
  \coordinate (p3) at (0.33,0.18); 
  \coordinate (p4) at (0.84,0.18); 

  \draw[black,line width=0.9pt]
    (p1) -- (p2) -- (p4) -- (p3) -- cycle;

  \fill[black] (p1) circle (0.008);
  \fill[black] (p2) circle (0.008);
  \fill[black] (p3) circle (0.008);
  \fill[black] (p4) circle (0.008);

  \def\sigHalf{0.022}     
  \def\sigHalfV{0.022}    
  \def\sigLW{1.25pt}      

  \draw[black,line width=\sigLW]
    ($(p1)+(-0.028,0.030)+(-\sigHalf,0)$) --
    ($(p1)+(-0.028,0.030)+(\sigHalf,0)$);
  \draw[black,line width=\sigLW]
    ($(p1)+(-0.028,0.030)+(0,-\sigHalfV)$) --
    ($(p1)+(-0.028,0.030)+(0,\sigHalfV)$);

  \draw[black,line width=\sigLW]
    ($(p2)+(0.028,0.030)+(-\sigHalf,0)$) --
    ($(p2)+(0.028,0.030)+(\sigHalf,0)$);

  \draw[black,line width=\sigLW]
    ($(p3)+(-0.028,-0.030)+(-\sigHalf,0)$) --
    ($(p3)+(-0.028,-0.030)+(\sigHalf,0)$);

  \draw[black,line width=\sigLW]
    ($(p4)+(0.028,-0.030)+(-\sigHalf,0)$) --
    ($(p4)+(0.028,-0.030)+(\sigHalf,0)$);
  \draw[black,line width=\sigLW]
    ($(p4)+(0.028,-0.030)+(0,-\sigHalfV)$) --
    ($(p4)+(0.028,-0.030)+(0,\sigHalfV)$);

  \node[font=\Large\bfseries,anchor=south east] at (0.98,1.00)
    {$\boldsymbol{K=[0,1]^2}$};

  \node[font=\large] at (0.67,0.30)
    {$\boldsymbol{G}$};

  \node[font=\large] at (0.48,0.56)
    {$\boldsymbol{F_\ell}$};

  \node[font=\large,anchor=west] at (0.77,0.82)
    {dyadic cells};

  \draw[black,line width=0.8pt]
    (0.78,0.8) -- (0.61,0.765);

  \node[font=\large,anchor=east,align=center] at (0.16,0.56)
    {4-point\\mass exchange};

  \draw[black,line width=0.8pt]
    (0.0,0.49) -- (0.32,0.32);

\end{tikzpicture}%
}}%
\caption[Marginal-preserving dyadic exchange increasing target-set mass]{Compact-core and dyadic-cell illustration for Theorem~\ref{thm_VaRExample} in two dimensions. The domain is the unit square $K=[0,1]^2$. The Borel target set $G\subseteq K$ is shown in grey, and a compact core $F_\ell\subseteq G$ is shown in darker grey. Dyadic cells intersecting $F_\ell$ are highlighted. The four corners of the rectangle depict an equal-mass $2\times2$ exchange:
mass is removed at the two points marked ``$-$'' and added at the two
points marked ``$+$''. This preserves both coordinate marginals. The
upper-left added-mass point lies in $G$, while the other three points
lie in $K\setminus G$, so the exchange strictly increases $G$-mass. See the proof for details.}
\label{fig:theorem219_dyadic_compact_core_exchange}
\end{figure}
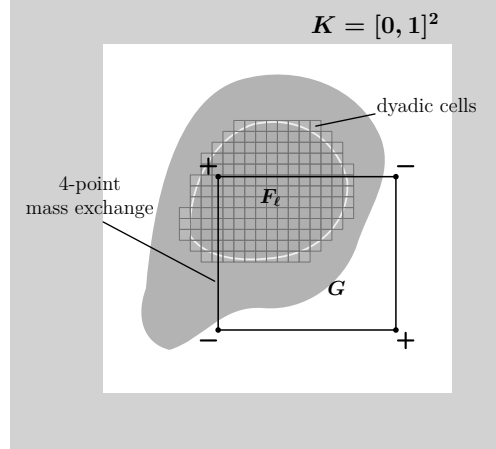

\medskip
\noindent
\textbf{(b) Fibrewise fixed points, extremal support, and tight mass exchanges.}

Let $\mathbb P\in\mathcal D$ and let
$0<\alpha\leqslant\mathbb P$. For $\rho\in\mathbb R$ and
$\beta\in\mathcal C(\alpha)$, define $h_{\rho,G}^{\alpha}(\beta)
:=
\beta(G)-\rho\,\beta(K)$
and $\ \Psi_{\alpha,G}(\rho)
:=
\sup_{\beta\in\mathcal C(\alpha)}
h_{\rho,G}^{\alpha}(\beta)$.
Writing $\Gamma_G(\alpha)
:=
\sup_{\beta\in\mathcal C(\alpha)}\beta(G)$,
one has $\Psi_{\alpha,G}(\rho)
=
\Gamma_G(\alpha)-\rho\,\alpha(K)$ 
whose unique root is $\widehat\rho_G(\alpha)
=
\frac{\Gamma_G(\alpha)}{\alpha(K)}$. 

A prior $\mathbb P^\ast\in\mathcal D$ is extremal for $V(G)$ if and only if
every nonzero submeasure $0<\alpha\leqslant\mathbb P^\ast$ satisfies the
fibrewise fixed-point conditions
\begin{align}
(E1)_{\alpha,G} &:\ 
\alpha
\in
\operatorname*{arg\,max}_{\beta\in\mathcal C(\alpha)}
h_{\rho_\alpha,G}^{\alpha}(\beta),
\label{eqn_fibre_E1}
\\
(E2)_{\alpha,G} &:\ 
\rho_\alpha
=
\frac{\alpha(G)}{\alpha(K)}.
\label{eqn_fibre_E2}
\end{align}
Equivalently, $\ \Gamma_G(\alpha)=\alpha(G)
\ \text{for every }
0<\alpha\leqslant\mathbb P^\ast$. Thus, an extremal prior is hereditarily fibre-extremal: every removable
portion of its support already maximizes $G$-mass among all positive
measures having the same induced coordinate marginals. 

For an extremiser $\mathbb P^\ast$, $\mathfrak X_G^0(\mathbb P^\ast)
:=
\bigl\{
(\alpha,\beta)\in\mathfrak X(\mathbb P^\ast)
\mathrel{\big|}
\allowbreak
\Delta_G(\alpha,\beta)=0
\bigr\}$ is the class of tight exchanges for $\mathbb P^\ast$. Then, $\ \operatorname*{arg\,max}_{\mathbb P\in\mathcal D}\mathbb P(G)
=
\left\{
\mathbb P^\ast-\alpha+\beta:
(\alpha,\beta)\in\mathfrak X_G^0(\mathbb P^\ast)
\right\}$. Hence, tight exchanges generate the complete family of extremal solutions.

\noindent
\textbf{(c) Fixed-point solutions of intermediate approximation problems.}

There exist increasing compact sets $F_\ell\subseteq G$ such that $V(F_\ell)\uparrow V(G)
\ 
\text{as }\ell\to\infty$. For every fixed $\ell$, there exists a nested sequence of discretised
optimisation problems, with $m$-th optimal value
$\phi_{\ell,m}^{\ast}$ such that $\phi_{\ell,m}^{\ast}\downarrow V(F_\ell)$ as $m\to\infty$.  It follows that $V(G)
=
\lim_{\ell\to\infty}
\lim_{m\to\infty}
\phi_{\ell,m}^\ast$. Each $\phi_{\ell,m}^{\ast}$ is the value component of a fixed-point
pair that solves the corresponding discretised problem. For every fixed $\ell$, there exists a sequence of discrete extremal
priors for these discretised problems that converges weakly to some $\mathbb P_\ell^\ast\in\mathcal D$
satisfying $\mathbb P_\ell^\ast(F_\ell)
=
V(F_\ell)$.
\end{theorem}

\begin{proof}
Appendix~P, Supplementary Material (\cite{Salako_Muhammad_2025_suppmat}).
\end{proof}


\noindent\textbf{Remark: }The support characterisation in Theorem~\ref{thm_VaRExample} differs from the other example support characterisations in the framework. In those results, the Dinkelbach difference identifies extremal support locations through local or cellwise minimisation
conditions. Here, the Dinkelbach difference instead acts on marginal-preserving mass exchanges. Because such exchanges
preserve total mass, their Dinkelbach difference reduces to $\beta(G)-\alpha(G)$, the resulting change in event probability. Its vanishing defines the tight exchanges that generate an equivalence class of extremal priors---an ``exchange-orbit'' of priors with potentially different support---that all attain the
extremal event value.
\begin{corollary}[Application to Value-at-Risk aggregation with prescribed
marginals and unknown dependence]
\label{Cor_VaRExample}
Let $K=[0,1]^n$, let $\mathcal D$ be the set of probability measures on
$K$ with uniform coordinate marginals, and let $H_1,\ldots,H_n$ be
continuous distribution functions with bounded supports. For
$j=1,\ldots,n$, define
\begin{equation}
Q_j^-(u)
:=
\inf\{x\in\mathbb R:H_j(x)\geqslant u\},
\quad
Q_j^+(u)
:=
\inf\{x\in\mathbb R:H_j(x)>u\},
\quad u\in(0,1),
\label{eqn_VaR_quantile_versions}
\end{equation}
with endpoint values given by the corresponding one-sided limits, and put
\begin{equation}
S^-(\mathbf u)
:=
\sum_{j=1}^n Q_j^-(u_j),
\qquad
S^+(\mathbf u)
:=
\sum_{j=1}^n Q_j^+(u_j),
\qquad \mathbf u\in K.
\label{eqn_VaR_aggregate_versions}
\end{equation}
For $\mathbb P\in\mathcal D$ and $\eta\in(0,1)$, define
\begin{equation}
\operatorname{VaR}_{\eta}(S;\mathbb P)
:=
\inf\{s\in\mathbb R:\mathbb P(S^-\leqslant s)>\eta\}
=
\sup\{s\in\mathbb R:\mathbb P(S^+\geqslant s)\geqslant1-\eta\}.
\label{eqn_symmetric_VaR_representation}
\end{equation}
The robust VaR bounds are
\begin{equation}
\operatorname{VaR}_{\eta}^{\mathrm{lower}}
:=
\inf_{\mathbb P\in\mathcal D}
\operatorname{VaR}_{\eta}(S;\mathbb P),
\qquad
\operatorname{VaR}_{\eta}^{\mathrm{upper}}
:=
\sup_{\mathbb P\in\mathcal D}
\operatorname{VaR}_{\eta}(S;\mathbb P).
\label{eqn_reduced_VaR_target_problems}
\end{equation}

For $s\in\mathbb R$, define the paired threshold events
\begin{equation}
C_s
:=
\{\mathbf u\in K:S^-(\mathbf u)\leqslant s\},
\qquad
G_s
:=
\{\mathbf u\in K:S^+(\mathbf u)\geqslant s\},
\label{eqn_VaR_threshold_events}
\end{equation}
and the extremal envelopes
\begin{equation}
G_r^+(s)
:=
\sup_{\mathbb P\in\mathcal D}\mathbb P(S^-\leqslant s)
=
V(C_s),
\qquad
G_r^-(s)
:=
\inf_{\mathbb P\in\mathcal D}\mathbb P(S^+<s)
=
1-V(G_s).
\label{eqn_reduced_envelopes}
\end{equation}
Then \eqref{eqn_reduced_VaR_target_problems} is equivalent to
\begin{equation}
\operatorname{VaR}_{\eta}^{\mathrm{lower}}
=
\inf\{s\in\mathbb R:G_r^+(s)>\eta\},
\qquad
\operatorname{VaR}_{\eta}^{\mathrm{upper}}
=
\max\{s\in\mathbb R:G_r^-(s)\leqslant\eta\}.
\label{eqn_reduced_lower_VaR}
\end{equation}

Using Theorem~\ref{thm_VaRExample}, for every $\mathbb P\in\mathcal D$, the exchange residuals are
\begin{align}
\mathscr R_{C_s}(\mathbb P)
&=
G_r^+(s)-\mathbb P(C_s),
\label{eqn_VaR_positive_residual}
\\
\mathscr R_{G_s}(\mathbb P)
&=
1-G_r^-(s)-\mathbb P(G_s).
\label{eqn_VaR_negative_residual}
\end{align}
For every $s\in\mathbb R$, there exist exchange fixed-point pairs
\[
\bigl(\mathbb P_s^{+,\ast},G_r^+(s)\bigr)
\quad\text{and}\quad
\bigl(\mathbb P_s^{-,\ast},1-G_r^-(s)\bigr)
\]
satisfying the respective
$(E1)_{C_s}$--$(E2)_{C_s}$ and
$(E1)_{G_s}$--$(E2)_{G_s}$ systems. 
The extremal priors are precisely those satisfying
\begin{align}
\beta(C_s)-\alpha(C_s)
&\leqslant0
\quad
\forall(\alpha,\beta)\in
\mathfrak X(\mathbb P_s^{+,\ast}),
\label{eqn_VaR_positive_fixed_point}
\\
\beta(G_s)-\alpha(G_s)
&\leqslant0
\quad
\forall(\alpha,\beta)\in
\mathfrak X(\mathbb P_s^{-,\ast}).
\label{eqn_VaR_negative_fixed_point}
\end{align}
Moreover,
\begin{align}
\alpha(C_s)
&=
\sup_{\beta\in\mathcal C(\alpha)}\beta(C_s)
\quad
\forall\,0<\alpha\leqslant\mathbb P_s^{+,\ast},
\label{eqn_VaR_positive_hereditary_support}
\\
\alpha(G_s)
&=
\sup_{\beta\in\mathcal C(\alpha)}\beta(G_s)
\quad
\forall\,0<\alpha\leqslant\mathbb P_s^{-,\ast}.
\label{eqn_VaR_negative_hereditary_support}
\end{align}
All other extremal priors are obtained by tight exchanges for the corresponding threshold event.

There exist coherently nested discretised problems with optimal values
$\phi_{s,m}^+$ and $\phi_{s,m}^-$ satisfying
\begin{align}
\phi_{s,m}^+
&\downarrow
G_r^+(s),
\label{eqn_VaR_upper_envelope_limit}
\\
\phi_{s,m}^-
&\downarrow
1-G_r^-(s)
=
V(G_s)
\label{eqn_VaR_lower_envelope_limit}
\end{align}
as $m\to\infty$. For every fixed $s$, there are sequences of extremal priors for the
discretised $C_s$ and $G_s$ problems, with weakly convergent
subsequences whose limits attain $V(C_s)$ and $V(G_s)$, respectively. For $\gamma\in(0,1)$, define
\begin{equation}
q_m^+(\gamma)
:=
\inf\{s\in\mathbb R:\phi_{s,m}^+>\gamma\},
\qquad
q_m^-(\gamma)
:=
\sup\{s\in\mathbb R:\phi_{s,m}^-\geqslant1-\gamma\}.
\label{eqn_finite_symmetric_VaR_inversions}
\end{equation}
For every $\eta\in(0,1)$, continuity levels
$\eta_r\downarrow\eta$ may be chosen such that
\begin{equation}
q_m^+(\eta_r)
\longrightarrow
\operatorname{VaR}_{\eta_r}^{\mathrm{lower}}
\ \text{as }m\to\infty,
\qquad
\operatorname{VaR}_{\eta_r}^{\mathrm{lower}}
\downarrow
\operatorname{VaR}_{\eta}^{\mathrm{lower}}
\ \text{as }r\to\infty.
\label{eqn_stable_level_lower_VaR_convergence}
\end{equation}
The upper robust bound is obtained directly at level $\eta$:
\begin{equation}
q_m^-(\eta)
\downarrow
\operatorname{VaR}_{\eta}^{\mathrm{upper}}
\ 
\text{as }m\to\infty.
\label{eqn_direct_upper_VaR_convergence}
\end{equation}

Writing
\begin{equation}
s_\eta^\ast
:=
\operatorname{VaR}_{\eta}^{\mathrm{upper}}
=
\max\{s\in\mathbb R:V(G_s)\geqslant1-\eta\},
\label{eqn_upper_VaR_extremal_threshold}
\end{equation}
every extremal prior $\mathbb P_{s_\eta^\ast}^{-,\ast}$ for
$G_{s_\eta^\ast}$ attains the robust upper VaR:
\begin{equation}
\operatorname{VaR}_{\eta}
\bigl(S;\mathbb P_{s_\eta^\ast}^{-,\ast}\bigr)
=
\operatorname{VaR}_{\eta}^{\mathrm{upper}}.
\label{eqn_upper_VaR_attainment}
\end{equation}
Thus, the two robust VaR bounds are obtained by threshold inversion of the extremal-value functions associated with the paired fixed-point problems for $C_s$ and $G_s$.
\end{corollary}

\begin{proof}
Appendix~P, Supplementary Material (\cite{Salako_Muhammad_2025_suppmat}).
\end{proof}
\subsection{Sharp Lower Bound for Two-Asset Basket Options in Finance}Consider the two-asset basket option problem studied by \cite{LaurenceWang2005Basket}. In that setting, the Problem~\eqref{eqn_genCBIprob} objective is the expected payoff of a contract whose value depends on the sum of two underlying quantities. Given only the one-dimensional marginal distributions of those quantities, one seeks the sharpest possible lower and upper bounds on that expected payoff that are consistent with absence of arbitrage. In the two-asset lower-bound case, \cite{LaurenceWang2005Basket} obtain a closed form. We re-derive this lower bound within the fixed-point framework and highlight the added insights this brings. 
Let
\[
K:=[0,1]^2,
\qquad
\mathcal{D}
:=
\left\{
\mathbb P\in\mathcal{P}(K):
(\pi_1)_{\#}\mathbb P=\lambda,\ (\pi_2)_{\#}\mathbb P=\lambda
\right\},
\]
where $\lambda$ denotes Lebesgue measure on $[0,1]$, the $\pi_1,\pi_2$ are coordinate
projections, and $(\pi_i)_{\#}\mathbb P$ is the push-forward measure of $\mathbb P$ under $\pi_i$. Thus, $\mathcal{D}$ is the class of all couplings of two
$\mathrm{Unif}[0,1]$ marginals; i.e. the $M=2$ special case of Proposition~\ref{prop_gen_multiple_marginals_piecewise_refined} with
$\rho_1\equiv \rho_2\equiv 1$ and $U_1=U_2=[0,1]$. Let $Q_1,Q_2:[0,1]\to\mathbb{R}$ be continuous increasing functions, let $k\in\mathbb{R}$,
and define
\[
f(\textnormal{\bf x})
:=
\bigl(Q_1(u)+Q_2(v)-k\bigr)^+,
\qquad
g(\textnormal{\bf x})\equiv 1,
\qquad
\textnormal{\bf x}=(u,v)\in K.
\]
Set
\begin{equation}
\phi^\ast
:=
\inf_{\mathbb P\in\mathcal{D}}
\frac{\mathbb{E}_{\mathbb P}[f(\textnormal{\bf X})]}{\mathbb{E}_{\mathbb P}[g(\textnormal{\bf X})]}
=
\inf_{\mathbb P\in\mathcal{D}}
\mathbb{E}_{\mathbb P}[f(\textnormal{\bf X})].
\label{eqn_assetbasketoptimisation}
\end{equation}
For each $m\in\mathbb{N}$, write $N_m:=2^m$, define the dyadic intervals
\[
I^{(m)}_1:=\Bigl[0,\frac{1}{N_m}\Bigr]\quad\text{ and }\quad I^{(m)}_i:=\Bigl(\frac{i-1}{N_m},\frac{i}{N_m}\Bigr],
\quad
2\leqslant i\leqslant N_m,
\]
and the dyadic product cells
\[C_{ij}^{(m)}
:=
I^{(m)}_i
\times
I^{(m)}_j,
\qquad
1\leqslant i,j\leqslant N_m.
\]
Define the lower and upper cell envelopes of the original basket payoff
\[
a_{ij}^{(m)}
:=
\inf_{\textnormal{\bf x}\in C_{ij}^{(m)}} f(\textnormal{\bf x}),
\qquad
b_{ij}^{(m)}
:=
\sup_{\textnormal{\bf x}\in C_{ij}^{(m)}} f(\textnormal{\bf x}).
\]

We also define the full-support stage-$m$ cell array by a choice of payoff for each cell: $F_{ij}^{(m)}
:=
f(\textbf{x}_{ij})$ for some chosen $\textbf{x}_{ij}^{(m)}\in \overline C_{ij}^{(m)}$ ($1\leqslant i,j\leqslant N_m$) such that $\pi_1(\mathbf{x}_{ij}^{(m)})=\pi_1(\mathbf{x}_{il}^{(m)})$ and $\pi_2(\mathbf{x}_{ij}^{(m)})=\pi_2(\mathbf{x}_{kj}^{(m)})$ for all $i,\; j,\; k,\; l$.
These define piecewise-constant stage-$m$ approximants, $f_m^{\mathrm{full}}(\textnormal{\bf x})
=
F_{ij}^{(m)}$ for $\textnormal{\bf x}\in C_{ij}^{(m)}$. Clearly, for every $m\in\mathbb{N}$ and every $1\leqslant i,j\leqslant N_m$, we have $a_{ij}^{(m)}\leqslant F_{ij}^{(m)}\leqslant b_{ij}^{(m)}$.

Finally, let
\[
{\mathcal W}_m
=
\left\{
\textnormal{\textbf{w}}=(w_{ij})_{1\leqslant i,j\leqslant N_m}\in[0,\infty)^{N_m\times N_m}:
\sum_{j=1}^{N_m} w_{ij}=\frac{1}{N_m},\
\sum_{i=1}^{N_m} w_{ij}=\frac{1}{N_m}
\right\},
\]
and set
\[
\underline{\phi}_m
:=
\inf_{\textnormal{\textbf{w}}\in {\mathcal W}_m}\sum_{i,j=1}^{N_m} w_{ij}a_{ij}^{(m)},
\qquad
\overline{\phi}_m
:=
\inf_{\textnormal{\textbf{w}}\in {\mathcal W}_m}\sum_{i,j=1}^{N_m} w_{ij}b_{ij}^{(m)},
\]
\[
\phi_m^{\mathrm{full}}
:=
\inf_{\textnormal{\textbf{w}}\in {\mathcal W}_m}\sum_{i,j=1}^{N_m} w_{ij}F_{ij}^{(m)}.
\]

We now recursively define stage-$m$ anti-diagonal dyadic families.
At level $r=0$, let $\mathscr{B}_{m,0}
:=
\bigl\{\{1,\dots,N_m\}\times\{1,\dots,N_m\}\bigr\}$.
Suppose $\mathscr{B}_{m,r-1}$ has been defined. For each block
\[
B=I\times J\in\mathscr{B}_{m,r-1},
\]
with $I$ and $J$ contiguous intervals of equal cardinality $2^{m-r+1}$, split
\[
I=I^{-}\sqcup I^{+},
\qquad
J=J^{-}\sqcup J^{+}
\]
into their lower and upper halves, and define the four children
\[
B^{\mathrm{SW}}:=I^{-}\times J^{-},
\quad
B^{\mathrm{NW}}:=I^{-}\times J^{+},
\quad
B^{\mathrm{SE}}:=I^{+}\times J^{-},
\quad
B^{\mathrm{NE}}:=I^{+}\times J^{+}.
\]
Set $\mathscr{B}_{m,r}
:=
\bigcup_{B\in\mathscr{B}_{m,r-1}}
\{B^{\mathrm{NW}},B^{\mathrm{SE}}\}$.
Let $S_{m,r}
:=
\bigcup_{B\in\mathscr{B}_{m,r}}
\bigcup_{(i,j)\in B} C_{ij}^{(m)}$.
In particular, $S_{m,m}
=
\bigcup_{i=1}^{N_m} C_{i,N_m+1-i}^{(m)}$, the union of the dyadic cells whose interiors intersect the anti-diagonal $v=1-u$.

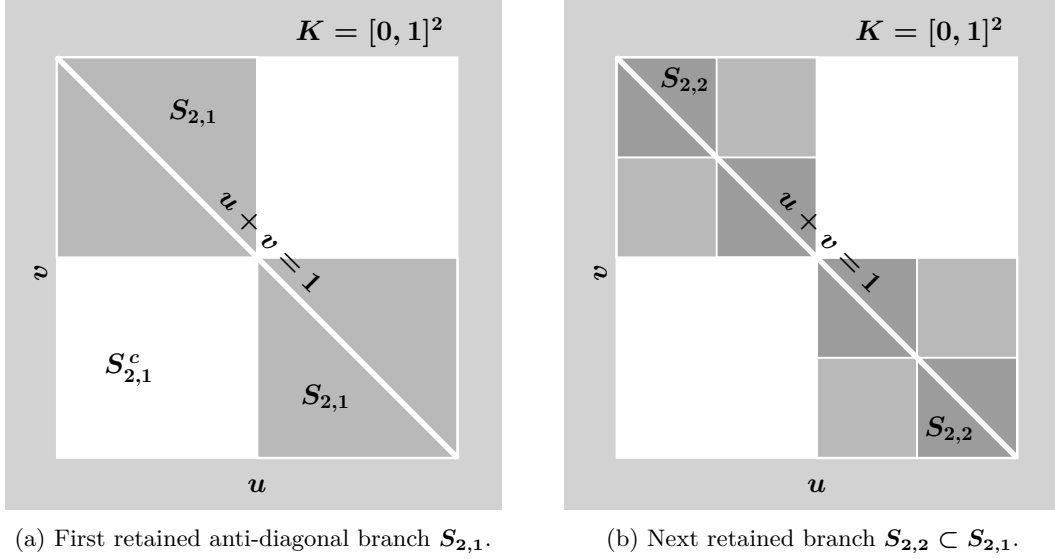
\begin{figure}[t!]
\centering

\begin{subfigure}[t]{0.47\linewidth}
\centering
\resizebox{\linewidth}{!}{%
\setlength{\fboxsep}{6pt}%
\colorbox{gray!35}{%
\begin{tikzpicture}[x=7.2cm,y=7.2cm,>=Latex]

  \fill[gray!35] (-0.10,-0.10) rectangle (1.08,1.08);
  \fill[white] (0,0) rectangle (1,1);

  \fill[gray!55] (0,0.5) rectangle (0.5,1);
  \fill[gray!55] (0.5,0) rectangle (1,0.5);

  \draw[white,line width=0.9pt] (0,0) rectangle (1,1);

  \draw[white,line width=1.2pt] (0,0.5) rectangle (0.5,1);
  \draw[white,line width=1.2pt] (0.5,0) rectangle (1,0.5);

  \draw[white,line width=2.8pt] (0,1) -- (1,0);

  \node[font=\Large\bfseries, anchor=north] at (0.50,-0.04) {$\boldsymbol{u}$};
  \node[font=\Large\bfseries, anchor=east, rotate=90] at (-0.04,0.50) {$\boldsymbol{v}$};

  \node[font=\Large\bfseries, anchor=south east] at (0.98,1.01)
    {$\boldsymbol{K=[0,1]^2}$};

  \node[font=\Large\bfseries, text=black] at (0.18,0.22)
    {$\boldsymbol{S_{2,1}^{\,c}}$};

  \node[font=\Large\bfseries, text=black] at (0.34,0.86)
    {$\boldsymbol{S_{2,1}}$};

  \node[font=\Large\bfseries, text=black] at (0.67,0.15)
    {$\boldsymbol{S_{2,1}}$};

  \node[font=\Large\bfseries, text=black, rotate=-45] at (0.53,0.54)
    {$\boldsymbol{u+v=1}$};

\end{tikzpicture}%
}}%
\caption{First retained anti-diagonal branch $\boldsymbol{S_{2,1}}$.}
\label{fig:anti_diagonal_s21}
\end{subfigure}
\hfill
\begin{subfigure}[t]{0.47\linewidth}
\centering
\resizebox{\linewidth}{!}{%
\setlength{\fboxsep}{6pt}%
\colorbox{gray!35}{%
\begin{tikzpicture}[x=7.2cm,y=7.2cm,>=Latex]

  \fill[gray!35] (-0.10,-0.10) rectangle (1.08,1.08);
  \fill[white] (0,0) rectangle (1,1);

  \fill[gray!55] (0,0.5) rectangle (0.5,1);
  \fill[gray!55] (0.5,0) rectangle (1,0.5);

  \fill[gray!78] (0,0.75) rectangle (0.25,1);
  \fill[gray!78] (0.25,0.5) rectangle (0.5,0.75);
  \fill[gray!78] (0.5,0.25) rectangle (0.75,0.5);
  \fill[gray!78] (0.75,0) rectangle (1,0.25);

  \draw[white,line width=0.9pt] (0,0) rectangle (1,1);

  \draw[white,line width=1.2pt] (0,0.5) rectangle (0.5,1);
  \draw[white,line width=1.2pt] (0.5,0) rectangle (1,0.5);

  \draw[white,line width=1.0pt] (0,0.75) rectangle (0.25,1);
  \draw[white,line width=1.0pt] (0.25,0.5) rectangle (0.5,0.75);
  \draw[white,line width=1.0pt] (0.5,0.25) rectangle (0.75,0.5);
  \draw[white,line width=1.0pt] (0.75,0) rectangle (1,0.25);

  \draw[white,line width=2.8pt] (0,1) -- (1,0);

  \node[font=\Large\bfseries, anchor=north] at (0.50,-0.04) {$\boldsymbol{u}$};
  \node[font=\Large\bfseries, anchor=east, rotate=90] at (-0.04,0.50) {$\boldsymbol{v}$};

  \node[font=\Large\bfseries, anchor=south east] at (0.98,1.01)
    {$\boldsymbol{K=[0,1]^2}$};




  \node[font=\Large\bfseries, text=black] at (0.17,0.94)
    {$\boldsymbol{S_{2,2}}$};

  \node[font=\Large\bfseries, text=black] at (0.83,0.07)
    {$\boldsymbol{S_{2,2}}$};

  \node[font=\Large\bfseries, text=black, rotate=-45] at (0.53,0.54)
    {$\boldsymbol{u+v=1}$};

\end{tikzpicture}%
}}%
\caption{Next retained branch $\boldsymbol{S_{2,2}\subset S_{2,1}}$.}
\label{fig:anti_diagonal_s22}
\end{subfigure}

\caption[Dyadic concentration of extremal mass on the anti-diagonal]{Dyadic anti-diagonal branch selection in the basket-payoff construction. Figure~\ref{fig:anti_diagonal_s21} shows the first retained set $\boldsymbol{S_{2,1}}$, consisting of the two level-1 dyadic squares whose interiors intersect the anti-diagonal. Figure~\ref{fig:anti_diagonal_s22} shows the next retained set inside these squares: the darker region is $\boldsymbol{S_{2,2}}$, obtained by keeping only the level-2 dyadic children that continue to hug the anti-diagonal $\boldsymbol{u+v=1}$. Thus, the nested sets $\boldsymbol{S_{m,r}}$ concentrate increasingly tightly around the anti-diagonal, consistent with $\operatorname{supp}(\mathbb P_m^{\mathrm{ad}})\subseteq S_{m,m}\subseteq \{(u,v)\in K: |u+v-1|\leqslant 2^{-m}\}$.}
\label{fig:anti_diagonal_nested_squares}
\end{figure}

\begin{theorem}[Canonical strike-payoff formula]
\label{thm:canonical-strike-recursive}
With the above notation, the following hold.

\begin{enumerate}
\item For every $m\in\mathbb{N}$ and every $r\in\{0,1,\dots,m\}$,
\begin{equation}
\phi_m^{\mathrm{full}}
=
\inf_{\substack{\mathbf{w}\in{\mathcal W}_m\\ \operatorname{supp}(\mathbf w)\subseteq\!\!\!\! \bigcup\limits_{B\in\mathscr{B}_{m,r}}\!\!\!\!\!\!\!\!B}}
\,\sum_{i,j=1}^{N_m} w_{ij}F_{ij}^{(m)}.
\label{eqn_BasketPayoff_1}
\end{equation}

\item In particular,
\begin{equation}
\phi_m^{\mathrm{full}}
=
\frac{1}{N_m}\sum_{i=1}^{N_m}F_{i,N_m+1-i}^{(m)}.
\label{eqn_BasketPayoff_2}
\end{equation}

\item The values $\phi_m^{\mathrm{full}}$ converge to
\begin{equation}
\phi^\ast
=
\int_0^1 \bigl(Q_1(u)+Q_2(1-u)-k\bigr)^+\,du.
\label{eqn_BasketPayoff_3}
\end{equation}

\item For each $m$, let
\[
\mathbb P_m^{\mathrm{ad}}
=
\frac1{N_m}\sum_{i=1}^{N_m}\delta_{\textnormal{\textbf{x}}_i^{(m)}},
\qquad
\textnormal{\textbf{x}}_i^{(m)}
=
\left(\frac{i-\frac12}{N_m},\,1-\frac{i-\frac12}{N_m}\right).
\]
Then
\begin{equation}
\operatorname{supp}(\mathbb P_m^{\mathrm{ad}})
\subseteq S_{m,m}
\subseteq
\{(u,v)\in K:\ |u+v-1|\leqslant 2^{-m}\}.
\label{eqn_BasketPayoff_4}
\end{equation}
Moreover, for each $m$, one may choose a stage-$m$ extremal anti-diagonal distribution
\[
\mathbb P_{m,\star}^{\mathrm{ad}}
=
\frac1{N_m}\sum_{i=1}^{N_m}\delta_{\textnormal{\textbf{x}}_{i,\star}^{(m)}},
\qquad
\textnormal{\textbf{x}}_{i,\star}^{(m)}\in \overline C_{i,N_m+1-i}^{(m)},
\]
such that
\[
\max_{1\leqslant i\leqslant N_m}\|\textnormal{\textbf{x}}_i^{(m)}-\textnormal{\textbf{x}}_{i,\star}^{(m)}\|
\leqslant \max_{1\leqslant i\leqslant N_m}\operatorname{diam}(C_{i,N_m+1-i}^{(m)})\to 0.
\]
In particular, there is convergence in Wasserstein-1 distance,
\begin{equation}
W_1(\mathbb P_m^{\mathrm{ad}},\mathbb P_{m,\star}^{\mathrm{ad}})\to 0.
\label{eqn_Wass1Dist}
\end{equation}
\end{enumerate}
\end{theorem}

\begin{proof}
Appendix~Q, Supplementary Material (\cite{Salako_Muhammad_2025_suppmat}).
\end{proof}
\noindent\textbf{Remarks: }
Theorem~\ref{thm:canonical-strike-recursive} does more than recover the Laurence--Wang lower bound in closed form. It also exhibits a \emph{selection mechanism} for the extremal branch. In particular,
the anti-diagonal/countermonotone solution is not only identified as an optimiser in the limit; it is selected at each finite stage $m$ by recursive uncrossing on dyadic blocks, which yields a canonical anti-diagonal-support representative of the stage-$m$ extremal class.

The theorem's proof also makes the \emph{support geometry} of the extremals explicit. The selected stage-$m$ extremal representatives are supported on cells contained in the shrinking tube $S_{m,m}\subseteq \{(u,v)\in K:\ |u+v-1|\leqslant 2^{-m}\}$, so the anti-diagonal manifold appears as an asymptotic support ``attractor'' for the extremal equivalence classes selected by the recursion. The Wasserstein estimate $W_1(\mathbb P_m^{\mathrm{ad}},\mathbb P_{m,\star}^{\mathrm{ad}})\to 0$ further shows that the
canonical midpoint branch is not merely a convenient visual representative, but is asymptotically equivalent to an exact stage-$m$ extremal solution branch.

More broadly, the proof of Theorem~\ref{thm:canonical-strike-recursive} illustrates a reusable methodological pattern within the fixed-point framework: one first controls theorem-covered stage-$m$ values by lower and upper envelope problems, then uses structural properties of approximants to select canonical surrogate branches (ideally with closed-form limits) that are shown, via a sandwiching argument, to be asymptotically extremal. Thus, beyond the closed-form lower bound itself, the framework supplies a finite-level branch-selection principle, explicit support localisation, and a general approximation strategy for extracting canonical extremisers from broader admissible classes.

%
%
%
%
%
\FloatBarrier
\subsection{Robust Next-State Inference for a 2-State Markov Chain with Random Transition Probabilities} 
Fix $n\geqslant 1$, $0<t<\epsilon<1$, $0<\theta<1$, and let $\Theta=[0,1]^{1+2n}$. Here, $n$ is the number of events in the observed history of a $2$-state Markov chain (with states $0$ and $1$) and $\Theta$ is the parameter space for the chain. Write a point $\textnormal{\textbf{q}}\in\Theta$ as $\textnormal{\textbf{q}}=
\bigl(
q^{(1)}_1,
q^{(2)}_{01},q^{(2)}_{11},
\ldots,
q^{(n+1)}_{01},q^{(n+1)}_{11}
\bigr)$. The initial probability of state-$1$ is $q^{(1)}_1$,
so that the initial distribution in the chain is
\[
Q^{(1)}_0(\textnormal{\textbf{q}}):=1-q^{(1)}_1,
\qquad
Q^{(1)}_1(\textnormal{\textbf{q}}):=q^{(1)}_1.
\]
For $i=2,\ldots,n+1$, define the transition probabilities
\[
Q^{(i)}_{01}(\textnormal{\textbf{q}}):=q^{(i)}_{01},
\qquad
Q^{(i)}_{00}(\textnormal{\textbf{q}}):=1-q^{(i)}_{01},\qquad Q^{(i)}_{11}(\textnormal{\textbf{q}}):=q^{(i)}_{11},
\qquad
Q^{(i)}_{10}(\textnormal{\textbf{q}}):=1-q^{(i)}_{11}.
\]
Equivalently, the $i$-th transition matrix is
\[
Q^{(i)}(\textnormal{\textbf{q}})
:=
\begin{pmatrix}
1-q^{(i)}_{01} & q^{(i)}_{01}\\
1-q^{(i)}_{11} & q^{(i)}_{11}
\end{pmatrix},
\qquad i=2,\ldots,n+1.
\]

Let $U=\{0,1\}^n$, so $U$ contains $m=2^n$ elements. For $\sigma=(\sigma^{(1)},\ldots,\sigma^{(n)})\in U$, define the likelihood of observing the $n$-history $\sigma$ by
\[
G_\sigma(\textnormal{\textbf{q}})
=
Q^{(1)}_{\sigma^{(1)}}(\textnormal{\textbf{q}})
\prod_{i=2}^{n}
Q^{(i)}_{\sigma^{(i-1)}\sigma^{(i)}}(\textnormal{\textbf{q}}),
\]
with the convention that the empty product is $1$ when $n=1$.

The probability of state-$0$ occurring next, conditional on the most recent state $\sigma^{(n)}$, is
\[
S_\sigma(\textnormal{\textbf{q}})
=
Q^{(n+1)}_{\sigma^{(n)}0}(\textnormal{\textbf{q}})
=
1-Q^{(n+1)}_{\sigma^{(n)}1}(\textnormal{\textbf{q}}).
\]
Thus, the likelihood of observing history $\sigma$ followed by event $n+1$ being state-$0$ is
\[
L_\sigma(\textnormal{\textbf{q}})
=
G_\sigma(\textnormal{\textbf{q}})S_\sigma(\textnormal{\textbf{q}})
=
Q^{(1)}_{\sigma^{(1)}}(\textnormal{\textbf{q}})
\prod_{i=2}^{n}
Q^{(i)}_{\sigma^{(i-1)}\sigma^{(i)}}(\textnormal{\textbf{q}})
\bigl(1-Q^{(n+1)}_{\sigma^{(n)}1}(\textnormal{\textbf{q}})\bigr).
\]
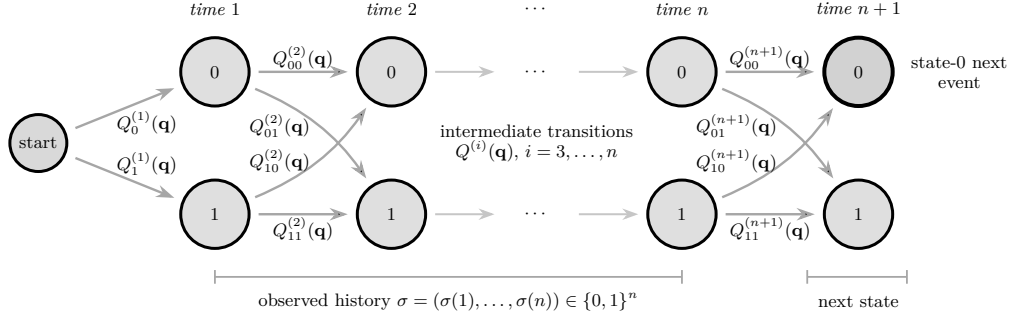
\begin{figure}[t!]
\resizebox{0.95\linewidth}{!}{%
\centering
\begin{tikzpicture}[
    x=1cm, y=1cm,
    >=Stealth,
    font=\small,
start/.style={
    circle,
    draw=black,
    line width=1.5pt,
    fill=gray!30,
    text=black,
    minimum size=10mm,
    inner sep=0pt
},
state/.style={
    circle,
    draw=black,
    line width=1.5pt,
    fill=gray!25,
    text=black,
    minimum size=12mm,
    inner sep=0pt
},
target/.style={
    circle,
    draw=black,
    line width=2pt,
    fill=gray!35,
    text=black,
    minimum size=12mm,
    inner sep=0pt
},
    trans/.style={
        draw=gray!70,
        line width=1.2pt,
        -{Stealth[length=3.2mm,width=2.2mm]},
        shorten <=4pt,
        shorten >=4pt
    },
    inittrans/.style={
        draw=gray!70,
        line width=1.2pt,
        -{Stealth[length=3.2mm,width=2.2mm]},
        shorten <=5pt,
        shorten >=4pt
    },
        soft/.style={
        draw=gray!45,
        line width=1.0pt,
        -{Stealth[length=2.8mm,width=2.0mm]},
        shorten <=4pt,
        shorten >=4pt
    },
    lab/.style={
        text=black,
        fill=white,
        fill opacity=0,
        text opacity=1,
        inner sep=1pt
    },
    timelab/.style={
        text=black,
        font=\small\itshape
    },
    annot/.style={
        text=black,
        fill=white,
        align=center,
        inner sep=2pt
    }
]

\node[start]  (start) at (-3.10,0) {start};

\node[state]  (s10) at (0,  1.25) {$0$};
\node[state]  (s11) at (0, -1.25) {$1$};

\node[state]  (s20) at (3.1,  1.25) {$0$};
\node[state]  (s21) at (3.1, -1.25) {$1$};

\node[state]  (sn0) at (8.2,  1.25) {$0$};
\node[state]  (sn1) at (8.2, -1.25) {$1$};

\node[target] (sp0) at (11.3,  1.25) {$0$};
\node[state]  (sp1) at (11.3, -1.25) {$1$};

\node[timelab] at (0,     2.35) {time $1$};
\node[timelab] at (3.1,   2.35) {time $2$};
\node[timelab] at (5.65,  2.35) {$\cdots$};
\node[timelab] at (8.2,   2.35) {time $n$};
\node[timelab] at (11.3,  2.35) {time $n+1$};

\draw[inittrans]
    (start) -- node[lab, below=1pt, xshift=8pt, pos=0.56] {$Q^{(1)}_{0}(\mathbf{q})$} (s10);

\draw[inittrans]
    (start) -- node[lab, above=1pt, xshift=8pt, pos=0.56] {$Q^{(1)}_{1}(\mathbf{q})$} (s11);

\draw[trans] (s10) edge node[lab, above] {$Q^{(2)}_{00}(\mathbf{q})$} (s20);
\draw[trans] (s10) edge[bend left=16] node[lab, left=2pt] {$Q^{(2)}_{01}(\mathbf{q})$} (s21);
\draw[trans] (s11) edge[bend right=16] node[lab, left=2pt] {$Q^{(2)}_{10}(\mathbf{q})$} (s20);
\draw[trans] (s11) edge node[lab, below] {$Q^{(2)}_{11}(\mathbf{q})$} (s21);

\node at (5.65,  1.25) {$\cdots$};
\node at (5.65, -1.25) {$\cdots$};

\draw[soft] (s20) -- (5.05,  1.25);
\draw[soft] (s21) -- (5.05, -1.25);
\draw[soft] (6.25,  1.25) -- (sn0);
\draw[soft] (6.25, -1.25) -- (sn1);

\node[annot] at (5.65,0)
    {intermediate transitions\\[-1pt] $Q^{(i)}(\mathbf{q})$, $i=3,\dots,n$};

\draw[trans] (sn0) edge node[lab, above] {$Q^{(n+1)}_{00}(\mathbf{q})$} (sp0);
\draw[trans] (sn0) edge[bend left=16] node[lab, left=2pt] {$Q^{(n+1)}_{01}(\mathbf{q})$} (sp1);
\draw[trans] (sn1) edge[bend right=16] node[lab, left=2pt] {$Q^{(n+1)}_{10}(\mathbf{q})$} (sp0);
\draw[trans] (sn1) edge node[lab, below] {$Q^{(n+1)}_{11}(\mathbf{q})$} (sp1);

\draw[gray!70, line width=0.9pt]
    (0,-2.35) -- (8.2,-2.35);
\draw[gray!70, line width=0.9pt]
    (0,-2.23) -- (0,-2.47);
\draw[gray!70, line width=0.9pt]
    (8.2,-2.23) -- (8.2,-2.47);

\node[annot] at (4.1,-2.78)
    {observed history $\sigma=(\sigma(1),\ldots,\sigma(n))\in\{0,1\}^n$};

\draw[gray!70, line width=0.9pt]
    (10.4,-2.35) -- (12.15,-2.35);
\draw[gray!70, line width=0.9pt]
    (10.4,-2.23) -- (10.4,-2.47);
\draw[gray!70, line width=0.9pt]
    (12.15,-2.23) -- (12.15,-2.47);

\node[annot] at (11.3,-2.78)
    {next state};

\node[annot, anchor=west] at (12.15,1.25)
    {state-$0$ next\\[-1pt] event};

\end{tikzpicture}%
}
\caption[The 2-state Markov chain structure for robust next-state inference]{The 2-state Markov chain structure used in Theorem~\ref{thm_Gen2stateMarkovChain}.
A parameter value $\mathbf{q}\in\Theta_{\mathrm{feas}}$ determines the initial law and the time-indexed
transition probabilities. The first $n$ states form the observed history $\sigma\in\{0,1\}^n$;
the terminal transition determines the next-state event used in the robust objective.}
\label{fig:theorem222_markov_chain}
\end{figure}
For each $\sigma_i\in U$, write $L_i(\textnormal{\textbf{q}}):=L_{\sigma_i}(\textnormal{\textbf{q}})$ and $G_i(\textnormal{\textbf{q}}):=G_{\sigma_i}(\textnormal{\textbf{q}})$. Fix a target history $\sigma^\star\in U$, and define $f(\textnormal{\textbf{q}}):=L_{\sigma^\star}(\textnormal{\textbf{q}})$, $g(\textnormal{\textbf{q}}):=G_{\sigma^\star}(\textnormal{\textbf{q}})$. All functions $G_\sigma$, $S_\sigma$, and $L_\sigma$ are polynomial, hence continuous,
on $\Theta$.

Let
\[
\Theta_{\mathrm{feas}}
=
\bigcap_{i=1}^m\{\textnormal{\textbf{q}}\in\Theta:L_i(\textnormal{\textbf{q}})\geqslant t\}.
\]
On $\Theta_{\mathrm{feas}}$ we have $g(\textnormal{\textbf{q}})=G_{\sigma^\star}(\textnormal{\textbf{q}})
\geqslant
L_{\sigma^\star}(\textnormal{\textbf{q}})
\geqslant t$, because $0\leqslant S_{\sigma^\star}(\textnormal{\textbf{q}})\leqslant1$ and $L_{\sigma^\star}(\textnormal{\textbf{q}})=G_{\sigma^\star}(\textnormal{\textbf{q}})S_{\sigma^\star}(\textnormal{\textbf{q}})$. Let $\mathcal D$ be the class of all ${\mathbb P}\in\mathcal P(\Theta)$ satisfying ${\mathbb P}(\Theta_{\mathrm{feas}})=1$ and, for every $i=1,\ldots,m$, ${\mathbb P}(t\leqslant L_i<\epsilon)=\theta$, ${\mathbb P}(\epsilon\leqslant L_i\leqslant 1)=1-\theta$. Assume $\mathcal D\neq\varnothing$ and define
\begin{equation}
\Phi^\ast
=
\inf_{{\mathbb P}\in\mathcal D}
\frac{{\mathbb E}_{\mathbb P}[f(\textnormal{\textbf{Q}})]}{{\mathbb E}_{\mathbb P}[g(\textnormal{\textbf{Q}})]}.
\label{eqn_2stateMarkovOptimisation}
\end{equation}
\eqref{eqn_2stateMarkovOptimisation} is the primary problem of interest here. The objective function is the probability that state $0$ occurs next after observing $n$ past events in the chain's history.

Recalling the definition of $m$ and using the notation $[m]:=\{1,\ldots,m\}$, for each $S\subseteq[m]$, define the half-open atom
\[
A_S
:=
\{\textnormal{\textbf{q}}\in\Theta_{\mathrm{feas}}:
L_i(\textnormal{\textbf{q}})\in[t,\epsilon)\text{ if }i\in S,\
L_i(\textnormal{\textbf{q}})\in[\epsilon,1]\text{ if }i\notin S\}.
\]
Let $\mathcal M_{\mathrm{miss}}
:=
\{S\subseteq[m]:A_S=\varnothing\}$. For each nonmissing atom, define its relative closure in $\Theta_{\mathrm{feas}}$ by $\overline A_S:=\operatorname{cl}_{\Theta_{\mathrm{feas}}}(A_S)$. For $S\subseteq[m]$, let $p_S$ denote the mass assigned to $A_S$. Define
\[
C_{\mathrm{full}}=[c_S]_{S\subseteq[m]},
\qquad
c_S=
\begin{pmatrix}
1\\
\mathbf 1_S
\end{pmatrix}
\in{\mathbb R}^{m+1},\qquad
b=
\begin{pmatrix}
1\\
\theta\\
\vdots\\
\theta
\end{pmatrix}.
\]
The atom-mass polytope is
\[
\mathcal P
=
\{p\in{\mathbb R}^{2^m}:
C_{\mathrm{full}}p=b,\
p_S=0\text{ for }S\in\mathcal M_{\mathrm{miss}},\
p_S\geqslant 0\text{ for }S\notin\mathcal M_{\mathrm{miss}}\}.
\]

For $\phi\geqslant 0$, the Dinkelbach difference is $h_\phi(\textnormal{\textbf{q}})=f(\textnormal{\textbf{q}})-\phi g(\textnormal{\textbf{q}})$ (\emph{cf.} Definition~\ref{def:Dinkelbachdifference}).
For every nonmissing atom, define
\[
m_S(\phi)
:=
\inf_{\textnormal{\textbf{q}}\in A_S}h_\phi(\textnormal{\textbf{q}})
=
\min_{\textnormal{\textbf{q}}\in\overline A_S}h_\phi(\textnormal{\textbf{q}}),
\]
and
\[
\overline N_S(\phi)
:=
\arg\min_{\textnormal{\textbf{q}}\in\overline A_S}h_\phi(\textnormal{\textbf{q}}).
\]
Write $m(\phi):=(m_S(\phi))_{S\subseteq[m]}$. For every $S\subseteq[m]$ with $|S|\geqslant 2$, define
\[
d_S
=
e_S+(|S|-1)e_\varnothing-\sum_{i\in S}e_{\{i\}},
\]
where $e_T$ denotes the standard coordinate vector indexed by $T\subseteq[m]$. Let $\underline D$ be the matrix whose columns are the $d_S$'s. Then the columns of $\underline D$ form a basis of $\ker C_{\mathrm{full}}$. For $p^\ast\in\mathcal P$, define the feasible exchange cone
\[
\mathcal A(p^\ast)
:=
\{\alpha:
({\underline D}\alpha)_S=0\text{ for }S\in\mathcal M_{\mathrm{miss}},
\quad
({\underline D}\alpha)_S\geqslant 0
\text{ for }S\notin\mathcal M_{\mathrm{miss}}\text{ with }p_S^\ast=0
\}.
\]

\begin{theorem}[General 2-state Markov chain problem]
\label{thm_Gen2stateMarkovChain}
There exist $p^\ast\in\mathcal P$ and points $\bar {\textnormal{\textbf{q}}}_S^\ast\in\overline A_S$, $S\in\mathcal S_{p^\ast}:=\{S\subseteq[m]:p_S^\ast>0\}$, such that a fixed-point tuple $(\phi^\ast,p^\ast,\bar {\textnormal{\textbf{q}}}^\ast)$
(with $\bar {\textnormal{\textbf{q}}}^\ast=(\bar {\textnormal{\textbf{q}}}_S^\ast)_{S\in\mathcal S_{p^\ast}}$) solves \eqref{eqn_2stateMarkovOptimisation} and satisfies the following three conditions:

(E1$_q^{\mathrm{cl}}$) \textbf{Extremal support-locations.}
\[
\bar {\textnormal{\textbf{q}}}_S^\ast\in\overline N_S(\phi^\ast)
=
\arg\min_{\textnormal{\textbf{q}}\in\overline A_S}h_{\phi^\ast}(\textnormal{\textbf{q}})\;\; \text{for } S\in\mathcal S_{p^\ast}.
\]

(E1$_p$) \textbf{Mass-exchange tightness.}
\[
\alpha^\top {\underline D}^\top m(\phi^\ast)\geqslant 0
\;\;\text{for every }\alpha\in\mathcal A(p^\ast).
\]
Equivalently, $p^\ast\in\arg\underset{p\in\mathcal P}{\min}\;\; p^\top m(\phi^\ast)$.


(E2$_{\mathrm{obj}}$) \textbf{Objective-value balance.} The infimum $\Phi^\ast$ has the value $\phi^\ast$, where 
\[
\phi^\ast =\frac{
\sum_{S\in\mathcal S_{p^\ast}}
p_S^\ast f(\bar {\textnormal{\textbf{q}}}_S^\ast)
}{
\sum_{S\in\mathcal S_{p^\ast}}
p_S^\ast g(\bar {\textnormal{\textbf{q}}}_S^\ast)
}.
\]

The closure support points need not lie in the half-open atoms $A_S$. However, for every $\eta>0$ and every active atom $S\in\mathcal S_{p^\ast}$, there exists ${\textnormal{\textbf{q}}}_{S,\eta}\in A_S$ such that $\|{\textnormal{\textbf{q}}}_{S,\eta}-\bar {\textnormal{\textbf{q}}}_S^\ast\|\leqslant \eta$ and 
$h_{\phi^\ast}({\textnormal{\textbf{q}}}_{S,\eta})
\leqslant
h_{\phi^\ast}(\bar {\textnormal{\textbf{q}}}_S^\ast)+\eta
=
m_S(\phi^\ast)+\eta$. Consequently, the discrete prior
\[
{\mathbb P}_\eta
=
\sum_{S\in\mathcal S_{p^\ast}}p_S^\ast\delta_{{\textnormal{\textbf{q}}}_{S,\eta}}
\]
belongs to $\mathcal D$ and satisfies $0
\leqslant
{\mathbb E}_{{\mathbb P}_\eta}[h_{\phi^\ast}(\textnormal{\textbf{Q}})]
\leqslant
\eta$. Hence,
\[
\frac{{\mathbb E}_{{\mathbb P}_\eta}[f(\textnormal{\textbf{Q}})]}
{{\mathbb E}_{{\mathbb P}_\eta}[g(\textnormal{\textbf{Q}})]}
\to
\Phi^\ast
\qquad
\text{as }\eta\downarrow 0.
\]

Finally, define the Dinkelbach residual function $F(\phi):=\min_{p\in\mathcal P}p^\top m(\phi)$. Then the scalar residual equation $F(\phi)=0$ has the unique solution $\phi=\Phi^\ast$. Moreover, $F$ is continuous and strictly
decreasing. Hence, any bracketed scalar root-finding method for $F$, in particular bisection,
converges to $\Phi^\ast$. 

If $\phi_r\to\Phi^\ast$ and $p^{(r)}\in\arg\min_{p\in\mathcal P}p^\top m(\phi_r)$, then every accumulation point of $\{p^{(r)}\}$ satisfies $(E1_p)$ and, after choosing
closure minimisers on its active atoms, forms a fixed-point tuple satisfying
$(E1_q^{\mathrm{cl}})$ and $(E2_{\mathrm{obj}})$ at $\Phi^\ast$. 
\end{theorem}

\begin{proof}
Appendix~R, Supplementary Material (\cite{Salako_Muhammad_2025_suppmat}).
\end{proof}
%
%
%
%
%
%
%
%
%
%
%
%
%
%
%
%
%
%
%
%
%
%
%
%
%

%
%
%
\section{Discussion: Interpretation, Convergence, and Extensions}
\label{sec_discussion}

For non-trivial problems, the extremal distributions are often ``degenerate'': although the parameters may, in principle, take any value in the hypercube $K$, most parameter values are impossible under these distributions. In Bayesian applications, this degeneracy does not conflict with Bayesian principles. In particular, for conservative Bayesian inference, it is the class of feasible priors defined by the partial prior specifications---rather than any single degenerate extremal prior---that represents one’s prior beliefs. As new observations arrive, every feasible prior is updated, and the extremal priors may change.

In the following sense, Bayesian applications of Problem~\eqref{eqn_genCBIprob} ``converge'' to traditional Bayesian inference, in the limit of more detailed partial specifications defined over refinements of the partition of $K$ (\emph{cf.} convergence of $\epsilon$-contaminations to traditional Bayesian inference as $\epsilon\to 0$; e.g. in \cite{berger1990_SensitivityToPrior}, \cite{moreno1991robust}): 
\begin{proposition}[Convergence to ordinary Bayesian inference]
\label{prop_convergence_to_ordinary_BayesianInference}
Let $\mu\in\mathcal P(K)$ be a fixed Borel probability measure over $K$. For each $m\in\mathbb N$, let  $\mathcal K^{(m)}=\{K^{(m)}_1,\dots,K^{(m)}_{n_m}\}$ be a finite Borel partition of $\mathrm{supp}(\mu)$ in $K$, such that $\mu(K_i^{(m)})>0$ for $i=1,\dots,n_m$. Assume $\max\limits_{1\leqslant i\leqslant n_m}\operatorname{diam}\!\left(K_i^{(m)}\right)\to 0$ (as $m\to\infty$). Let $\mathcal D_m:=\Bigl\{\mathbb P\in\mathcal P(K):\mathbb P\!\left(K_i^{(m)}\right)=\mu\!\left(K_i^{(m)}\right)
\text{for }i=1,\dots,n_m\Bigr\}$. Let $f,g:K\to[0,\infty)$ be continuous and $\mathbb E_{\mu}[g]>0$. Restrict to all sufficiently large $m$, thereby ensuring $\mathbb E[g]>0$ for every $\mathbb P\in\mathcal D_m$. Define $\,L_m:=\inf_{\mathcal D_m}\frac{\mathbb E[f]}{\mathbb E[g]}\,$ and $\,U_m:=\sup_{\mathcal D_m}\frac{\mathbb E[f]}{\mathbb E[g]}$. Then,
\[
L_m\to \frac{\mathbb E_\mu[f]}{\mathbb E_\mu[g]},
\qquad
U_m\to \frac{\mathbb E_\mu[f]}{\mathbb E_\mu[g]}.
\]
\end{proposition}
\begin{proof}Appendix~S, Supplementary Material (\cite{Salako_Muhammad_2025_suppmat}).\end{proof}
This convergence is consistent with the general principle in robust Bayesian inference that the more constrained an inference problem becomes, the tighter the bounds that sandwich the objective function in question (e.g. see Section~4 of \cite{WassermanLavineWolpert1993}).

The infimum of the objective function in \eqref{eqn_weak_opt_equiv_disc} is attained by certain functional forms related to the objective function: forms ultimately determined by (E1)--(E2) conditions in the Theorems (e.g. Theorem~\ref{thm:global_fp_tuple}). An extremal prior assigns probability mass to $n$ locations in its domain---probability $p_i$ is assigned to location $i$. When locations coincide, the probability the prior assigns at the common location is the sum of the corresponding $p_i$'s. Different extremal priors can have different support locations. The $i$-th location of an extremal prior lies in the closure of the measurable subset $K_i$, and hence is the limit of points in $K_i$. Many previously published explicit solutions of \eqref{eqn_genCBIprob} have each been constrained by finitely many such functional forms. (E1)--(E2) also make clear how the extremal priors depend on $f$, $g$, the $p_i$ probabilities, and the partition of $K$. 

Disparate analogues of (E$1$) and (E$2$) exist in the literature, but without the computational and approximating extremal prior convergence guarantees presented here (\emph{cf.} Remark~1 of \cite{LiseoMorenoSalinetti1996-GivenMarginals} for (E$1$) and Example~1 of \cite{WassermanLavineWolpert1993} for (E$2$)). The guarantees come at the price of less generality in the classes of $f$, $g$ functions. 


A common theme across many of these solutions is that fixed points often behave like ``attractors'' or ``repellers''---attracting or repelling locations where probability masses are assigned in the domain of the prior. The optimisation constraints induce a partition of the prior's domain, and the fixed points dictate those locations where probability mass should be assigned (in each subset of the partition) to achieve conservative results. For some problems, the fixed point exclusively attracts or repels. For other problems, the fixed point attracts locations in some subsets while repelling locations in other subsets. An example is the five solutions in the proof of the theorem in \cite{zhao_assessing_2019}; the placement of probability masses for each of these solutions is determined by the attraction and repulsion of the fixed point $r/(r+k)$ (i.e. the MLE for $x^r(1-x)^{k}$). 

\def\rval{1}
\def\mval{3}
\def\kval{2}
\def\xsing{\fpeval{\rval / (\rval + \kval)}}
\def\xroot{\fpeval{\rval / (\rval + \mval + \kval)}}
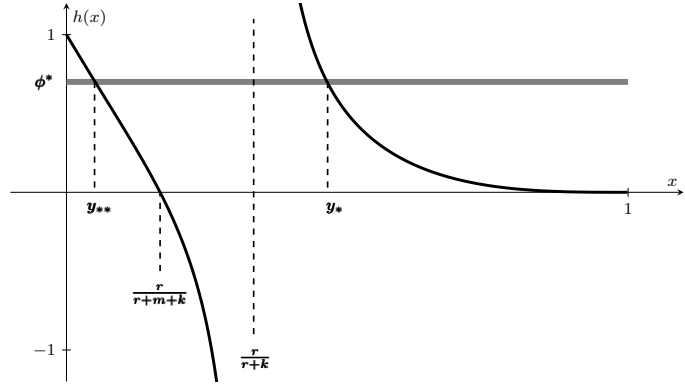
\begin{figure}[htbp!]
\begin{center}
\scalebox{0.78}
{\begin{tikzpicture}
  \begin{axis}[
    width=13cm,
    height=8cm,
    domain=0:1,
    samples=400,
    xlabel={$x$},
    ylabel={$h(x)$},
    axis lines=middle,
    xtick distance=1,
    ytick distance=1,
    enlargelimits=true,
    tick style={black},
    ymin=-1, ymax=1,
  ]

  \draw[line width=1mm,gray] ({axis cs:0,0.7}) -- ({axis cs:1,0.7});
  \node[black] at (axis cs:-0.01,0.7) [anchor=east] {$\pmb{\phi^*}$};

  \addplot[
    domain=0:\fpeval{\xsing - 0.001},
    line width = 0.5mm,
    black,
  ]
  {(1 - x)^\mval * (\rval - x*(\mval + \kval + \rval)) / ( \rval - x*(\kval + \rval) )};

  \addplot[
    domain=\fpeval{\xsing + 0.001}:0.999,
    line width = 0.5mm,
    black,
  ]
  {(1 - x)^\mval * (\rval - x*(\mval + \kval + \rval)) / ( \rval - x*(\kval + \rval) )};


  \draw[dashed,thick,black] ({axis cs:\xsing,-0.9}) -- ({axis cs:\xsing,1.1});
  \node[black] at (axis cs:\xsing+0.003,-1.06) {$\pmb{\frac{r}{r+k}}$};


  \draw[dashed,thick,black] ({axis cs:\xroot,-0.5}) -- ({axis cs:\xroot,0});
  \node[black] at (axis cs:\xroot,-0.65) {$\pmb{\frac{r}{r+m+k}}$};

  \draw[dashed,thick,black] ({axis cs:0.05,0.7}) -- ({axis cs:0.05,0});
  \node[black] at (axis cs:0.06,-0.1) {$\pmb{y_{**}}$};

  \draw[dashed,thick,black] ({axis cs:0.465,0.7}) -- ({axis cs:0.465,0});
  \node[black] at (axis cs:0.479,-0.1) {$\pmb{y_*}$};

  \end{axis}
\end{tikzpicture}}
\end{center}\caption[Attractor $y^*$ and repeller $y^{**}$ from $h(x)=\phi^*$]{How $\phi^*$, $y_{**}$, and $y_*$ are related by $h(x) = (1 - x)^m\left(\frac{r - x(m + k + r)}{r - x(k + r)}\right)$. Reproduced from \cite{salako2025conservative}.} 
\label{fig_h_stationarypoints}
\end{figure}

As a particular example, the components of Theorem~\ref{thm_CBIwithfails_sol}'s fixed--point triplet~---~i.e. $\phi^*$, $y_{**}$, $y_*$~---~are related by the function $h: [0,1]\setminus\{\frac{r}{r+k}\}\rightarrow \mathbb R$, $h(x) = (1 - x)^m\left(\frac{r - x(m + k + r)}{r - x(k + r)}\right)$; see Figure~\ref{fig_h_stationarypoints}. The gradient of the objective function defines $h$, where $h$ determines when gradient components are non-trivially zero. The roots of the difference $h-\phi$, between $h$ and the objective function $\phi$, imply the fixed point ``attracts'' (w.r.t. $y_{*}$) and ``repels'' (w.r.t. $y_{**}$). In terms of the more general setting of Theorem~\ref{thm:global_fp_tuple}, the $h-\phi$ difference is a rescaling of the Dinkelbach slope
$\nabla h_{\phi,f,g}(\mathbf{x})$ (see Definition~\ref{def:Dinkelbachdifference}), with the sign of the scaling factor changing at $x=r/(r+k)$. In this one-dimensional example, the ``repeller'' $y^{**}<r/(r+k)$ is a root at which $\nabla h_{\phi^*,f,g}(\mathbf{x})$ passes monotonically from
positive to negative, whereas the ``attractor''
$y^*>r/(r+k)$ is a root at which it passes monotonically from negative to positive. More generally, any interior point $x_i^*$
satisfying (E1) of Theorem~\ref{thm:global_fp_tuple} is a local ``attractor'' in this terminology, since it is a local minimiser of $h_{\phi^*,f,g}$. See Appendix~G of the Supplementary Material (\cite{Salako_Muhammad_2025_suppmat}).

As highlighted by Theorem~\ref{thm:global_fp_tuple}, ``attractors'' and ``repellers'' are not necessarily ``points''; they can be ``line segments'', ``surfaces'', or topological manifolds in the prior's domain. For example, the extremal priors for \eqref{eqn_xKlotzlklhdFn_maintxt} have line segments in their domain that ``attract'' univariate probability density (e.g. see \cite{SalakoZhao_TSE_2023}). 

Another example arises in Problem~\eqref{eqn_assetbasketoptimisation} and its solution in Theorem~\ref{thm:canonical-strike-recursive} (extending the solution in \cite{LaurenceWang2005Basket}): the attractor is the anti-diagonal of the unit square (see Figure~\ref{fig:anti_diagonal_nested_squares}), which is the support of the limiting extremal distribution. Moreover, for non-negative $X$ and $Y$, the classical lower bound for the
$\mathbb E[XY]$ problem with prescribed marginals
(see \cite{Whitt1976BivariateDistributions}) can be recovered; e.g. consider the following special case obtained in a manner directly
analogous to Theorem~\ref{thm:canonical-strike-recursive}. Writing
$X=F^{-1}(U)$ and $Y=G^{-1}(V)$, where $F^{-1}$ and $G^{-1}$
denote quantile functions that extend to continuous functions on $[0,1]$ and $U,V\sim{\rm Unif}(0,1)$, the problem reduces to minimizing $\mathbb{E}\!\left[F^{-1}(U)G^{-1}(V)\right]$ over all couplings of $U$ and $V$. The finite marginal approximants lead to the same uncrossing structure
as in Theorem~\ref{thm:canonical-strike-recursive}: mass assigned to
positively ordered pairs can be rearranged to oppositely ordered pairs
without increasing the objective. Consequently, the extremal approximants
concentrate on the anti-diagonal in quantile coordinates; in the limit
the extremal prior is supported on $\{(u,v)\in[0,1]^2: v=1-u\}$.
The lower bound is $\inf_{X\sim F,\;Y\sim G}\mathbb E[XY]=\int_0^1 F^{-1}(u)G^{-1}(1-u)\,du$. Equivalently, in $(X,Y)$-coordinates, the anti-diagonal is mapped to the one-dimensional curve $u\mapsto \bigl(F^{-1}(u),G^{-1}(1-u)\bigr)$: the intermediate extremal $(X,Y)$-priors concentrate on this curve, generalising the Theorem~\ref{thm:canonical-strike-recursive} anti-diagonal attractor.


Theorems~\ref{thm:2d-gk-fixed-point-extension-corrected} and \ref{thm_VaRExample} offer more general examples of ``attractors''---Borel subsets of the unit square and unit hypercube, respectively. The $K\cap C^c$ attractor in Theorem~\ref{thm:2d-gk-fixed-point-extension-corrected} only attracts probability mass that can reach it through vertical movement, leaving the remaining mass in vertical sections from which no such escape is possible; see Figure~\ref{fig:fibrewise_repair_theorem218}. In Theorem~\ref{thm_VaRExample}, mass is attracted from multiple locations and assigned to multiple locations (some lying within the attractor $G$) only when all of these locations form part of an admissible mass exchange: see Figure~\ref{fig:theorem219_dyadic_compact_core_exchange}.

The exchange formulation of Theorem~\ref{thm_VaRExample} also reveals a two-level orbit structure. At the level of admissibility, every admissible prior in the fixed-marginal class $\mathcal{D}$ can be obtained from any other admissible prior by a complete marginal-preserving mass exchange; thus, viewed from any $\mathbb P\in\mathcal{D}$, the entire feasible class $\mathcal D$ is $\mathbb P$'s complete exchange orbit. At the level of extremality, whenever the target event-probability bound is attained, the exchanges having zero Dinkelbach difference---the tight exchanges---generate the complete family of extremal priors from any extremal prior $\mathbb P^*$. The extremal solution set is, in general, a smaller tight-exchange orbit nested within the admissible exchange orbit. This global classification sharpens the attractor interpretation above. An extremising attractor includes the target event $G$ and describes the objective-improving direction of admissible mass exchange: mass can be reassigned so as to increase $\mathbb P(G)$ only through exchanges that preserve the prescribed marginals. The orbit structure describes the corresponding reachability: the complete exchange orbit through any admissible prior is the entire feasible class, and the tight-exchange orbit through any extremal prior is the entire extremal prior class; see Theorem~\ref{thm_VaRExample}(a)--(b), Figure~\ref{fig:theorem219_dyadic_compact_core_exchange}, and Appendix~P, Steps~1 and~6. Corollary~\ref{Cor_VaRExample} carries the same structure into prescribed-marginal robust VaR problems.

Degenerate extremal distributions supported on lower-dimensional manifolds also occur for multidimensional $\epsilon$-contamination problems; e.g. 2-dimensional extreme priors with support on 1-dimensional curves in the priors' domains---see Theorem~1 and Remark~1 of \cite{LiseoMorenoSalinetti1996-GivenMarginals}, Remark~2 of \cite{MorenoCano1995-CollectionOfSets}, and \cite{LavineWassermanWolpert1991-SpecifiedMarginals}.

Most of the results of this paper are subject to finite or discretised setwise mass constraints: Borel partition masses, overlapping Borel-event masses, fixed marginal laws expressed as cylinder-set probabilities, and combinations/refinements/limits of these. A notable exception is Proposition~\ref{prop_cbi_transform}, which considers moment constraints. A natural extension would be more general linear integral constraints (e.g. see \cite{karr1983extreme,BetroRuggeriMeczarski1994-JSPI}) where one fixes finitely many quantities of the form $\mathbb E[f_i]=c_i$; then, the current paper's finite atom-mass polytopes arising from Borel refinements would be replaced, or augmented, by consistency polytopes that encode the values of constraint functions on refined atoms. Karr-type results describe which finite supports can be extremal, while extensions using this fixed-point framework could potentially also identify where extremal support should concentrate, and approximate the corresponding extremal distributions by recursively solving related finite discretised problems.

\bibliographystyle{ba}
\bibliography{bibliography}

@article{BetroRuggeriMeczarski1994-JSPI,
  author  = {Bruno Betr{\`o} and Fabrizio Ruggeri and Marek M\k{e}czarski},
  title   = {Robust Bayesian analysis under generalized moments conditions},
  journal = {Journal of Statistical Planning and Inference},
  year    = {1994},
  volume  = {41},
  number  = {3},
  pages   = {257--266},
  month   = oct,
  issn    = {0378-3758},
  doi     = {10.1016/0378-3758(94)90022-1},
  url     = {https://doi.org/10.1016/0378-3758(94)90022-1}
}

@article{klotz_statistical_1973,
	title = {Statistical {inference} in {Bernoulli} {trials} with {dependence}},
	volume = {1},
	issn = {00905364},
	number = {2},
	journal = {The Annals of Statistics},
	author = {Klotz, Jerome},
	year = {1973},
	pages = {373--379},
    url = {https://doi.org/10.1214/aos/1176342377}
}

@article{1991_Lavine,
 ISSN = {01621459, 1537274X},
 URL = {http://www.jstor.org/stable/2290583},
 author = {Michael Lavine},
 journal = {Journal of the American Statistical Association},
 number = {414},
 pages = {396--399},
 publisher = {[American Statistical Association, Taylor & Francis, Ltd.]},
 title = {Sensitivity in Bayesian Statistics: The Prior and the Likelihood},
 urldate = {2025-11-04},
 volume = {86},
 year = {1991}
}

@incollection{LiseoMorenoSalinetti1996-GivenMarginals,
  author    = {Brunero Liseo and El\'{\i}as Moreno and Gabriella Salinetti},
  title     = {Bayesian Robustness for Classes of Bidimensional Priors with Given Marginals},
  booktitle = {Bayesian Robustness},
  series    = {Institute of Mathematical Statistics Lecture Notes{\textemdash}Monograph Series},
  volume    = {29},
  editor    = {Joseph O. Berger and Bruno Betr{\`o} and El\'{\i}as Moreno and Luis R. Pericchi and Fabrizio Ruggeri and Gabriella Salinetti and Larry Wasserman},
  publisher = {Institute of Mathematical Statistics},
  address   = {Hayward, CA},
  year      = {1996},
  pages     = {101--118},
  doi       = {10.1214/lnms/1215453063},
  url       = {https://doi.org/10.1214/lnms/1215453063}
}

@article{LavineWassermanWolpert1991-SpecifiedMarginals,
  author  = {Michael Lavine and Larry Wasserman and Robert L. Wolpert},
  title   = {Bayesian Inference with Specified Prior Marginals},
  journal = {Journal of the American Statistical Association},
  year    = {1991},
  volume  = {86},
  number  = {416},
  pages   = {964--971},
  doi     = {10.1080/01621459.1991.10475139},
  url     = {https://doi.org/10.1080/01621459.1991.10475139}
}

@article{MorenoCano1995-CollectionOfSets,
  author  = {El\'{\i}as Moreno and Juan A. Cano},
  title   = {Classes of bidimensional priors specified on a collection of sets: Bayesian robustness},
  journal = {Journal of Statistical Planning and Inference},
  year    = {1995},
  volume  = {46},
  number  = {3},
  pages   = {325--334},
  issn    = {0378-3758},
  doi     = {10.1016/0378-3758(94)00134-H},
  url     = {https://doi.org/10.1016/0378-3758(94)00134-H}
}

@article{salako2025constructive,
  title         = {Constructive Proofs of Generalized Boole--Frechet Bounds: A Dynamic Programming Approach},
  author        = {Salako, Kizito},
  journal       = {{arXiv preprint}},
  year          = {2025},
  eprint        = {2512.09161},
  archivePrefix = {arXiv},
  primaryClass  = {math.PR},
  url={https://doi.org/10.48550/arXiv.2512.09161},
  doi           = {10.48550/arXiv.2512.09161}
}

@mastersthesis{Karush1939,
  author = {Karush, William},
  title  = {Minima of Functions of Several Variables with Inequalities as Side Conditions},
  school = {University of Chicago},
  address = {Chicago, IL},
  year   = {1939},
  note   = {Department of Mathematics}
}

@inproceedings{KuhnTucker1951,
  author    = {Kuhn, H. W. and Tucker, A. W.},
  title     = {Nonlinear Programming},
  booktitle = {Proceedings of the Second Berkeley Symposium on Mathematical Statistics and Probability},
  editor    = {Neyman, Jerzy},
  pages     = {481--492},
  publisher = {University of California Press},
  address   = {Berkeley},
  year      = {1951}
}

@book{NocedalWright2006,
  author    = {Nocedal, Jorge and Wright, Stephen J.},
  title     = {Numerical Optimization},
  edition   = {2nd},
  series    = {Springer Series in Operations Research and Financial Engineering},
  publisher = {Springer},
  address   = {New York},
  year      = {2006},
  doi       = {10.1007/978-0-387-40065-5},
  url = {https://doi.org/10.1007/978-0-387-40065-5}
}

@article{Hailperin_1965,
 ISSN = {00029890, 19300972},
url={https://www.jstor.org/stable/2313491},
 doi = {10.2307/2313491},
 author = {Theodore Hailperin},
 journal = {The American Mathematical Monthly},
 number = {4},
 pages = {343--359},
 publisher = {[Taylor & Francis, Ltd., Mathematical Association of America]},
 title = {Best Possible Inequalities for the Probability of a Logical Function of Events},
 urldate = {2025-01-03},
 volume = {72},
 year = {1965}
}

@article{Whitt1976BivariateDistributions,
  author  = {Whitt, Ward},
  title   = {Bivariate Distributions with Given Marginals},
  journal = {The Annals of Statistics},
  year    = {1976},
  volume  = {4},
  number  = {6},
  pages   = {1280--1289},
  doi     = {10.1214/aos/1176343660},
  url     = {https://doi.org/10.1214/aos/1176343660}
}

@article{karr1983extreme,
  author  = {Karr, Alan F.},
  title   = {Extreme Points of Certain Sets of Probability Measures, with Applications},
  journal = {Mathematics of Operations Research},
  year    = {1983},
  volume  = {8},
  number  = {1},
  pages   = {74--85},
  doi     = {10.1287/moor.8.1.74},
  month   = feb,
  publisher = {INFORMS},
  url = {https://doi.org/10.1287/moor.8.1.74}
}

@article{GiacominiKitagawa2021RobustBayesian,
  author  = {Giacomini, Raffaella and Kitagawa, Toru},
  title   = {Robust {B}ayesian Inference for Set-Identified Models},
  journal = {Econometrica},
  volume  = {89},
  number  = {4},
  pages   = {1519--1556},
  year    = {2021},
  month   = jul,
  doi     = {10.3982/ECTA16773},
  url = {https://doi.org/10.3982/ECTA16773}
}

@article{Ruger1978,
  author  = {R{\"u}ger, B.},
  title   = {Das maximale Signifikanzniveau des Tests: ``Lehne $H_{0}$ ab, wenn $k$ unter $n$ gegebenen Tests zur Ablehnung f{\"u}hren''},
  journal = {Metrika},
  year    = {1978},
  volume  = {25},
  pages   = {171--178},
  url     = {https://doi.org/10.1007/BF02204362}
}

@misc{Salako_Muhammad_2025_suppmat,
  author = {Salako, Kizito and Muhammad, Rabiu Tsoho},
  title = {Supplement to ``Fixed-Point Characterisations of Extremal Distributions under Partial Distributional Constraints''},
  year         = {2026},
  note = {See appendices of the main paper}
}

@article{CharnesCooper1962LinearFractionalFunctionals,
  author  = {Charnes, A. and Cooper, W. W.},
  title   = {Programming with Linear Fractional Functionals},
  journal = {Naval Research Logistics Quarterly},
  year    = {1962},
  volume  = {9},
  number  = {3--4},
  pages   = {181--186},
  doi     = {10.1002/nav.3800090303},
  url = { https://doi.org/10.1002/nav.3800090303}
}

@article{salako2025conservative,
  author        = {Kizito Salako and Rabiu Tsoho Muhammad},
  title         = {Conservative Software Reliability Assessments Using Collections of Bayesian Inference Problems},
  journal = {arXiv preprint},
  year          = {2025},
  eprint        = {2511.07038},
  archivePrefix = {arXiv},
  primaryClass  = {stat.AP},
  url={https://doi.org/10.48550/arXiv.2511.07038},
  doi           = {10.48550/arXiv.2511.07038}
}

@article{Schaible1976Dinkelbach,
  author  = {Siegfried Schaible},
  title   = {Fractional Programming. II, On Dinkelbach's Algorithm},
  journal = {Management Science},
  year    = {1976},
  month   = {April},
  volume  = {22},
  number  = {8},
  pages   = {868--873},
  publisher = {INFORMS},
  doi     = {10.1287/mnsc.22.8.868},
  issn    = {0025-1909},
  eissn   = {1526-5501},
  url     = {https://doi.org/10.1287/mnsc.22.8.868}
}

@article{WassermanLavineWolpert1993,
  author  = {Wasserman, Larry and Lavine, Michael and Wolpert, Robert L.},
  title   = {Linearization of Bayesian Robustness Problems},
  journal = {Journal of Statistical Planning and Inference},
  year    = {1993},
  volume  = {37},
  number  = {3},
  pages   = {307--316},
  doi     = {10.1016/0378-3758(93)90109-J},
  url     = {https://doi.org/10.1016/0378-3758(93)90109-J}
}

@article{Richter1957Parameterfreie,
  author  = {Richter, Hans},
  title   = {Parameterfreie Absch{\"a}tzung und Realisierung von Erwartungswerten},
  journal = {Bl. Deutsch. Ges. Versicherungsmath.},
  volume  = {3},
  year    = {1957},
  pages   = {147--162}
}

@article{Winkler1988ExtremePointsMomentSets,
  author  = {Winkler, Gerhard},
  title   = {Extreme Points of Moment Sets},
  journal = {Mathematics of Operations Research},
  year    = {1988},
  volume  = {13},
  number  = {4},
  pages   = {581--587},
  month   = nov,
  url = {https://doi.org/10.1287/moor.13.4.581}
}

@article{Cozman1999PosteriorBounds,
  author  = {Cozman, Fabio G.},
  title   = {Calculation of Posterior Bounds Given Convex Sets of Prior Probability Measures and Likelihood Functions},
  journal = {Journal of Computational and Graphical Statistics},
  year    = {1999},
  volume  = {8},
  number  = {4},
  pages   = {824--838}
}

@book{BoydVandenberghe2004ConvexOptimization,
  author    = {Boyd, Stephen and Vandenberghe, Lieven},
  title     = {Convex Optimization},
  publisher = {Cambridge University Press},
  address   = {Cambridge},
  year      = {2004},
  isbn      = {978-0-521-83378-3},
  url = {
https://doi.org/10.1017/CBO9780511804441}
}

@article{Choquet1954TheoryOfCapacities,
  author    = {Choquet, Gustave},
  title     = {Theory of capacities},
  journal   = {Annales de l'Institut Fourier},
  volume    = {5},
  year      = {1954},
  pages     = {131--295},
  publisher = {Institut Fourier},
  address   = {Grenoble},
  doi       = {10.5802/aif.53},
  mrnumber  = {18,295g},
  zbl       = {0064.35101},
  language  = {en},
  url       = {https://www.numdam.org/articles/10.5802/aif.53/}
}

@article{Birkhoff1946Tres,
  author  = {Birkhoff, Garrett},
  title   = {Tres observaciones sobre el {\'a}lgebra lineal},
  journal = {Univ. Nac. Tucum{\'a}n. Revista A},
  volume  = {5},
  pages   = {147--151},
  year    = {1946}
}

@incollection{vonNeumann1953Assignment,
  author    = {von Neumann, John},
  title     = {A Certain Zero-Sum Two-Person Game Equivalent to the Optimal Assignment Problem},
  editor    = {Kuhn, H. W. and Tucker, A. W.},
  booktitle = {Contributions to the Theory of Games, Volume II},
  series    = {Annals of Mathematics Studies},
  number    = {28},
  pages     = {5--12},
  publisher = {Princeton University Press},
  address   = {Princeton, NJ},
  year      = {1953},
  url={https://doi.org/10.1515/9781400881970-002}
}

@article{LaurenceWang2005Basket,
  author       = {Laurence, Peter and Wang, {Tai-Ho}},
  title        = {Sharp Upper and Lower Bounds for Basket Options},
  journal      = {Applied Mathematical Finance},
  year         = {2005},
  volume       = {12},
  number       = {3},
  pages        = {253--282},
  doi          = {10.1080/1350486042000325179},
  url = {https://doi.org/10.1080/1350486042000325179},
  note    = {See also the shorter exposition ``What's a Basket Worth?'' in Risk Magazine, February 2004.},
}

@article{burkard1996perspectives,
  author  = {Burkard, Rainer E. and Klinz, Bettina and Rudolf, R{\"u}diger},
  title   = {Perspectives of Monge Properties in Optimization},
  journal = {Discrete Applied Mathematics},
  volume  = {70},
  number  = {2},
  pages   = {95--161},
  year    = {1996},
  doi     = {10.1016/0166-218X(95)00103-X},
  url = {https://doi.org/10.1016/0166-218X(95)00103-X}
}

@article{EmbrechtsPuccettiRueschendorf2013,
  author  = {Embrechts, Paul and Puccetti, Giovanni and R{\"u}schendorf, Ludger},
  title   = {Model uncertainty and {VaR} aggregation},
  journal = {Journal of Banking \& Finance},
  volume  = {37},
  number  = {8},
  pages   = {2750--2764},
  year    = {2013},
  issn    = {0378-4266},
  doi     = {10.1016/j.jbankfin.2013.03.014},
  url = {https://doi.org/10.1016/j.jbankfin.2013.03.014}
}

@article{Dinkelbach1967,
  author  = {Dinkelbach, Wolfgang},
  title   = {On Nonlinear Fractional Programming},
  journal = {Management Science},
  year    = {1967},
  volume  = {13},
  number  = {7},
  pages   = {492--498},
  doi     = {10.1287/mnsc.13.7.492},
  url     = {https://pubsonline.informs.org/doi/10.1287/mnsc.13.7.492}
}

@article{kizito_QRE_2024,
author={Salako,Kizito and Zhao,Xingyu},
year={2024},
title={Demonstrating software reliability using possibly correlated tests: Insights from a conservative Bayesian approach},
journal={Quality and Reliability Engineering International},
volume={40},
number={3},
pages={1197-1220},
isbn={0748-8017},
language={English},
url={https://doi.org/10.1002/qre.3460}
}

@article{moreno1991robust,
  title={Robust Bayesian Analysis with $\epsilon$-Contaminations Partially Known},
  author={Moreno, E and Cano, J A},
  journal={Journal of the Royal Statistical Society: Series B (Methodological)},
  volume={53},
  number={1},
  pages={143--155},
  year={1991},
  publisher={Wiley Online Library},
  url = {https://doi.org/10.1111/j.2517-6161.1991.tb01814.x}
}

@ARTICLE {SalakoZhao_TSE_2023,
author = {K. Salako and X. Zhao},
journal = {IEEE Transactions on Software Engineering},
title = {The Unnecessity of Assuming Statistically Independent Tests in Bayesian Software Reliability Assessments},
year = {2023},
volume = {49},
number = {4},
issn = {1939-3520},
pages = {2829-2838},
doi = {10.1109/TSE.2022.3233802},
publisher = {IEEE Computer Society},
address = {Los Alamitos, CA, USA},
month = {apr},
url = {https://doi.org/10.1109/TSE.2022.3233802}
}

@inproceedings{zhao_assessing_2019,
	address = {Berlin, Germany},
	title = {Assessing the {safety} and {reliability} of {autonomous} {vehicles} from {road} {testing}},
	booktitle = {the 30th {Int.} {Symp.} on {Software} {Reliability} {Engineering}},
	publisher = {IEEE},
	author = {Zhao, X. and Robu, V. and Flynn, D. and Salako, K. and Strigini, L.},
	year = {2019},
	pages = {13--23},
    url = {https://doi.org/10.1109/ISSRE.2019.00012}
}

@article{Berger_1994_RobustnessInBidinesionalModels,
author={Berger,James and Moreno,Elías},
year={1994},
title={{Bayesian} robustness in bidimensional models: Prior independence},
journal={Journal of Statistical Planning and Inference},
volume={40},
number={2},
pages={161-176},
isbn={0378-3758},
url = {https://doi.org/10.1016/0378-3758(94)90118-X}
}

@article{berger1990_SensitivityToPrior,
  title={Robust {Bayesian} Analysis: Sensitivity to the Prior},
  author={Berger, James O},
  year={1990},
  journal={Journal of Statistical Planning and Inference},
  volume={25},
  number={},
  publisher={North-Holland},
  location={Amsterdam, Netherlands},
  pages={303--328},
  url = {https://doi.org/10.1016/0378-3758(90)90079-A}
}

@article{littlewood_reliability_2020,
	title = {On reliability assessment when a software-based system is replaced by a thought-to-be-better one},
	volume = {197},
	copyright = {All rights reserved},
	issn = {0951-8320},
	doi = {10.1016/j.ress.2019.106752},
	journal = {Reliability Engineering \& System Safety},
	author = {Littlewood, Bev and Salako, Kizito and Strigini, Lorenzo and Zhao, Xingyu},
	year = {2020},
	pages = {106752},
    url = {https://doi.org/10.1016/j.ress.2019.106752}
}

@inproceedings{salako_conservative_2021,
	address = {Taipei, Taiwan},
	series = {{DSN}'21},
	title = {Conservative {confidence} {bounds} in {safety}, from {generalised} {claims} of {improvement} \& {statistical} {evidence}},
	booktitle = {51st {Annual} {IEEE}/{IFIP} {Int.} {Conf.} on {Dependable} {Systems} and {Networks}},
	publisher = {IEEE/IFIP},
	author = {Salako, Kizito and Strigini, Lorenzo and Zhao, Xingyu},
	year = {2021},
	pages={451--462},
    url = {https://doi.org/10.1109/DSN48987.2021.00055}
}

@inproceedings{strigini_software_2013,
	title = {Software fault-freeness and reliability predictions},
	isbn = {978-3-642-40793-2},
	booktitle = {Computer {Safety}, {Reliability}, and {Security}},
	publisher = {Springer Berlin Heidelberg},
	author = {Strigini, Lorenzo and Povyakalo, Andrey},
	year = {2013},
	volume={8153},
	series={LNCS},
	pages = {106--117},
    url={https://doi.org/10.1007/978-3-642-40793-2_10}
}

@article{bishop_toward_2011,
	title = {Toward a formalism for conservative claims about the dependability of software-based systems},
	volume = {37},
	number = {5},
	journal = {IEEE Transactions on Software Engineering},
	author = {Bishop, P. and Bloomfield, R. and Littlewood, B. and Povyakalo, A. and Wright, D.},
	year = {2011},
	pages = {708--717},
    url = {https://doi.org/10.1109/TSE.2010.67}
}

\newpage

\section*{Supplementary Material / Appendices}\label{supplementary-material}




\begin{description}
\item[Title:]
Supplement to ``Fixed-Point Characterisations of Extremal Distributions under Partial Distributional Constraints''.
\item[Appendix A:]
Proof of Proposition~\ref{prop_cbi_transform};
\item[Appendix B:]
Proof of Proposition~\ref{prop_genprob1_v1};
\item[Appendix C:]
Proof of Proposition~\ref{prop_cont_marginal_approx};
\item[Appendix D:]
Proof of Proposition~\ref{cor_refined_strip_oscillations};
\item[Appendix E:]
Proof of Proposition~\ref{prop_gen_multiple_marginals_piecewise_refined};
\item[Appendix F:]
Proof of Corollary~\ref{cor_fgvanishingoscillations};
\item[Appendix G:] 
Proof of Theorem~\ref{thm:global_fp_tuple};
\item[Appendix H:]
Proof of convergence of Algorithm~\ref{alg:dinkelbach};
\item[Appendix I:] 
Proof of Theorem~\ref{thm_gensol};
\item[Appendix J:] 
Proof of Proposition~\ref{prop:semicont_refinement};
\item[Appendix K:] 
Proof of Theorem~\ref{thm_refinedgensol};
\item[Appendix L:] 
Proof of Theorem~\ref{thm:regulated_extension};
\item[Appendix M:]
Proof of Theorem~\ref{thm_CBIwithfails_sol};
\item[Appendix N:]
Proof of Theorem~\ref{thm_atleastkoutofnbounds};
\item[Appendix O:]
Proof of Theorem~\ref{thm:2d-gk-fixed-point-extension-corrected};
\item[Appendix P:]
Proof of Theorem~\ref{thm_VaRExample} and Corollary~\ref{Cor_VaRExample};
\item[Appendix Q:]
Proof of Theorem~\ref{thm:canonical-strike-recursive};
\item[Appendix R:]
Proof of Theorem~\ref{thm_Gen2stateMarkovChain};
\item[Appendix S:]
Proof of Proposition~\ref{prop_convergence_to_ordinary_BayesianInference}.
\end{description}

\appendix

\newpage
\section{Proof of Proposition\ \ref{prop_cbi_transform}}
\label{sec_app_prop1}
\begin{proof}
Fix $\mathbb P\in\mathcal D_{p,\boldsymbol m}$ and set $\phi_{\mathbb P}
:=
\frac{\mathbb E_{\mathbb P}[f({\mathbf X})]}
{\mathbb E_{\mathbb P}[g({\mathbf X})]}$. Since $f,g\geqslant 0$ and $\mathbb E_{\mathbb P}[g({\mathbf X})]>0$, we have
$\phi_{\mathbb P}\geqslant 0$. By Definition~\ref{def:Dinkelbachdifference}, the Dinkelbach difference is $h_{\phi_{\mathbb P},f,g}({\mathbf x})
=
f({\mathbf x})-\phi_{\mathbb P}g({\mathbf x})$. Note that
\[
\mathbb E_{\mathbb P}[h_{\phi_{\mathbb P},f,g}({\mathbf X})]
=
\mathbb E_{\mathbb P}[f({\mathbf X})]
-
\phi_{\mathbb P}\mathbb E_{\mathbb P}[g({\mathbf X})]
=
0.
\]
Define the constraint map $T:K\to\mathbb R^{n+q}$ by \[T(\mathbf{x})
:=
\bigl(
\mathbf{1}_{\{\mathbf{x}\in K_1\}},
\allowbreak \ldots,
\allowbreak \mathbf{1}_{\{\mathbf{x}\in K_n\}},
\allowbreak \psi_1(\mathbf{x}),
\allowbreak \ldots,
\allowbreak \psi_q(\mathbf{x})
\bigr)\,.
\]
The first $n$ coordinates satisfy $\sum_{i=1}^n {\textbf 1}_{\{\mathbf x\in K_i\}}=1$ for all $\,{\mathbf x}\in K$,
so the set $T(K)$ lies in an affine subspace of dimension at most $n-1+q$. 

Now, extend $T$ by adjoining the Dinkelbach difference. Define
$Y_{\phi_{\mathbb P}}:K\to\mathbb R^{n+q+1}$ by $Y_{\phi_{\mathbb P}}({\mathbf x})
:=
\left(
T({\mathbf x}),
h_{\phi_{\mathbb P},f,g}({\mathbf x})
\right)$. Adjoining $h_{\phi_{\mathbb P},f,g}$ increases affine dimension by at most one, so
$Y_{\phi_{\mathbb P}}(K)$ lies in an affine subspace of dimension at most $n+q$.
Moreover, $\mathbb E_{\mathbb P}[Y_{\phi_{\mathbb P}}({\mathbf X})]
=
(p_1,\ldots,p_n,m_1,\ldots,m_q,0)$. By Lemma~\ref{lem_finite_dimensional_barycentric_reduction}, applied to
$Y_{\phi_{\mathbb P}}$ and the probability measure $\mathbb P$, there exist
points ${\mathbf x}_1,\ldots,{\mathbf x}_N\in K$, with
$N\leqslant n+q+1$, and weights
$w_1,\ldots,w_N\geqslant 0$, $\sum_{\ell=1}^N w_\ell=1$, such that
$\sum_{\ell=1}^N w_\ell
Y_{\phi_{\mathbb P}}({\mathbf x}_\ell)
=
\mathbb E_{\mathbb P}[Y_{\phi_{\mathbb P}}({\mathbf X})]$. Equivalently, the atomic probability measure $\mathbb Q_0
:=
\sum_{\ell=1}^N w_\ell\delta_{{\mathbf x}_\ell}$ satisfies
\[
\begin{aligned}
\mathbb Q_0(K_i) &= p_i,
&& i=1,\ldots,n, \\
\mathbb E_{\mathbb Q_0}[\psi_j(\mathbf X)] &= m_j,
&& j=1,\ldots,q, \\
\mathbb E_{\mathbb Q_0}[h_{\phi_{\mathbb P},f,g}(\mathbf X)] &= 0.
\end{aligned}
\]
Thus, $\mathbb Q_0\in\mathcal D_{p,\boldsymbol m}$. 

If $N\leqslant n+q$, then
$\mathbb Q_0\in\mathcal D_{p,\boldsymbol m}^{(n+q)}$ and $\frac{\mathbb E_{\mathbb Q_0}[f({\mathbf X})]}
{\mathbb E_{\mathbb Q_0}[g({\mathbf X})]}
=
\phi_{\mathbb P}$. There is then nothing more to prove. It remains to consider the case $N=n+q+1$. After merging repeated support points and
discarding zero weights, assume $w_\ell>0$ for $\ell=1,\ldots,N$. For each support point
${\mathbf x}_\ell$, let $i(\ell)$ denote the unique index such that
${\mathbf x}_\ell\in K_{i(\ell)}$. Define the linear map $L:\mathbb R^N\to\mathbb R^{n+q}$ by
\[
L({\mathbf c})
=
\left(
\sum_{\ell:{\mathbf x}_\ell\in K_1}c_\ell,\ldots,
\sum_{\ell:{\mathbf x}_\ell\in K_n}c_\ell,
\sum_{\ell=1}^N c_\ell\psi_1({\mathbf x}_\ell),\ldots,
\sum_{\ell=1}^N c_\ell\psi_q({\mathbf x}_\ell)
\right),
\]
where ${\mathbf c}=(c_1,\ldots,c_N)\in\mathbb R^N$. Since $N=n+q+1$,
$\ker L$ is non-trivial. Choose $0\neq{\mathbf c}\in\ker L$. The equations
$L({\mathbf c})=0$ mean exactly that perturbing the weights in the direction
${\mathbf c}$ preserves all partition-mass constraints and all moment constraints. Also, since $K_1,\ldots,K_n$ form a partition of $K$, $\sum_{\ell=1}^N c_\ell
=
\sum_{i=1}^n
\sum_{\ell:{\mathbf x}_\ell\in K_i}c_\ell
=
0$. Hence, ${\mathbf c}$ has at least one positive and at least one negative component. Let $H({\mathbf c})
:=
\sum_{\ell=1}^N c_\ell
h_{\phi_{\mathbb P},f,g}({\mathbf x}_\ell)$. Replacing ${\mathbf c}$ by $-{\mathbf c}$, if necessary, we may assume
$H({\mathbf c})\leqslant 0$. For $t\geqslant 0$, define $w_\ell(t):=w_\ell+t c_\ell,\qquad \ell=1,\ldots,N$, and let $t_*:=
\min_{\ell:c_\ell<0}
\frac{w_\ell}{-c_\ell}$. Since ${\mathbf c}$ has a negative component and all $w_\ell>0$, we have
$0<t_*<\infty$. By definition, $w_\ell(t_*)\geqslant 0$ for every $\ell$, and at least one
of these weights is equal to zero. Define
\[
\mathbb Q_1
:=
\sum_{\ell=1}^N w_\ell(t_*)\delta_{{\mathbf x}_\ell}.
\]
Since ${\mathbf c}\in\ker L$, the partition masses and moment constraints are preserved.
Thus,
\[
\begin{aligned}
\mathbb Q_1(K_i) &= p_i,
&& i=1,\ldots,n, \\
\mathbb E_{\mathbb Q_1}[\psi_j(\mathbf X)] &= m_j,
&& j=1,\ldots,q.
\end{aligned}
\]
Therefore, $\mathbb Q_1\in\mathcal D_{p,\boldsymbol m}$. Furthermore, at least one weight
has vanished, so $\mathbb Q_1$ has support size at most $N-1=n+q$. 

Finally, by altering the $w_\ell$ weights along a kernel vector $\mathbf c$ direction, $\mathbb E_{\mathbb Q_1}[h_{\phi_{\mathbb P},f,g}({\mathbf X})]
=
\mathbb E_{\mathbb Q_0}[h_{\phi_{\mathbb P},f,g}({\mathbf X})]
+
t_*H({\mathbf c})
\leqslant
0$. Hence, $\mathbb E_{\mathbb Q_1}[f({\mathbf X})]
-
\phi_{\mathbb P}
\mathbb E_{\mathbb Q_1}[g({\mathbf X})]
\leqslant
0$. From $\mathbb E_{\mathbb Q_1}[g({\mathbf X})]>0$, it follows that
\[
\frac{\mathbb E_{\mathbb Q_1}[f({\mathbf X})]}
{\mathbb E_{\mathbb Q_1}[g({\mathbf X})]}
\leqslant
\phi_{\mathbb P}.
\]
Thus, for every $\mathbb P\in\mathcal D_{p,\boldsymbol m}$, we have constructed
$\mathbb Q_1\in\mathcal D_{p,\boldsymbol m}^{(n+q)}$ whose fractional objective value is
no larger than that of $\mathbb P$.

Taking the infimum over $\mathbb P\in\mathcal D_{p,\boldsymbol m}$ gives
\[
\inf_{\mathbb Q\in\mathcal D_{p,\boldsymbol m}^{(n+q)}}
\frac{\mathbb E_{\mathbb Q}[f({\mathbf X})]}
{\mathbb E_{\mathbb Q}[g({\mathbf X})]}
\leqslant
\inf_{\mathbb P\in\mathcal D_{p,\boldsymbol m}}
\frac{\mathbb E_{\mathbb P}[f({\mathbf X})]}
{\mathbb E_{\mathbb P}[g({\mathbf X})]}.
\]
The reverse inequality is immediate because
$\mathcal D_{p,\boldsymbol m}^{(n+q)}\subseteq\mathcal D_{p,\boldsymbol m}$.
Hence, the two infima are equal. In particular, when $q=0$, every $\mathbb Q\in\mathcal D_p^{(n)}$ has exactly one positive-weight
support point in each $K_i$, since $\mathbb Q(K_i)=p_i>0$ for every $i$.
Thus, $\mathbb Q=\sum_{i=1}^n p_i\delta_{{\mathbf x}_i}$ with
${\mathbf x}_i\in K_i$, and the equality in \eqref{eqn_weak_opt_equiv_momentgen_disc} gives the discrete formulation in \eqref{eqn_weak_opt_equiv_disc}.

For the special case of continuous $f$, $g$, let $\mathcal K:=K_1\times\cdots\times K_n$, let 
$\underline{\mathcal K}:=\overline K_1\times\cdots\times \overline K_n$,
and let $\Phi({\mathbf x}_1,\dots,{\mathbf x}_n)
:=
(\,\sum_{j=1}^n p_j f({\mathbf x}_j)\,)/(\,\sum_{j=1}^n p_j g({\mathbf x}_j)\,)$. Since $\mathcal K\subseteq \underline{\mathcal K}$, we have
\begin{equation}
\inf_{({\mathbf x}_1,\dots,{\mathbf x}_n)\in \mathcal K}\Phi({\mathbf x}_1,\dots,{\mathbf x}_n)
\geqslant
\min_{({\mathbf x}_1,\dots,{\mathbf x}_n)\in \underline{\mathcal K}}\Phi({\mathbf x}_1,\dots,{\mathbf x}_n).
\label{eqn_appA_ineq1}
\end{equation}
For the reverse inequality, note that each $\overline K_j$ is closed and bounded in $\mathbb R^d$, hence compact. Therefore, $\underline{\mathcal K}$ is compact. By continuity of $\Phi$ on $\underline{\mathcal K}$, there exists ${\mathbf x}^*=({\mathbf x}_1^*,\dots,{\mathbf x}_n^*)\in \underline{\mathcal K}$ such that $\Phi({\mathbf x}^*)=\min_{{\mathbf x}\in \underline{\mathcal K}}\Phi({\mathbf x})=\inf_{{\mathbf x}\in \underline{\mathcal K}}\Phi({\mathbf x})$. Now, $\mathcal K$ is dense in $\underline{\mathcal K}$, since each $K_j$ is dense in $\overline K_j$. Hence, there exists a sequence ${\mathbf x}^{(m)}=({\mathbf x}_1^{(m)},\dots,{\mathbf x}_n^{(m)})\in \mathcal K$ such that ${\mathbf x}^{(m)}\to {\mathbf x}^*$. By continuity of $\Phi$,
\[
\Phi({\mathbf x}^{(m)})\to \Phi({\mathbf x}^*).
\]
Since $\inf_{{\mathbf x}\in \mathcal K}\Phi(\mathbf x)\leqslant \Phi({\mathbf x}^{(m)})$ for every $m$, passing to the limit yields 
\begin{equation}
\inf_{{\mathbf x}\in \mathcal K}\Phi({\mathbf x})\leqslant \Phi({\mathbf x}^*)=\min_{{\mathbf x}\in\underline{\mathcal K}}\Phi({\mathbf x}).
\label{eqn_appA_ineq2}
\end{equation}
Combining inequalities \eqref{eqn_appA_ineq1} and \eqref{eqn_appA_ineq2}, we conclude $\inf_{{\mathbf x}\in \mathcal K}\Phi({\mathbf x})=\min_{{\mathbf x}\in \underline{\mathcal K}}\Phi({\mathbf x})$.
\end{proof}
\begin{lemma}[Convex form of Richter's Theorem (\cite{Richter1957Parameterfreie}, Satz~4)]
\label{lem_finite_dimensional_barycentric_reduction}
Let $(S,\mathcal A,\mu)$ be a probability space and let
$Y:S\to\mathbb R^r$ be an integrable measurable map. Suppose that
$Y(S)$ is contained in an affine subspace of $\mathbb R^r$ of dimension at most
$d$. Then, there exist points $s_1,\ldots,s_N\in S$ and weights
$w_1,\ldots,w_N\geqslant0$, with $N\leqslant d+1$ and $\sum_{\ell=1}^N w_\ell=1$, 
such that $\int_S Y(s)\,d\mu(s)
=
\sum_{\ell=1}^N w_\ell Y(s_\ell)$.
\end{lemma}

\begin{proof}
Let ${\mathbf y}:=\int_S Y(s)\,d\mu(s)$, $A:=Y(S)$, $C:=\operatorname{conv}(A)$, $\overline C:=\operatorname{cl}(C)$. Here, $\operatorname{conv}(A)$ denotes the set of all finite convex combinations
of points of $A$.

We first show that ${\mathbf y}\in\overline C$. Suppose, to the contrary, that
${\mathbf y}\notin\overline C$. Since $\overline C$ is a closed convex subset
of the finite-dimensional Euclidean space $\mathbb R^r$, there exist
${\mathbf u}\in\mathbb R^r$ and $\alpha\in\mathbb R$ such that ${\mathbf u}\cdot{\mathbf y}<\alpha
\leqslant
{\mathbf u}\cdot{\mathbf z}$ for every ${\mathbf z}\in\overline C$.
Since $Y(s)\in A\subseteq\overline C$ for every $s\in S$, this gives ${\mathbf u}\cdot Y(s)\geqslant\alpha$
for every $s\in S$. Integrating yields
\[
{\mathbf u}\cdot{\mathbf y}
=
{\mathbf u}\cdot\int_S Y(s)\,d\mu(s)
=
\int_S {\mathbf u}\cdot Y(s)\,d\mu(s)
\geqslant \alpha,
\]
which contradicts ${\mathbf u}\cdot{\mathbf y}<\alpha$. Hence, ${\mathbf y}\in\overline C$.

We next show that, in fact, ${\mathbf y}\in C$. This is proved by induction on $d_0:=\dim L$, where $L$ is the smallest affine space containing $\overline{C}$; i.e. $L:=\operatorname{aff}(\overline C)$. By the definitions of $A$, $C$, and the fact that affine subspaces of $\mathbb R^r$ are closed, we have $\operatorname{aff}(\overline C)=\operatorname{aff}(A)$. Thus, if $Y(S)$ is contained in an affine subspace of dimension at most $d$, then $d_0\leqslant d$. The induction hypothesis is that, for any integrable measurable map whose
associated closed convex hull has affine dimension strictly smaller than
$d_0$, the integral of that map belongs to the convex hull of its image.

If $d_0=0$, then $\overline C$ consists of a single point. Since $C$ is nonempty and $C\subseteq\overline C$, it follows that $C=\overline C$, and therefore ${\mathbf y}\in C$. Assume now that $d_0\geqslant1$ and that the assertion has already been proved in all smaller affine dimensions. Since ${\mathbf y}\in\overline C$, either $ {\mathbf y}\in\operatorname{ri}(\overline C)$ or $ {\mathbf y}\in\operatorname{rbd}(\overline C)$,  where both relative interior $\operatorname{ri}(\overline C)$ and relative boundary $\operatorname{rbd}(\overline C)$ are taken in the affine space $L$.

Suppose first that ${\mathbf y}\in\operatorname{ri}(\overline C)$. By the definition of relative interior, there exists $\rho>0$ such that the
relative open ball $B_L({\mathbf y},\rho)
:=
\{ {\mathbf z}\in L:\|{\mathbf z}-{\mathbf y}\|<\rho\}$ is contained in $\overline C$. Construct ${\mathbf z}_0,\ldots,{\mathbf z}_{d_0}$ that lie in
$B_L({\mathbf y},\rho)\subseteq\overline C$ as follows. Let $V:=L-{\mathbf y}$ be the $d_0$-dimensional linear space parallel to $L$,
and choose a basis ${\mathbf v}_1,\ldots,{\mathbf v}_{d_0}$ of $V$. Define
\[
{\mathbf z}_0
:=
{\mathbf y}
-
\varepsilon\sum_{k=1}^{d_0}{\mathbf v}_k,
\qquad
{\mathbf z}_j
:=
{\mathbf y}
+
\varepsilon{\mathbf v}_j,
\qquad j=1,\ldots,d_0 .
\]
For sufficiently small $\varepsilon$, the points
${\mathbf z}_0,\ldots,{\mathbf z}_{d_0}$ lie in
$B_L({\mathbf y},\rho)\subseteq\overline C$. Moreover, ${\mathbf y}
=
\frac{1}{d_0+1}
\sum_{j=0}^{d_0}{\mathbf z}_j$ and the points ${\mathbf z}_0,\ldots,{\mathbf z}_{d_0}$ are affinely independent.
Indeed, affine independence is equivalent to linear independence of ${\mathbf z}_1-{\mathbf z}_0,\ldots,
{\mathbf z}_{d_0}-{\mathbf z}_0$. If $\sum_{j=1}^{d_0}a_j({\mathbf z}_j-{\mathbf z}_0)=0$ then, since ${\mathbf z}_j-{\mathbf z}_0
=
\varepsilon
\left(
{\mathbf v}_j+\sum_{k=1}^{d_0}{\mathbf v}_k
\right)$, we obtain $\sum_{j=1}^{d_0}a_j{\mathbf v}_j
+
\left(\sum_{j=1}^{d_0}a_j\right)
\sum_{k=1}^{d_0}{\mathbf v}_k
=0$. Writing $a_{\mathrm{ sum}}:=\sum_{j=1}^{d_0}a_j$, the coefficient of each
${\mathbf v}_k$ is $a_k+a_{\mathrm{ sum}}$. Since
${\mathbf v}_1,\ldots,{\mathbf v}_{d_0}$ is a basis, $a_k+a_{\mathrm{ sum}}=0$ for every
$k$. Summing over $k$ gives $a_{\mathrm{ sum}}=-d_0 a_{\mathrm{ sum}}$, hence $a_{\mathrm{ sum}}=0$, and therefore
$a_k=0$ for every $k$. Thus, the difference vectors are linearly independent.

So, since each ${\mathbf z}_j\in\overline C$, for every
$\delta>0$ there exist points ${\mathbf c}_j\in C$ such that $\|{\mathbf c}_j-{\mathbf z}_j\|<\delta$ for $(j=0,\ldots,d_0)$.
Choose $\delta$ sufficiently small that
${\mathbf c}_0,\ldots,{\mathbf c}_{d_0}$ remain affinely independent. Since
these affinely independent points lie in the $d_0$-dimensional affine space $L$, their affine hull
is again $L$. This means the barycentric coordinates of
${\mathbf y}$ with respect to ${\mathbf c}_0,\ldots,{\mathbf c}_{d_0}$ are the unique scalars
$\lambda_0,\ldots,\lambda_{d_0}$ satisfying
\[
{\mathbf y}
=
\sum_{j=0}^{d_0}\lambda_j{\mathbf c}_j,
\qquad
\sum_{j=0}^{d_0}\lambda_j=1.
\]
These coordinates solve a nonsingular affine linear system, and therefore depend continuously on the vertices ${\mathbf c}_0,\ldots,{\mathbf c}_{d_0}$ as long as affine independence is
preserved. For the unperturbed vertices
${\mathbf z}_0,\ldots,{\mathbf z}_{d_0}$, the barycentric coordinates of
${\mathbf y}$ are all equal to $\frac{1}{d_0+1}>0$. Therefore, by choosing the approximating points ${\mathbf c}_j\in C$
sufficiently close to the corresponding ${\mathbf z}_j$, the barycentric
coordinates of ${\mathbf y}$ with respect to
${\mathbf c}_0,\ldots,{\mathbf c}_{d_0}$ remain strictly positive. Hence,
${\mathbf y}
=
\sum_{j=0}^{d_0}\lambda_j{\mathbf c}_j$, $\sum_{j=0}^{d_0}\lambda_j=1$ and $\lambda_j>0$ for $j=0,...,d_0$.
Thus, $
{\mathbf y}\in
\operatorname{conv}({\mathbf c}_0,\ldots,{\mathbf c}_{d_0})$. And, since each ${\mathbf c}_j$ belongs to $C$ and $C$ is convex, $\operatorname{conv}({\mathbf c}_0,\ldots,{\mathbf c}_{d_0})\subseteq C$. Therefore, ${\mathbf y}\in C$.

It remains to consider the case $ {\mathbf y}\in\operatorname{rbd}(\overline C)$. By the supporting-hyperplane theorem \cite[Section~2.5.2]{BoydVandenberghe2004ConvexOptimization} in the affine space $L$, there exists ${\mathbf u}\in\mathbb R^r$ whose restriction to the direction space of $L$ is not identically zero, such that ${\mathbf u}\cdot{\mathbf z} \geqslant {\mathbf u}\cdot{\mathbf y}$ for every ${\mathbf z}\in\overline C$.  Consequently, ${\mathbf u}\cdot Y(s) \geqslant {\mathbf u}\cdot{\mathbf y}$ for every $s\in S$. Define $Z(s):={\mathbf u}\cdot Y(s)-{\mathbf u}\cdot{\mathbf y}$. Then, $Z(s)\geqslant0$ for every $s\in S$, and \[ \int_S Z(s)\,d\mu(s) = {\mathbf u}\cdot\int_S Y(s)\,d\mu(s) - {\mathbf u}\cdot{\mathbf y} = 0. \] Therefore, $Z=0$ $\mu$-almost everywhere. Hence, there exists $S_0\in\mathcal A$ defined as $S_0:=\{s\in S: Z(s)=0\}$, satisfying $\mu(S_0)=1$ and ${\mathbf u}\cdot Y(s) = {\mathbf u}\cdot{\mathbf y}$ for every $s\in S_0$. Thus, $Y(S_0)$ is contained in the affine hyperplane $H:= \left\{ {\mathbf z}\in L: {\mathbf u}\cdot{\mathbf z} = {\mathbf u}\cdot{\mathbf y} \right\}$. Since the restriction of ${\mathbf u}$ to the direction space of $L$ is not identically zero, $H$ is a proper affine hyperplane of $L$. Hence, $\dim\operatorname{aff}(Y(S_0))<d_0$. Now, view $(S_0,\mathcal A|_{S_0},\mu|_{S_0})$ as a probability space. This is valid because $\mu(S_0)=1$. Since $Y$ is integrable and $S_0$ has full measure, $ {\mathbf y} = \int_S Y(s)\,d\mu(s) = \int_{S_0}Y(s)\,d\mu(s)$. The restricted map $Y|_{S_0}$ has image contained in an affine subspace of dimension smaller than $d_0$. Thus, by the induction hypothesis applied to $Y|_{S_0}$, \[ {\mathbf y} = \int_{S_0}Y(s)\,d\mu(s) \in \operatorname{conv}(Y(S_0)). \] Because $Y(S_0)\subseteq Y(S)=A$, we have $ \mathbf{y}\in\operatorname{conv}(Y(S_0))\subseteq \operatorname{conv}(A)=C$. Thus, ${\mathbf y}\in C$ in the relative-boundary case as well. Therefore, ${\mathbf y}\in C$ in all cases.

Since ${\mathbf y}\in C=\operatorname{conv}(A)$, there exist
$a_1,\ldots,a_M\in A$ and coefficients
$\alpha_1,\ldots,\alpha_M\geqslant0$, with
\[
\sum_{m=1}^M\alpha_m=1,
\qquad
{\mathbf y}=\sum_{m=1}^M\alpha_m a_m.
\]
Discarding zero coefficients if necessary, assume $\alpha_m>0$ for all $m$. If $M\leqslant d_0+1$, then the desired representation is obtained by choosing
$s_m\in S$ such that $Y(s_m)=a_m$. Suppose instead that $M>d_0+1$. Since
$a_1,\ldots,a_M$ lie in the affine space $L$ of dimension $d_0$, they are
affinely dependent. Hence, there exist real numbers
$\beta_1,\ldots,\beta_M$, not all zero, such that
\[
\sum_{m=1}^M\beta_m a_m=0,
\qquad
\sum_{m=1}^M\beta_m=0.
\]
Since $\sum_m\beta_m=0$ and the $\beta_m$ are not all zero, at least one
$\beta_m$ is positive and at least one is negative. Define $t_*:=\min_{\beta_m>0}\frac{\alpha_m}{\beta_m}$ and 
set $\alpha_m':=\alpha_m-t_*\beta_m$ for 
$m=1,\ldots,M$. Then, $\alpha_m'\geqslant0$ for every $m$, at least one $\alpha_m'$ is zero, and
\[
\sum_{m=1}^M\alpha_m'
=
\sum_{m=1}^M\alpha_m
-
t_*\sum_{m=1}^M\beta_m
=
1.
\]
Moreover,
\[
\sum_{m=1}^M\alpha_m'a_m
=
\sum_{m=1}^M\alpha_m a_m
-
t_*\sum_{m=1}^M\beta_m a_m
=
{\mathbf y}.
\]
Thus, one point with positive coefficient has been removed and we still have a convex combination that gives $\mathbf y$. Repeating this finite elimination procedure (if necessary) gives a
representation
\[
{\mathbf y}
=
\sum_{\ell=1}^N w_\ell a_\ell,
\qquad
N\leqslant d_0+1\leqslant d+1,
\qquad
w_\ell\geqslant0,
\qquad
\sum_{\ell=1}^N w_\ell=1.
\]
Finally, since $a_\ell\in A=Y(S)$, choose $s_\ell\in S$ with
$Y(s_\ell)=a_\ell$. Then, $\int_S Y(s)\,d\mu(s)
=
{\mathbf y}
=
\sum_{\ell=1}^N w_\ell Y(s_\ell)$,
with $N\leqslant d+1$. This proves the lemma.
\end{proof}

\newpage
\section{Proof of Proposition~\ref{prop_genprob1_v1}}
\begin{proof}
The sets $\{A_I:I\in\mathcal I\}$ form a Borel partition of $K$ generated by
$\tilde K_1,\ldots,\tilde K_n$.
For any $\mathbb P\in\widetilde{\mathcal D}$, define $p_I:=\mathbb P(A_I)$.
Then $p_I\geqslant 0$, $\sum_{I\in\mathcal I} p_I=1$, and, for each $i=1,\ldots,n$,
\[
\sum_{I\in\mathcal I:\, i\in I} p_I
=
\sum_{I\in\mathcal I:\, i\in I} \mathbb P(A_I)
=
\mathbb P(\tilde K_i)
=
\tilde p_i.
\]
Conversely, any nonnegative collection $\{p_I\}_{I\in\mathcal I}$ satisfying these linear
constraints determines a subclass $\widetilde{\mathcal D}(p)\subseteq\widetilde{\mathcal D}$.
Hence,
\[
\inf_{\mathbb P\in\widetilde{\mathcal D}}
\frac{\mathbb E[f({\mathbf X})]}{\mathbb E[g({\mathbf X})]}
=
\inf_{\substack{p_I\geqslant 0\text{ for all }I\in\mathcal I,\\
\sum_{I\in\mathcal I} p_I = 1\\
\sum_{I\in\mathcal I:\, i\in I} p_I = \tilde p_i,\ i=1,\ldots,n}}
\!\!\!\!
\inf_{\mathbb P\in\widetilde{\mathcal D}(p)}
\frac{\mathbb E[f({\mathbf X})]}{\mathbb E[g({\mathbf X})]},
\]
where the outer infimum is over all feasible $\{p_I\}_{I\in\mathcal I}$.
The $A_I$ (such that $p_I>0$) form a partition of some non-empty subset of $K$; indeed, there must be some $A_I$ with non-zero $p_I$ in each $\tilde{K}_i$, since $\tilde{p}_i>0$ for all $i$. Thus, Proposition~\ref{prop_cbi_transform} applies to this partition, yielding the inner optimisation in \eqref{eq:overlap_refinement}. The inner and outer optimisations can be switched: the outer optimisation becomes an inner linear-fractional program subject to, at most, $n+1$ linearly independent constraints. The extremal prior can assign probability masses to, at most, $n+1$ atoms in its domain. Hence, the form of the extremal prior must be as stated. The final claim is immediate with disjoint $\tilde{K}_i$.
\end{proof}

\newpage\section{Proof of Proposition~\ref{prop_cont_marginal_approx}}
\begin{proof}
For each $m$, the sets $K_i^{(m)}$ with $p_i^{(m)}>0$ partition $K$ up to cells of zero prescribed mass, and their $p_i^{(m)}$ sum to $1$.
Hence, Proposition~\ref{prop_cbi_transform} applies to this partition and the restriction of all priors in $ {{\mathcal D}^{(m)}}$ to this partition.

We first show that $\phi_m^*\leqslant \phi^*$ for every $m$.
Indeed, if $\mathbb P\in\mathcal D$ then, for every $i=1,\ldots,n_m$,
\[
\mathbb P\bigl(K_i^{(m)}\bigr)
=
\mathbb P\bigl(B_i^{(m)}\times V\bigr)
=
\int_{B_i^{(m)}}\rho({\mathbf u})\,d{\mathbf u}
=
p_i^{(m)},
\]
so $\mathbb P\in {{\mathcal D}^{(m)}}$.
Thus $\mathcal D\subseteq  {{\mathcal D}^{(m)}}$, hence
\begin{equation}
\phi_m^*\leqslant \phi^*.
\label{eqn_phimstarlephistar}
\end{equation}

For the reverse inequality, fix $m\in\mathbb N$ and $\mathbb Q\in {{\mathcal D}^{(m)}}$.
For each $i$ with $p_i^{(m)}>0$, define a probability measure $\nu_i^{(m)}$ on $V$ by
\[
\nu_i^{(m)}(C):=\frac{\mathbb Q(B_i^{(m)}\times C)}{p_i^{(m)}},
\qquad C\in\mathcal B(V).
\]
If $p_i^{(m)}=0$, choose $\nu_i^{(m)}$ arbitrarily in $\mathcal P(V)$.
Now define a probability measure $S_m(\mathbb Q)$ on $K$ by
\[
S_m(\mathbb Q)(d{\mathbf u},d{\mathbf v})
:=
\sum_{i=1}^{n_m}\mathbf 1_{B_i^{(m)}}({\mathbf u})\,
\rho({\mathbf u})\,d{\mathbf u}\,\nu_i^{(m)}(d{\mathbf v}).
\]
Then, for every Borel-measurable $A\subseteq U$,
\[
S_m(\mathbb Q)(A\times V)
=
\sum_{i=1}^{n_m}\int_{A\cap B_i^{(m)}} \rho({\mathbf u})\,d{\mathbf u}
=
\int_A \rho({\mathbf u})\,d{\mathbf u},
\]
so $S_m(\mathbb Q)\in\mathcal D$.

Next, compare the $f$- and $g$-integrals under $\mathbb Q$ and $S_m(\mathbb Q)$.
For each $i$ with $p_i^{(m)}>0$, choose an arbitrary point ${\mathbf u}_i^{(m)}\in B_i^{(m)}$.
Since both $\mathbb Q$ and $S_m(\mathbb Q)$ have the same $V$-marginal on the strip $B_i^{(m)}\times V$, namely $p_i^{(m)}\nu_i^{(m)}$, we have
\[
\int_{B_i^{(m)}\times V} f({\mathbf u}_i^{(m)},{\mathbf v})\,\mathbb Q(d{\mathbf u},d{\mathbf v})
=
\int_{B_i^{(m)}\times V} f({\mathbf u}_i^{(m)},{\mathbf v})\,S_m(\mathbb Q)(d{\mathbf u},d{\mathbf v}).
\]
Therefore,
\begin{align*}
&\Biggl|
\int_{B_i^{(m)}\times V} f({\mathbf u},{\mathbf v})\,S_m(\mathbb Q)(d{\mathbf u},d{\mathbf v})
-
\int_{B_i^{(m)}\times V} f({\mathbf u},{\mathbf v})\,\mathbb Q(d{\mathbf u},d{\mathbf v})
\Biggr| \\
=&\;\Biggl| \int_{B_i^{(m)}\times V}
\bigl(f({\mathbf u},{\mathbf v})-f({\mathbf u}_i^{(m)},{\mathbf v})\bigr)\,
S_m(\mathbb Q)(d{\mathbf u},d{\mathbf v})
-
\int_{B_i^{(m)}\times V}
\bigl(f({\mathbf u},{\mathbf v})-f({\mathbf u}_i^{(m)},{\mathbf v})\bigr)\,
\mathbb Q(d{\mathbf u},d{\mathbf v})\Biggr| \\
\leqslant
&\;\int_{B_i^{(m)}\times V}
\bigl|f({\mathbf u},{\mathbf v})-f({\mathbf u}_i^{(m)},{\mathbf v})\bigr|\,
S_m(\mathbb Q)(d{\mathbf u},d{\mathbf v})
+
\int_{B_i^{(m)}\times V}
\bigl|f({\mathbf u},{\mathbf v})-f({\mathbf u}_i^{(m)},{\mathbf v})\bigr|\,
\mathbb Q(d{\mathbf u},d{\mathbf v}) \\
\leqslant &\; 2\,p_i^{(m)}\,\omega_f(m).
\end{align*}
Summing over all $i$ with $p_i^{(m)}>0$ yields $\Bigl|\mathbb E_{S_m(\mathbb Q)}[f({\mathbf X})]-\mathbb {\mathbb E}_{\mathbb Q}[f({\mathbf X})]\Bigr|
\leqslant 2\,\omega_f(m)$. The same argument gives $
\Bigl|\mathbb E_{S_m(\mathbb Q)}[g({\mathbf X})]-\mathbb {\mathbb E}_{\mathbb Q}[g({\mathbf X})]\Bigr|
\leqslant 2\,\omega_g(m)$. For every probability measure $\mathbb P$ on $K$, $\mathbb E_{\mathbb P}[g({\mathbf X})]\geqslant c$ and $\mathbb E_{\mathbb P}[f({\mathbf X})]\leqslant M_f$. Hence, for every $\mathbb Q\in {{\mathcal D}^{(m)}}$,
\[
\left|
\frac{\mathbb E_{S_m(\mathbb Q)}[f({\mathbf X})]}{\mathbb E_{S_m(\mathbb Q)}[g({\mathbf X})]}
-
\frac{\mathbb {\mathbb E}_{\mathbb Q}[f({\mathbf X})]}{\mathbb {\mathbb E}_{\mathbb Q}[g({\mathbf X})]}
\right|
\leqslant
\frac{2\,\omega_f(m)}{c}
+
\frac{2M_f\,\omega_g(m)}{c^2}
=
\eta_m.
\]
Since $S_m(\mathbb Q)\in\mathcal D$, we have
\[
\phi^*
\leqslant
\frac{\mathbb E_{S_m(\mathbb Q)}[f({\mathbf X})]}{\mathbb E_{S_m(\mathbb Q)}[g({\mathbf X})]}
\leqslant
\frac{\mathbb {\mathbb E}_{\mathbb Q}[f({\mathbf X})]}{\mathbb {\mathbb E}_{\mathbb Q}[g({\mathbf X})]}
+\eta_m.
\]
Taking the infimum over $\mathbb Q\in {{\mathcal D}^{(m)}}$ yields
\begin{equation}
\phi^*\leqslant \phi_m^*+\eta_m.
\label{eqn_phistarlephimstar}
\end{equation}
Together, \eqref{eqn_phimstarlephistar} and \eqref{eqn_phistarlephimstar} imply
\[
\phi^*-\eta_m\leqslant \phi_m^*\leqslant \phi^*.
\]
In particular,
\[
\phi^*-\eta_m\leqslant \phi_m^*\leqslant \phi^*+\eta_m.
\]
Since $\omega_f(m)\to 0$ and $\omega_g(m)\to 0$, we have $\eta_m\to 0$, and therefore $\phi_m^*\to \phi^*$.
\end{proof}
\newpage

\section{Proof of Proposition~\ref{cor_refined_strip_oscillations}}
\begin{proof}
We prove item~(1) first. We first justify the decomposition 
\begin{equation}
\inf_{\substack{\mathbb P\in {\mathcal D_{F,q}}^{(m)}}}
\frac{\mathbb E_{\mathbb P}[f(\textnormal{\textbf{X}})]}{\mathbb E_{\mathbb P}[g(\textnormal{\textbf{X}})]}
=
\inf_{\boldsymbol{\alpha} \in {\mathcal W_{F,q,m}}}\ \inf_{\substack{\mathbb P\in {\mathcal D_{F,q}}^{(m)}\\ \mathbb P(R_\tau^{(m)})=\alpha_\tau,\ \tau=1,\ldots,L_m}}
\frac{\mathbb E_{\mathbb P}[f(\textnormal{\textbf{X}})]}{\mathbb E_{\mathbb P}[g(\textnormal{\textbf{X}})]}.
\label{eqn_appD_item1decomposition}
\end{equation}
Fix $m \in \mathbb N$ and define
\[
I_m^+ := \{\, i \in \{1,\ldots,n_m\} : p_i^{(m)} > 0 \,\},
\qquad
J^+ := \{\, \ell \in \{1,\ldots,L\} : q_\ell > 0 \,\}.
\]
Since $\sum_{i=1}^{n_m} p_i^{(m)} = 1$, we also have $\sum_{i\in I_m^+} p_i^{(m)} = 1$. Since $\{F_\ell\}_{\ell=1}^L$ is a partition of $K$ and $\sum_{\ell=1}^L q_\ell=1$, we
also have $\sum_{\ell\in J^+} q_\ell = 1$. Because the strips $\{K_i^{(m)}\}_{i=1}^{n_m}$ form a partition of $K$, it follows that
\[
{\mathcal D_{F,q}}^{(m)}
=
\Bigl\{
\mathbb P \in \mathcal P(K) :
\mathbb P\bigl(K_i^{(m)}\bigr)=p_i^{(m)},\ i\in I_m^+,\ 
\mathbb P(F_\ell)=q_\ell,\ \ell\in J^+
\Bigr\}.
\]
Indeed, if $\mathbb P$ satisfies the right-hand side, then
\[
1
=
\mathbb P(K)
=
\sum_{i=1}^{n_m} \mathbb P\bigl(K_i^{(m)}\bigr)
=
\sum_{i\in I_m^+} p_i^{(m)}
+
\sum_{i\notin I_m^+} \mathbb P\bigl(K_i^{(m)}\bigr)
=
1 + \sum_{i\notin I_m^+} \mathbb P\bigl(K_i^{(m)}\bigr),
\]
so $\mathbb P\bigl(K_i^{(m)}\bigr)=0=p_i^{(m)}$ for all $i\notin I_m^+$.

Likewise, because $\{F_\ell\}_{\ell=1}^L$ is a partition of $K$,
\[
1
=
\mathbb P(K)
=
\sum_{\ell=1}^{L} \mathbb P(F_\ell)
=
\sum_{\ell\in J^+} q_\ell
+
\sum_{\ell\notin J^+} \mathbb P(F_\ell)
=
1 + \sum_{\ell\notin J^+} \mathbb P(F_\ell),
\]
so $\mathbb P(F_\ell)=0=q_\ell$ for all $\ell\notin J^+$.
Hence, the omitted zero-mass strip constraints and zero-probability $F$-constraints hold.

Now let $\mathbb P \in {\mathcal D_{F,q}}^{(m)}$. Since the family $\{R_\tau^{(m)}\}_{\tau=1}^{L_m}$ is a
finite Borel partition of $K$ refining both $\{K_i^{(m)}\}_{i=1}^{n_m}$ and
$\{F_\ell\}_{\ell=1}^L$, define $\alpha_\tau := \mathbb P\bigl(R_\tau^{(m)}\bigr)$ for $\tau=1,\ldots,L_m$. Because the refined atoms partition each strip, $\sum_{\tau:\,R_\tau^{(m)} \subseteq K_i^{(m)}} \alpha_\tau
=
\mathbb P\bigl(K_i^{(m)}\bigr)
=
p_i^{(m)}$ for $i=1,\ldots,n_m$. Likewise, because the refined atoms partition each $F_\ell$, $\sum_{\tau:\,R_\tau^{(m)} \subseteq F_\ell} \alpha_\tau
=
\mathbb P(F_\ell)
=
q_\ell\,$ for 
$\,\ell=1,\ldots,L$. Hence, $\boldsymbol{\alpha} \in {\mathcal W_{F,q,m}}$.

Conversely, let $\boldsymbol{\alpha} \in {\mathcal W_{F,q,m}}$. For each $\tau$ with $\alpha_\tau>0$, choose
${\textnormal{\textbf{x}}}_\tau \in R_\tau^{(m)}$, and define $\mathbb P_{\boldsymbol{\alpha}} := \sum_{\tau:\,\alpha_\tau>0} \alpha_\tau \delta_{{\textnormal{\textbf{x}}}_\tau}$. Since the strips $\{K_i^{(m)}\}_{i=1}^{n_m}$ partition $K$, summing the strip constraints
gives $\sum_{\tau=1}^{L_m} \alpha_\tau
=
\sum_{i=1}^{n_m} p_i^{(m)}
=
1$, so $\mathbb P_{\boldsymbol{\alpha}}$ is a probability measure. Moreover,
\[
\mathbb P_{\boldsymbol{\alpha}}\bigl(K_i^{(m)}\bigr)
=
\sum_{\tau:\,R_\tau^{(m)} \subseteq K_i^{(m)}} \alpha_\tau
=
p_i^{(m)},
\qquad i=1,\ldots,n_m,
\]
and
\[
\mathbb P_{\boldsymbol{\alpha}}(F_\ell)
=
\sum_{\tau:\,R_\tau^{(m)} \subseteq F_\ell} \alpha_\tau
=
q_\ell,
\qquad \ell=1,\ldots,L.
\]
Thus, $\mathbb P_{\boldsymbol{\alpha}} \in {\mathcal D_{F,q}}^{(m)}$. This justifies the decomposition \eqref{eqn_appD_item1decomposition}.

The following argument further shows
\begin{equation}
\inf_{\substack{\mathbb P\in {\mathcal D_{F,q}}^{(m)}\\ \mathbb P(R_\tau^{(m)})=\alpha_\tau,\ \tau=1,\ldots,L_m}}
\frac{\mathbb E_{\mathbb P}[f(\textnormal{\textbf{X}})]}{\mathbb E_{\mathbb P}[g(\textnormal{\textbf{X}})]}
=
V_{F,q,m}(\boldsymbol{\alpha}).
\label{eqn_appD_Vfqmequivalence}
\end{equation}
Since ${\mathcal D_{F,q}}\neq\varnothing$ and ${\mathcal D_{F,q}}\subseteq {\mathcal D_{F,q}}^{(m)}$, the polytope
${\mathcal W_{F,q,m}}$ is nonempty. Fix $\boldsymbol{\alpha}\in {\mathcal W_{F,q,m}}$, and let $\mathcal T(\boldsymbol{\alpha}):=\{\tau\in\{1,\ldots,L_m\}:\alpha_\tau>0\}$, $K_{\boldsymbol{\alpha}}:=\bigcup_{\tau\in\mathcal T(\boldsymbol{\alpha})} R_\tau^{(m)}$. Then every probability measure $\mathbb P$ on $K$ that satisfies $\mathbb P\bigl(R_\tau^{(m)}\bigr)=\alpha_\tau$ ($\tau=1,\ldots,L_m$) is concentrated on $K_{\boldsymbol{\alpha}}$. For each $\tau\in\mathcal T(\boldsymbol{\alpha})$, define the
conditional probability measure
\[
\mathbb P_\tau(A):=\frac{\mathbb P(A\cap R_\tau^{(m)})}{\alpha_\tau},
\qquad A\in\mathcal B(K).
\]
Then, $\mathbb E_{\mathbb P}[f(\textnormal{\textbf{X}})]
=
\sum_{\tau\in\mathcal T(\boldsymbol{\alpha})} \alpha_\tau \mathbb E_{\mathbb P_\tau}[f(\textnormal{\textbf{X}})]$ and similarly for $g(\mathbf{X})$. Moreover, for each $\tau\in\mathcal T(\boldsymbol{\alpha})$, the pair $\bigl(\mathbb E_{\mathbb P_\tau}[f(\textnormal{\textbf{X}})],\,\mathbb E_{{\mathbb P}_\tau}[g(\textnormal{\textbf{X}})]\bigr)$ lies in the convex hull of $\{(f(\textnormal{\textbf{x}}),g(\textnormal{\textbf{x}})):\textnormal{\textbf{x}}\in R_\tau^{(m)}\}$. Hence, by the same reduction argument as in the proof of Proposition~\ref{prop_cbi_transform}, minimising
the ratio over measures with fixed refined masses $\boldsymbol{\alpha}$ is equivalent to minimising
over one support point ${\textnormal{\textbf{x}}}_\tau\in R_\tau^{(m)}$ for each $\tau\in\mathcal T(\boldsymbol{\alpha})$:
\[
\inf_{\substack{\mathbb P\in {\mathcal D_{F,q}}^{(m)}\\ \mathbb P(R_\tau^{(m)})=\alpha_\tau,\ \tau=1,\ldots,L_m}}
\frac{\mathbb E_{\mathbb P}[f(\textnormal{\textbf{X}})]}{\mathbb E_{\mathbb P}[g(\textnormal{\textbf{X}})]}
=
V_{F,q,m}(\boldsymbol{\alpha}),
\]
which proves \eqref{eqn_appD_Vfqmequivalence}.

Taking the infimum over all feasible $\boldsymbol{\alpha}\in {\mathcal W_{F,q,m}}$ and using \eqref{eqn_appD_item1decomposition} yields
\[
\inf_{\mathbb P \in {\mathcal D_{F,q}}^{(m)}} \frac{\mathbb E_{\mathbb P}[f(\textnormal{\textbf{X}})]}{\mathbb E_{\mathbb P}[g(\textnormal{\textbf{X}})]}
=
\inf_{\boldsymbol{\alpha} \in {\mathcal W_{F,q,m}}} V_{F,q,m}(\boldsymbol{\alpha})
=
\phi_{F,q,m}.
\]
This proves item~(1).

We next prove item~(2). Since every $\mathbb P \in {\mathcal D_{F,q}}$ satisfies
\[
\mathbb P\bigl(K_i^{(m)}\bigr)
=
\mathbb P\bigl(B_i^{(m)} \times V\bigr)
=
\int_{B_i^{(m)}} \rho(\textnormal{\textbf{u}})\,d\textnormal{\textbf{u}}
=
p_i^{(m)},
\qquad i=1,\ldots,n_m,
\]
and also $\mathbb P(F_\ell)=q_\ell$ for every $\ell$, we have ${\mathcal D_{F,q}} \subseteq {\mathcal D_{F,q}}^{(m)}$. Therefore,
\[
\inf_{\mathbb P \in {\mathcal D_{F,q}}^{(m)}} \frac{\mathbb E_{\mathbb  P}[f(\textnormal{\textbf{X}})]}{\mathbb E_{\mathbb  P}[g(\textnormal{\textbf{X}})]}
\leqslant
\inf_{\mathbb P \in {\mathcal D_{F,q}}} \frac{\mathbb E_{\mathbb P}[f(\textnormal{\textbf{X}})]}{{\mathbb E}_{\mathbb P}[g(\textnormal{\textbf{X}})]};
\]
that is,
\begin{equation}
\phi_{F,q,m} \leqslant \phi_{F,q}^*.
\label{eqn_appD_ineq1}
\end{equation}

The reverse inequality is shown by construction and an approximation-error-control argument. Fix $m \in \mathbb N$ and let $\mathbb Q \in {\mathcal D_{F,q}}^{(m)}$ be arbitrary. Define $\alpha_\tau := \mathbb Q\bigl(R_\tau^{(m)}\bigr)\,$ for $\,\tau=1,\ldots,L_m$ and, for each strip index $i$, let
\[
T_i^{(m)} := \{\, \tau \in \{1,\ldots,L_m\} : R_\tau^{(m)} \subseteq K_i^{(m)} \,\},
\qquad
T_i^{(m,{\mathbb Q})} := \{\, \tau \in T_i^{(m)} : \alpha_\tau>0 \,\}.
\]
For each $i\in I_m^+$ and $\textnormal{\textbf{u}}\in B_i^{(m)}$, define a probability kernel on $V$ by
\[
\kappa_{i,\textnormal{\textbf{u}}}^{(m,{\mathbb Q})}
:=
\sum_{\tau\in T_i^{(m,{\mathbb Q})}} \frac{\alpha_\tau}{p_i^{(m)}} \,\nu_{\tau,\textnormal{\textbf{u}}}^{(m)}.
\]
This is well-defined because if $\alpha_\tau>0$, then $\tau\in\mathcal T_m^+$ by
definition, so (by the kernel assumption in the Proposition) there is the associated kernel $\nu_{\tau,\textnormal{\textbf{u}}}^{(m)}$. Moreover, since $\sum_{\tau\in T_i^{(m,{\mathbb Q})}} \alpha_\tau
=
\sum_{\tau\in T_i^{(m)}} \alpha_\tau
=
p_i^{(m)}$, the kernel $\kappa_{i,\textnormal{\textbf{u}}}^{(m,{\mathbb Q})}$ is a probability measure on $V$ for each
$i\in I_m^+$ and $\rho$-a.e. $\textnormal{\textbf{u}}\in B_i^{(m)}$. 

Now, define a probability measure $\widetilde{\mathbb Q}$ on $K$ by
\[
\widetilde{\mathbb Q}(A)
:=
\sum_{i\in I_m^+}
\int_{B_i^{(m)}} \rho(\textnormal{\textbf{u}})\,
\kappa_{i,\textnormal{\textbf{u}}}^{(m,{\mathbb Q})}(A_{\textnormal{\textbf{u}}})\,d\textnormal{\textbf{u}},
\qquad A\in\mathcal B(K),
\]
where
\[
A_{\textnormal{\textbf{u}}} := \{\, \textnormal{\textbf{v}}\in V : (\textnormal{\textbf{u}},\textnormal{\textbf{v}})\in A \,\}.
\]
Since each $\kappa_{i,\textnormal{\textbf{u}}}^{(m,{\mathbb Q})}$ is a probability measure and
$\sum_{i\in I_m^+}\int_{B_i^{(m)}}\rho(\textnormal{\textbf{u}})\,d\textnormal{\textbf{u}}=1$, the measure $\widetilde{\mathbb Q}$ is a
probability measure. 

We show that $\widetilde{\mathbb Q}\in {\mathcal D_{F,q}}$ and that it preserves all refined-atom masses. First, for any Borel $A\subseteq U$,
\[
\widetilde{\mathbb Q}(A\times V)
=
\sum_{i\in I_m^+}
\int_{A\cap B_i^{(m)}} \rho(\textnormal{\textbf{u}})\,
\kappa_{i,\textnormal{\textbf{u}}}^{(m,{\mathbb Q})}(V)\,d\textnormal{\textbf{u}}
=
\sum_{i\in I_m^+}
\int_{A\cap B_i^{(m)}} \rho(\textnormal{\textbf{u}})\,d\textnormal{\textbf{u}}
=
\int_A \rho(\textnormal{\textbf{u}})\,d\textnormal{\textbf{u}}.
\]
Here the omitted indices $i\notin I_m^+$ contribute zero, since $p_i^{(m)}=\int_{B_i^{(m)}}\rho(\mathbf u)\,d\mathbf u=0$, and hence
$\int_{A\cap B_i^{(m)}}\rho(\mathbf u)\,d\mathbf u=0$. So, $\widetilde{\mathbb Q}$ has the required $U$-marginal. Next, fix $\tau\in\{1,\ldots,L_m\}$ and set $i=i(\tau)$. If $p_i^{(m)}=0$ then $\alpha_\tau=0$, because $0 \leqslant \alpha_\tau \leqslant \sum_{\tau':\,R_{\tau'}^{(m)}\subseteq K_i^{(m)}} \alpha_{\tau'}
= p_i^{(m)} = 0$. Also, $p_i^{(m)}=0$ means $i\notin I_m^+$, so the sum defining
$\widetilde{\mathbb Q}$ contains no term integrating over
$B_i^{(m)}$. Since $R_\tau^{(m)}\subseteq K_i^{(m)}
= B_i^{(m)}\times V$, its vertical section is empty for every
$\mathbf u\in B_j^{(m)}$ with $j\in I_m^+$. Hence,
$\widetilde{\mathbb Q}(R_\tau^{(m)})=0$. Now, suppose $p_i^{(m)}>0$. If $\alpha_\tau=0$ then $\tau\notin T_i^{(m,{\mathbb Q})}$. Since $R_\tau^{(m)}\subseteq K_i^{(m)}$, the vertical section
$R_{\tau,\mathbf{u}}^{(m)}$ is empty for $\mathbf{u}\in B_j^{(m)}$ whenever $j\neq i$.
Thus, only the $i$-th strip contributes to
$\widetilde{\mathbb Q}(R_\tau^{(m)})$. Consider
\begin{equation}
\widetilde{\mathbb Q}\bigl(R_\tau^{(m)}\bigr)
=
\int_{B_i^{(m)}} \rho(\textnormal{\textbf{u}})\,
\sum_{\tau'\in T_i^{(m,{\mathbb Q})}} \frac{\alpha_{\tau'}}{p_i^{(m)}}\,
\nu_{\tau',\textnormal{\textbf{u}}}^{(m)}\!\bigl(R_{\tau,\textnormal{\textbf{u}}}^{(m)}\bigr)\,d\textnormal{\textbf{u}}.
\label{eqn_appD_ConstrQonRtm}
\end{equation}
For every $\tau' \in T_i^{(m,{\mathbb Q})}$ we have $\alpha_{\tau'}>0$, hence
$\tau'\in\mathcal T_m^+$, so the related kernel $\nu_{\tau',\textnormal{\textbf{u}}}^{(m)}$ is supported
on $R_{\tau',\textnormal{\textbf{u}}}^{(m)}$. Since the refined atoms partition $K_i^{(m)}$, the sections
$R_{\tau',\textnormal{\textbf{u}}}^{(m)}$ and $R_{\tau,\textnormal{\textbf{u}}}^{(m)}$ are disjoint whenever $\tau'\neq\tau$.
Because $\tau\notin T_i^{(m,{\mathbb Q})}$, every term in the sum in \eqref{eqn_appD_ConstrQonRtm} is $0$ for $\rho$-a.e. $\textnormal{\textbf{u}}$. Therefore, $\widetilde{\mathbb Q}\bigl(R_\tau^{(m)}\bigr)=0=\alpha_\tau$.
Finally, suppose $\alpha_\tau>0$. Then $\tau\in T_i^{(m,{\mathbb Q})}$, and since the refined
atoms partition $K_i^{(m)}$, their vertical sections partition $V$ for each fixed
$\textnormal{\textbf{u}}\in B_i^{(m)}$. Hence, for $\tau'\in T_i^{(m,{\mathbb Q})}$ with $\tau'\neq\tau$, the sections
$R_{\tau',\textnormal{\textbf{u}}}^{(m)}$ and $R_{\tau,\textnormal{\textbf{u}}}^{(m)}$ are disjoint, and therefore
\[
\nu_{\tau',\textnormal{\textbf{u}}}^{(m)}\!\bigl(R_{\tau,\textnormal{\textbf{\textnormal{\textbf{u}}}}}^{(m)}\bigr)=0
\quad\text{for $\rho$-a.e. }\textnormal{\textbf{u}}\in B_i^{(m)}.
\]
Also, by assumption,
\[
\nu_{\tau,\textnormal{\textbf{u}}}^{(m)}\!\bigl(R_{\tau,\textnormal{\textbf{u}}}^{(m)}\bigr)=1
\quad\text{for $\rho$-a.e. }\textnormal{\textbf{u}}\in B_i^{(m)}.
\]
Therefore,
\begin{align*}
\widetilde{\mathbb Q}\bigl(R_\tau^{(m)}\bigr)
&=
\int_{B_i^{(m)}} \rho(\textnormal{\textbf{u}})\,
\kappa_{i,\textnormal{\textbf{u}}}^{(m,{\mathbb Q})}\!\bigl(R_{\tau,\textnormal{\textbf{u}}}^{(m)}\bigr)\,d\textnormal{\textbf{u}} \\
&=
\int_{B_i^{(m)}} \rho(\textnormal{\textbf{u}})\,
\sum_{\tau'\in T_i^{(m,{\mathbb Q})}} \frac{\alpha_{\tau'}}{p_i^{(m)}}\,
\nu_{\tau',\textnormal{\textbf{u}}}^{(m)}\!\bigl(R_{\tau,\textnormal{\textbf{u}}}^{(m)}\bigr)\,d\textnormal{\textbf{u}} \\
&=
\int_{B_i^{(m)}} \rho(\textnormal{\textbf{u}})\,\frac{\alpha_\tau}{p_i^{(m)}}\,d\textnormal{\textbf{u}} \\
&=
\frac{\alpha_\tau}{p_i^{(m)}} \int_{B_i^{(m)}} \rho(\textnormal{\textbf{u}})\,d\textnormal{\textbf{u}} \\
&=\frac{\alpha_\tau}{p_i^{(m)}} p^{(m)}_i \\
&= \alpha_\tau.
\end{align*}
Thus, $\widetilde{\mathbb Q}\bigl(R_\tau^{(m)}\bigr)=\mathbb Q\bigl(R_\tau^{(m)}\bigr)$ for all $\tau=1,\ldots,L_m$. Since each $F_\ell$ is the union of those refined atoms contained in it, $\widetilde{\mathbb Q}(F_\ell)
=
\sum_{\tau:\,R_\tau^{(m)}\subseteq F_\ell} \widetilde{\mathbb Q}\bigl(R_\tau^{(m)}\bigr)
=
\sum_{\tau:\,R_\tau^{(m)}\subseteq F_\ell} \alpha_\tau
=
q_\ell$. We have shown $\widetilde{\mathbb Q}\in {\mathcal D_{F,q}}$.

We can use this constructed $\widetilde{Q}$ to prove an approximate reverse inequality to \eqref{eqn_appD_ineq1}. Note that
\[
{\mathbb E}_{\mathbb Q}[f(\textnormal{\textbf{X}})]
=
\sum_{\tau=1}^{L_m} \mathbb {\mathbb E}_{\mathbb Q}[f(\textnormal{\textbf{X}})\textbf{1}_{R_\tau^{(m)}}],
\qquad
{\mathbb E}_{\widetilde{\mathbb Q}}[f(\textnormal{\textbf{X}})]
=
\sum_{\tau=1}^{L_m} {\mathbb E}_{\widetilde{\mathbb Q}}[f(\textnormal{\textbf{X}})\textbf{1}_{R_\tau^{(m)}}].
\]
If $\alpha_\tau=0$, both expectations over $R_\tau^{(m)}$ are $0$. If $\alpha_\tau>0$, then
\[
\mathbb {\mathbb E}_{\mathbb Q}[f(\textnormal{\textbf{X}})\textbf{1}_{R_\tau^{(m)}}]
=
\alpha_\tau \bar f_{\tau,\mathbb Q},
\qquad
{\mathbb E}_{\widetilde{\mathbb Q}}[f(\textnormal{\textbf{X}})\textbf{1}_{R_\tau^{(m)}}]
=
\alpha_\tau \bar f_{\tau,\widetilde{\mathbb Q}},
\]
for some numbers $\bar f_{\tau,\mathbb Q}$ and $\bar f_{\tau,\widetilde{\mathbb Q}}$ lying between
$\inf_{\textnormal{\textbf{x}} \in R_\tau^{(m)}} f(\textnormal{\textbf{x}})$ and $\sup_{\textnormal{\textbf{x}} \in R_\tau^{(m)}} f(\textnormal{\textbf{x}})$. Hence, for every
$\tau=1,\ldots,L_m$,
\[
\left|
{\mathbb E}_{\mathbb Q}[f(\textnormal{\textbf{X}})\textbf{1}_{R_\tau^{(m)}}]
-
{\mathbb E}_{\widetilde{\mathbb Q}}[f(\textnormal{\textbf{X}})\textbf{1}_{R_\tau^{(m)}}]
\right|
\leqslant
\alpha_\tau \sup_{\textnormal{\textbf{x}},\textnormal{\textbf{y}} \in R_\tau^{(m)}} |f(\textnormal{\textbf{x}})-f(\textnormal{\textbf{y}})|
\leqslant
\alpha_\tau \widehat{\omega}_f(m).
\]
Summing over $\tau$ and using $\sum_{\tau=1}^{L_m}\alpha_\tau=1$, we obtain $|{\mathbb E}_{\mathbb Q}[f(\textnormal{\textbf{X}})] - {\mathbb E}_{\widetilde{\mathbb Q}}[f(\textnormal{\textbf{X}})]|
\leqslant
\widehat{\omega}_f(m)$. The same argument gives $|{\mathbb E}_{\mathbb Q}[g(\textnormal{\textbf{X}})] - {\mathbb E}_{\widetilde{\mathbb Q}}[g(\textnormal{\textbf{X}})]|
\leqslant
\widehat{\omega}_g(m)$. Write
\[
a := {\mathbb E}_{\mathbb Q}[f(\textnormal{\textbf{X}})], \qquad b := {\mathbb E}_{\mathbb Q}[g(\textnormal{\textbf{X}})], \qquad
\widetilde{a} := {\mathbb E}_{\widetilde{\mathbb Q}}[f(\textnormal{\textbf{X}})], \qquad
\widetilde{b} := {\mathbb E}_{\widetilde{\mathbb Q}}[g(\textnormal{\textbf{X}})].
\]
Since $g \geqslant c>0$ on $K$ for some $c\in\mathbb R$, $b \geqslant c$ and $\widetilde{b} \geqslant c$. Also, $0 \leqslant f \leqslant M_f$ implies $0 \leqslant \widetilde{a} \leqslant M_f$. Hence,
\[
\left|
\frac{a}{b} - \frac{\widetilde{a}}{\widetilde{b}}
\right|
=
\left|
\frac{(a-\widetilde{a})\widetilde{b} + \widetilde{a}(\widetilde{b}-b)}
{b\widetilde{b}}
\right|
\leqslant
\frac{|a-\widetilde{a}|}{c}
+
\frac{M_f\,|b-\widetilde{b}|}{c^2}
\leqslant
\eta_{F,q,m}.
\]
We have $\frac{\widetilde{a}}{\widetilde{b}}
\geqslant
\phi_{F,q}^*$ because $\widetilde{\mathbb Q} \in {\mathcal D_{F,q}}$. Therefore,
\[
\frac{{\mathbb E}_{\mathbb Q}[f(\textnormal{\textbf{X}})]}{{\mathbb E}_{\mathbb Q}[g(\textnormal{\textbf{X}})]}
=
\frac{a}{b}
\geqslant \frac{\widetilde{a}}{\widetilde{b}} - \eta_{F,q,m}
\geqslant
\phi_{F,q}^* - \eta_{F,q,m}.
\]
Furthermore, because $\mathbb Q \in {\mathcal D_{F,q}}^{(m)}$ was arbitrary, taking the infimum over
$\mathbb Q \in {\mathcal D_{F,q}}^{(m)}$ gives the approximate reverse inequality to \eqref{eqn_appD_ineq1},
\begin{equation}
\phi_{F,q,m}
=
\inf_{\mathbb Q \in {\mathcal D_{F,q}}^{(m)}} \frac{{\mathbb E}_{\mathbb Q}[f(\textnormal{\textbf{X}})]}{{\mathbb E}_{\mathbb Q}[g(\textnormal{\textbf{X}})]}
\geqslant
\phi_{F,q}^* - \eta_{F,q,m}.
\label{eqn_appD_ineq2}
\end{equation}
By combining inequalities \eqref{eqn_appD_ineq1} and \eqref{eqn_appD_ineq2} we finally deduce
\[
\phi_{F,q}^* - \eta_{F,q,m}
\leqslant
\phi_{F,q,m}
\leqslant
\phi_{F,q}^*.
\]
Since $\widehat{\omega}_f(m)\to 0$ and $\widehat{\omega}_g(m)\to 0$, we have
$\eta_{F,q,m}\to 0$, so $\phi_{F,q,m} \to \phi_{F,q}^*$. This proves item~(2).
\end{proof}

\newpage\section{Proof of Proposition~\ref{prop_gen_multiple_marginals_piecewise_refined}}
\begin{proof}
We first prove item (1). 

Fix $m\in\mathbb N$. Let $\mathbb P\in {{\mathcal D}^{(m)}}$ and define $w^{\mathbb P}:=\bigl(\mathbb P(C^{(m)}_\nu)\bigr)_{\nu=1}^{N_m}$. Since the cells $\{C^{(m)}_\nu\}_{\nu=1}^{N_m}$ partition $K$, and each strip
$K^{(m)}_{j,r}$ is a union of those cells contained in it,
\[
\sum_{\nu:\,C^{(m)}_\nu\subseteq K^{(m)}_{j,r}} w_\nu^{\mathbb P}
=
\sum_{\nu:\,C^{(m)}_\nu\subseteq K^{(m)}_{j,r}} \mathbb P(C^{(m)}_\nu)
=
\mathbb P\bigl(K^{(m)}_{j,r}\bigr)
=
p^{(m)}_{j,r}.
\]
Hence, $w^{\mathbb P}\in {\mathcal W}_m$.

Let $I_{\mathbb P}:=\{\nu\in\{1,\dots,N_m\}: w_\nu^{\mathbb P}>0\}$. By the same argument as in Proposition~\ref{prop_cbi_transform} applied to the positive-mass subpartition $\{C^{(m)}_\nu:\nu\in I_{\mathbb P}\}$ we obtain
\[
\inf_{\substack{\mathbb Q\in\mathcal P(K)\\ \mathbb Q(C^{(m)}_\nu)=w_\nu^{\mathbb P}\ \forall \nu}}
\frac{{\mathbb E}_{\mathbb Q}[f(\textnormal{\textbf{X}})]}{{\mathbb E}_{\mathbb Q}[g(\textnormal{\textbf{X}})]}
=
\inf_{\substack{\textbf{x}_\nu\in C^{(m)}_\nu\\ \nu\in I_{\mathbb P}}}
\frac{\sum_{\nu\in I_{\mathbb P}} w_\nu^{\mathbb P} f(\textbf{x}_\nu)}
{\sum_{\nu\in I_{\mathbb P}} w_\nu^{\mathbb P} g(\textbf{x}_\nu)}
=
V_m(w^{\mathbb P}).
\]
Since $\mathbb P$ belongs to the constraining class on the left-hand side,
\[
\frac{{\mathbb E}_{\mathbb P}[f(\textnormal{\textbf{X}})]}{{\mathbb E}_{\mathbb P}[g(\textnormal{\textbf{X}})]}\geqslant V_m(w^{\mathbb P})\geqslant \inf_{\textnormal{\textbf{w}}\in {\mathcal W}_m}V_m(w)=\phi_m.
\]
So, taking the infimum over $\mathbb P\in {{\mathcal D}^{(m)}}$ gives
\begin{equation}
\inf_{\mathbb P\in {{\mathcal D}^{(m)}}}\frac{{\mathbb E}_{\mathbb P}[f(\textnormal{\textbf{X}})]}{{\mathbb E}_{\mathbb P}[g(\textnormal{\textbf{X}})]}\geqslant \phi_m.
\label{eqn_appE_ineq1}
\end{equation}

Conversely, let $\textnormal{\textbf{w}}\in {\mathcal W}_m$ and $\varepsilon>0$. Since
\[
\sum_{\nu=1}^{N_m}w_\nu
=
\sum_{r=1}^{n_j(m)}
\sum_{\nu:\,C^{(m)}_\nu\subseteq K^{(m)}_{j,r}} w_\nu
=
\sum_{r=1}^{n_j(m)} p^{(m)}_{j,r}
=
1,
\]
the coefficients $(w_\nu)_{\nu=1}^{N_m}$ define a
probability vector. Choose points ${\textbf{x}}_\nu\in C^{(m)}_\nu$ for those $\nu$ with
$w_\nu>0$, so that
\[
\frac{\sum_{\nu=1}^{N_m} w_\nu f({\textbf{x}}_\nu)}
{\sum_{\nu=1}^{N_m} w_\nu g({\textbf{x}}_\nu)}
\leqslant V_m(\textbf{w})+\varepsilon.
\]
Define the atomic measure $\mathbb P_{\textbf{w}}:=\sum_{\nu=1}^{N_m} w_\nu\delta_{\mathbf x_\nu}$.
Then, for every $j=1,\dots,M$ and $r=1,\dots,n_j(m)$, $\mathbb P_{\textbf{w}}\bigl(K^{(m)}_{j,r}\bigr)
=
\sum_{\nu:\,C^{(m)}_\nu\subseteq K^{(m)}_{j,r}} w_\nu
=
p^{(m)}_{j,r}$ because every cell is contained in exactly one strip $K^{(m)}_{j,r}$ in the
$j$th coordinate. Hence, ${\mathbb P}_w\in {{\mathcal D}^{(m)}}$, and therefore
\[
\inf_{\mathbb P\in {{\mathcal D}^{(m)}}}\frac{{\mathbb E}_{\mathbb P}[f(\textnormal{\textbf{X}})]}{{\mathbb E}_{\mathbb P}[g(\textnormal{\textbf{X}})]}
\leqslant
\frac{\mathbb E_{{\mathbb P}_{\textbf{w}}}[f(\textnormal{\textbf{X}})]}{\mathbb E_{{\mathbb P}_{\textbf{w}}}[g(\textnormal{\textbf{X}})]}
=
\frac{\sum_{\nu=1}^{N_m} w_\nu f({\textbf{x}}_\nu)}
{\sum_{\nu=1}^{N_m} w_\nu g({\textbf{x}}_\nu)}
\leqslant
V_m({\textbf{w}})+\varepsilon.
\]
Taking $\varepsilon\downarrow 0$ and then the infimum over $\textnormal{\textbf{w}}\in {\mathcal W}_m$ gives
\begin{equation}
\inf_{\mathbb P\in {{\mathcal D}^{(m)}}}\frac{{\mathbb E}_{\mathbb P}[f(\textnormal{\textbf{X}})]}{{\mathbb E}_{\mathbb P}[g(\textnormal{\textbf{X}})]}\leqslant \phi_m.
\label{eqn_appE_ineq2}
\end{equation}
Inequalities \eqref{eqn_appE_ineq1} and \eqref{eqn_appE_ineq2} prove item (1).

We now prove item (2). For each $m\in\mathbb N$, define the lower and upper cellwise step
envelopes
\[
f_m^-(\textbf{x}):=\inf_{\textbf{y}\in C^{(m)}_\nu} f(\textbf{y}),\qquad
f_m^+(\textbf{x}):=\sup_{\textbf{y}\in C^{(m)}_\nu} f(\textbf{y}),
\qquad \textbf{x}\in C^{(m)}_\nu,
\]
and similarly
\[
g_m^-(\textbf{x}):=\inf_{\textbf{y}\in C^{(m)}_\nu} g(\textbf{y}),\qquad
g_m^+(\textbf{x}):=\sup_{\textbf{y}\in C^{(m)}_\nu} g(\textbf{y}),
\qquad \textbf{x}\in C^{(m)}_\nu.
\]
Thus, $f_m^- \leqslant f \leqslant f_m^+\,$ and $\,g_m^- \leqslant g \leqslant g_m^+$. By the definitions of $\omega_f(m)\,$ and $\,\omega_g(m)$, the bounds $\|f-f_m^\pm\|_\infty\leqslant \omega_f(m)\,$ and 
$\,\|g-g_m^\pm\|_\infty\leqslant \omega_g(m)$ hold.
Define
\[
J({\mathbb P}):=\frac{{\mathbb E}_{\mathbb P}[f(\textnormal{\textbf{X}})]}{{\mathbb E}_{\mathbb P}[g(\textnormal{\textbf{X}})]},
\qquad
J_m^-({\mathbb P}):=\frac{{\mathbb E}_{\mathbb P}[f_m^-(\textbf{X})]}{{\mathbb E}_{\mathbb P}[g_m^+(\textbf{X})]},
\qquad
J_m^+({\mathbb P}):=\frac{{\mathbb E}_{\mathbb P}[f_m^+(\textbf{X})]}{{\mathbb E}_{\mathbb P}[g_m^-(\textbf{X})]}.
\]
For every probability measure $\mathbb P$ on $K$ and all $m\in\mathbb N$, $J_m^-({\mathbb P})\leqslant J({\mathbb P})\leqslant J_m^+({\mathbb P}),\,$
$|J({\mathbb P})-J_m^-({\mathbb P})|\leqslant \eta_m,\,$
and $\,|J_m^+({\mathbb P})-J({\mathbb P})|\leqslant \eta_m,\,$ where $\,M_f:=\sup\limits_{\textnormal{\textbf{x}}\in K}f(\textnormal{\textbf{x}})\,$ and $\,\eta_m:=\frac{\omega_f(m)}{c}
+\frac{M_f\,\omega_g(m)}{c^2}$. Indeed, writing
\begin{align*}
a&:={\mathbb E}_{\mathbb P}[f],\;\;a^-:={\mathbb E}_{\mathbb P}[f_m^-],\;\;a^+:={\mathbb E}_{\mathbb P}[f_m^+],\\
b&:={\mathbb E}_{\mathbb P}[g],\;\;b^-:={\mathbb E}_{\mathbb P}[g_m^-],\;\;\;b^+:={\mathbb E}_{\mathbb P}[g_m^+],
\end{align*}
and recalling $g\geqslant c>0$ on $K$ for some $c\in\mathbb R$, we have
\[
|a-a^\pm|\leqslant \omega_f(m),\qquad |b-b^\pm|\leqslant \omega_g(m),
\qquad b,b^-,b^+\geqslant c,
\]
and therefore
\begin{equation}
\left|\frac{a}{b}-\frac{a^\pm}{b^\mp}\right|
\leqslant
\frac{|a-a^\pm|}{c}
+\frac{M_f|b-b^\mp|}{c^2}
\leqslant \eta_m.
\end{equation}

Next we show that, for step functions constant on the product cells, the
optimisation over $\mathcal D$ and over ${\mathcal D^{(m)}}$ coincide. Let $s,t$ be bounded
functions constant on each cell $C^{(m)}_\nu$, with values
$s_\nu,t_\nu$ on $C^{(m)}_\nu$, and assume $\inf_K t>0$. Then,
\begin{equation}
\inf_{\mathbb P\in {{\mathcal D}^{(m)}}}\frac{{\mathbb E}_{\mathbb P}[s(\textbf{X})]}{{\mathbb E}_{\mathbb P}[t(\textbf{X})]}
=
\inf_{\textnormal{\textbf{w}}\in {\mathcal W}_m}\frac{\sum_{\nu=1}^{N_m} w_\nu s_\nu}
{\sum_{\nu=1}^{N_m} w_\nu t_\nu}
=
\inf_{{\mathbb P}\in{\mathcal D}}\frac{{\mathbb E}_{\mathbb P}[s(\textbf{X})]}{{\mathbb E}_{\mathbb P}[t(\textbf{X})]}.
\label{eq:step-equality}
\end{equation}
The first equality follows by repeating the proof of item~(1) with $(s,t)$ in
place of $(f,g)$; since $s$ and $t$ are constant on each cell $C^{(m)}_\nu$,
the inner minimisation in the definition of $V_m$ is trivial. For the second
equality, every ${\mathbb P}\in{\mathcal D}$ induces a weight vector
$w^{\mathbb P}=({\mathbb P}(C^{(m)}_\nu))_{\nu=1}^{N_m}\in {\mathcal W}_m$, so
\begin{equation}
\inf_{{\mathbb P}\in{\mathcal D}}\frac{{\mathbb E}_{\mathbb P}[s(\textbf{X})]}{{\mathbb E}_{\mathbb P}[t(\textbf{X})]}
\geqslant
\inf_{\textnormal{\textbf{w}}\in {\mathcal W}_m}\frac{\sum_{\nu=1}^{N_m} w_\nu s_\nu}
{\sum_{\nu=1}^{N_m} w_\nu t_\nu}.
\label{eqn_appE_ineq4}
\end{equation}
Conversely, let $\textnormal{\textbf{w}}\in {\mathcal W}_m$. For every $j,r$ with $p^{(m)}_{j,r}>0$, define
\[
\mu^{(m)}_{j,r}(A):=
\frac{\mu_j(A\cap B^{(m)}_{j,r})}{p^{(m)}_{j,r}},
\qquad A\in\mathcal B(U_j).
\]
If $p^{(m)}_{j,r}=0$ then every $w_\nu$ with
$C^{(m)}_\nu\subseteq K^{(m)}_{j,r}$ is necessarily zero, so those strips may be
ignored. Write $\nu\leftrightarrow (i_1,\dots,i_M)$ according to the product-cell
indexing. If \(w_\nu>0\) and \(\nu\leftrightarrow(i_1,\ldots,i_M)\), then
\(p_{j,i_j}^{(m)}>0\) for every \(j\), so all the measures
\(\mu_{j,i_j}^{(m)}\) are defined. Define the mixture of products of conditional measures
\[
Q_w := \sum_{\nu:\,w_\nu>0}
w_\nu \bigotimes_{j=1}^M \mu_{j,i_j}^{(m)} .
\]
Then ${\mathbb Q}_w\in {\mathcal D}$, ${\mathbb Q}_w(C^{(m)}_\nu)=w_\nu$ for every $\nu$, and hence
\[
\inf_{{\mathbb P}\in{\mathcal D}}\frac{{\mathbb E}_{\mathbb P}[s(\textbf{X})]}{{\mathbb E}_{\mathbb P}[t(\textbf{X})]}
\leqslant
\frac{{\mathbb E}_{{\mathbb Q}_w}[s(\textbf{X})]}{{\mathbb E}_{{\mathbb Q}_w}[t(\textbf{X})]}
=
\frac{\sum_{\nu=1}^{N_m} w_\nu s_\nu}
{\sum_{\nu=1}^{N_m} w_\nu t_\nu}.
\]
Taking the infimum over $\textnormal{\textbf{w}}\in {\mathcal W}_m$ gives the inequality 
\begin{equation}
\inf_{{\mathbb P}\in{\mathcal D}}\frac{{\mathbb E}_{\mathbb P}[s(\textbf{X})]}{{\mathbb E}_{\mathbb P}[t(\textbf{X})]}
\leqslant
\inf_{\textnormal{\textbf{w}}\in {\mathcal W}_m}\frac{\sum_{\nu=1}^{N_m} w_\nu s_\nu}
{\sum_{\nu=1}^{N_m} w_\nu t_\nu}.
\label{eqn_appE_ineq5}
\end{equation}
Inequalities ~\eqref{eqn_appE_ineq4} and \eqref{eqn_appE_ineq5} justify the second equality in \eqref{eq:step-equality}.

Apply \eqref{eq:step-equality} with $(s,t)=(f_m^-,g_m^+)$ and with
$(s,t)=(f_m^+,g_m^-)$. Since we showed, for every ${\mathbb P}\in{\mathcal D}$,
\[
J({\mathbb P})-\eta_m \leqslant J_m^-({\mathbb P})\leqslant J({\mathbb P})\leqslant J_m^+({\mathbb P})\leqslant J({\mathbb P})+\eta_m,
\]
taking infima over ${\mathbb P}\in{\mathcal D}$ yields
\[
\phi^*-\eta_m
\leqslant
\inf_{{\mathbb P}\in{\mathcal D}}J_m^-({\mathbb P})
\leqslant
\phi^*
\leqslant
\inf_{{\mathbb P}\in{\mathcal D}}J_m^+({\mathbb P})
\leqslant
\phi^*+\eta_m.
\]
By \eqref{eq:step-equality},
\begin{equation}
\phi^*-\eta_m
\leqslant
\inf_{\mathbb P\in {{\mathcal D}^{(m)}}}J_m^-({\mathbb P})
\leqslant
\phi^*
\leqslant
\inf_{\mathbb P\in {{\mathcal D}^{(m)}}}J_m^+({\mathbb P})
\leqslant
\phi^*+\eta_m.
\label{eqn_appE_ineq6}
\end{equation}

Finally, for every $\mathbb P\in {{\mathcal D}^{(m)}}$,
\[
J_m^-({\mathbb P})\leqslant J({\mathbb P})\leqslant J_m^+({\mathbb P}),
\]
so taking infima over $\mathbb P\in {{\mathcal D}^{(m)}}$ and using the result in item~(1),
\begin{equation}
\inf_{\mathbb P\in {{\mathcal D}^{(m)}}}J_m^-({\mathbb P})
\leqslant
\phi_m
\leqslant
\inf_{\mathbb P\in {{\mathcal D}^{(m)}}}J_m^+({\mathbb P}).
\label{eqn_appE_ineq7}
\end{equation}
Combining inequalities \eqref{eqn_appE_ineq6} and \eqref{eqn_appE_ineq7} gives $\phi^*-\eta_m\leqslant \phi_m\leqslant \phi^*+\eta_m$ for all $m\in\mathbb N$. Since $\omega_f(m)\to 0$ and $\omega_g(m)\to 0$, so $\eta_m\to0$, it follows that $\phi_m\to\phi^*$. This proves item~(2).
\end{proof}
\newpage
\section{Proof of Corollary~\ref{cor_fgvanishingoscillations}}
\begin{proof}
For each fixed $m\in\mathbb N$, the class ${\mathcal D^{(m)}}$ is specified by the finitely many strip constraints
$K^{(m)}_{j,r}$. The sets $\{R^{(m)}_\tau\}_{\tau=1}^{L_m}$ form a finite Borel partition
of $K$ refining the product-cell partition $\{C^{(m)}_\nu\}_{\nu=1}^{N_m}$.

Let $\mathbb P\in {{\mathcal D}^{(m)}}$, and define $\alpha_\tau:=\mathbb P(R^{(m)}_\tau)$. Since the refined atoms
partition each strip,
\[
\sum_{\tau:\,R^{(m)}_\tau\subseteq K^{(m)}_{j,r}}\alpha_\tau
=
\mathbb P(K^{(m)}_{j,r})
=
p^{(m)}_{j,r},
\]
so $\boldsymbol{\alpha}\in \widetilde{\mathcal W}_m$. Also, for any fixed $j$, the sets
$\{K^{(m)}_{j,r}\}_{r=1}^{n_j(m)}$ partition $K$, whence $\sum_{\tau=1}^{L_m}\alpha_\tau=\mathbb P(K)=1$.

Conversely, if $\boldsymbol{\alpha}\in \widetilde{\mathcal W}_m$, then $\mathbb P_{\boldsymbol{\alpha}}:=\sum_{\tau=1}^{L_m}\alpha_\tau\delta_{{\textnormal{\textbf{x}}}_\tau}$ (${\textnormal{\textbf{x}}}_\tau\in R^{(m)}_\tau$) is a probability measure and belongs to ${\mathcal D^{(m)}}$, because
\[
\mathbb P_{\boldsymbol{\alpha}}(K^{(m)}_{j,r})
=
\sum_{\tau:\,R^{(m)}_\tau\subseteq K^{(m)}_{j,r}}\alpha_\tau
=
p^{(m)}_{j,r}.
\]
First optimise over feasible refined masses and then over support points on the refined atoms. Atoms with mass $\alpha_\tau=0$ may be discarded without affecting
the optimisation. Thus, applying
Proposition~\ref{prop_cbi_transform} to the positive-mass subpartition $\{R^{(m)}_\tau:\alpha_\tau>0\}$ yields
\[
\inf_{\mathbb P\in {{\mathcal D}^{(m)}}}\frac{\mathbb E_{\mathbb  P}[f(\textnormal{\textbf{X}})]}{\mathbb E_{\mathbb P}[g(\textnormal{\textbf{X}})]}
=
\inf_{\boldsymbol{\alpha}\in \widetilde{\mathcal W}_m}\widetilde{V}_m(\boldsymbol{\alpha})
=
\widetilde{\phi}_m.
\]
This proves item $(1)$.

For (ii), it is enough to verify the dominance hypothesis of Proposition~\ref{prop_gen_multiple_marginals_piecewise_refined} for the
original product cells. Fix $m$ and $\nu$, and let
\[
I_\nu^{(m)}:=\{\ell\in\{1,\ldots,L\}: C^{(m)}_\nu\cap E_\ell\neq\varnothing\}.
\]
Then
\[
\overline{C^{(m)}_\nu}
=
\bigcup_{\ell\in I_\nu^{(m)}} \overline{C^{(m)}_\nu\cap E_\ell}.
\]
Since $\overline{C^{(m)}_\nu}$ is connected and the family on the right is finite and
closed, its intersection graph is connected: otherwise the family could be partitioned
into two nonempty subfamilies with no cross-intersections, and the unions of the sets
in these two subfamilies would give a decomposition of $\overline{C^{(m)}_\nu}$ into two
disjoint nonempty closed sets, contradicting connectedness. Therefore, for any
$\textnormal{\textbf{x}},\textnormal{\textbf{y}}\in C^{(m)}_\nu$, there exist indices $\ell_0,\ell_1,\ldots,\ell_q\in I_\nu^{(m)}$ with $q+1\leqslant |I_\nu^{(m)}|\leqslant L$, such that
\[
\mathbf x\in \overline{C^{(m)}_\nu\cap E_{\ell_0}},\qquad
\mathbf y\in \overline{C^{(m)}_\nu\cap E_{\ell_q}},
\]
and
\[
\overline{C^{(m)}_\nu\cap E_{\ell_{t-1}}}\cap
\overline{C^{(m)}_\nu\cap E_{\ell_t}}
\neq\varnothing
\qquad (t=1,\ldots,q).
\]
Choose $\textnormal{\textbf{z}}_t$ in each such intersection. Then,
\begin{align*}
|f(\textnormal{\textbf{x}})-f(\textnormal{\textbf{y}})|
&\leqslant
|f(\textnormal{\textbf{x}})-f(\textnormal{\textbf{z}}_1)|+\sum_{t=1}^{q-1}|f(\textnormal{\textbf{z}}_t)-f(\textnormal{\textbf{z}}_{t+1})|+|f(\textnormal{\textbf{z}}_q)-f(\textnormal{\textbf{y}})| \\
&\leqslant
(q+1)\widetilde{\omega}_f(m)
\leqslant
L\,\widetilde{\omega}_f(m).
\end{align*}
Taking the supremum over $\textnormal{\textbf{x}},\textnormal{\textbf{y}}\in C^{(m)}_\nu$, and then the maximum over $\nu$, gives $\omega_f(m)\leqslant L\,\widetilde{\omega}_f(m)$. The same argument also gives $\omega_g(m)\leqslant L\,\widetilde{\omega}_g(m)$.
Since $\widetilde{\omega}_f(m)\to 0$ and $\widetilde{\omega}_g(m)\to 0$, the dominance
hypothesis of Proposition~\ref{prop_gen_multiple_marginals_piecewise_refined} holds. Therefore,
\[
\widetilde{\phi}_m=\inf_{\mathbb P\in {{\mathcal D}^{(m)}}}\frac{\mathbb E_{\mathbb P}[f(\textnormal{\textbf{X}})]}{\mathbb E_{\mathbb P}[g(\textnormal{\textbf{X}})]}\to \phi^*.
\]
Moreover, substituting the bounds $\omega_f(m)\leqslant L\,\widetilde{\omega}_f(m)$, $\omega_g(m)\leqslant L\,\widetilde{\omega}_g(m)$ into the error estimate in Proposition~\ref{prop_gen_multiple_marginals_piecewise_refined} yields, for all $m\in\mathbb N$,
\[
\phi^*-\widetilde{\eta}_m\leqslant \widetilde{\phi}_m\leqslant \phi^*+\widetilde{\eta}_m.
\]
This proves item $(2)$.
\end{proof}
\newpage
\section{Proof of Theorem~\ref{thm:global_fp_tuple}}
\label{app:proofofgenCBIsolnExistence}
\begin{proof}
By Proposition~\ref{prop_cbi_transform} we can consider $\underline{\mathcal K}:=\overline{K}_1\times\cdots\times \overline{K}_n$ for the optimisation. Define
\[
\Phi({\mathbf x}_1,\ldots,{\mathbf x}_n):=
\frac{\sum_{i=1}^n f({\mathbf x}_i)p_i}{\sum_{i=1}^n g({\mathbf x}_i)p_i},
\qquad
({\mathbf x}_1,\ldots,{\mathbf x}_n)\in \underline{\mathcal K}.
\]
$\Phi$ is continuous on the compact domain $\underline{\mathcal K}$. To see this, since each $\overline{K}_i$ is compact and $f,g$ are continuous, it suffices to check positivity of the denominator.
For every $({\mathbf x}_1,\ldots,{\mathbf x}_n)\in \underline{\mathcal K}$, $\sum_{i=1}^n g({\mathbf x}_i)p_i
\geqslant p_j \inf_{\mathbf x\in K_j} g(\mathbf x)
>0$. 
Hence, the infimum of \eqref{eqn_weak_opt_equiv_disc} is attained at some point $({\mathbf x}_1^*,\ldots,{\mathbf x}_n^*)\in \underline{\mathcal K}$. 
Define
\[
\phi^*:=\Phi({\mathbf x}_1^*,\ldots,{\mathbf x}_n^*).
\]

Now fix $\phi\geqslant 0$. For any feasible $({\mathbf x}_1,\ldots,{\mathbf x}_n)$,
\[
\Phi({\mathbf x}_1,\ldots,{\mathbf x}_n)\geqslant \phi
\iff
\sum_{i=1}^n f({\mathbf x}_i)p_i-\phi\sum_{i=1}^n g({\mathbf x}_i)p_i \geqslant 0\iff
\sum_{i=1}^n p_i h_{\phi,f,g}({\mathbf x}_i)\geqslant 0.
\]
In particular, since $\phi^*$ is the infimum of $\Phi$, at $\phi=\phi^*$ we have $\sum_{i=1}^n p_i h_{\phi^*,f,g}({\mathbf x}_i)\geqslant 0$ for all $({\mathbf x}_1,\ldots,{\mathbf x}_n)\in \underline{\mathcal K}$, with equality at $({\mathbf x}_1^*,\ldots,{\mathbf x}_n^*)$.

The residual objective $R_{\phi^*}({\mathbf x}_1,\ldots,{\mathbf x}_n):=\sum_{i=1}^n p_i h_{\phi^*,f,g}({\mathbf x}_i)$ is separable. Since each $p_i>0$, the tuple $({\mathbf x}_1^*,\ldots,{\mathbf x}_n^*)$ minimises $R_{\phi^*}$ if and only if each coordinate minimises its own term. Thus
\[
{\mathbf x}_i^*\in \arg\min_{\mathbf x\in \overline K_i} h_{\phi^*,f,g}(\mathbf x),
\qquad i=1,\ldots,n,
\]
which is \textnormal{(E1)}.

Conversely, if a tuple satisfies \textnormal{(E1)} then $\sum_{i=1}^n p_i h_{\phi^*,f,g}({\mathbf x}_i)
\geqslant
\sum_{i=1}^n p_i h_{\phi^*,f,g}({\mathbf x}_i^*)$ for every feasible $({\mathbf x}_1,\ldots,{\mathbf x}_n)$. If also \textnormal{(E2)} holds, then $\sum_{i=1}^n p_i h_{\phi^*,f,g}({\mathbf x}_i^*)
=
\sum_{i=1}^n p_i\bigl(f({\mathbf x}_i^*)-\phi^*g({\mathbf x}_i^*)\bigr)=0$,
so $\Phi({\mathbf x}_1,\ldots,{\mathbf x}_n)\geqslant \phi^*$ for every feasible tuple, with equality at $({\mathbf x}_1^*,\ldots,{\mathbf x}_n^*)$.
This shows the fixed-point tuple $(\phi^*,{\mathbf x}_1^*,\ldots,{\mathbf x}_n^*)$ satisfying (E1) and (E2) solves \eqref{eqn_weak_opt_equiv_disc}.

Clearly, $\phi^*$ is obtained from the discrete objective in \eqref{eqn_weak_opt_equiv_disc} by using the extremal prior \eqref{eqn_analyticextremalprior}; i.e. using $
\mathbb P=\sum_{i=1}^n p_i\delta_{{\mathbf x}_i^*}$. If ${\mathbf x}_i^*\in K_i$ for every $i$, then $\mathbb P$ is feasible. In general, however,
some support points may lie in $\overline K_i\setminus K_i$. In that case, for each
$i$ choose a sequence ${\mathbf x}_{i,r}\in K_i$ such that
${\mathbf x}_{i,r}\to{\mathbf x}_i^*$, and define $\mathbb P_r:=\sum_{i=1}^n p_i\delta_{{\mathbf x}_{i,r}}.\,$
Then $\mathbb P_r$ is feasible for every $r$, and $\mathbb P_r\Rightarrow\mathbb P$.
Moreover, by continuity of $f$ and $g$,
\begin{equation}
\mathbb E_{\mathbb P_r}[f(\mathbf{X})]
=
\sum_{i=1}^n p_i f({\mathbf x}_{i,r})
\longrightarrow
\sum_{i=1}^n p_i f({\mathbf x}_i^*)
=
\mathbb E_{\mathbb P}[f(\mathbf{X})],
\label{eqn_appG_finiteExpConvf}
\end{equation}
and
\begin{equation}
\mathbb E_{\mathbb P_r}[g(\mathbf{X})]
=
\sum_{i=1}^n p_i g({\mathbf x}_{i,r})
\longrightarrow
\sum_{i=1}^n p_i g({\mathbf x}_i^*)
=
\mathbb E_{\mathbb P}[g(\mathbf{X})].
\label{eqn_appG_finiteExpConvg}
\end{equation}
Since $\,\mathbb E_{\mathbb P}[g(\mathbf{X})]
=
\sum_{i=1}^n p_i g({\mathbf x}_i^*)
\geqslant
p_j\inf_{{\mathbf x}\in\overline K_j}g({\mathbf x})
>0,\,$ the quotient map is continuous at
\((\mathbb E_{\mathbb P}[f(\mathbf{X})],\mathbb E_{\mathbb P}[g(\mathbf{X})])\). Therefore, \eqref{eqn_appG_finiteExpConvf} and \eqref{eqn_appG_finiteExpConvg} imply
\[
\frac{\mathbb E_{\mathbb P_r}[f(\mathbf{X})]}{\mathbb E_{\mathbb P_r}[g(\mathbf{X})]}
\longrightarrow
\frac{\mathbb E_{\mathbb P}[f(\mathbf{X})]}{\mathbb E_{\mathbb P}[g(\mathbf{X})]}
=
\frac{\sum_{i=1}^n p_i f({\mathbf x}_i^*)}
{\sum_{i=1}^n p_i g({\mathbf x}_i^*)}
=
\phi^*,
\]
where the last equality is \textnormal{(E2)}. Hence, $\mathbb P$ is the weak limit of feasible
priors whose objective values converge to $\phi^*$.

It remains to justify the differential refinements in \textnormal{(E1)}. Suppose first that ${\mathbf x}_i^*\in \mathrm{int}(\overline{K}_i)$. Since ${\mathbf x}_i^*$ is a local minimiser of $h_{\phi^*,f,g}$ on the open set $\mathrm{int}(\overline{K}_i)$, for any $\textbf v\in\mathbb{R}^d$, the function $\psi_{\textbf v}(t):=h_{\phi^*,f,g}({\mathbf x}_i^*+t{\textbf v})$ has a local minimum at $t=0$. That is, for sufficiently small $t$, $\psi_{\textbf v}(0)=h_{\phi^*,f,g}({\mathbf x}_i^*)\leqslant h_{\phi^*,f,g}({\mathbf x}_i^*+t{\textbf v})=\psi_{\textbf v}(t)$. Also, $t\mapsto {\mathbf x}_i^*+t{\textbf v}$ is analytic, and $h_{\phi^*,f,g}$ is analytic on an open neighbourhood of ${\mathbf x}_i^*$, hence $\psi_{\textbf v}(t)$ is analytic ($C^2$, in particular) on some open neighbourhood of $t=0$. Consequently, by the chain rule, $\psi_{\textbf v}'(0)=\nabla h_{\phi^*,f,g}({\mathbf x}^*_i)\cdot {\textbf v}=0$. Since $\textbf v$ was arbitrary,  $\nabla h_{\phi^*,f,g}({\mathbf x}^*_i)= 0$. Moreover, since $\psi_{\textbf v}(t)$ is $C^2$ and has a local minimum at $0$, $\psi_{\textbf v}''(0)\geqslant 0$. Again, by the chain rule, $\psi_{\textbf v}''(0)={\textbf v}^\top\nabla^2 h_{\phi^*,f,g}({\mathbf x}_i^*){\textbf v}\geqslant 0$. This holds for all $\textbf v$, so $\nabla^2 h_{\phi^*,f,g}({\mathbf x}_i^*)$ is positive semidefinite.  


Now assume that $\overline{K}_i$ is convex and that ${\mathbf x}_i^*\in \partial\overline{K}_i$ minimises $h_{\phi^*,f,g}$ on $\overline{K}_i$.
Then for every $\mathbf y\in \overline{K}_i$ and every $t\in[0,1]$, $(1-t){\mathbf x}_i^*+t\mathbf y\in \overline{K}_i$. Thus, the function $\varphi(t):=h_{\phi^*,f,g}\bigl((1-t){\mathbf x}_i^*+t\mathbf y\bigr)$ satisfies $\varphi(t)\geqslant \varphi(0)$ for $t\in[0,1]$.
Therefore, by the chain rule,
\[
\nabla h_{\phi^*,f,g}({\mathbf x}_i^*) ({\mathbf y}-{\mathbf x}_i^*)=\varphi'(0)=\varphi_{+}'(0)\geqslant 0
\qquad \forall\, \mathbf y\in \overline{K}_i,
\]
which is the asserted variational inequality, equivalently $-\nabla h_{\phi^*,f,g}({\mathbf x}_i^*)\in N_{\overline{K}_i}({\mathbf x}_i^*)$.

Finally, suppose that, for some neighbourhood $U_i$ of ${\mathbf x}_i^*$, the set
\[
M_i:=\{{\mathbf x}\in U_i\cap \overline{K}_i:\ h_{\phi^*,f,g}(\mathbf x)=h_{\phi^*,f,g}({\mathbf x}_i^*)\}
\]
is a $C^2$ embedded submanifold contained in $\operatorname{int}(\overline K_i)$.
Since every point ${\mathbf x}\in M_i$ is an interior local minimiser of $h_{\phi^*,f,g}$, the previous argument gives $\nabla h_{\phi^*,f,g}({\mathbf x})=0$ for all ${\mathbf x}\in M_i$. Now fix ${\mathbf x}\in M_i$ and $\mathbf w\in T_{\mathbf x}M_i$.
There exists a $C^2$ curve $\gamma:(-\varepsilon,\varepsilon)\to M_i$ such that $\gamma(0)={\mathbf x}$, $\gamma'(0)=\mathbf w$. Differentiating the identity $\nabla h_{\phi^*,f,g}(\gamma(t))=0$ at $t=0$ yields $\nabla^2 h_{\phi^*,f,g}({\mathbf x})\,\mathbf w=0$. Hence, $T_{\mathbf x}M_i\subseteq \ker \nabla^2 h_{\phi^*,f,g}({\mathbf x})$.
In the Morse--Bott case, the final statement is precisely the non-degeneracy condition along the critical manifold $M_i$.
\end{proof}

\newpage
\section{Proof of Convergence of Algorithm~\ref{alg:dinkelbach}}
\begin{proof}
For $\textbf{x}:=(\textbf{x}_1,\dots,\textbf{x}_n)\in K_1\times\cdots\times K_n$, write
\[
N(\textbf{x}):=\sum_{i=1}^n p_i f(\textbf{x}_i),\qquad
D(\textbf{x}):=\sum_{i=1}^n p_i g(\textbf{x}_i),\qquad
\phi(\mathbf x):=\frac{N(\textbf{x})}{D(\textbf{x})}.
\]
$f$ and $g$ are non-negative and continuous on the compact set $K$, hence $N$ and
$D$ are continuous on the compact set $K:=\overline{K}_1\times\cdots\times \overline{K}_n$.
Moreover, $D$ is positive since $\inf_{\overline K_j} g>0$ and $p_j>0$ for some $j$,
so there exists a $\delta>0$ such that $D(\textbf{x})\geqslant p_j\delta>0$, $\textbf{x}\in K$.
Hence $\phi(\textbf{x})$ is non-negative continuous and the infimum value $\phi^*$ in \eqref{eqn_weak_opt_equiv_disc} is attained on $K$. For a $\phi(\textbf{x})$-value $\phi\geqslant 0$, define the parametric value function
\[
F(\phi):=\min_{\textbf{x}\in K}\bigl(N(\textbf{x})-\phi D(\textbf{x})\bigr).
\]
Since $N(\textbf{x})-\phi D(\textbf{x})=\sum_{i=1}^n p_i\bigl(f(\textbf{x}_i)-\phi g(\textbf{x}_i)\bigr)$, 
step 1 of Algorithm~\ref{alg:dinkelbach} computes $\hat {\textbf{x}}^{(t)}=(\hat {\textbf{x}}^{(t)}_1,\dots,\hat {\textbf{x}}^{(t)}_n)\in K$
satisfying
\[
\hat {\textbf{x}}^{(t)}\in \arg\min_{{\textbf{x}}\in K}\bigl(N({\textbf{x}})-\hat\phi_t D({\textbf{x}})\bigr),
\]
because minimizing the sum over $K$ separates coordinatewise over the intervals $K_i$.
Therefore, $F(\hat\phi_t)=N(\hat {\textbf{x}}^{(t)})-\hat\phi_t D(\hat {\textbf{x}}^{(t)})$. In step 3 of Algorithm~\ref{alg:dinkelbach},
$\hat\phi_{t+1}=\frac{N(\hat {\textbf{x}}^{(t)})}{D(\hat {\textbf{x}}^{(t)})}$ implies 
\[
F(\hat\phi_t)=D(\hat {\mathbf x}^{(t)})(\hat\phi_{t+1}-\hat\phi_t). \tag{A1}\label{A1}
\]
The positivity of $D$ implies the sign of $F(\hat{\phi}_t)$ is the sign of $(\hat{\phi}_{t+1}-\hat{\phi}_t)$.

$F$ is monotonically decreasing: if $\phi_1\leqslant\phi_2$ then $N({\textbf{x}})-\phi_1 D({\textbf{x}})\geqslant N({\textbf{x}})-\phi_2 D({\textbf{x}})$ for every ${\textbf{x}}\in K$, so $F(\phi_1)\geqslant F(\phi_2)$. Moreover, $F$ has a root: since $F$ is the minimum of affine functions in $\phi$ it is continuous in $\phi$, so a root follows from monotonicity, the fact that $F\downarrow-\infty$ as $\phi\uparrow\infty$, $F\uparrow\infty$ as $\phi\downarrow-\infty$, and the \emph{Intermediate Value Theorem} (IVT). This root must be $\phi^*$, so is unique: for any $\phi$ and ${\textbf{x}}\in K$,
\[
N({\textbf{x}})-\phi D({\textbf{x}})=D({\textbf{x}})(\phi({\textbf{x}})-\phi)\geqslant D({\textbf{x}})(\phi^*-\phi)
\]
and $D({\textbf{x}})(\phi^*-\phi)=0$ \emph{iff} $\phi=\phi^*$ (since $D({\textbf{x}})>0$ and $\phi^*$ can be attained). 

Now generate $\{\hat\phi_t\}_{t\in\mathbb Z^{\geqslant0}}$ by Algorithm~\ref{alg:dinkelbach}. By the properties of $F$ and \eqref{A1}, and the fact that $\hat\phi_t\geqslant\phi^*$ for all $t\geqslant 1$,
\[
\hat\phi_t>\phi^* \implies F(\hat\phi_t)<F(\phi^*)=0 \implies \hat\phi_{t+1}<\hat\phi_t .
\]
Thus $\hat\phi_1,\hat\phi_2,\ldots$ decreases monotonically to a limit $L\geqslant\phi^*$. Thus, $\hat\phi_{t+1}-\hat\phi_t\longrightarrow 0$ and, from \eqref{A1}, since $D$ is continuous (thus bounded) on $K$, $F(\phi_t)\longrightarrow0$. The continuity of $F$ and the uniqueness of its root $\phi^*$ imply $L=\phi^*$. That is, $\hat\phi_t \downarrow \phi^*$ and $F(\hat\phi_t)\uparrow F(\phi^*)=0$. 

Let  $(\phi^\dagger,{\textbf{x}}^\dagger_1,\dots,{\textbf{x}}^\dagger_n)$ be any accumulation point of
$\{(\hat\phi_t,\hat {\textbf{x}}^{(t)}_1,\dots,\hat {\textbf{x}}^{(t)}_n)\}_{t\in\mathbb Z^{\geqslant 0}}$. Since $\hat\phi_t\to\phi^*$, we
must have $\phi^\dagger=\phi^*$. Passing to a convergent subsequence
$t_k\to\infty$ with
\[
(\hat\phi_{t_k},\hat {\textbf{x}}^{(t_k)}_1,\dots,\hat {\textbf{x}}^{(t_k)}_n)\to
(\phi^*,{\textbf{x}}^\dagger_1,\dots,{\textbf{x}}^\dagger_n),
\]
and using continuity of $N$ and $D$, together with
\[
N(\hat {\textbf{x}}^{(t_k)})-\hat\phi_{t_k}D(\hat {\textbf{x}}^{(t_k)})=F(\hat\phi_{t_k})\to 0,
\]
we get
\[
N({\textbf{x}}^\dagger)-\phi^*D({\textbf{x}}^\dagger)=0,
\qquad\text{where }{\mathbf x}^\dagger:=({\textbf{x}}^\dagger_1,\dots,{\textbf{x}}^\dagger_n).
\]
Since $D({\textbf{x}}^\dagger)>0$, this implies
\[
\phi({\textbf{x}}^\dagger)=\frac{N({\textbf{x}}^\dagger)}{D({\textbf{x}}^\dagger)}=\phi^*,
\]
so ${\textbf{x}}^\dagger$ is a global minimiser of \eqref{eqn_weak_opt_equiv_disc}. By Theorem~\ref{thm:global_fp_tuple}, every global minimiser
of \eqref{eqn_weak_opt_equiv_disc} satisfies the fixed-point conditions (E1) and (E2). Hence $(\phi^*,{\textbf{x}}^\dagger_1,\dots,{\textbf{x}}^\dagger_n)$ satisfies (E1) and (E2).

\end{proof}

\newpage\section{Proof of Theorem~\ref{thm_gensol}}
\begin{proof}
By Proposition~\ref{prop_cbi_transform} we can consider $\underline{\mathcal K}:=\overline{K}_1\times\cdots\times \overline{K}_n$ for the optimisation. Since $g$ is continuous on the compact set $\overline{K}_j$ and
$\inf_{\overline{K}_j}g>0$, there exists $\eta>0$ such that $g(\mathbf x_j)\geqslant \eta$ for all $\mathbf x_j\in \overline{K}_j$.
Hence, for all $\mathbf x=({\mathbf x}_1,\dots,{\mathbf x}_n)\in \underline{\mathcal K}$,
\[
\sum_{i=1}^n g({\mathbf x}_i)p_i \geqslant p_j g({\mathbf x}_j)\geqslant p_j\eta>0,
\]
so $\tilde{\phi}$ is well-defined on $\underline{\mathcal K}$.

By the \emph{Stone-Weierstrass Approximation Theorem}, there are sequences $\{\underline{f_m}\}_{m\in\mathbb N}$, $\{\underline{g_m}\}_{m\in\mathbb N}$ of analytic functions that uniformly converge to $f$ and $g$. These sequences can be made non-negative by sufficiently large perturbation constants $\varepsilon_m,\,\delta_m>0$ that vanish as $m\to\infty$. Thus, define $f_m:=\underline{f_m}+\varepsilon_m$ and $g_m:=\underline{g_m}+\delta_m$. Since $\|g_m-g\|_{\infty}\to 0$, there exists $m_0$ such that for all $m\geqslant m_0$, $\|g_m-g\|_{\infty}<\eta/2$. Therefore, for all ${\mathbf x}_j\in \overline{K}_j$ and all $m\geqslant m_0$,
\[
g_m({\mathbf x}_j)\geqslant g({\mathbf x}_j)-|g_m({\mathbf x}_j)-g({\mathbf x}_j)|\geqslant \eta-\eta/2=\eta/2,
\]
and so, for every ${\mathbf x}=({\mathbf x}_1,\dots,{\mathbf x}_n)\in \underline{\mathcal K}$ and every $m\geqslant m_0$,
\[
\sum_{i=1}^n g_m({\mathbf x}_i)p_i \geqslant p_j g_m({\mathbf x}_j)\geqslant p_j\eta/2>0.
\]
Thus, $\phi_m$ is also well-defined on $\underline{\mathcal K}$ for all sufficiently large $m$, and Theorem~\ref{thm:global_fp_tuple} applies to ($f_m$, $g_m$) pairs from the approximating sequences $\{f_m\}_{m\geqslant m_0}$ and $\{g_m\}_{m\geqslant m_0}$. In what follows, we restrict to these sequences, re-indexing over all $m\in\mathbb N$. 

Define
\[
N_m({\mathbf x}):=\sum_{i=1}^n f_m({\mathbf x}_i)p_i,
\qquad
D_m({\mathbf x}):=\sum_{i=1}^n g_m({\mathbf x}_i)p_i,
\]
and
\[
\tilde{N}({\mathbf x}):=\sum_{i=1}^n f({\mathbf x}_i)p_i,
\qquad
\tilde{D}({\mathbf x}):=\sum_{i=1}^n g({\mathbf x}_i)p_i.
\]
For ${\mathbf x}\in \underline{\mathcal K}$,
\[
\phi_m({\mathbf x})-\tilde{\phi}({\mathbf x})
=
\frac{N_m({\mathbf x})}{D_m({\mathbf x})}-\frac{\tilde{N}({\mathbf x})}{\tilde{D}({\mathbf x})}
=
\frac{(N_m({\mathbf x})-\tilde{N}({\mathbf x}))\tilde{D}({\mathbf x})+\tilde{N}({\mathbf x})(\tilde{D}({\mathbf x})-D_m({\mathbf x}))}
{D_m({\mathbf x})\tilde{D}({\mathbf x})}.
\]
Hence,
\[
|\phi_m({\mathbf x})-\tilde{\phi}({\mathbf x})|
\leqslant
\frac{|N_m({\mathbf x})-\tilde{N}({\mathbf x})|}{D_m({\mathbf x})}
+
\frac{|\tilde{N}({\mathbf x})|\,|D_m({\mathbf x})-\tilde{D}({\mathbf x})|}{D_m({\mathbf x})\tilde{D}({\mathbf x})}.
\]
Using the lower bounds on $D_m$ and $\tilde{D}$ gives
\[
|\phi_m({\mathbf x})-\tilde{\phi}({\mathbf x})|
\leqslant
\frac{2}{p_j\eta}|N_m({\mathbf x})-\tilde{N}({\mathbf x})|
+
\frac{2}{(p_j\eta)^2}|\tilde{N}({\mathbf x})|\,|D_m({\mathbf x})-\tilde{D}({\mathbf x})|.
\]
Next,
\[
|N_m({\mathbf x})-\tilde{N}({\mathbf x})|
=
\left|
\sum_{i=1}^n (f_m({\mathbf x}_i)-f({\mathbf x}_i))p_i
\right|
\leqslant
\sum_{i=1}^n p_i\|f_m-f\|_{\infty}
=
\|f_m-f\|_{\infty},
\]
and similarly
\[
|D_m({\mathbf x})-\tilde{D}({\mathbf x})|\leqslant \|g_m-g\|_{\infty}.
\]
Since $f$ is continuous on compact $K$, it is bounded; let $
M_f:=\|f\|_{\infty}$. Then, for all ${\mathbf x}\in \underline{\mathcal K}$, $|\tilde{N}({\mathbf x})|\leqslant \sum_{i=1}^n |f({\mathbf x}_i)|p_i\leqslant M_f$. Therefore, for all ${\mathbf x}\in \underline{\mathcal K}$,
\begin{equation}
|\phi_m({\mathbf x})-\tilde{\phi}({\mathbf x})|
\leqslant
\frac{2}{p_j\eta}\|f_m-f\|_{\infty}
+
\frac{2M_f}{(p_j\eta)^2}\|g_m-g\|_{\infty}.
\label{eqn_boundonphistarmphistardiff}
\end{equation}
Hence, $\sup_{{\mathbf x}\in \underline{\mathcal K}}|\phi_m({\mathbf x})-\tilde{\phi}({\mathbf x})|\to 0$, as the \acrshort*{rhs} of \eqref{eqn_boundonphistarmphistardiff} has no ${\mathbf x}$ dependence and vanishes as $m\to\infty$. So, $\phi_m\to \tilde{\phi}$ uniformly on $\underline{\mathcal K}$. Uniform convergence immediately implies convergence of the infima:
\[
|\phi_m^*-\tilde{\phi}^*|
=
\left|
\inf_{{\mathbf x}\in \underline{\mathcal K}}\phi_m({\mathbf x})-\inf_{{\mathbf x}\in \underline{\mathcal K}}\tilde{\phi}({\mathbf x})
\right|
\leqslant
\sup_{{\mathbf x}\in \underline{\mathcal K}}|\phi_m({\mathbf x})-\tilde{\phi}({\mathbf x})|
\to 0.
\]
Consequently,
\begin{equation}
\phi_m^*\to \tilde{\phi}^*\qquad\text{and}\qquad\phi_m({\mathbf x}^{(m)})\to \tilde{\phi}^*,
\label{eqn_phimtendstophistar}
\end{equation}
since $\phi_m({\mathbf x}^{(m)})=\phi_m^*$. Now, let $\mathbf x^{(m)}\in \underline{\mathcal K}$ be any minimiser of $\phi_m$. Since $K$ is compact, the sequence
$\{{\mathbf x}^{(m)}\}_{m\geqslant 1}$ has an accumulation point $\tilde{{\mathbf x}}\in \underline{\mathcal K}$; pass to a subsequence,
not relabelled, such that ${\mathbf x}^{(m)}\to \tilde{{\mathbf x}}$. Since $\tilde{\phi}$ is continuous on $\underline{\mathcal K}$,
\begin{equation}
\tilde{\phi}({\mathbf x}^{(m)})\to \tilde{\phi}(\tilde{{\mathbf x}}).
\label{eqn_xmtendstotildex}
\end{equation}
Also,
\begin{equation}
|\phi_m({\mathbf x}^{(m)})-\tilde{\phi}({\mathbf x}^{(m)})|
\leqslant
\sup_{{\mathbf x}\in \underline{\mathcal K}}|\phi_m({\mathbf x})-\tilde{\phi}({\mathbf x})|
\to 0,
\label{eqn_phimtendstophitilde}
\end{equation}
Thus, by \eqref{eqn_phimtendstophistar}, \eqref{eqn_phimtendstophitilde}, and $|\tilde{\phi}({\mathbf x}^{(m)})-\tilde{\phi}^*|\leqslant|\tilde{\phi}({\mathbf x}^{(m)})-\phi_m({\mathbf x}^{(m)})|+|\phi_m({\mathbf x}^{(m)})-\tilde{\phi}^*|$,
\begin{equation}
\tilde{\phi}({\mathbf x}^{(m)})\to \tilde{\phi}^*.
\label{eqn_phitildetendstophistar}
\end{equation}
Furthermore, 
\begin{equation}
\tilde{\phi}(\tilde{{\mathbf x}})=\tilde{\phi}^*,
\label{eqn_continuousfg_E1}
\end{equation}
which follows from \eqref{eqn_xmtendstotildex}, \eqref{eqn_phitildetendstophistar}, and $|\tilde{\phi}(\tilde{{\mathbf x}})-\tilde{\phi}^*|\leqslant|\tilde{\phi}(\tilde{{\mathbf x}})-\tilde{\phi}({\mathbf x}^{(m)})|+|\tilde{\phi}({\mathbf x}^{(m)})-\tilde{\phi}^*|$.

So, every accumulation point of minimisers of the analytic problems minimises the limiting
continuous problem. 
This justifies (E2).

From \eqref{eqn_continuousfg_E1}, the tuple $(\tilde{\phi}^*,\tilde{\mathbf x})$ satisfies $\sum_{i=1}^{n}p_i(f(\tilde{\mathbf x}_i)-\tilde{\phi}^*g(\tilde{\mathbf x}_i))=\sum_{i=1}^n p_i h_{\tilde\phi^*,f,g}(\tilde{\mathbf x_i})=0$. If there exists ${\mathbf y}\in\overline K_i$ for some $i$ such that $h_{\tilde\phi^*,f,g}(\mathbf y)<h_{\tilde\phi^*,f,g}(\tilde{\mathbf x}_i)$, then replacing $\tilde{\mathbf x_i}$ with ${\mathbf y}$ in the zero sum makes the sum negative: this means $(\tilde {\mathbf x}_1,\ldots,{\mathbf y},\ldots,\tilde {\mathbf x}_n)$ gives a smaller objective function value than $(\tilde {\mathbf x}_1,\ldots,\tilde {\mathbf x}_i,\ldots,\tilde {\mathbf x}_n)$, contradicting $\tilde\phi^*$ being the infimum. Thus, (E1) also holds.

The sequence of priors $\mathbb P_{m_k}$ converges weakly to $\mathbb P$ in Theorem~\ref{thm_gensol} because, for any bounded continuous function $\psi:K\to [0,\infty)$,
\[
\mathbb E_{\mathbb P_{m_k}}[\psi({\mathbf X})]=\sum_{i=1}^n p_i\psi({\mathbf x}_i^{(m_k)})\rightarrow \sum_{i=1}^n p_i\psi(\tilde {\mathbf x}_i)=\mathbb E_{\mathbb P}[\psi(\mathbf X)]\,.
\]
Finally, each minimiser ${\mathbf x}^{(m)}$ determines, via Theorem~\ref{thm:global_fp_tuple}, an extremal prior $\mathbb P_m$
supported on the coordinates of ${\mathbf x}^{(m)}$.
Any accumulation point of $({\mathbf x}^{(m)})$ therefore yields, by the same support construction,
an extremal prior for the limiting problem with objective $\tilde{\phi}$. Different subsequences
may converge to different minimisers of $\tilde{\phi}$, so the limiting extremal prior need
not be unique.

\end{proof}
\newpage
\section{Proof of Proposition~\ref{prop:semicont_refinement}}
\begin{proof}
Let $\mathbb P$ be any distribution satisfying the constraints of
\eqref{eqn_genCBIprob}.
Define
\[
p_{i,\ell}:=\mathbb P(\mathbf X\in K_{i,\ell}),\qquad
q_{i}:=\mathbb P(\mathbf X\in D_i).
\]
The $K_{i,\ell}$ and $D_i$ sets partition $K_i$, so $\sum_{\ell=1}^{r_i}p_{i,\ell}+q_{i}
=
\mathbb P(K_i)
=
p_i$. Conversely, any nonnegative collection $\{p_{i,\ell},q_{i}\}$
satisfying these constraints determines a subclass of
distributions in $\mathcal D$ whose masses on the refined partition coincide
with these quantities. Hence,
\begin{equation}
\inf_{\mathbb P\in \mathcal D \atop \mathbb P(K_i)=p_i}
\frac{\mathbb E[f(\mathbf X)]}{\mathbb E[g(\mathbf X)]}
=
\inf_{\substack{p_{i,\ell},q_{i}\geqslant0\\
\sum_{\ell}p_{i,\ell}+q_{i}=p_i}}
\;
\inf_{\mathbb P\in \mathcal D(p_{i,\ell},q_{i})}
\frac{\mathbb E[f(\mathbf X)]}{\mathbb E[g(\mathbf X)]},
\label{eqn_proofRefineMorenoCano}
\end{equation}
where $\mathcal D(p_{i,\ell},q_{i})$ denotes the set of distributions
with these refined masses.

By definition of $\mathcal D(p_{i,\ell},q_{i})$, the refinement of the partition of $K$ consists of Borel-measurable, hence $\mathbb P$-measurable, sets (for all $\mathbb P\in \mathcal D(p_{i,\ell},q_{i})$). Since $f, g\in\mathcal L(\mathbb P)$ for all $\mathbb P\in\mathcal D(p_{i,\ell},q_{i})$, Proposition~\ref{prop_cbi_transform} yields
\begin{equation}
\inf_{\mathbb P\in \mathcal D(p_{i,\ell},q_{i})}
\frac{\mathbb E[f(\mathbf X)]}{\mathbb E[g(\mathbf X)]}
=
\inf_{{\mathbf x}_{i,\ell}\in K_{i,\ell},\,\mathbf d_i\in D_i}
\frac{
\sum_{i=1}^n\left(
\sum_{\ell=1}^{r_i} f({\mathbf x}_{i,\ell})\,p_{i,\ell}
+
f(\mathbf d_i)\,q_{i}
\right)
}{
\sum_{i=1}^n\left(
\sum_{\ell=1}^{r_i} g({\mathbf x}_{i,\ell})\,p_{i,\ell}
+
g(\mathbf d_i)\,q_{i}
\right)
}.
\label{eqn_proofSubRefineMorenoCano}
\end{equation}
Using \eqref{eqn_proofSubRefineMorenoCano} in \eqref{eqn_proofRefineMorenoCano} gives \eqref{eqn_v2_morenoetalCBIprob}.
\end{proof}
\newpage
\section{Proof of Theorem~\ref{thm_refinedgensol}}
\label{app_TheoremBoundedPiecewiseContinuous}
\begin{proof}
By Proposition~\ref{prop:semicont_refinement}, the original problem
\eqref{eqn_genCBIprob} is equivalent to the refined problem \eqref{eqn_v2_morenoetalCBIprob},
\begin{equation}
\phi^*
=
\inf_{\substack{p_{i,\ell}\geqslant 0,\ q_i\geqslant 0\,\forall\,i,\ell\\ \sum_{\ell=1}^{r_i}p_{i,\ell}+q_i=p_i\,\forall\,i}}
\ \inf_{\substack{{\mathbf x}_{i,\ell}\in K_{i,\ell}\,\forall\,i,\ell\\ {\mathbf d}_i\in D_i\,\forall\,i}}
\frac{
\sum_{i=1}^n \left(
\sum_{\ell=1}^{r_i} f({\mathbf x}_{i,\ell})\,p_{i,\ell}
+
f({\mathbf d}_i)\,q_i
\right)}
{
\sum_{i=1}^n \left(
\sum_{\ell=1}^{r_i} g({\mathbf x}_{i,\ell})\,p_{i,\ell}
+
g({\mathbf d}_i)\,q_i
\right)} .
\label{eq:refined-proof-problem}
\end{equation}
In what follows, terms involving $D_i$ should be considered omitted when $D_i=\varnothing$;
equivalently, in that case we set $q_i=0$.

The inner optimisation in \eqref{eq:refined-proof-problem} is separable and so can be solved via Dinkelbach iteration. Indeed, for potential extremal objective-function value $\varphi\geqslant 0$, the Dinkelbach value for $\varphi$ is \begin{align}F(\varphi)&:=\inf_{\substack{p_{i,\ell}\geqslant 0,\ q_i\geqslant 0\,\forall\,i,\ell\\ \sum_{\ell=1}^{r_i}p_{i,\ell}+q_i=p_i\,\forall\,i}}
\ \inf_{\substack{{\mathbf x}_{i,\ell}\in K_{i,\ell}\,\forall\,i,\ell\\ {\mathbf d}_i\in D_i\,\forall\,i}}\,\,\sum_{i=1}^{n}\biggl(\sum_{\ell=1}^{r_i}\bigl(p_{i,\ell}  h_{\varphi,f,g}(\mathbf{x}_{i,\ell})\bigr)\,+\,q_{i}  h_{\varphi,f,g}(\mathbf{d}_i)\biggr) \nonumber\\
&=\inf_{\substack{p_{i,\ell}\geqslant 0,\ q_i\geqslant 0\,\forall\,i,\ell\\ \sum_{\ell=1}^{r_i}p_{i,\ell}+q_i=p_i\,\forall\,i}}
\ \sum_{i=1}^{n}\biggl(\sum_{\ell=1}^{r_i}\bigl(p_{i,\ell} \inf_{\substack{{\mathbf x}_{i,\ell}\in K_{i,\ell}}} h_{\varphi,f,g}(\mathbf{x}_{i,\ell})\bigr)\,+\,q_{i} \inf_{\substack{\mathbf{d}_i\in D_i}} h_{\varphi,f,g}(\mathbf{d}_i)\biggr)
,\label{eqn_appK_Dinkelbachvalue1}\end{align}
where $h_{\varphi,f,g}$ is the Dinkelbach difference from Definition~\ref{def:Dinkelbachdifference}. By the definition and continuity of the extensions $\tilde{f}_{i,\ell}$, $\tilde{g}_{i,\ell}$, $\tilde{f}_{i}$, $\tilde{g}_{i}$, over the compact closures $\overline{K}_{i,\ell}$, $\overline{D}_i$, the Dinkelbach value \eqref{eqn_appK_Dinkelbachvalue1} is equivalently
\begin{equation}F(\varphi)=\inf_{\substack{p_{i,\ell}\geqslant 0,\ q_i\geqslant 0\,\forall\,i,\ell\\ \sum_{\ell=1}^{r_i}p_{i,\ell}+q_i=p_i\,\forall\,i}}
\ \sum_{i=1}^{n}\biggl(\sum_{\ell=1}^{r_i}\bigl(p_{i,\ell} \min_{\substack{{\mathbf x}_{i,\ell}\in \overline{K}_{i,\ell}}} h_{\varphi,\tilde{f}_{i,\ell},\tilde{g}_{i,\ell}}(\mathbf{x}_{i,\ell})\bigr)\,+\,q_{i} \min_{\substack{\mathbf{d}_i\in \overline{D}_i}} h_{\varphi,\tilde{f}_i,\tilde{g}_i}(\mathbf{d}_i)\biggr).\label{eqn_appK_Dinkelbachvalue2}\end{equation}
Note that the expressions being summed over for each $i$ form a convex combination. Hence, the minimum for each $i$ is given by placing all probability mass at a support location $\mathbf{x}_i(\varphi)\in\overline{K}_i\cup\overline{D}_i$ that gives the minimum value among the Dinkelbach minima defined over $\overline{K}_{i,\ell}$, $\overline{D}_i$ for fixed $i$ and $\ell=1,\ldots,r_i$. If there is at least  one such $\mathbf{x}_i(\varphi)$ in $\overline{K}_{i,\ell}$ then set $p_{i,\ell}=p_i$; otherwise, there is at least one such $\mathbf{x}_i(\varphi)$ in $\overline{D}_i$ so set $q_i=p_i$. Thus, \eqref{eqn_appK_Dinkelbachvalue2} becomes
\begin{equation}
F(\varphi)=\sum_{i=1}^n p_i m_i(\varphi),
\label{eqn_appK_Dinkelbachvalue3}
\end{equation}
where 
\begin{align*}
m_{i,\ell}(\varphi)
&:=
\min_{\mathbf{x}\in \overline K_{i,\ell}}
h_{\varphi,\tilde f_{i,\ell},\tilde g_{i,\ell}}(\mathbf{x}),
\qquad \ell=1,\ldots,r_i, \\
m_{i,D}(\varphi)
&:=
\begin{cases}
\displaystyle
\min_{\mathbf{x}\in \overline D_i}
h_{\varphi,\tilde f_i,\tilde g_i}(\mathbf{x}),
& D_i\neq\varnothing,\\[1ex]
+\infty, & D_i=\varnothing .
\end{cases} \\
m_i(\varphi)
&:=
\min\Bigl\{
m_{i,1}(\varphi),\ldots,m_{i,r_i}(\varphi),
m_{i,D}(\varphi)
\Bigr\}.
\end{align*}
The definition of $F(\varphi)$ in \eqref{eqn_appK_Dinkelbachvalue1} is an infimum over affine functions of $\varphi$, where each affine function is non-increasing in $\varphi$. Thus, $F(\varphi)$ is a monotonically non-increasing function of $\varphi$. Using the definition of $F(\varphi)$ and $g\geqslant\eta>0$ for some $\eta\in\mathbb R$ on some $K_j$ with $p_j>0$ (so, uniformly bounded positive objective function denominator), if $\phi^*< \varphi$ then there exists some value of the objective in \eqref{eq:refined-proof-problem} that gives a value smaller than $\varphi$, so $F(\varphi)<0$. On the other hand, if $\varphi<\phi^*$, no value of the objective function can be smaller than $\varphi$, so $F(\varphi)>0$. Also note that $F(\varphi)$ is continuous because each $m_{i,\ell}(\varphi)$ is continuous as a minimum (over a compact set) of an affine-in-$\varphi$ continuous function and $m_i(\varphi)$ is continuous as the minimum of continuous functions. Hence, the continuity and monotonicity of $F(\varphi)$, and the continuity of the $f$, $g$ extensions over compact closures, imply   
\begin{equation}
\sum_{i=1}^n p_i\,m_i(\phi^*)=F(\phi^*)=0.
\label{eq:root-equation}
\end{equation}
Since each related extremal support point $\mathbf{x}^*_i:=\mathbf{x}_i(\phi^*)$ satisfies $\mathbf{x}^*_i\in\overline{K}_{i,\ell}$ (for some $\ell$) or $\mathbf{x}^*_i\in\overline{D}_i$, then $\mathbf{x}^*_i\in M_i(\phi^*)$ by the definition of $M_i(\phi^*)$, for all $i$. In particular, either 
${\mathbf x}_i^*\in
\underset{{\mathbf x}\in\overline K_{i,\ell}}{\mathrm{argmin}}\;
h_{\phi^*,\tilde f_{i,\ell},\tilde g_{i,\ell}}({\mathbf x})$ for some $\ell,\,$  
or ${\mathbf x}_i^*\in
\underset{{\mathbf x}\in\overline D_i}{\mathrm{argmin}}\;
h_{\phi^*,\tilde f_i,\tilde g_i}({\mathbf x})$.
Thus, \textnormal{(E1)} holds. Define
\[
(\overset{\hspace{0.2em}*}{f}({\mathbf x}_i),\overset{\hspace{0.2em}*}{g}({\mathbf x}_i))
=
\begin{cases}
(\tilde f_{i,\ell}({\mathbf x}_i),\tilde g_{i,\ell}({\mathbf x}_i)),
& \text{if }{\mathbf x}_i\in
\underset{{\mathbf x}\in\overline K_{i,\ell}}{\mathrm{argmin}}\;
h_{\phi^*,\tilde f_{i,\ell},\tilde g_{i,\ell}}({\mathbf x})
\subseteq M_{i,0}(\phi^*),\\[4pt]
(\tilde f_i({\mathbf x}_i),\tilde g_i({\mathbf x}_i)),
& \text{if }{\mathbf x}_i\in M_{i,1}(\phi^*)\setminus M_{i,0}(\phi^*).
\end{cases}
\]
With a choice of extremal support ${\mathbf x}_i^*$ that satisfies $m_i(\phi^*)
=
\overset{\hspace{0.2em}*}{f}({\mathbf x}_i^*)
-
\phi^*\,\overset{\hspace{0.2em}*}{g}({\mathbf x}_i^*)$, it follows from \eqref{eq:root-equation} that $0
=
\sum_{i=1}^n p_i
\Bigl(
\overset{\hspace{0.2em}*}{f}({\mathbf x}_i^*)
-
\phi^*\,\overset{\hspace{0.2em}*}{g}({\mathbf x}_i^*)
\Bigr)$. Hence,
\[
\phi^*
=
\frac{\sum_{i=1}^n \overset{\hspace{0.2em}*}{f}({\mathbf x}_i^*)\,p_i}
{\sum_{i=1}^n \overset{\hspace{0.2em}*}{g}({\mathbf x}_i^*)\,p_i}.
\]
This is \textnormal{(E2)}.

Consider $\mathbb P:=\sum_{i=1}^n p_i\,\delta_{{\mathbf x}_i^*}$. Since ${\mathbf x}_i^*\in \overline K_i$, this is a probability measure of the required form
\eqref{eqn_semicontprior}. It remains to show that it is the weak limit of feasible priors
for \eqref{eqn_v2_morenoetalCBIprob}, and that the objective values along those feasible
priors converge to $\phi^*$.

Fix $i$. Since ${\mathbf x}_i^*\in \overline K_i$, there exists a sequence
$\{{\mathbf z}_i^{(m)}\}_{m\geqslant 1}\subseteq K_i$ such that
${\mathbf z}_i^{(m)}\to {\mathbf x}_i^*$.
We choose this sequence more precisely as follows.

\begin{itemize}
\item
If ${\mathbf x}_i^*\in M_{i,1}(\phi^*)\setminus M_{i,0}(\phi^*)\subseteq \overline D_i$,
choose ${\mathbf z}_i^{(m)}\in D_i$ with ${\mathbf z}_i^{(m)}\to {\mathbf x}_i^*$.

\item
If ${\mathbf x}_i^*\in M_{i,0}(\phi^*)$, choose $\ell$ such that
\[
{\mathbf x}_i^*\in
\underset{{\mathbf x}\in\overline K_{i,\ell}}{\mathrm{argmin}}\;
h_{\phi^*,\tilde f_{i,\ell},\tilde g_{i,\ell}}({\mathbf x}),
\]
and choose ${\mathbf z}_i^{(m)}\in K_{i,\ell}$ with
${\mathbf z}_i^{(m)}\to {\mathbf x}_i^*$.
\end{itemize}

Define $\mathbb P_m:=\sum_{i=1}^n p_i\,\delta_{{\mathbf z}_i^{(m)}}$. Each $\mathbb P_m$ is feasible for \eqref{eqn_v2_morenoetalCBIprob}, because
${\mathbf z}_i^{(m)}\in K_i$ for every $i$, so $\mathbb P_m(K_i)=p_i$.
Also, since ${\mathbf z}_i^{(m)}\to {\mathbf x}_i^*$ for each $i$, one has weak convergence $\mathbb P_m \Rightarrow \mathbb P$. By construction of the sequences and by continuity of the relevant extensions,
\[
f({\mathbf z}_i^{(m)})\to \overset{\hspace{0.2em}*}{f}({\mathbf x}_i^*),
\qquad
g({\mathbf z}_i^{(m)})\to \overset{\hspace{0.2em}*}{g}({\mathbf x}_i^*).
\]
Thus,
\[
\frac{\mathbb E_{\mathbb P_m}[f(\mathbf X)]}{\mathbb  E_{\mathbb P_m}[g(\mathbf X)]}
=
\frac{\sum_{i=1}^n f({\mathbf z}_i^{(m)})\,p_i}
{\sum_{i=1}^n g({\mathbf z}_i^{(m)})\,p_i}
\longrightarrow
\frac{\sum_{i=1}^n \overset{\hspace{0.2em}*}{f}({\mathbf x}_i^*)\,p_i}
{\sum_{i=1}^n \overset{\hspace{0.2em}*}{g}({\mathbf x}_i^*)\,p_i}
=
\phi^*.
\]
Consequently, the infimum is obtained in the limit by a sequence of feasible priors whose weak
limit is $\mathbb P$.

Fix $i$ and $\ell$. On the compact sets $\overline K_{i,\ell}$ and $\overline{D}_i$, the functions
$\tilde f_{i,\ell}$, $\tilde g_{i,\ell}$, $\tilde{f}_i$, $\tilde{g}_i$ are continuous. Moreover, over the compact $\overline{K}_{i,\ell}$ and $\overline{D}_i$ sets, each $h_{\varphi,\tilde{f}_{i,\ell},\tilde{g}_{i,\ell}}$ and $h_{\varphi,\tilde{f}_{i},\tilde{g}_{i}}$ is bounded (because the $\tilde{f}$ and $\tilde{g}$ extensions are continuous), and so the Dinkelbach differences can be made non-negative by affine translation by constants. Thus, global minimisers of these Dinkelbach differences can be obtained via Theorem~\ref{thm_gensol}: using
analytic approximating sequences for $\tilde f_{i,\ell}$, $\tilde g_{i,\ell}$, $\tilde{f}_i$, and $\tilde{g}_i$ to obtain global minimisers of intermediate approximate problems, then
passing to accumulation points of these intermediate minimisers. These accumulation points are minimisers of $h_{\phi,\tilde f_{i,\ell},\tilde g_{i,\ell}}\,$ on $\,\overline K_{i,\ell}$ and $h_{\phi,\tilde f_{i},\tilde g_{i,}}\,$ on $\,\overline D_{i}$ that are either:
\begin{enumerate}
\item
an interior point of $\overline{K}_{i,\ell}$;
\item
a boundary point of $\overline{K}_{i,\ell}$;
\item an interior point of $\overline{D}_i$;
\item a boundary point of $\overline{D}_i$.
\end{enumerate}
The sets $\mathcal C_{\overline{K}_{i,\ell}}(\phi)$ and $\mathcal C_{\overline{D}_{i}}(\phi)$ in Theorem~\ref{thm_refinedgensol}'s statement are the collections of all such accumulation points. Therefore, if ${\mathbf x}_i^*\in M_{i,0}(\phi^*)$ is from such an approximating sequence, then ${\mathbf x}_i^*\in \bigcup_{\ell=1}^{r_i}\mathcal C_{\overline{K}_{i,\ell}}(\phi^*)$. If instead ${\mathbf x}_i^*\in M_{i,1}(\phi^*)\setminus M_{i,0}(\phi^*)$ is from an approximating sequence, then
${\mathbf x}_i^*\in \mathcal C_{\overline{D}_{i}}(\phi^*)$.
Hence, ${\mathbf x}_i^*\in
\Bigl(\bigcup_{\ell=1}^{r_i}\mathcal C_{\overline{K}_{i,\ell}}(\phi^*)\Bigr)\cup \mathcal C_{\overline{D}_{i}}(\phi^*)
=:\mathcal C_i(\phi^*)$ when ${\mathbf x}_i^*$ is from an approximating sequence.

Since the analytic-approximation argument applies on every compact branch
$\overline K_{i,\ell}$ and $\overline D_i$, each branch minimum
$m_{i,\ell}(\phi^*)$ and $m_{i,D}(\phi^*)$ is attained by at least one
minimiser belonging to $\mathcal C_{\overline K_{i,\ell}}(\phi^*)$ or
$\mathcal C_{\overline D_i}(\phi^*)$, respectively. Since $m_i(\phi^*)$ is the
minimum of finitely many such branch minima, $m_i(\phi^*)$ is attained by
some $\mathbf{x}_i^*\in \mathcal C_i(\phi^*)$. Choosing the extremal support points
$\mathbf{x}_i^*$ in this way preserves the (\textnormal{E1}) fixed-point condition $\mathbf{x}_i^*\in M_i(\phi^*)$ and $m_i(\phi^*) = {}\overset{\hspace{0.2em}*}{f}({\mathbf x}_i^*)-\phi^*{}\overset{\hspace{0.2em}*}{g}({\mathbf x}_i^*)$. Therefore, $F(\phi^*)=\sum_{i=1}^n p_i m_i(\phi^*)=0$ yields the (\textnormal{E2})
fixed-point equation
\[
\phi^*
=
\frac{\sum_{i=1}^n {}\overset{\hspace{0.2em}*}{f}({\mathbf x}_i^*)p_i}
{\sum_{i=1}^n {}\overset{\hspace{0.2em}*}{g}({\mathbf x}_i^*)p_i}.
\]

Clearly $\mathcal C_i(\phi^*)\subseteq M_i(\phi^*)$. To see the reverse inclusion $M_i(\phi^*)\subseteq\mathcal C_i(\phi^*) $, choose $\mathbf{x}^*\in M_i(\phi^*)$. Then either $\mathbf{x}^*\in M_{i,0}(\phi^*)$ or $\mathbf{x}^*\in M_{i,1}(\phi^*)$. Suppose $\mathbf{x}^*\in M_{i,0}(\phi^*)$ w.l.o.g.---the following argument can be applied when $\mathbf{x}^*\in M_{i,1}(\phi^*)$. So $\mathbf{x}^*\in\overline{K}_{i,\ell}$ for some $\ell$. Now, for any relatively open neighbourhood $U_m$ of
$\mathbf{x}^*$ (so $U_m\subset \overline{K}_{i,\ell}$), since $\overline{K}_{i,\ell}\setminus U_m$ is compact
and disjoint from $\mathbf{x}^*$, there exists $\rho_{U_m}>0$ such that
$\|\mathbf{x}-\mathbf{x}^*\|\geqslant \rho_{U_m}$ for all $\mathbf{x}\in \overline{K}_{i,\ell}\setminus U_m$. Let $h=h_{\phi^*,\tilde{f}_{i,\ell},\tilde{g}_{i,\ell}}$. Since $\mathbf{x}^*$ minimises $h$ on $ \overline{K}_{i,\ell}$, we have
\[
h(\mathbf{x})+\eta_m\|\mathbf{x}-\mathbf{x}^*\|^2
\geqslant h(\mathbf{x}^*)+\eta_m\rho_{U_m}^2
\qquad (\mathbf{x}\in \overline{K}_{i,\ell}\setminus U_m),
\]
for every $\eta_m>0$. Thus, the perturbed Dinkelbach difference on the \emph{lhs} of the inequality
has its unique global minimiser at $\mathbf{x}^*$, and hence in $U_m$:
the smallest value of $h$ is $h(\mathbf{x}^*)$, while the smallest value of 
$\|\mathbf{x}-\mathbf{x}^*\|^2$ is $0$, which is attained only at
$\mathbf{x}=\mathbf{x}^*$. Moreover, on
$ \overline{K}_{i,\ell}\setminus U_m$, the lower bound
$\|\mathbf{x}-\mathbf{x}^*\|\geqslant \rho_{U_m}$ gives the positive gap
$\eta_m\rho_{U_m}^2$. Hence, any Dinkelbach difference uniformly within, say, $\eta_m\rho_{U_m}^2/3$ of the perturbed Dinkelbach difference preserves this
strict separation and has all its global minimisers in $U_m$. Choosing
non-negative analytic approximants to
$\tilde{f}_{i,\ell}+\eta_m\|\cdot-\mathbf{x}^*\|^2$ and $\tilde{g}_{i,\ell}$ sufficiently close uniformly
ensures that the corresponding analytic Dinkelbach difference has this
property. As $m\to\infty$, taking a shrinking neighbourhood basis $U_m\downarrow\{\mathbf{x}^*\}$ with $\eta_m\downarrow 0$, produces analytic approximating Dinkelbach minimisers converging to $\mathbf{x}^*$. Hence, $\mathbf{x}^*\in\mathcal C_{\overline{K}_{i,\ell}}\subseteq\mathcal C_i(\phi^*)$. An analogous argument shows $\mathbf{x}^*\in M_{i,1}(\phi^*)$ implies $\mathbf{x}^*\in\mathcal C_{\overline{D}_{i}}\subseteq\mathcal C_i(\phi^*)$. We conclude $M_i(\phi^*)\subseteq\mathcal C_i(\phi^*) $, so $M_i(\phi^*)=\mathcal C_i(\phi^*)$.
\end{proof}

\newpage
\section{Proof of Theorem~\ref{thm:regulated_extension}}
\label{app_theoremboundedregulated}
\begin{proof}
Since $f,g\in \overline{\mathcal R}_{2.13}^{\|\cdot\|_\infty}$, choose any sequences
$\{f_m\},\{g_m\}\subset \mathcal R_{2.13}$ such that $\|f_m-f\|_\infty\to 0$, $\|g_m-g\|_\infty\to 0$.
As uniform limits of bounded functions, $f$ and $g$ are bounded. Note that, for $\eta:=\inf_{\mathbf{x}\in K_j}g(\mathbf{x})>0$, the uniform convergence of $g_m$ implies there exists $m_0$ such that $\,\inf_{\mathbf{x}\in K_j}g_m(\mathbf{x})\geqslant\eta/2\,$ for all $m\geqslant m_0$. Restrict the rest of the proof to all such $m$.

Fix $m$. Since $f_m\in\mathcal R_{2.13}$, for each $i$, there is a partition of $K_i$ associated with $f_m$ into disjoint Borel subsets:
\[
K_i=\Bigl(\bigsqcup_{a=1}^{r_i^{(m)}} A_{i,a}^{(m)}\Bigr)\sqcup E_i^{(m)},
\]
such that $f_m|_{A_{i,a}^{(m)}}$ admits a continuous extension
$\widetilde f_{i,a}^{(m)}$ to $\overline{A_{i,a}^{(m)}}$, and
$f_m|_{E_i^{(m)}}$ admits a continuous extension $\widetilde f_{i,0}^{(m)}$
to $\overline{E_i^{(m)}}$. Likewise $g_m$ has an associated partition of $K_i$:
\[
K_i=\Bigl(\bigsqcup_{b=1}^{s_i^{(m)}} B_{i,b}^{(m)}\Bigr)\sqcup F_i^{(m)},
\]
such that $g_m|_{B_{i,b}^{(m)}}$ admits a continuous extension
$\widetilde g_{i,b}^{(m)}$ to $\overline{B_{i,b}^{(m)}}$, and
$g_m|_{F_i^{(m)}}$ admits a continuous extension $\widetilde g_{i,0}^{(m)}$
to $\overline{F_i^{(m)}}$. Set $\,A_{i,0}^{(m)}:=E_i^{(m)}$ and $\,B_{i,0}^{(m)}:=F_i^{(m)}$. Define the common refinement $S_{i,a,b}^{(m)}:=A_{i,a}^{(m)}\cap B_{i,b}^{(m)}$ for 
$a=0,\dots,r_i^{(m)}$, $b=0,\dots,s_i^{(m)}$. The nonempty sets among these form a finite Borel partition of $K_i$. On each
$S_{i,a,b}^{(m)}$, the restriction of $f_m$ admits a continuous extension to
$\overline{S_{i,a,b}^{(m)}}$, namely the restriction of the relevant
$\widetilde f_{i,a}^{(m)}$; similarly, the restriction of $g_m$ admits a continuous
extension to $\overline{S_{i,a,b}^{(m)}}$, namely the restriction of the relevant
$\widetilde g_{i,b}^{(m)}$. Hence, after relabelling the non-empty
$S_{i,a,b}^{(m)}$, the pair $(f_m,g_m)$ satisfies the hypotheses of
Proposition~\ref{prop:semicont_refinement} and Theorem~\ref{thm_refinedgensol}. Therefore, for each $m$, there exists an extremal prior
\[
\mathbb P_m=\sum_{i=1}^n p_i\,\delta_{{\mathbf x}_i^{(m)}},
\qquad {\mathbf x}_i^{(m)}\in \overline K_i,
\]
and a sequence of feasible priors in $\mathcal D$ whose weak limit is
$\mathbb P_m$ and whose objective values converge to $\phi_m^*$.
This proves \textnormal{(ii)}.

We next prove \textnormal{(i)}. For every $\mathbb P\in \mathcal D$, ${\mathbb E}_{\mathbb P}[g(\textnormal{\textbf{X}})]\geqslant p_j\eta\,$ and $\,{\mathbb E}_{\mathbb P}[g_m(\textnormal{\textbf{X}})]\geqslant p_j\eta/2$. Also, by the triangle inequality and the definition of the uniform norm,
\[
|{\mathbb E}_{\mathbb P}[f_m(\textnormal{\textbf{X}})]-{\mathbb E}_{\mathbb P}[f(\textnormal{\textbf{X}})]|
\leqslant 
{\mathbb E}_{\mathbb P}[|f_m(\textnormal{\textbf{X}})-f(\textnormal{\textbf{X}})|] 
\leqslant 
{\mathbb E}_{\mathbb P}[\|f_m-f\|_\infty]
=\|f_m-f\|_\infty.
\]
Similarly, $|{\mathbb E}_{\mathbb P}[g_m(\textnormal{\textbf{X}})]-{\mathbb E}_{\mathbb P}[g(\textnormal{\textbf{X}})]|
 \leqslant \|g_m-g\|_\infty$. Since $f$ is bounded and non-negative, $0\leqslant {\mathbb E}_{\mathbb P}[f(\textnormal{\textbf{X}})]\leqslant \|f\|_\infty$. Therefore, for $m\geqslant m_0$,
\begin{align*}
\left|
\frac{{\mathbb E}_{\mathbb P}[f_m(\textnormal{\textbf{X}})]}{{\mathbb E}_{\mathbb P}[g_m(\textnormal{\textbf{X}})]}
-
\frac{{\mathbb E}_{\mathbb P}[f(\textnormal{\textbf{X}})]}{{\mathbb E}_{\mathbb P}[g(\textnormal{\textbf{X}})]}
\right|
&\leqslant
\frac{|{\mathbb E}_{\mathbb P}[f_m(\textnormal{\textbf{X}})]-{\mathbb E}_{\mathbb P}[f(\textnormal{\textbf{X}})]|}{{\mathbb E}_{\mathbb P}[g_m(\textnormal{\textbf{X}})]}
\\
&\quad+
\frac{{\mathbb E}_{\mathbb P}[f(\textnormal{\textbf{X}})]\,|{\mathbb E}_{\mathbb P}[g_m(\textnormal{\textbf{X}})]-{\mathbb E}_{\mathbb P}[g(\textnormal{\textbf{X}})]|}
     {{\mathbb E}_{\mathbb P}[g_m(\textnormal{\textbf{X}})]{\mathbb E}_{\mathbb P}[g(\textnormal{\textbf{X}})]}
\\
&\leqslant
\frac{2}{p_j\eta}\,\|f_m-f\|_\infty
+
\frac{2\|f\|_\infty}{p_j^2\eta^2}\,\|g_m-g\|_\infty .
\end{align*}
Taking the supremum over $\mathbb P\in \mathcal D$, $\sup_{\mathbb P\in \mathcal D}
\left|
\frac{{\mathbb E}_{\mathbb P}[f_m(\textnormal{\textbf{X}})]}{{\mathbb E}_{\mathbb P}[g_m(\textnormal{\textbf{X}})]}
-
\frac{{\mathbb E}_{\mathbb P}[f(\textnormal{\textbf{X}})]}{{\mathbb E}_{\mathbb P}[g(\textnormal{\textbf{X}})]}
\right|
\to 0\,$ as $\,m\to\infty$.
Since
\[
\phi_m^*=\inf_{\mathbb P\in \mathcal D}\frac{{\mathbb E}_{\mathbb P}[f_m(\textnormal{\textbf{X}})]}{{\mathbb E}_{\mathbb P}[g_m(\textnormal{\textbf{X}})]},
\qquad
\phi^*=\inf_{\mathbb P\in \mathcal D}\frac{{\mathbb E}_{\mathbb P}[f(\textnormal{\textbf{X}})]}{{\mathbb E}_{\mathbb P}[g(\textnormal{\textbf{X}})]},
\]
it follows that
\[
|\phi_m^*-\phi^*|
\leqslant
\sup_{\mathbb P\in \mathcal D}
\left|
\frac{{\mathbb E}_{\mathbb P}[f_m(\textnormal{\textbf{X}})]}{{\mathbb E}_{\mathbb P}[g_m(\textnormal{\textbf{X}})]}
-
\frac{{\mathbb E}_{\mathbb P}[f(\textnormal{\textbf{X}})]}{{\mathbb E}_{\mathbb P}[g(\textnormal{\textbf{X}})]}
\right|
\to 0.
\]
Thus $\phi_m^*\to\phi^*$, proving \textnormal{(i)}.

We now prove \textnormal{(iii)}. Since each $\overline K_i\subseteq K=[0,1]^d$ is compact,
the product $\overline K_1\times\cdots\times \overline K_n$ is compact. Hence,
$({\mathbf x}_1^{(m)},\dots,{\mathbf x}_n^{(m)})$ has a convergent subsequence,
say $({\mathbf x}_1^{(m_k)},\dots,{\mathbf x}_n^{(m_k)})
\to
({\mathbf x}_1^*,\dots,{\mathbf x}_n^*)
\in \overline K_1\times\cdots\times \overline K_n$. Define $\mathbb P^*:=\sum_{i=1}^n p_i\,\delta_{{\mathbf x}_i^*}$ and recall $\mathbb P_m$. Then, for every bounded continuous $\psi:K\to\mathbb R$,
\[
\mathbb E_{\mathbb P_{m_k}}[\psi(\mathbf{X})]
=
\sum_{i=1}^n p_i\,\psi({\mathbf x}_i^{(m_k)})
\to
\sum_{i=1}^n p_i\,\psi({\mathbf x}_i^*)
=
\mathbb E_{\mathbb P^*}[\psi(\mathbf{X})].
\]
Hence $\mathbb P_{m_k}\Rightarrow \mathbb P^*$, proving the subsequential compactness and weak-convergence assertions in \textnormal{(iii)}. Part \textnormal{(iv)} below shows that $\mathbb P^*$ is extremal.

Finally, we prove \textnormal{(iv)}. For each $k$, because
$\mathbb P_{m_k}$ is an extremal prior for the approximating pair
$(f_{m_k},g_{m_k})$, Theorem~\ref{thm_refinedgensol} provides a sequence of feasible
priors in $\mathcal D$ converging weakly to $\mathbb P_{m_k}$, with objective values converging
to $\phi_{m_k}^*$. Choose one element of that sequence, denoted $\mathbb Q_k\in \mathcal D$, such that $d_{\mathrm w}(\mathbb Q_k,\mathbb P_{m_k})<\frac1k$
for some metric $d_{\mathrm w}$ generating weak convergence on $\mathcal P(K)$, and $\left|
\frac{\mathbb E_{\mathbb Q_k}[f_{m_k}(\mathbf{X})]}{\mathbb E_{\mathbb Q_k}[g_{m_k}(\mathbf{X})]}
-
\phi_{m_k}^*
\right|
<\frac1k$. Since $\mathbb P_{m_k}\Rightarrow \mathbb P^*$, the choice
$d_{\mathrm w}(\mathbb Q_k,\mathbb P_{m_k})<1/k$ implies
$\mathbb Q_k\Rightarrow \mathbb P^*$ via the triangle inequality.

It remains to transfer the objective from $(f_{m_k},g_{m_k})$ to $(f,g)$.
Exactly as above, for all sufficiently large $k$, $\,\mathbb E_{\mathbb Q_k}[g(\textnormal{\textbf{X}})]\geqslant p_j\eta\,$ and $\,\mathbb E_{\mathbb Q_k}[g_{m_k}(\mathbf{X})]\geqslant p_j\eta/2$,
therefore
\begin{align*}
\left|
\frac{\mathbb E_{\mathbb Q_k}[f_{m_k}(\mathbf{X})]}{\mathbb E_{\mathbb Q_k}[g_{m_k}(\mathbf{X})]}
-
\frac{\mathbb E_{\mathbb Q_k}[f(\textnormal{\textbf{X}})]}{\mathbb E_{\mathbb Q_k}[g(\textnormal{\textbf{X}})]}
\right|
&\leqslant
\frac{2}{p_j\eta}\,\|f_{m_k}-f\|_\infty
+
\frac{2\|f\|_\infty}{p_j^2\eta^2}\,\|g_{m_k}-g\|_\infty\longrightarrow 0.
\end{align*}
Hence,
\[
\left|
\frac{\mathbb E_{\mathbb Q_k}[f(\textnormal{\textbf{X}})]}{\mathbb E_{\mathbb Q_k}[g(\textnormal{\textbf{X}})]}
-
\phi^*
\right|
\leqslant
\left|
\frac{\mathbb E_{\mathbb Q_k}[f(\textnormal{\textbf{X}})]}{\mathbb E_{\mathbb Q_k}[g(\textnormal{\textbf{X}})]}
-
\frac{\mathbb E_{\mathbb Q_k}[f_{m_k}(\mathbf{X})]}{\mathbb E_{\mathbb Q_k}[g_{m_k}(\mathbf{X})]}
\right|
+
\left|
\frac{\mathbb E_{\mathbb Q_k}[f_{m_k}(\mathbf{X})]}{\mathbb E_{\mathbb Q_k}[g_{m_k}(\mathbf{X})]}
-
\phi_{m_k}^*
\right|
+
|\phi_{m_k}^*-\phi^*|.
\]
The first term tends to $0$ by the displayed estimate, the second is $<1/k$, and
the third tends to $0$ by \textnormal{(i)}. Therefore,
\[
\frac{\mathbb E_{\mathbb Q_k}[f(\textnormal{\textbf{X}})]}{\mathbb E_{\mathbb Q_k}[g(\textnormal{\textbf{X}})]}\to \phi^*.
\]
This proves \textnormal{(iv)} and completes the proof.
\end{proof}

\newpage
\section{Proof of Theorem~\ref{thm_CBIwithfails_sol}}
\label{app_proofofCBIproblem}
\begin{proof}
The following three-step proof is reproduced from \cite{salako2025conservative}: 
\begin{enumerate}
   \item[]\textbf{step I:} the gradient of the objective function determines the functional forms the objective function can assume when minimized. Prove that two of the gradient's components each have a non-trivial root in $[0,1]$ when these components equal zero, then use these roots and the signs of the gradient's components to deduce finitely many functional forms for the objective function when minimized.
  \item[]\textbf{step II:} prove these functional forms have finitely many infima, thus the objective function has finitely many infima: the smallest infimum is the global infimum.
  \item[]\textbf{step III:} based on the previous steps, state the infimum, state \eqref{eqn_morenoetalBer}'s solution as a fixed--point system, and state the prior that gives the infimum.
\end{enumerate}

\noindent\textbf{Step I:}
There are a finite number of plausible functional forms that $\phi$ can take when minimized. To deduce these, first, we prove the existence of a pair of points at which two components of the gradient of $\phi$ are zero. Then, using this pair to deduce the signs of gradient components, we deduce the functional forms.

Define $f,g:[0,1] \longrightarrow [0,1]$, $f(x)=(1-x)^{m}g(x)$ and $g(x)=x^r(1-x)^{k}$, with $k>0,\,r > 0$. Consider $x_i\in K_i$ for $i=1,\ldots, n$. The  gradient of $\phi$ has $i$-th component
\begin{equation}
     \frac{\partial \phi}{\partial x_{i}}=\frac{g(x_i)(r-x_i(r+k))p_i}{x_i(1-x_i){\sum_{j=1}^{n}g(x_j)p_{j}}}\left(h(x_i)-\phi\right)\,,
     \label{eq_grad}
 \end{equation} where $h$ is the function $h:[0,1]\setminus\{\frac{r}{r+k}\}\longrightarrow \mathbb R$, $h(x)=(1-x)^m\left(\frac{r-x(m+k+r)}{r-x(k+r)}\right)$. From \eqref{eq_grad}, the $i$-th gradient component is non-trivially zero if and only if 
\begin{equation}    
h(x_i)=\phi
\label{eq_stationarity_cond}
\end{equation}
Otherwise, the gradient component is positive or negative, depending on combinations of whether $h(x_i)>\phi$ or $h(x_i)<\phi$, and $x_i<\frac{r}{r+k}$ or $x_i>\frac{r}{r+k}$.

The signs of the gradient components determine a preferred location in each compact interval $\overline{K}_i$ ($\overline{K}_i$ is the closure of $K_i$) where probability mass should be assigned to minimize $\phi$. This restricts the possible functional forms of $\phi$ that equal the infimum, $\phi^*$, of $\phi$. 
Probability $p_i$ could be placed at $y_{i-1}$ if $\partial\phi/\partial x_i>0$ in a neighborhood of $y_{i-1}$, or it could be placed at $y_i$ if $\partial\phi/\partial x_i<0$ in a neighborhood of $y_{i}$. It could also be placed at $x$, if $x$ is a stationary point (i.e. $\partial\phi/\partial x_i=0$ at $x_i=x$) at which $\phi$ has a local infimum. 

Whenever the $p_i$s are assigned within each $\overline{K}_i$, there are only two stationary points in $[0,1]$ that satisfy $\partial \phi/\partial x_i=0$~---~one point which is a local infimum, and the other a local supremum. This follows from the graph of $h$, which we deduce now.

$h$ is continuously differentiable on its domain. The derivative of $h$ shows $h$ is monotonically decreasing over $[0,\frac{r}{r+m+k})$ and $(\frac{r}{r+k},1]$, and any stationary points of $h$ must lie in $(\frac{r}{r+m+k},\frac{r}{r+k})$. Differentiating $h$ for $x\neq \frac{r}{r+k}$, \begin{equation}
    \dfrac{\mathtt{d} h}{\mathtt{d}x}=(1-x)^m\left(\dfrac{r-x(r+m+k)}{r-x(r+k)}\right)'+\left(\dfrac{r-x(r+m+k)}{r-x(r+k)}\right)\left((1-x)^m\right)'
    \label{eq_grad_h}
\end{equation} 
    Observe, $\left((1-x)^m\right)'=-m(1-x)^{m-1}<0$ and $\left(\frac{r-x(r+m+k)}{r-x(r+k)}\right)'=\frac{-mr}{(r(1-x)-kx)^2}<0$. Thus, $\mathtt{d} h/\mathtt{d}x<0$ for $x<\frac{r}{r+m+k}$, because $r-x(r+m+k)>0$ and ${r-x(r+k)}>0$ whenever $x<\frac{r}{r+m+k}$. Similarly, $\mathtt{d} h/\mathtt{d}x<0$ for $x>\frac{r}{r+k}$, because $r-x(r+m+k)<0$ and ${r-x(r+k)}<0$ whenever $x>\frac{r}{r+k}$. So, $h$ is monotonically decreasing over these intervals, as claimed. However, for $\frac{r}{r+m+k}< x< \frac{r}{r+k}$ (i.e. $r-x(r+m+k)<0$ and ${r-x(r+k)}>0$), the sign of $\mathtt{d} h/\mathtt{d}x$ depends on $r,\ k$, and $m$.
    
    Interestingly, $h(x)<0$ over the interval  $\frac{r}{r+m+k}< x< \frac{r}{r+k}$. Furthermore, $h$ has, at most, a pair of stationary points in this interval; because, from \eqref{eq_grad_h}, $\mathtt{d} h/\mathtt{d}x=0$ non-trivially if and only if\begin{equation*}x=\dfrac{2r^2+(2k+m+1)r\pm\sqrt{-4rk^2-4kr(m+r)+r^2(m-1)^2}}{2(r+k)(r+m+k)}\end{equation*}
These stationary points are real \emph{iff} $(-4rk^2-4kr(m+r)+r^2(m-1)^2)\geqslant0$. Otherwise, there are no stationary points: the points are a complex conjugate pair. 

When real, this pair of points lies in $(\frac{r}{r+m+k},\frac{r}{r+k})$. The smaller of the two points is greater than $\frac{r}{r+m+k}$, because 
\begin{align*}
    &\dfrac{2r^2+(2k+m+1)r-\sqrt{-4rk^2-4kr(m+r)+r^2(m-1)^2}}{2(r+k)(r+m+k)}>\frac{r}{r+m+k}\\
    &\iff r(m+1)-\sqrt{-4rk^2-4kr(m+r)+r^2(m-1)^2}>0
\end{align*} The last inequality is true because $r(m+1)>r(m-1)$. 
Analogously, the larger of the two points is less than $\frac{r}{r+k}$, because 
\begin{align*}
    &\dfrac{2r^2+(2k+m+1)r+\sqrt{-4rk^2-4kr(m+r)+r^2(m-1)^2}}{2(r+k)(r+m+k)}<\frac{r}{r+k}\\
    &\iff \dfrac{-r(m-1)+\sqrt{-4rk^2-4kr(m+r)+r^2(m-1)^2}}{2(r+k)(r+m+k)}<0\\
    &\iff -r(m-1)+\sqrt{-4rk^2-4kr(m+r)+r^2(m-1)^2}<0
\end{align*} The last inequality is true because the radicand subtracts the positive number $4rk^2+4kr(m+r)$ from $r^2(m-1)^2$, so the square-root of the radicand must be smaller than $r(m-1)$. We conclude that both stationary points of $h$ lie in $(\frac{r}{r+m+k},\frac{r}{r+k})$.

 So far, we have shown $h$ is monotonically decreasing  on $[0,\frac{r}{r+m+k})$ and $(\frac{r}{r+k},1]$. Notice, from the definition of $h$, $h(x)\geqslant 0$ over these intervals. We have also shown $h$ has a pair of stationary points (if any) in $(\frac{r}{r+m+k},\frac{r}{r+k})$. And, over this interval, $h<0$ by definition. Altogether, this means \eqref{eq_stationarity_cond} has a root~---~$\phi$ can equal $h$~---~only where $h$ is positive and monotonically decreasing; that is, only over $[0,1]\setminus(\frac{r}{r+m+k},\frac{r}{r+k})$. In fact, there are two reasons why $0<\phi<1$, so that $\phi$ is bounded by $h$:
 \textbf{i)} the terms in $\phi$'s numerator are strictly less than the corresponding terms in its denominator; \textbf{ii)} both numerator and denominator are strictly positive. 
 
 Figure~\ref{fig_function_h} illustrates the graph of $h$. There is an $x^*\in(\frac{r}{r+k},1]$ such that $h(x^*)=1$. So, $h$ is $0$ at $1$ and $\frac{r}{r+m+k}$, while $h$ is $1$ at $0$ and $x^*$. 
 Moreover, for an arbitrary value of $\phi$, denoted $\hat{\phi}$, this value is represented as a horizontal line at height $\hat{\phi}$. And, for any $x_u\in[x^*,1]$ such that $h(x_u)=\hat{\phi}$, a corresponding $x_l\in[0,\frac{r}{r+m+k}]$ also satisfies $h(x_l)=\hat{\phi}$. Since $h$ bounds $\phi$, and the horizontal line of height $\hat{\phi}$ intersects $h$ at exactly two points, these imply the functional form of $\phi$ when minimized must be based on exactly two components of the gradient of $\phi$ vanishing. The local supremum acts like a ``source'' that ``repels'' probability masses in nearby intervals, while the local infimum acts like a ``sink'' that ``attracts'' probability masses. This proves that at most two components of the gradient have non-trivial roots.


\def\rval{1}
\def\mval{3}
\def\kval{2}

\def\xsing{\fpeval{\rval / (\rval + \kval)}}

\def\xroot{\fpeval{\rval / (\rval + \mval + \kval)}}

\begin{figure}[htbp!]
\begin{center}
\scalebox{0.85}
{\begin{tikzpicture}
  \begin{axis}[
    width=13cm,
    height=8cm,
    domain=0:1,
    samples=400,
    xlabel={$x$},
    ylabel={$h(x)$},
    axis lines=middle,
    xtick distance=1,
    ytick distance=1,
    enlargelimits=true,
    tick style={black},
    ymin=-1, ymax=1,
  ]

  \draw[line width=1mm,gray] ({axis cs:0,0.7}) -- ({axis cs:1,0.7});
  \node[black] at (axis cs:-0.01,0.7) [anchor=east] {$\pmb{\hat{\phi}}$};

  \addplot[
    domain=0:\fpeval{\xsing - 0.001},
    line width = 0.5mm,
    black,
  ]
  {(1 - x)^\mval * (\rval - x*(\mval + \kval + \rval)) / ( \rval - x*(\kval + \rval) )};

  \addplot[
    domain=\fpeval{\xsing + 0.001}:0.999,
    line width = 0.5mm,
    black,
  ]
  {(1 - x)^\mval * (\rval - x*(\mval + \kval + \rval)) / ( \rval - x*(\kval + \rval) )};


  \draw[dashed,thick,black] ({axis cs:\xsing,-0.9}) -- ({axis cs:\xsing,1.1});
  \node[black] at (axis cs:\xsing+0.003,-1.06) {$\pmb{\frac{r}{r+k}}$};

  \draw[dashed,thick,black] ({axis cs:\xsing+0.095,-0.3}) -- ({axis cs:\xsing+0.1,1});
  \node[black] at (axis cs:\xsing+0.105,-0.4)  {$\pmb{x^\ast}$};

  \draw[dashed,thick,black] ({axis cs:\xroot,-0.5}) -- ({axis cs:\xroot,0});
  \node[black] at (axis cs:\xroot,-0.65) {$\pmb{\frac{r}{r+m+k}}$};

  \draw[dashed,thick,black] ({axis cs:0.05,0.7}) -- ({axis cs:0.05,0});
  \node[black] at (axis cs:0.06,-0.1) {$\pmb{x_l}$};

  \draw[dashed,thick,black] ({axis cs:0.465,0.7}) -- ({axis cs:0.465,0});
  \node[black] at (axis cs:0.479,-0.1) {$\pmb{x_u}$};

  \end{axis}
\end{tikzpicture}}
\end{center}\caption[Two-root structure of $h(x)=\widehat{\phi}$]{An example graph of $h(x) = (1 - x)^m\left(\frac{r - x(m + k + r)}{r - x(k + r)}\right)$. Reproduced from \cite{salako2025conservative}.} \label{fig_function_h}
\end{figure}
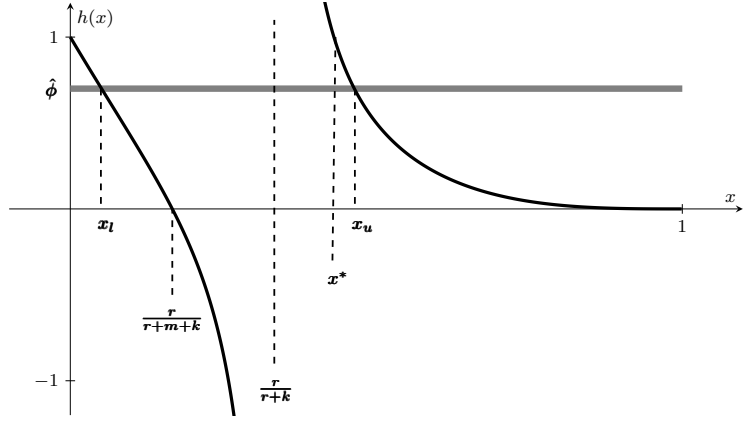

 Consequently, to minimise $\phi$, all of the $x_i$ must be assigned in their respective $K_i$ as dictated by a ``source'' and ``sink'' pair. All $x_i$ to the left of the ``source'' must be located at the lower endpoints of their respective $K_i$ intervals, while $x_i$ between the ``source'' and  ``sink'' must be located at the upper endpoints of their $K_i$. The $x_i$ for the interval containing the ``source'' must be located at one of the endpoints of the interval~---~whichever endpoint gives the lower objective function value. Analogous reasoning in terms of the ``sink'' requires all $x_i$ to the right of the ``sink'' to be located at the lower endpoints of their respective $K_i$. The $x_i$ in the same interval as the ``sink'' should be located \emph{at} the ``sink''.
 


In interval $K_{1}$, $x_1$ must be $0$ when $n\geqslant 2$. Indeed, express $\phi$ as a function of $\alpha$, \begin{equation*}
    \phi(\alpha)=(1-x_1)^{m}\dfrac{g(x_1)p_1}{\sum_{i=1}^ng(x_i)p_i}+\sum_{j=2}^{n}(1-x_j)^{m}\dfrac{g(x_j)p_j}{\sum_{i=1}^n g(x_i)p_i}=(1-x_1)^{m}\alpha+(1-\alpha)\beta
\end{equation*} where  $\alpha=g(x_1)p_1/\sum_{i=1}^n g(x_i)p_i$ and $\beta=\sum_{j=2}^{n}(1-x_j)^{m}g(x_j)p_j/\sum_{i=2}^n g(x_i)p_i$. Since $(1-x_1)^m\geqslant(1-x_2)^m\geqslant\dots\geqslant(1-x_n)^m$, for all $\alpha\in[0,1]$, we have $\phi(\alpha)\geqslant \beta$. Thus, $\phi(\alpha)\geqslant \phi(0)$, and $\phi$ is smallest necessarily when $x_1=0$ in $\overline{K}_1$. Otherwise, when $n=1$ instead, $\phi$ is $(1-x_1)^m$, so $x_1=1$ minimizes $\phi$. 

All of the foregoing justifies there being only two possible functional forms, $\tilde{\phi}_1$ and $\tilde{\phi}_2$, for $\phi$ at its infima. Define $x^*_1$ to be $\max\{x^*,y_1\}$ and choose $x\in[x^*_1,1]$. Denote the interval containing $x$ as the ${j_2}^{th}$ interval, and the ${j_1}^{th}$ interval as the interval containing the corresponding $x_l$ under $h$. Then, $\tilde{\phi_1}:[x^*_1,1]\longrightarrow[0,1]$ and $\tilde{\phi_2}:[x^*_1,1]\longrightarrow[0,1]$ are the functional forms,  \begin{align*}  &\tilde{\phi_1}(x):=
    \begin{cases}
  \dfrac{\sum\limits_{i=2}^{j_{1}}f(y_{i-1})p_i+\sum\limits_{i=j_{1}+1}^{j_{2}-1}f(y_{i})p_i+f(x)p_{j_2}+\sum\limits_{i=j_{2}+1}^{n}f(y_{i-1})p_i}{\sum\limits_{i=2}^{j_{1}}g(y_{i-1})p_i+\sum\limits_{i=j_{1}+1}^{j_{2}-1}g(y_{i})p_i+g(x)p_{j_2}+\sum\limits_{i=j_{2}+1}^{n}g(y_{i-1})p_i}\,,\hfill j_1<j_2\\
     \dfrac{\sum\limits_{i=2}^{n}f(y_{i-1})p_i}{\sum\limits_{i=2}^{n}g(y_{i-1})p_i}\,,\hfill j_1=j_2
\end{cases}
\\
  &\tilde{\phi_2}(x):=\begin{cases}
   \dfrac{\sum\limits_{i=2}^{j_1-1}\!\!f(y_{i-1})p_i+\sum\limits_{i=j_{1}}^{j_{2}-1}\!f(y_{i})p_i+f(x)p_{j_2}+\sum\limits_{i=j_2+1}^{n}\!\!f(y_{i-1})p_i}{\sum\limits_{i=2}^{j_1-1}\!\!g(y_{i-1})p_i+\sum\limits_{i=j_{1}}^{j_{2}-1}\!g(y_{i})p_i+g(x)p_{j_2}+\sum\limits_{i=j_2+1}^{n}\!\!g(y_{i-1})p_i}\,,\hfill  j_1<j_2\\
    \dfrac{\sum\limits_{i=2}^{j_{2}-1}f(y_{i-1})p_i+f(x)p_{j_2}+\sum\limits_{i=j_{2}+1}^{n}f(y_{i-1})p_i}{\sum\limits_{i=2}^{j_{2}-1}g(y_{i-1})p_i+g(x)p_{j_2}+\sum\limits_{i=j_{2}+1}^{n}g(y_{i-1})p_i}\,,\hfill  j_1=j_2 \end{cases}
\end{align*} 
Notice how $j_1$ and $j_2$ can change as $x$ changes, since the domain $[x_1^*,1]$ of $x$ values can intersect multiple $K_i$ intervals. Also note that, since $x_1=0$, we have $f(x_1)=f(0)=0$ and $g(x_1)=g(0)=0$, so these terms do not appear in the functional forms.

\noindent\textbf{Step II:}
$\phi$ has a finite number of local infima. This is equivalent to only a finite number of possible placements of the $p_i$ probabilities in the $\overline{K}_i$ intervals, in such a way that the infimum $\phi^*$ is either $\tilde{\phi}_1$ or $\tilde{\phi}_2$ while, simultaneously, $\phi^*$ satisfies $h(x_i)=\phi^*$ for some $x_i\in\overline{K}_i$. To show this, consider all possible placements of the $p_i$ consistent with the following three cases:

\noindent Let the interval containing $x^*$ be $K_{i^*}$.
\begin{enumerate}
    \item[] \textbf{case I:} Assume $x$ and $x_1^*$ lie in separate intervals. For each $i\leqslant i^*$ place $p_i$ at an endpoint of $K_i$. Otherwise, place $p_i$ as close as possible to $x$ for the other intervals above $K_{i^*}$; in particular, place $p_i$ at $x$ for the interval containing $x$ (see Figure~\ref {fig_mass_placement_case1}).

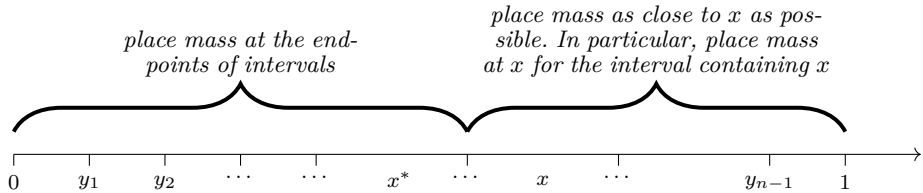
\begin{figure}[htbp!]\centering
\begin{tikzpicture}
  \draw[->] (0,0) -- (12,0);

 \draw (0,0) -- (0,-0.15) node[below=5pt] at (0,0) {$0$};
  \foreach \x/\label in {
    1/{$y_1$},
    2/{$y_2$},
    3/{\dots},
    4/{\dots},
    6/{\dots},
    8/{\dots},
    10/{$y_{n-1}$},
     11/{$1$}
  }{
    \draw (\x,0) -- (\x,-0.15);
    \node[below=5pt] at (\x,0) {\label};
  }

  \node[below=3.5pt] at (5.1,0) {$x^*$};

  \node[below=5pt] at (7,0) {$x$};

  \draw [decorate,decoration={brace,amplitude=18pt},yshift=0.3cm,line width=1.5pt]
    (0,0) -- (6,0) 
    node[midway,above=17pt,text width=4cm,align=center]
    {\small\textit{place mass at the endpoints of intervals}};
    
  \draw [decorate,decoration={brace,amplitude=18pt},yshift=0.3cm,line width=1.5pt]
    (6,0) -- (11,0) 
    node[midway,above=17pt,text width=6cm,align=center]
    {\small\textit{place mass as close to $x$ as possible. In particular, place mass at $x$ for the interval containing $x$}};
\end{tikzpicture}
    \caption[Separate-interval mass placement for $x$ and $x_1^*$]{Illustration of  possible probability mass placement: case I. Reproduced from \cite{salako2025conservative}.}\label{fig_mass_placement_case1}
\end{figure}
      \item[] \textbf{case II:} Assume $x$ and $x_1^*$ are in the same interval. For each $i< i^*$ place $p_i$ at an endpoint of $K_i$. Place $p_i$ as close as possible to $x$ in all other intervals; in particular, place $p_{i^*}$ at $x$ in interval $K_{i^*}$ containing $x$ and $x^*$ (see Figure~\ref {fig_mass_placement_case2}).
      \begin{figure}[ht!]\centering
      \begin{tikzpicture}
  \draw[->] (0,0) -- (12,0);

   \draw (0,0) -- (0,-0.15) node[below=5pt] at (0,0) {$0$};
  \foreach \x/\label in {
    1/{$y_1$},
    3/{$y_2$},
    4/{\dots},
    6/{\dots},
    7/{\dots},
    9/{\dots},
    10/{$y_{n-1}$},
    11/{$1$}
  }{
    \draw (\x,0) -- (\x,-0.15);
    \node[below=5pt] at (\x,0) {\label};
  }

  \node[below=2.5pt] at (4.6,0) {$x^*$};
  
    \node[below=2.5pt] at (8,0) {\space};

  \node[below=5pt] at (5.4,0) {$x$};

  \draw [decorate,decoration={brace,amplitude=18pt},yshift=0.3cm,line width=1.5pt]
    (0,0) -- (4,0) 
    node[midway,above=19pt,text width=5.5cm,align=center]
    {\small\textit{place mass at the endpoints of intervals}};
  \draw [decorate,decoration={brace,amplitude=18pt},yshift=0.3cm,line width=1.5pt]
    (4,0) -- (6,0) 
    node[midway,above=19pt,text width=2cm,align=center]
    {\small\textit{place mass at $x$}};
   
  \draw [decorate,decoration={brace,amplitude=19pt},yshift=0.3cm,line width=1.5pt]
    (6,0) -- (11,0) 
    node[midway,above=19pt,text width=5.5cm,align=center]
    {\small\textit{place mass at the lower endpoints of intervals}};
\end{tikzpicture}
          \caption[Shared-interval mass placement at $x$]{Illustration of possible probability mass placement: case II. Reproduced from \cite{salako2025conservative}.}
          \label{fig_mass_placement_case2}
      \end{figure}
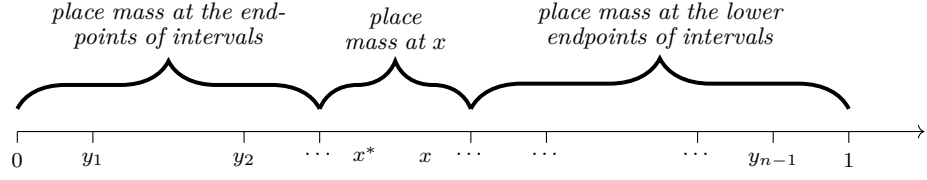
      \item[] \textbf{case III:} Assume $x$ and $x_1^*$ are in the same interval. For each $i< i^*$ place mass at an endpoint of $K_i$. Place $p_{i^*}$ at the lower endpoint of the interval $K_{i^*}$ containing $x$ and $x^*$. Otherwise, place mass as close as possible to $x$ in the intervals above $K_{i^*}$ (see Figure~\ref {fig_mass_placement_case3}).
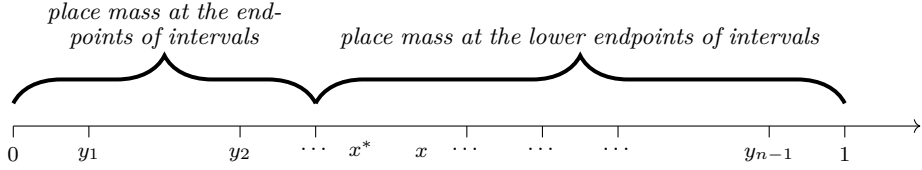
\begin{figure}[ht!]\centering 
   \begin{tikzpicture}
  \draw[->] (0,0) -- (12,0);

   \draw (0,0) -- (0,-0.15) node[below=5pt] at (0,0) {$0$};
  \foreach \x/\label in {
    1/{$y_1$},
    3/{$y_2$},
    4/{\dots},
    6/{\dots},
    7/{\dots},
    8/{\dots},
    10/{$y_{n-1}$},
    11/{$1$}
  }{
    \draw (\x,0) -- (\x,-0.15);
    \node[below=5pt] at (\x,0) {\label};
  }

  \node[below=2.5pt] at (4.6,0) {$x^*$};

  \node[below=5pt] at (5.4,0) {$x$};

  \draw [decorate,decoration={brace,amplitude=18pt},yshift=0.3cm,line width=1.5pt]
    (0,0) -- (4,0) 
    node[midway,above=17pt,text width=4.6cm,align=center]
    {\small\textit{place mass at the endpoints of intervals}};
   
  \draw [decorate,decoration={brace,amplitude=18pt},yshift=0.3cm,line width=1.5pt]
    (4,0) -- (11,0) 
    node[midway,above=17pt,text width=7cm,align=center]
    {\small\textit{place mass at the lower endpoints of intervals}};
\end{tikzpicture}

	\caption[Shared-interval lower-endpoint mass placement]{Illustration of possible probability mass placement: case III. Reproduced from \cite{salako2025conservative}. 
    }
	\label{fig_mass_placement_case3}
\end{figure}
\end{enumerate}
    Placements of the $p_i$s that are consistent with these three cases define functional forms. For example, from case (I), define $\phi_{z_i,\ldots,z_{i^*},j_{2}}: \overline{K}_{j_2}\cap[x^*_1,1]\longrightarrow [0,1]$ as 
    \begin{align*}
    \phi_{z_i,\ldots,z_{i^*},j_2}(x)=\dfrac{\sum\limits_{i=1}^{i^*}f(z_i)p_i+\sum\limits_{i=i^*+1}^{j_{2}-1}f(y_i)p_i+f(x_i)p_{j_{2}}+\sum\limits_{i=j_{2}+1}^{n}f(y_{i-1})p_i}{\sum\limits_{i=1}^{i^*}g(z_i)p_i+\sum\limits_{i=i^*+1}^{j_{2}-1}g(y_i)p_i+g(x_i)p_{j_{2}}+\sum\limits_{i=j_{2}+1}^{n}g(y_{i-1})p_i}
    \end{align*}
where $j_2$ is an arbitrary choice from $\{(i^*+1),\ldots,n\}$, while $z_1,\ldots,z_{i^*}$ are an arbitrary placement of probability masses at an endpoint in each of the intervals $K_1,\ldots, K_{i^*}$. Analogous univariate quotient functions can be defined based on cases (II) and (III). Altogether, there is a finite number of such quotient functions. Note that $\tilde{\phi}_1$ and  $\tilde{\phi}_2$ must agree with some of these functional forms. So, the infimum of $\phi$ must also agree with some of these functional forms. 

Consider all such functional forms that agree on the locations in the intervals to the left of $x_1^*$; for example, the functions $\phi_{z_i,\ldots,z_{i^*},j_2}$ for all $j_2$. These define a piecewise continuous function $F:[x^*_1,1]\rightarrow [0,1]$, $F(x)=\phi_{z_i,\ldots,z_{i^*},j_2}(x)-h(x)$ where $x\in \overline{K}_{j_2}$. Since $F(x_1^*)$ is non-positive (because $h(x_1^*)\geqslant(1-x_1^*)^m\geqslant\phi_{z_i,\ldots,z_{i^*},i^*}(x_1^*)$ for $x\geqslant x_1^*$) and $F(1)$ is non-negative (because $h(1)=0$ and the $\phi_{z_i,\ldots,z_{i^*},j_2}$ are positive), the \acrshort*{ivt} implies that $F(x)=0$ for some $j_2$ interval containing this $x$. Moreover, at such an $x$, $\partial\phi/\partial x_{j_2}(x)$ is zero and $F$ must be increasing, since $F^\prime(x)=\phi^\prime_{z_i,\ldots,z_{i^*},j_2}(x)-h^\prime(x)=\partial\phi/\partial x_{j_2}(x)-h^\prime(x)=-h^\prime(x)>0$. If $x$ lies at an interval boundary, the same inequality holds for the appropriate one-sided derivative of $F$, and adjacent $F$ pieces glue continuously with consistent one-sided slopes; hence, the $F$ root is still simple. Thus, $F=0$ can only be satisfied once over $[x_1^*,1]$. Since there are finitely many such $F$, there are finitely many of these functional forms that satisfy $\phi_{z_i,\ldots,z_{i^*},j_2}=h$, so there are finitely many points in the domain of $\phi$ at which $\phi$ could be at a minimum. These points, up to a choice of endpoint in the $j_1$ interval, must give distinct $\phi$ values: since otherwise, if they have the same value $\hat{\phi}$ of $\phi$, they must have the same value of $\hat{x}\in[x_1^*,1]$ that solves $h(\hat{x})=\hat{\phi}$, and this $\hat{x}$ determines the locations in all other intervals up to a choice of endpoint in interval $j_1$~---~the same functional form follows from the same $\phi$ value. 

Finally, since $h$ is monotonically decreasing over $[x_1^*,1]$, the horizontal lines at the minima values of $\phi$ intersect $h$ at various heights, with only one unique value that has the lowest intersection with $h$: this is the global minimum $\phi^*$.

\noindent\textbf{Step III:} In summary, there exists the unique fixed point consisting of $\phi^*,\,y_{*},\,y_{**}$ that satisfies the system of equations $\phi^* = \underset{x\in[x^*,1]}{\inf} \{\tilde{\phi}(x)\}=\tilde{\phi}(y_{*})$ and $h(y_{**})=\phi^*=h(y_{*})$, for some $i$-indices $j_1,j_2$ such that $y_{**}\in\overline{K}_{j_1}$, $y_{*}\in\overline{K}_{j_2}$ and $y_{**}< \frac{r}{r+m+k}<\frac{r}{r+k}<y_{*}$. And, $\phi$ attains its infimum $\phi^*$ if the prior distribution of $X$ in Theorem\ \ref{thm_CBIwithfails_sol} is used for inference.
\end{proof}

\newpage
\section{Proof of Theorem~\ref{thm_atleastkoutofnbounds}}
\begin{proof}
Let $K=[0,1]^n$ and define $\widetilde K_i:=\{\textbf{x}\in K:x_i\geqslant \tfrac12\}$ for each $i=1,\dots,n$. The upper $k$-out-of-$n$ problem is the optimisation of the success indicator
\[
\mathbf 1\!\left\{\sum_{i=1}^n \mathbf 1_{\widetilde K_i}(\textbf{x})\geqslant k\right\}
\]
under the marginal constraints $\mathbb P(\widetilde K_i)=p_i$. Here, $\widetilde{K}_i$ is the geometric representation of the event $A_i$ for each $i$. By Proposition~\ref{prop_genprob1_v1}, the overlapping sets $\widetilde K_1,\dots,\widetilde K_n$ induce the
atomisation
\[
A_I:=\Bigl(\bigcap_{i\in I}\widetilde K_i\Bigr)\cap
\Bigl(\bigcap_{i\notin I}\widetilde K_i^{\,c}\Bigr),
\qquad I\subseteq[n],
\]
and the problem becomes
\[
\Phi_k
=
\sup_{\substack{p_I\geqslant 0\\ \sum_{I\subseteq[n]}p_I=1\\ \sum_{I\ni i}p_I=p_i,\ i=1,\dots,n}}
\ \sup_{\substack{\mathbf x_I\in A_I\\ p_I>0}}
\sum_{|I|\geqslant k}p_I.
\]
Equivalently, 
\[
\Phi_k
=
1-
\inf_{\substack{p_I\geqslant 0\\ \sum_{I\subseteq[n]}p_I=1\\ \sum_{I\ni i}p_I=p_i,\ i=1,\dots,n}}
\ \inf_{\substack{\textbf{x}_I\in A_I\\ p_I>0}}
\frac{\sum_{I\subseteq[n]} f(\textbf{x}_I)p_I}{\sum_{I\subseteq[n]} g(\textbf{x}_I)p_I},
\]
where
\[
f(\textbf{x}):=\mathbf 1\!\left\{\sum_{i=1}^n \mathbf 1_{\widetilde K_i}(\textbf{x})<k\right\},
\qquad
g(\textbf{x})\equiv 1.
\]
By Proposition~\ref{prop:semicont_refinement}, the atomised domain may be further refined without changing the
objective's optimal value, and by Theorem~\ref{thm_refinedgensol} the corresponding refined problems admit the (E1)--(E2) extremal solution characterisation.

For each $j=1,\dots,k$, define the stage-$j$ indicator objective by
\[
f_j(\textbf{x}):=
\mathbf 1\!\left\{
\sum_{i=1}^{m_j}\mathbf 1_{\widetilde K_i}(\textbf{x})<j
\right\},
\qquad
g_j(\textbf{x})\equiv 1,
\]
where $m_j=n-k+j$. Thus, $f_j$ records failure of a $j$-out-of-$m_j$ objective using only the smallest $m_j$ marginals. This finite sequence of approximants converges pointwise to f in finitely many steps, so the natural extension to an infinite sequence converges uniformly to $f$.

The proof now uses the following three observations in four proof steps.

\smallskip

\noindent
\emph{Observation 1.}
After atomisation via Proposition~\ref{prop_genprob1_v1}, the outermost optimisation can be chosen to enforce $\sum_I p_I=1$, with the inner optimisation producing uncapped solutions that must satisfy the marginal event constraints; the outer problem then truncates the inner solution at $1$.

\smallskip

\noindent
\emph{Observation 2.}
For the inner optimisation, one may restrict attention to \emph{exactly} $k$-contributor
atoms: if a successful atom has $m>k$ contributors, then $m-k$ contributors may have their mass reassigned to atoms consisting of only one marginal event, while the remaining $k$ contributors may have their mass reassigned to a $k$-atom consisting only of these $k$ contributors, without decreasing the objective value.

\smallskip

\noindent
\emph{Observation 3.}
For each $j=1,\dots,k$, one may consider the stage-$j$ reduced objective obtained by
counting only the smallest $m_j=n-k+j$ marginals as genuinely contributing to the
objective function. Equivalently, the largest $k-j$ marginals are treated as already dominating.

\noindent
\paragraph{Step 1: stage $j$ contributes exactly one new quotient form.}
Fix $j\in\{1,\dots,k\}$ and consider the stage-$j$ problem determined by $f_j$.
By Observation~2, we may analyse only exact-$k$ atoms. Inside stage $j$, there may be dominating marginals among the smallest
$m_j$ eligible marginals. Suppose there are $s\in\{0,\dots,j-1\}$ such dominating marginals. Then the number of non-dominating contributors is
\[
t:=j-s\in\{1,\dots,j\}.
\]
Accordingly, the total number of dominating marginals if these $j$-atoms are extended trivially to $k$-atoms is $(k-j)+s=k-t$. Hence, the set of non-dominating marginals has cardinality $n-(k-t)=n-k+t$. Summing the marginal constraints over the non-dominating marginals counts each
successful exact-$k$ atom exactly $t$ times. Therefore, the corresponding (E2) quotient must have the
uncapped form
\[
\frac{\sum_{i=1}^{n-k+t}p_i}{t}
\]
and, by Observation~1, it is capped at $1$. Thus, the stage-$j$ candidate optimal objective values are precisely
\[
U_t=\min\!\left\{\frac{\sum_{i=1}^{n-k+t}p_i}{t},\,1\right\},
\qquad
t=1,\dots,j.
\]
That is, stage $j$ contributes one \emph{new} form $U_j$, with its total candidate set being $\{U_1,\dots,U_j\}$. For $j=k$, one has $m_k=n$, so $f_k=f$ and Observation~3 imposes no further restriction: the stage-$k$ problem is exactly the original inner
exact-$k$ problem obtained after Observation~2.

\noindent
\paragraph{Step 2: each $U_j$ is an upper bound.}
Fix $j\in\{1,\dots,k\}$. Again by Observation~2, it is enough to count only exact-$k$
successful atoms. Any exact-$k$ atom contains at most $k-j$ indices outside the smallest $m_j=n-k+j$ marginals. Hence, every exact-$k$ atom must contain at least $j$ indices from $\{1,\dots,m_j\}$. So, summing the marginal constraints over $\{1,\dots,m_j\}$ counts each exact-$k$ atom at least $j$ times, so $j\,\Phi_k\leqslant \sum_{i=1}^{m_j}p_i$.
Therefore, $\Phi_k\leqslant \frac{\sum_{i=1}^{m_j}p_i}{j}$.
Applying Observation~1 then yields $\Phi_k\leqslant U_j$ $(j=1,\dots,k)$. Thus, each $U_j$ is an upper bound.

\noindent
\paragraph{Step 3: select the stage-$k$ extremal objective value.}
From the foregoing, $\Phi_k$ must be one of $U_1,\dots,U_k$, which is guaranteed by a trivial application of Theorem~\ref{thm_refinedgensol} using the $f_j$ approximants. Since $\Phi_k$ is bounded above by every one of them, it must equal their minimum: $\Phi_k=\min_{1\leqslant j\leqslant k}U_j$.

\noindent
\paragraph{Step 4: characterise the support of extremal priors.}
Finally, temporarily define
\[
j^*:=\begin{cases}\min\{j\in\{1,\ldots,k-1\}:j\in\arg\min\limits_{1\leqslant j\leqslant k}U_j,\;U_j\leqslant p_{n-k+1+j}\},\text{ if }\min\text{exists}\\k,\text{ otherwise}\end{cases}
\]
and $r^*:=k-j^*$.

Assuming $j^*=1$ and $j^*<k$, then $\Phi_k=U_{1}\leqslant p_{n-k+2}<\ldots<p_n$, so the $r^*=k-1$ largest marginal probabilities dominate $\Phi_k$. 

Assume $1<j^*<k$, with $\Phi_k=U_{j^*}$ and $m_{j^*}=n-k+j^*$. Then, by the definition of $j^*$, $U_{j^*}\leqslant p_{m_{j^*}+1}<1$. So, $U_{j^*}=\frac{\sum^{m_{j^*}}_{i=1} p_i}{j^*}\leqslant p_{m_{j^*+1}}<p_{m_{j^*+2}}<\ldots<p_n<1$; there are at least $n-m_{j^*}=n-(n-k+j^*)=k-j^*$ many marginal probabilities greater than or equal to $\Phi_k$. There are no smaller $p_i$s that dominate $U_{j^*}$. By way of contradiction, assume (in particular) that $p_{m_{j^*}}\geqslant U_{j^*}$. Thus $p_{m_{j^*}}\geqslant U_{j^*-1}$. On the other hand, $U_{j^*}< U_{j^*-1}$ by definition of $j^*$ above: that is, we also have $p_{m_{j^*}}<U_{j^*-1}$ which is a contradiction. We conclude, in particular, that $p_{m_{j^*}}<U_{j^*}$; therefore $0<p_1<\ldots<p_{m_{j^*}}<U_{j^*}$. And so, exactly $r^*=k-j^*$ many marginal probabilities are greater than or equal to $\Phi_k$.

If instead $j^*=k$, then $\Phi_k=U_k>p_i$ for $i=1,\ldots,n$. There are no marginal probabilities that dominate $\Phi_k$; i.e. $r^*=0$.


Notice, elementary algebra and the strict monotonicity $p_1<\cdots<p_n$ imply that $U_j>p_{n-k+j+1}$ forces $U_j\geqslant U_{j+1}$, while $U_j\leqslant p_{n-k+j+1}$ forces $U_j\leqslant U_{j+1}$. Thus, $\{U_j\}_{j=1,...,k}$ is decreasing up to the first index satisfying $U_j\leqslant p_{n-k+j+1}$ and nondecreasing thereafter. Consequently, that index is automatically the leftmost minimizer, so the ``$\arg\min$'' part of the $j^*$ definition above is unnecessary.

In summary,
\[
\Phi_k
=
U_{j^*}
=
\min\!\left\{\frac{\sum_{i=1}^{n-r^*}p_i}{k-r^*},\,1\right\},\qquad r^*=\sum_{i=1}^{k-1}{\mathbf{1}}_{p_{n-k+1+i\geqslant\Phi_k}}=k-j^*.
\]
By the definition of the intermediate stage-$j$ optimisation problems, at stage $j^*$, there are exactly $k-j^*=r^*$ dominating marginals that contribute to all $k$-out-of-$n$ events in the support of the extremal prior:
namely, the $r^*$ many $A_i$ events with the largest $p_i$s. The remaining $j^*$ marginals in these $k$-out-of-$n$ events come from the $n-r^*$ many $A_i$ events with the smallest $p_i$s. This characterises the support of the extremal priors (and thus is an (E1) characterisation) up to this requirement on $k$-out-of-$n$ events.
\end{proof}
\newpage

\section{Proof of Theorem~\ref{thm:2d-gk-fixed-point-extension-corrected}}
Some preliminary definitions and constructs. Let $\mathcal J_m$ be a dyadic partition of $U$, and define an $S_C$-adapted partition $\mathcal B_m^-:= \{J\cap S_C,\ J\cap S_C^c:J\in\mathcal J_m\}$, discarding $\mu$-null cells. For $B\in\mathcal B_m^-$, write $p_B:=\mu(B)$ and $K_B:=K\cap(B\times V)$. Let \[ \mathcal D_m^-:= \{\mathbb P\in\mathcal P(K):\mathbb P(K_B)=p_B,\ B\in\mathcal B_m^-\}. \] We will consider the stage-$m$ lower fixed-point extremal value \[ \Lambda_m^-:= \inf_{\mathbb P\in\mathcal D_m^-}\frac{{\mathbb E}_{\mathbb P}[f]}{{\mathbb E}_{\mathbb P}[g]}. \] For $B\in\mathcal B_m^-$, write \[ \ell_B:=\inf_{\varphi\in B}L(\varphi), \qquad u_B:=\sup_{\varphi\in B}L(\varphi). \] Set \[ A_m^\ell:=\sum_{B\subseteq S_C}p_B\ell_B, \qquad A_m^u:=\sum_{B\subseteq S_C}p_Bu_B, \] and \[ B_m^\ell:=\sum_{B\subseteq S_C^c}p_B\ell_B, \qquad B_m^u:=\sum_{B\subseteq S_C^c}p_Bu_B. \] 

\begin{proof} First, reduce the finitely constrained problem to a discrete extremal-prior problem; then identify the discrete support configurations through Dinkelbach fixed-point conditions. Finally, pass to the limit by a sandwiching argument. 

\emph{Step 1: stage-m discrete optimisation.} Fix $m$. The class $\mathcal D_m^-$ imposes finitely many cell-mass constraints, $\mathbb P(K_B)=p_B$ for $B\in\mathcal B_m^-$. This finite relaxation is a finite distributional-constraint problem. The infimum over $\mathcal D_m^-$ is given by a discrete extremal distribution, $\mathbb P_m^-=\sum_{B\in\mathcal B_m^-}p_B\delta_{\mathbf{x}_B}$ for $\mathbf{x}_B\in \overline{K_B}$. Hence, 
\begin{equation}
\Lambda_m^-= \inf_{B\in\mathcal B_m^-,\ \mathbf{x}_B\in K_B} \frac{\sum_{B\in\mathcal B_m^-}p_Bf(\mathbf{x}_B)} {\sum_{B\in\mathcal B_m^-}p_Bg(\mathbf{x}_B)}. 
\label{eqn_stagemdiscretelambda}
\end{equation} 
This is the stage-$m$ discrete objective constrained optimisation problem. 

\emph{Step 2: Dinkelbach fixed-point characterisation.} Since $0\leqslant f\leqslant g$, the stage-$m$ objective value belongs to $[0,1]$; so, it is enough to compute the following Dinkelbach residual on $\lambda\in[0,1]$. Define $h_{\lambda,f,g}(\varphi,\eta) = f(\varphi,\eta)-\lambda g(\varphi,\eta) = L(\varphi)\{\mathbf 1_C(\varphi,\eta)-\lambda\}$. So, the stage-$m$ Dinkelbach residual is \[ F_m^-(\lambda) := \inf_{\mathbf{x}_B\in K_B} \sum_{B\in\mathcal B_m^-} p_Bh_{\lambda,f,g}(\mathbf{x}_B). \] When rewritten in Dinkelbach form, the finite objective in \eqref{eqn_stagemdiscretelambda} is separable across the $K_B$ cells, so  \[ F_m^-(\lambda) = \sum_{B\in\mathcal B_m^-} p_B \inf_{\mathbf{x}\in K_B}h_{\lambda,f,g}(\mathbf{x}). \] 
$\Lambda_m^-$ is the unique zero of $F_m^-$; we deduce its value as follows. If $B\subseteq S_C$, then $K_\varphi\subseteq C_\varphi$ for every $\varphi\in B$. Hence, every feasible point in $K_B$ lies in $C$, and $h_{\lambda,f,g}(\varphi,\eta) = (1-\lambda)L(\varphi)$ on $K_B$. Therefore, $\inf_{\mathbf{x}\in K_B}h_{\lambda,f,g}(\mathbf{x}) = (1-\lambda)\ell_B$. If $B\subseteq S_C^c$, then for every $\varphi\in B$ there exists some $\eta\in K_\varphi\setminus C_\varphi$. Thus, the event $C$ can be avoided over the cell. Moreover, on $C^c$, we have $h_{\lambda,f,g}(\varphi,\eta)=-\lambda L(\varphi)$, so the minimising support point is chosen in $K_B\cap C^c$ at maximal projected likelihood: $\inf_{\mathbf{x}\in K_B}h_{\lambda,f,g}(\mathbf{x}) = -\lambda u_B$. Consequently, \[ F_m^-(\lambda) = (1-\lambda)A_m^\ell-\lambda B_m^u. \] Thus, the unique root of $F_m^-(\lambda)=0$ is $\Lambda_m^-:= \frac{A_m^\ell}{A_m^\ell+B_m^u}$. 

\emph{Step 3: Monotone lower bounds on $I^-$.} Note that, by definition, $\mathcal D\subseteq\mathcal D_m^-$. Therefore, \[ \Lambda_m^-= \inf_{\mathbb P\in\mathcal D_m^-} \frac{{\mathbb E}_{\mathbb P}[f]}{{\mathbb E}_{\mathbb P}[g]} \leqslant \inf_{\mathbb P\in\mathcal D} \frac{{\mathbb E}_{\mathbb P}[f]}{{\mathbb E}_{\mathbb P}[g]} = I^-. \] Since the $S_C$-adapted stage-m partitions are nested, i.e. $\mathcal D\subseteq\mathcal D_{m+1}^-\subseteq\mathcal D_m^-$, we further have \[ \Lambda_m^-\leqslant\Lambda_{m+1}^-\leqslant I^-. \] 

\emph{Step 4: Upper bound on $I^-$.} There exist $\mathbb P_n\in\mathcal D$ that satisfy the vanishing limit condition in Lemma~\ref{lem:choquet-dyadic-vanishing-leakage}. The objective function denominator value w.r.t. $\mathbb P_n$ is fixed for all $n$ (because $g=L$ depends only on $\varphi$): \[ {\mathbb E}_{\mathbb P_n}[g] = \int_U L(\varphi)\mu(d\varphi). \] Moreover, \[ {\mathbb E}_{\mathbb P_n}[f] = \int_{S_C}L(\varphi)\mu(d\varphi) + {\mathbb E}_{\mathbb P_n}\!\left[ L\mathbf 1_C\mathbf 1_{S_C^c} \right]. \] Since $L\leqslant M$, the sandwich $0\leqslant {\mathbb E}_{\mathbb P_n}\!\left[ L\mathbf 1_C\mathbf 1_{S_C^c} \right] \leqslant M\,\mathbb P_n((S_C^c\times V)\cap C) \to0$ holds. Hence, \[ \limsup_{n\to\infty} \frac{{\mathbb E}_{\mathbb P_n}[f]}{{\mathbb E}_{\mathbb P_n}[g]} = \frac{\int_{S_C}L\,d\mu} {\int_U L\,d\mu}. \] Now, $\int_{S_C}L\,d\mu\leqslant A_m^u$ and $\int_{S_C^c}L\,d\mu\geqslant B_m^\ell$. Therefore, \[ \frac{\int_{S_C}L\,d\mu} {\int_U L\,d\mu} = \frac{\int_{S_C}L\,d\mu} {\int_{S_C}L\,d\mu+\int_{S_C^c}L\,d\mu} \leqslant \frac{A_m^u}{A_m^u+B_m^\ell} =: \Gamma_m^-. \] Because $I^-\leqslant {\mathbb E}_{\mathbb P_n}[f]/{\mathbb E}_{\mathbb P_n}[g]$ for every $n$, it follows that \[ I^-\leqslant\Gamma_m^-. \] Combining this with Step 3 gives the sandwich 
\begin{equation}
\Lambda_m^-\leqslant I^-\leqslant\Gamma_m^-.
\label{eqn_LambdaIGammasandwich}
\end{equation}

\emph{Step 5: Sandwich \eqref{eqn_LambdaIGammasandwich} converges.} Write $A:=\int_{S_C}L\,d\mu$ and $B:=\int_{S_C^c}L\,d\mu$. Since $L$ is continuous on compact $U$, it is uniformly continuous. Use \[ \omega_L(\delta) := \sup\{|L(\varphi)-L(\varphi')|:\varphi,\varphi'\in U,\ |\varphi-\varphi'|\leqslant\delta\}. \] Every adapted cell $B\in\mathcal B_m^-$ lies inside a dyadic interval of length $2^{-m}$. So \[ 0\leqslant A_m^u-A_m^\ell \leqslant \omega_L(2^{-m})\mu(S_C) \to0, \] and \[ 0\leqslant B_m^u-B_m^\ell \leqslant \omega_L(2^{-m})\mu(S_C^c) \to0. \] Also, \[ A_m^\ell,A_m^u\to A, \qquad B_m^\ell,B_m^u\to B. \] Therefore, $\Lambda_m^-= \frac{A_m^\ell}{A_m^\ell+B_m^u} \to \frac{A}{A+B}\ $ and $\ \Gamma_m^-= \frac{A_m^u}{A_m^u+B_m^\ell} \to \frac{A}{A+B}$, thus \eqref{eqn_LambdaIGammasandwich} implies \[I^-= \frac{A}{A+B} = \frac{\int_{S_C}L\,d\mu} {\int_U L\,d\mu}.\] 

\emph{Step 6: weak convergence of prior and posterior $\varphi$-marginals of stage-m extremal priors.} Let $\mathbb P_m^-=\sum_{B\in\mathcal B_m^-}p_B\delta_{\mathbf{x}_B^-}$ (where $\mathbf{x}_B^-=(\varphi^-_B,\eta^-_B)\in\overline{K_B}$ with $\varphi^-_B\in\overline B$) be a stage-$m$ extremal prior. For every continuous bounded $\psi:U\to\mathbb R$, \[ \left| \sum_Bp_B\psi(\varphi_B^-) - \int_U\psi(\varphi)\mu(d\varphi) \right| \leqslant \omega_\psi(2^{-m}), \] where $\varphi_B^-\in\overline B$ and $\omega_\psi(\delta)
:=
\sup\{|\psi(\varphi)-\psi(\varphi')|:
\varphi,\varphi'\in U,\ |\varphi-\varphi'|\leqslant \delta\}$. Since $\psi$ is continuous on compact $U$, $\omega_\psi(\delta)\to0$ as $\delta\downarrow0$. That is, $(\mathbb P_m^-)_\varphi\Rightarrow\mu$. 

Similarly, since $L\psi$ is continuous, \[ \sum_Bp_BL(\varphi_B^-)\psi(\varphi_B^-) \to \int_U L(\varphi)\psi(\varphi)\mu(d\varphi), \] and \[ \sum_Bp_BL(\varphi_B^-) \to \int_U L(\varphi)\mu(d\varphi). \] Therefore, the posterior $\varphi$-marginals satisfy \[ \frac{L(\varphi)(\mathbb P_m^-)_\varphi(d\varphi)} {\int_U L(\psi)(\mathbb P_m^-)_\varphi(d\psi)} \Rightarrow \frac{L(\varphi)\mu(d\varphi)} {\int_U L(\psi)\mu(d\psi)}.\qedhere \]
\end{proof}

\begin{lemma}[Choquet--dyadic construction of a vanishing-leakage sequence]
\label{lem:choquet-dyadic-vanishing-leakage}
Let $G:=K\setminus C$ and $H^-:=(S_C^c\times V)\cap C$. Under the assumptions of
Theorem~\ref{thm:2d-gk-fixed-point-extension-corrected}, there exists a
sequence $\{\mathbb P_m\}_{m\geqslant1}\subseteq\mathcal D$ such that $\mathbb P_m(H^-)\leqslant 2^{-m}$,
and hence $\mathbb P_m\bigl((S_C^c\times V)\cap C\bigr)\longrightarrow0$.

More precisely, for every $m\geqslant1$, there exists compact set $F_m\subseteq G\cap(S_C^c\times V)$, $\mathbb P_m\in\mathcal D$, and, for every
$r\geqslant1$, a finitely supported prior
$\mathbb Q_{m,r}\in\mathcal D_r^-$ such that:
\begin{enumerate}
\item $\mu\bigl(S_C^c\setminus\pi_U(F_m)\bigr)\leqslant2^{-m}$;
\item
$\mathbb P_m(H^-)\leqslant2^{-m},
\qquad
\mathbb Q_{m,r}(H^-)\leqslant2^{-m}$;
\item the two-dimensional dyadic cell-mass arrays associated with
$\{\mathbb Q_{m,r}\}_{r\geqslant1}$ are projectively consistent under
refinement;
\item for every continuous \(\Psi:K\to\mathbb R\),
\[
\left|
\int_K\Psi\,d\mathbb Q_{m,r}
-
\int_K\Psi\,d\mathbb P_m
\right|
\leqslant
\omega_\Psi\!\left(\sqrt{2}\,2^{-r}\right),
\]
where
\[
\omega_\Psi(\delta)
:=
\sup\left\{
|\Psi(\mathbf x)-\Psi(\mathbf y)|:
\mathbf x,\mathbf y\in K,\
\|\mathbf x-\mathbf y\|\leqslant\delta
\right\}.
\]
Consequently, for each fixed $m$, $\,\mathbb Q_{m,r}\Rightarrow\mathbb P_m\,$ as $\,r\to\infty$.
\end{enumerate}
\end{lemma}

\begin{proof}
Set the favourable relation $R:=G\cap(S_C^c\times V)$; in particular,  $R=G\,$ since $G\subseteq S_C^c\times V$. For every $\varphi\in S_C^c$, the definition of $S_C$ implies that
$K_\varphi\nsubseteq C_\varphi$. Hence, $G_\varphi=K_\varphi\setminus C_\varphi\neq\varnothing$, and therefore
\begin{equation}
\pi_U(R)=S_C^c.
\label{eqn_projection_R_SCc}
\end{equation}

For $E\subseteq K$, define the projection capacity
\[
\operatorname{Cap}(E)
:=
\mu^*\bigl(\pi_U(E)\bigr),
\]
where $\mu^*$ denotes the outer measure induced by $\mu$. This is
a Choquet projection capacity on the compact metric space
$K$. By Choquet's capacitability theorem---see 
\cite[Theorem~30.1, Example~37.2]{Choquet1954TheoryOfCapacities}, every analytic subset of a
compact metrizable space is capacitable with respect to a Choquet
capacity. Since $R$ is Borel, it is analytic and therefore capacitable:
\[
\operatorname{Cap}(R)
=
\sup\left\{
\operatorname{Cap}(F):
F\subseteq R,\ F\text{ compact}
\right\}.
\]
By \eqref{eqn_projection_R_SCc}, since $S^c_C$ is Borel hence $\mu$-measurable, $\operatorname{Cap}(R)=\mu^*(\pi_U(R))=\mu^*(S^c_C)=\mu(S_C^c)$. Thus, for every $m\geqslant1$, there exists a compact set $F_m\subseteq R$ such that $\,\operatorname{Cap}(F_m)
\geqslant
\operatorname{Cap}(R)-2^{-m}=\mu(S_C^c)-2^{-m}$.
Let $F_m^\pi:=\pi_U(F_m)$. Since $F_m$ is compact and $\pi_U$ is continuous,
$F_m^\pi$ is compact in $U$, hence Borel and therefore $\mu$-measurable. So $\operatorname{Cap}(F_m)=\mu^*(\pi_U(F_m))=\mu(F_m^\pi)$. Moreover,
$F_m^\pi\subseteq S_C^c$. It follows that
\begin{align}
\mu(S_C^c\setminus F_m^\pi)
&=
\mu(S_C^c)-\mu(F_m^\pi) \notag\\
&=
\operatorname{Cap}(R)-\operatorname{Cap}(F_m)
\leqslant 2^{-m}.
\label{eqn_compact_core_projection_error}
\end{align}

We next construct a Borel selector of the compact relation $F_m$ by
nested dyadic refinement in the $V$-coordinate. For
$\varphi\in F_m^\pi$, write $(F_m)_\varphi
:=
\{\eta\in V:(\varphi,\eta)\in F_m\}$. Every such section is the continuous $V$-projection of $F_m\cap(\{\varphi\}\times V)$,
which is nonempty because $\varphi\in\pi_U(F_m)$, and compact because the
intersection of two compact sets is compact. Hence, $(F_m)_\varphi$ is nonempty and
compact. Let $J_0(\varphi):=[0,1]$. Suppose that the closed dyadic interval
$J_r(\varphi)$ has been chosen. Write the left and right closed dyadic halves of $J_r(\varphi)$ as
$J_r^0(\varphi)$ and $J_r^1(\varphi)$, respectively. Define
\[
J_{r+1}(\varphi)
:=
\begin{cases}
J_r^0(\varphi),
&
(F_m)_\varphi\cap J_r^0(\varphi)\neq\varnothing,
\\[1mm]
J_r^1(\varphi),
&
(F_m)_\varphi\cap J_r^0(\varphi)=\varnothing.
\end{cases}
\]
At stage $r+1$, the recursion selects as $J_{r+1}(\varphi)$ a
dyadic child of $J_r(\varphi)$ that intersects the section
$(F_m)_\varphi$. The $0$-child at stage $r$ is chosen as the
tie-breaker when both children intersect the section. The construction is well-defined because $(F_m)_\varphi\cap J_r(\varphi)\neq\varnothing$: this implies at least one dyadic child of $J_r(\varphi)$ has a non-empty intersection with $(F_m)_\varphi$.

For any fixed closed dyadic interval $J\subseteq V$, $\{\varphi\in F_m^\pi:(F_m)_\varphi\cap J\neq\varnothing\}
=
\pi_U\bigl(F_m\cap(U\times J)\bigr)$. The set on the right-hand side is compact, because it is the
projection of a compact set. Hence, every branch decision in the above
recursion is Borel measurable.

The intervals $J_r(\varphi)$ are nested and satisfy $\operatorname{diam}J_r(\varphi)=2^{-r}$.
Thus, $\bigcap_{r\geqslant0}J_r(\varphi)
=
\{s_m(\varphi)\}$ for a unique point $s_m(\varphi)\in V$. The map $s_m:F_m^\pi\to V$ is Borel measurable because there is a sequence of
Borel-measurable step functions $c_r:F_m^\pi\to V$ that converge
pointwise to $s_m$. This sequence is constructed out of the following Borel sets. Let $\mathscr D_r$ denote the finite collection of level-$r$ dyadic
intervals in $V$. For each $I\in\mathscr D_r$, define $A_{r,I}
:=
\{\varphi\in F_m^\pi:J_r(\varphi)=I\}$. The sets $A_{r,I}$ (for $I\in\mathscr D_r$) form a finite Borel
partition of $F_m^\pi$. To see their Borelness, observe first that,
for every closed dyadic interval \(I\subseteq V\), the hitting set
\[
H_I
:=
\{\varphi\in F_m^\pi:(F_m)_\varphi\cap I\neq\varnothing\}
=
\pi_U\bigl(F_m\cap(U\times I)\bigr)
\]
is compact, and hence Borel. The Borelness of the sets $A_{r,I}$ now
follows inductively from the recursive construction. Indeed, if
$I^0$ and $I^1$ are the two children of $I\in\mathscr D_r$, then $\,A_{r+1,I^0}
=
A_{r,I}\cap H_{I^0}\,$ and $\,A_{r+1,I^1}
=
A_{r,I}\cap H_{I^0}^{\,c}$ are each Borel as intersections of Borel sets.

Now, for each $r$, let $c_r(\varphi)$ be the midpoint of
$J_r(\varphi)$. Then $c_r$ is a step function: on each set
$A_{r,I}$, it takes the constant value $\operatorname{mid}(I)$.
Consequently, for every Borel set $E\subseteq V$,
\[
c_r^{-1}(E)
=
\bigcup_{\substack{I\in\mathscr D_r\\
\operatorname{mid}(I)\in E}}
A_{r,I}.
\]
This is a finite union of Borel sets, so $c_r$ is Borel measurable. Finally, both $c_r(\varphi)$ and $s_m(\varphi)$ belong to
$J_r(\varphi)$. Since $c_r(\varphi)$ is the midpoint of this
interval and $\operatorname{diam}J_r(\varphi)=2^{-r}$, we have $\lvert c_r(\varphi)-s_m(\varphi)\rvert
\leqslant 2^{-r-1}$. Thus, $c_r(\varphi)\to s_m(\varphi)$ for every
$\varphi\in F_m^\pi$. Since a pointwise limit of real-valued Borel
functions is Borel measurable, $s_m$ is Borel measurable.

Preferred points in $G$ for each $\varphi$, so  points in $R$, can be chosen using $s_m$. Indeed, for each $r$, choose $\eta_r\in(F_m)_\varphi\cap J_r(\varphi)$. Then $\eta_r\to s_m(\varphi)$. Since $(F_m)_\varphi$ is compact, $s_m(\varphi)\in(F_m)_\varphi$. Therefore,
\begin{equation}
(\varphi,s_m(\varphi))
\in F_m\subseteq G
\qquad
\text{for every }\varphi\in F_m^\pi.
\label{eqn_compact_core_selector_in_G}
\end{equation}
Apply the same nested-dyadic construction to the compact relation
$K$. Since every section $K_\varphi$ is nonempty, there exists a
Borel map $t:U\to V$ such that
\begin{equation}
(\varphi,t(\varphi))\in K
\qquad
\text{for every }\varphi\in U.
\label{eqn_global_K_selector}
\end{equation}
Define
\[
\widehat s_m(\varphi)
:=
\begin{cases}
s_m(\varphi),
&
\varphi\in F_m^\pi,
\\
t(\varphi),
&
\varphi\notin F_m^\pi,
\end{cases}
\]
The map $T_m:U\to K$ defined as $T_m(\varphi)
:=
(\varphi,\widehat s_m(\varphi))$ is Borel, since its coordinate functions are. Define
\begin{equation}
\mathbb P_m:=(T_m)_\#\mu.
\label{eqn_exact_repaired_prior_Pm}
\end{equation}
This prior is concentrated on the graph of $\widehat s_m$. Since the first coordinate of \(T_m(\varphi)\) is \(\varphi\), the U-marginal of $\mathbb P_m$ is $\mu$. Hence, $\mathbb P_m\in\mathcal D$.

By \eqref{eqn_compact_core_selector_in_G}, the selected point belongs
to $G$ whenever $\varphi\in F_m^\pi$. Since $F_m^\pi\subseteq S_C^c$,
leakage into $C$ over $S_C^c$ can occur only when
$\varphi\in S_C^c\setminus F_m^\pi$. Therefore,
\[
\begin{aligned}
\mathbb P_m(H^-)
&=\mathbb P_m((S^c_C\times V)\cap C) \\
&=
\mu\left(
\left\{
\varphi\in S_C^c:
(\varphi,\widehat s_m(\varphi))\in C
\right\}
\right)
\\
&\leqslant
\mu(S_C^c\setminus F_m^\pi)
\\
&\leqslant
2^{-m},
\end{aligned}
\]
where the final inequality follows from
\eqref{eqn_compact_core_projection_error}. Thus, $\,\mathbb P_m\bigl((S_C^c\times V)\cap C\bigr)
\leqslant2^{-m}
\longrightarrow0$.

It remains to relate the repaired prior $\mathbb P_m$ to
two-dimensional dyadic cell repairs satisfying the finite
$U$-row-mass constraints.

Fix $m$. For $r\geqslant1$, let $\mathcal J_r$ be the level-$r$
dyadic partition of $U$, and let $\mathcal I_r$ be the level-$r$
dyadic partition of $V$. Use half-open dyadic intervals, with the
final interval closed at $1$, so that each point belongs to a unique
cell. Partition $U$ into the three Borel sets
\[
L_{m,1}:=F_m^\pi,
\qquad
L_{m,2}:=S_C^c\setminus F_m^\pi,
\qquad
L_{m,3}:=S_C.
\]
For $\,I\in\mathcal J_r$ and $\,J\in\mathcal I_r$ (where $q\in\{1,2,3\}$),
define
\begin{equation}
E_{I,J,q}^{m,r}
:=
\left\{
\varphi\in I\cap L_{m,q}:
\widehat s_m(\varphi)\in J
\right\}.
\label{eqn_dyadic_graph_atoms}
\end{equation}
These sets form a finite Borel partition of $U$. Put $\,w_{I,J,q}^{m,r}
:=
\mu(E_{I,J,q}^{m,r})$. For every nonempty atom $E_{I,J,q}^{m,r}$, choose
$\,\varphi_{I,J,q}^{m,r}
\in E_{I,J,q}^{m,r}\,$ and define $\,\mathbf x_{I,J,q}^{m,r}
:=
T_m(\varphi_{I,J,q}^{m,r})$. Then $\,\mathbf x_{I,J,q}^{m,r}
\in K\cap(I\times J)$. Define the finitely supported prior
\begin{equation}
\mathbb Q_{m,r}
:=
\sum_{I\in\mathcal J_r}
\sum_{J\in\mathcal I_r}
\sum_{q=1}^3
w_{I,J,q}^{m,r}
\delta_{\mathbf x_{I,J,q}^{m,r}},
\label{eqn_dyadic_repaired_prior_Qmr}
\end{equation}
where zero-mass terms are omitted.

For every $I\in\mathcal J_r$,
$\,\sum_{J\in\mathcal I_r}
w_{I,J,3}^{m,r}
=
\mu(I\cap S_C)\,$
and $\,\sum_{J\in\mathcal I_r}
\left(
w_{I,J,1}^{m,r}
+
w_{I,J,2}^{m,r}
\right)
=
\mu(I\cap S_C^c)$. Every atom indexed by $q=3$ lies in $K\cap((I\cap S_C)\times V)$,
whereas every atom indexed by $q=1$ or $q=2$ lies in $K\cap((I\cap S_C^c)\times V)$.
Consequently, $\mathbb Q_{m,r}(K_B)=\mu(B)\,$ for every $\,B\in\mathcal B_r^-$,
and therefore
\begin{equation}
\mathbb Q_{m,r}\in\mathcal D_r^-.
\label{eqn_Qmr_in_Dr}
\end{equation}

If $q=1$, then $\mathbf x_{I,J,1}^{m,r}
\in F_m\subseteq G$. If $q=3$, then the first coordinate belongs to $S_C$, so the atom
cannot belong to $H^-$. Hence, only the $q=2$ atoms can contribute
to $H^-$, and therefore
\[
\begin{aligned}
\mathbb Q_{m,r}(H^-)
&\leqslant
\sum_{I\in\mathcal J_r}
\sum_{J\in\mathcal I_r}
w_{I,J,2}^{m,r}
\\
&=
\mu(S_C^c\setminus F_m^\pi)
\\
&\leqslant
2^{-m}.
\end{aligned}
\]

The cell-mass arrays are projectively consistent. Indeed, if
$I\in\mathcal J_r$ and $J\in\mathcal I_r$, then
\[
E_{I,J,q}^{m,r}
=
\bigsqcup_{\substack{
I'\in\mathcal J_{r+1},\ I'\subseteq I\\
J'\in\mathcal I_{r+1},\ J'\subseteq J
}}
E_{I',J',q}^{m,r+1}.
\]
Consequently,
\begin{equation}
w_{I,J,q}^{m,r}
=
\sum_{\substack{
I'\in\mathcal J_{r+1},\ I'\subseteq I\\
J'\in\mathcal I_{r+1},\ J'\subseteq J
}}
w_{I',J',q}^{m,r+1}.
\label{eqn_projective_cell_mass_consistency}
\end{equation}
Thus, every finer repair aggregates to the same coarser two-dimensional cell-mass array.

Finally, let $\Psi:K\to\mathbb R$ be continuous. Couple
$T_m(\varphi)$, distributed according to $\mathbb P_m$, with the
representative point
$\mathbf x_{I,J,q}^{m,r}$ whenever
$\varphi\in E_{I,J,q}^{m,r}$. Both points lie in the same product
cell $I\times J$, whose Euclidean diameter is at most $\sqrt{2}\,2^{-r}$. Therefore,
\[
\begin{aligned}
\left|
\int_K\Psi\,d\mathbb Q_{m,r}
-
\int_K\Psi\,d\mathbb P_m
\right|
\ \leqslant\ 
\sum_{I,J,q}
\int_{E_{I,J,q}^{m,r}}
\left|
\Psi(\mathbf x_{I,J,q}^{m,r})
-
\Psi(T_m(\varphi))
\right|
\mu(d\varphi)\ \leqslant
\ \omega_\Psi\!\left(\sqrt{2}\,2^{-r}\right).
\end{aligned}
\]
Since $\Psi$ is uniformly continuous on compact $K$, $\,\omega_\Psi\!\left(\sqrt{2}\,2^{-r}\right)\longrightarrow0$. Hence, $\mathbb Q_{m,r}\Rightarrow\mathbb P_m\,$ as $\,r\to\infty$. In particular, the diagonal sequence $\{\mathbb Q_{m,m}\}$ gives finitely supported priors satisfying $\,\mathbb Q_{m,m}\in\mathcal D_m^-$, 
$\,\mathbb Q_{m,m}(H^-)\leqslant2^{-m}$, while the exact-marginal priors $\mathbb P_m\in\mathcal D$ satisfy $\,\mathbb P_m(H^-)\leqslant2^{-m}$. By the bounds established in the main proof of
Theorem~\ref{thm:2d-gk-fixed-point-extension-corrected}, the objective
values of the diagonal priors $\mathbb Q_{m,m}$ converge to the
target robust lower bound, while the priors $\mathbb P_m$ form the
required vanishing-leakage sequence.
\end{proof}

\newpage
\section{Proof of Theorem~\ref{thm_VaRExample} and Corollary~\ref{Cor_VaRExample}}
\begin{proof}
We organise the proof around the two constrained optimisation problems in
\eqref{eqn_event_extremal_problems}. Since $I(G)=1-V(G^c)$,
it is sufficient to solve the maximizing problem
\begin{equation}
\phi_G^+
=
\sup_{\mathbb P\in\mathcal D}\mathbb P(G).
\label{eqn_proof_target_problem}
\end{equation}
The minimizing problem then follows by applying the same arguments to
$G^c$.

The proofs of Theorem~\ref{thm_VaRExample} (including \eqref{eqn_proof_target_problem}) and Corollary~\ref{Cor_VaRExample} proceed through four
successive stages.

\begin{enumerate}
\item
\emph{Reduction of the constrained feasible set to exchange space.}
Starting from a candidate prior $\mathbb P\in\mathcal D$, we show that
every feasible competitor in $\mathcal D$ is obtained from $\mathbb P$ by
a complete admissible mass exchange. The constrained optimisation problem
can therefore be rewritten as an optimisation over
$\mathfrak X(\mathbb P)$. This reduction produces the exchange-space
Dinkelbach difference, the global $(E1)_G$--$(E2)_G$ system, and the residual identity in part \textup{(a)}.

\item
\emph{Reduction to marginal-completion fibres.}
To determine what global extremality implies for the support and masses of
an extremal prior, we apply the same exchange reduction to every removable
submeasure $0<\alpha\leqslant\mathbb P^\ast$. The resulting constrained
problem is the optimisation of $\beta(G)$ over
$\mathcal C(\alpha)$. Its Dinkelbach form yields the fibrewise fixed-point
conditions, hereditary fibre-extremality, and the tight-exchange
description of the complete extremal family in part \textup{(b)}.

\item
\emph{Finite approximation of Borel-event problems.}
Since $G$ is Borel (in particular, analytic), we first approximate $G$ from within by
increasing compact sets $F_\ell$. For each fixed compact core, we then
construct a nested family of dyadic whole-cell optimisation problems.
Their feasible points are represented by cell-mass tables, and each such
table admits a realisation in $\mathcal D$ with the same
whole-cell objective value. This yields the discretised extremal values
appearing in part \textup{(c)}.

\item
\emph{Passage to limits and application to robust VaR aggregation.}
For each compact core, weak compactness, together with the dyadic
bounded-Lipschitz estimate, yields a sequence of discrete extremal priors
converging weakly to a prior attaining the compact-core extremal value.
The compact-core and discretisation limits then prove part \textup{(c)}.

For the VaR application, the paired quantile versions in
\eqref{eqn_VaR_quantile_versions} generate the two threshold-event
families $C_s$ and $G_s$ in \eqref{eqn_VaR_threshold_events}. The
compact-target arguments above apply directly to both families, yielding attained extremal probabilities, exchange fixed-point
pairs, and convergent finite fixed-marginal approximations. The symmetric
representations in \eqref{eqn_symmetric_VaR_representation} then identify
the lower and upper robust VaR bounds with the corresponding threshold
inversions. Compact-set continuity closes the upper feasible-threshold
set and shows that its extremal threshold, and hence the robust upper VaR,
is attained. This proves Corollary~\ref{Cor_VaRExample}.
\end{enumerate}

We now carry out these four stages.

\medskip
\noindent
\textbf{I. Reduction of the constrained feasible set to exchange space.}

\noindent
\emph{Step 1: Complete exchanges describe exactly the feasible competitors
in \eqref{eqn_proof_target_problem}.}

Let $\mathbb P\in\mathcal D$ and
$(\alpha,\beta)\in\mathfrak X(\mathbb P)$. Since $0\leqslant\alpha\leqslant\mathbb P$,
one has $\mathbb P-\alpha\geqslant0$.
Thus, $T_{\alpha,\beta}\mathbb P
=
\mathbb P-\alpha+\beta$ is a positive measure. Moreover, $\bigl(T_{\alpha,\beta}\mathbb P\bigr)(K)
=
1-\alpha(K)+\beta(K)
=
1$. For every $j=1,\ldots,d$, $\,(\operatorname{pr}_j)_\#
T_{\alpha,\beta}\mathbb P
=
(\operatorname{pr}_j)_\#\mathbb P
-
(\operatorname{pr}_j)_\#\alpha
+
(\operatorname{pr}_j)_\#\beta =
\lambda$. Hence,
\begin{equation}
T_{\alpha,\beta}\mathbb P\in\mathcal D.
\label{eqn_exchange_feasibility_proof}
\end{equation}
The corresponding change in the target objective is
\begin{equation}
\bigl(T_{\alpha,\beta}\mathbb P\bigr)(G)-\mathbb P(G)
=
\beta(G)-\alpha(G)
=
\Delta_G(\alpha,\beta).
\label{eqn_exchange_change_proof}
\end{equation}
Conversely, let $\mathbb Q\in\mathcal D$ be an arbitrary feasible
prior in \eqref{eqn_proof_target_problem} different from $\mathbb P$. Take the Jordan
decomposition $\mathbb Q-\mathbb P
=
\beta-\alpha$, where $\alpha$ and $\beta$ are mutually-singular non-negative measures. Since $\alpha=-(\mathbb Q-\mathbb P)|_N$ where $N$ is a negative set of the decomposition, we have $\alpha(A)=\mathbb P(A\cap N)-\mathbb Q(A\cap N)\leqslant \mathbb P(A\cap N)\leqslant \mathbb P(A)$ for all Borel sets $A$. Hence, $0\leqslant\alpha\leqslant\mathbb P$. An analogous argument shows $0\leqslant\beta\leqslant\mathbb Q$. Since $\mathbb Q$ and $\mathbb P$ also have the same coordinate marginals, $(\operatorname{pr}_j)_\#
(\mathbb Q-\mathbb P)
=
0$ implies $(\operatorname{pr}_j)_\#\beta
=
(\operatorname{pr}_j)_\#\alpha$ by the linearity of push-forwards, for $j=1,\ldots,d$. We conclude that $(\alpha,\beta)\in\mathfrak X(\mathbb P)$ and $\mathbb Q
=
T_{\alpha,\beta}\mathbb P$.

Therefore the feasible set of \eqref{eqn_proof_target_problem}, viewed from
any $\mathbb P\in\mathcal D$, is precisely its complete exchange orbit:
\begin{equation}
\mathcal D
=
\left\{
T_{\alpha,\beta}\mathbb P:
(\alpha,\beta)\in\mathfrak X(\mathbb P)
\right\}.
\label{eqn_exchange_orbit_equals_D}
\end{equation}
Consequently,
\begin{equation}
V(G)
=
\mathbb P(G)
+
\sup_{(\alpha,\beta)\in\mathfrak X(\mathbb P)}
\Delta_G(\alpha,\beta).
\label{eqn_target_problem_exchange_reduction}
\end{equation}
This is the fundamental reduction of the target constrained problem to the
fixed-point framework.

\medskip
\noindent
\emph{Step 2: The global $(E1)_G$--$(E2)_G$ system solves the target
problem.}

Suppose first that $\mathbb P^\ast$ solves
\eqref{eqn_proof_target_problem}, and put $\phi^\ast
:=
\mathbb P^\ast(G)$. For every
$(\alpha,\beta)\in\mathfrak X(\mathbb P^\ast)$,
\eqref{eqn_exchange_feasibility_proof} gives $T_{\alpha,\beta}\mathbb P^\ast\in\mathcal D$. Optimality therefore implies $\bigl(T_{\alpha,\beta}\mathbb P^\ast\bigr)(G)
\leqslant
\mathbb P^\ast(G)
=
\phi^\ast$. Equivalently, $h_{\phi^\ast,G}^{\mathbb P^\ast}(\alpha,\beta)
\leqslant 0$. The zero exchange has value $h_{\phi^\ast,G}^{\mathbb P^\ast}(0,0)
=
\mathbb P^\ast(G)-\phi^\ast
=
0$. Hence $(E1)_G$ holds, while $(E2)_G$ holds by the definition of
$\phi^\ast$.

Conversely, suppose that
$(\mathbb P^\ast,\phi^\ast)$ satisfies $(E1)_G$ and $(E2)_G$. Let
$\mathbb Q\in\mathcal D$. By
\eqref{eqn_exchange_orbit_equals_D}, there exists
$(\alpha,\beta)\in\mathfrak X(\mathbb P^\ast)$ such that $\mathbb Q
=
T_{\alpha,\beta}\mathbb P^\ast$. Using $(E2)_G$, condition $(E1)_G$ yields $0
\geqslant
h_{\phi^\ast,G}^{\mathbb P^\ast}(\alpha,\beta)
=
\mathbb Q(G)-\mathbb P^\ast(G)$.
Thus, $\mathbb Q(G)\leqslant\mathbb P^\ast(G)$ for all $\mathbb Q\in\mathcal D$. Therefore, $\mathbb P^\ast$ solves
\eqref{eqn_proof_target_problem} and $\phi^\ast
=
\mathbb P^\ast(G)
=
V(G)$. This proves the fixed-point characterisation in part \textup{(a)}.

\medskip
\noindent
\emph{Step 3: The exchange Dinkelbach difference and the global
residual.}

By \eqref{eqn_exchange_change_proof}, the Dinkelbach difference associated
with a feasible exchange is $\Delta_G(\alpha,\beta)
=
\beta(G)-\alpha(G)$. For every
$(\alpha,\beta)\in\mathfrak X(\mathbb P)$, $T_{\alpha,\beta}\mathbb P\in\mathcal D$, and hence $\Delta_G(\alpha,\beta)
=
\bigl(T_{\alpha,\beta}\mathbb P\bigr)(G)-\mathbb P(G)
\leqslant
V(G)-\mathbb P(G)$. Taking the supremum gives $\mathscr R_G(\mathbb P)
\leqslant
V(G)-\mathbb P(G)$.

Conversely, for every $\mathbb Q\in\mathcal D$, Step~1 produces an
admissible exchange such that $\Delta_G(\alpha,\beta)
=
\mathbb Q(G)-\mathbb P(G)$. Taking the supremum over $\mathbb Q\in\mathcal D$ gives $\mathscr R_G(\mathbb P)
\geqslant
V(G)-\mathbb P(G)$. Therefore $\mathscr R_G(\mathbb P)
=
V(G)-\mathbb P(G)$. 

It follows immediately that $\mathscr R_G(\mathbb P)=0
\Longleftrightarrow
\mathbb P\in
\operatorname*{arg\,max}_{\mathbb Q\in\mathcal D}\mathbb Q(G)$,
and $\,\mathbb P_r(G)\longrightarrow V(G)
\Longleftrightarrow
\mathscr R_G(\mathbb P_r)\longrightarrow0$. Thus, the residual $\mathscr R_G(\mathbb P)$ is not merely a local stationarity quantity: over the
complete exchange class it is exactly the remaining value in the target
optimisation problem. If the residual is instead computed over a proper exchange subfamily, its
vanishing establishes only stability with respect to that subfamily unless
the subfamily is proved to generate every feasible displacement in
\eqref{eqn_exchange_orbit_equals_D}.

This completes the proof of part \textup{(a)} for the maximizing problem.
The minimizing assertions follow by replacing $G$ with $G^c$ and using $I(G)=1-V(G^c)$.

\medskip
\noindent
\textbf{II. Reduction to marginal-completion fibres.}

The global system above determines whether a prior solves the target
problem. We now determine what that solution implies for each removable
part of the extremal prior. Fix an extremiser $\mathbb P^\ast$ and a
nonzero submeasure $0<\alpha\leqslant\mathbb P^\ast$. Replacing $\alpha$ while retaining its coordinate marginals leads to the
constrained fibre problem
\begin{equation}
\Gamma_G(\alpha)
=
\sup_{\beta\in\mathcal C(\alpha)}\beta(G).
\label{eqn_proof_fibre_target}
\end{equation}

\noindent
\emph{Step 4: Dinkelbach reduction of the fibre problem.}

Every $\beta\in\mathcal C(\alpha)$ has $\beta(K)=\alpha(K)$. Therefore, $\Psi_{\alpha,G}(\rho)=
\sup_{\beta\in\mathcal C(\alpha)}
\left[
\beta(G)-\rho\beta(K)
\right]=
\sup_{\beta\in\mathcal C(\alpha)}
\beta(G)
-
\rho\alpha(K)=
\Gamma_G(\alpha)-\rho\alpha(K)$.
Since $\alpha(K)>0$, this affine function has the unique root $\widehat\rho_G(\alpha)
=
\frac{\Gamma_G(\alpha)}{\alpha(K)}$. If $(E2)_{\alpha,G}$ holds, then $\rho
=
\frac{\alpha(G)}{\alpha(K)}\,$ and $\,h_{\rho,G}^{\alpha}(\alpha)
=
\alpha(G)
-
\frac{\alpha(G)}{\alpha(K)}
\alpha(K)
=
0$.
Thus, $(E1)_{\alpha,G}$ holds if and only if $h_{\rho,G}^{\alpha}(\beta)\leqslant 0$ for all $\beta\in\mathcal C(\alpha)$,
which is equivalent to $\beta(G)\leqslant\alpha(G)$ for all $\beta\in\mathcal C(\alpha)$. Because $\alpha\in\mathcal C(\alpha)$,
this is equivalent to $\Gamma_G(\alpha)=\alpha(G)$. Hence, the fibrewise $(E1)$--$(E2)$ system solves
\eqref{eqn_proof_fibre_target} exactly.

\medskip
\noindent
\emph{Step 5: Hereditary fibre-extremality characterises extremal priors.}

Suppose that $\mathbb P^\ast$ solves
\eqref{eqn_proof_target_problem}. Let
$0<\alpha\leqslant\mathbb P^\ast$ and let
$\beta\in\mathcal C(\alpha)$. Then $\mathbb P^\ast-\alpha+\beta\in\mathcal D$. Optimality of $\mathbb P^\ast$ gives $\mathbb P^\ast(G)-\alpha(G)+\beta(G)
\leqslant
\mathbb P^\ast(G)$. Therefore, $\beta(G)\leqslant\alpha(G)$ for all $\beta\in\mathcal C(\alpha)$, and hence $\Gamma_G(\alpha)=\alpha(G)$. Thus, every removable submeasure of an extremal prior solves its own
marginal-completion problem \eqref{eqn_proof_fibre_target}.

Conversely, suppose that $\Gamma_G(\alpha)=\alpha(G)$ for all $0<\alpha\leqslant\mathbb P^\ast$. Take $\alpha=\mathbb P^\ast$. Since the marginals of $\mathbb P^\ast$ are all $\lambda$, $\mathcal C(\mathbb P^\ast)
=
\mathcal D$. Consequently, $\mathbb P^\ast(G)
=
\Gamma_G(\mathbb P^\ast)
=
\sup_{\mathbb Q\in\mathcal D}\mathbb Q(G)
=
V(G)$. Thus, $\mathbb P^\ast$ solves the original target problem.

This proves the hereditary extremal prior characterisation in part \textup{(b)}.

\medskip
\noindent
\emph{Step 6: Tight exchanges describe the complete solution set.}

Let $\mathbb P^\ast$ be an extremiser. By the global fixed-point condition, $\Delta_G(\alpha,\beta)\leqslant0$ for all $(\alpha,\beta)\in\mathfrak X(\mathbb P^\ast)$. If $\Delta_G(\alpha,\beta)=0$,
then $\bigl(T_{\alpha,\beta}\mathbb P^\ast\bigr)(G)
=
\mathbb P^\ast(G)
=
V(G)$, so the exchanged prior is also an extremiser.

Conversely, let $\mathbb Q^\ast$ be any other solution of
\eqref{eqn_proof_target_problem}. By Step~1, the Jordan decomposition $\mathbb Q^\ast-\mathbb P^\ast
=
\beta-\alpha$ gives an admissible exchange from $\mathbb P^\ast$. Since both priors are
extremal, $\Delta_G(\alpha,\beta)
=
\mathbb Q^\ast(G)-\mathbb P^\ast(G)
=
0$. Thus, $\operatorname*{arg\,max}_{\mathbb P\in\mathcal D}\mathbb P(G)
=
\left\{
\mathbb P^\ast-\alpha+\beta:
(\alpha,\beta)\in\mathfrak X_G^0(\mathbb P^\ast)
\right\}$.

Therefore, zero Dinkelbach difference identifies exchanges tangent to the
extremal solution set, while negative Dinkelbach difference identifies
strictly worsening feasible displacements. This completes the proof of
part \textup{(b)}.

\medskip
\noindent
\textbf{III. Finite approximation of Borel-event problems.}
The preceding arguments identify the structure of the set of extremisers of \eqref{eqn_proof_target_problem}. Since the Borel event $G$ is analytic, direct
attainment need not be available. We therefore approximate $G$ from
within by compact-core problems and approximate each compact-core problem
by finite whole-cell fixed-point problems.

\medskip
\noindent
\emph{Step 7: Compactness of the feasible class and capacity of the target
value.}

Since $K$ is compact metric, $\mathcal P(K)$ is weakly compact. For each
$j$, the map $\mathbb P
\longmapsto
(\operatorname{pr}_j)_\#\mathbb P$ is weakly continuous. Indeed, for every $f\in C([0,1])$, $\int_{[0,1]} f\,d((\operatorname{pr}_j)_\#\mathbb P)
=
\int_Kf\circ\operatorname{pr}_j\,d\mathbb P$. Since the projection map $\operatorname{pr}_j$ is continuous, so is the composition $f\circ\operatorname{pr}_j$. Hence, $\mathbb P_n\Rightarrow\mathbb P$ implies $\int_{[0,1]} f\,d((\operatorname{pr}_j)_\#\mathbb P_n)=\int_Kf\circ\operatorname{pr}_j\,d\mathbb P_n
\longrightarrow
\int_Kf\circ\operatorname{pr}_j\,d\mathbb P=\int_{[0,1]} f\,d((\operatorname{pr}_j)_\#\mathbb P)$, so $(\operatorname{pr}_j)_\#\mathbb P_n\Rightarrow(\operatorname{pr}_j)_\#\mathbb P$. 

Therefore, each constraint $(\operatorname{pr}_j)_\#\mathbb P=\lambda$ defines a weakly closed subset of $\mathcal P(K)$: if $(\operatorname{pr}_j)_\#\mathbb P_n=\lambda$ for all $n$ and $\mathbb P_n\Rightarrow\mathbb P$, then $\lambda=(\operatorname{pr}_j)_\#\mathbb P_n\Rightarrow (\operatorname{pr}_j)_\#\mathbb P$ so $(\operatorname{pr}_j)_\#\mathbb P=\lambda$ as well, which means $\mathbb P$ satisfies the constraint. The finite
intersection of these weakly closed subsets is weakly closed and equals $\mathcal D$. Since $\mathcal P(K)$ is weakly compact, every weakly closed subset is weakly compact, so $\mathcal D$ is weakly compact.

We next verify the capacity properties needed to approximate the target event. On nested events, the monotonicity of $V$ is immediate from the monotonicity of probability measures. Now, let $A_r\uparrow A$.
For every $\mathbb P\in\mathcal D$, $\mathbb P(A_r)\uparrow\mathbb P(A)$,
and therefore $V(A)
=
\sup_{\mathbb P\in\mathcal D}\sup_r\mathbb P(A_r)=
\sup_r\sup_{\mathbb P\in\mathcal D}\mathbb P(A_r)=
\sup_rV(A_r)$. Thus, $V$ is continuous from below.

Next let $F_r\downarrow F$, where every $F_r$ is compact. Put $a_r:=V(F_r)$. Then $a_r\downarrow a$ for some $a\geqslant V(F)$. For each $r$, by the definition of $V(F_r)$ as a supremum over $\mathcal D$, there is a sequence of priors $\{\mathbb P_{r,\ell}\}\subseteq\mathcal D$ that satisfy $\mathbb P_{r,\ell}(F_r)\rightarrow V(F_r)$ as $\ell\rightarrow\infty$. Since $\mathcal D$ is weakly compact, a weakly convergent subsequence $\{\mathbb P_{r,\ell_k}\}$ satisfies $\mathbb P_{r,\ell_k}\Rightarrow\mathbb P_r\in\mathcal D$. Since $F_r$ is compact (thus closed), the Portmanteau theorem implies $V(F_r)=\operatorname{limsup}_{k\to\infty}\mathbb P_{r,\ell_k}(F_r)\leqslant \mathbb P_r(F_r)\leqslant \sup_{\mathbb P\in\mathcal D}\mathbb P(F_r)=V(F_r)$. Hence, $\mathbb P_r(F_r)=V(F_r)$ for all $r$.

Since $\{\mathbb P_r\}\subseteq\mathcal D$ and $\mathcal D$ is weakly compact, extract a convergent subsequence $\{\mathbb P_{r_k}\}$ that satisfies $\mathbb P_{r_k}\Rightarrow\mathbb P^\ast\in\mathcal D$. Fix $q$. For all sufficiently large $k$, $F_{r_k}\subseteq F_q$, so $a_{r_k}
=
\mathbb P_{r_k}(F_{r_k})
\leqslant
\mathbb P_{r_k}(F_q)$. Since $F_q$ is closed, the Portmanteau theorem gives $a=\lim_{r\to\infty}V(F_r)=\lim_{k\to\infty}V(F_{r_k})=\lim_{k\to\infty}\mathbb P_{r_k}(F_{r_k})
\leqslant
\limsup_{k\to\infty}\mathbb P_{r_k}(F_q)
\leqslant
\mathbb P^\ast(F_q)$. Letting $q\to\infty$ and using continuity from above for the fixed measure
$\mathbb P^\ast$ yields $a
\leqslant
\mathbb P^\ast(F)
\leqslant
V(F)$. Since $V(F)\leqslant a$, it follows that $V(F_r)\downarrow V(F)$. Finally, note that $V(\varnothing)=0$.

Thus, $V$ has the properties of a Choquet capacity on Borel sets. Define the Borel outer extension of $V$ by $\overline V(A)
:=
\inf\left\{
V(B):
A\subseteq B,\  B\in\mathcal B(K)
\right\}$ for $A\subseteq K$. Monotonicity shows that $\overline V=V$ on $\mathcal B(K)$.
The continuity properties established above extend to $\overline V$.
Indeed, let $A_n\uparrow A$ and set
$c_n:=\overline V(A_n)$; since the $c_n$ are nondecreasing and bounded above by $1$, $c_n\uparrow c$ for some $c\leqslant\overline V(A)$. For $\varepsilon>0$, choose
Borel sets $B_n\supseteq A_n$ such that $V(B_n)\leqslant c_n+\varepsilon 2^{-n}$,
and define $C_n:=\bigcap_{m\geqslant n}B_m$, $C:=\bigcup_{n\geqslant1}C_n$. Then, $C_n\uparrow C$, $A_n\subseteq C_n$, and hence $A\subseteq C$.
Moreover, $C_n\subseteq B_n$, so continuity from below of $V$ on
Borel sets gives
\[
\overline V(A)
\leqslant V(C)
=\lim_{n\to\infty}V(C_n)
\leqslant\lim_{n\to\infty}
\bigl(c_n+\varepsilon2^{-n}\bigr)
=c.
\]
Thus, $\overline V(A_n)\uparrow\overline V(A)$. If
$F_n\downarrow F$ are compact, then $\overline V(F_n)=V(F_n)\downarrow V(F)=\overline V(F)$. Together with monotonicity and $\overline V(\varnothing)=0$, this
shows $\overline V$ is a Choquet capacity on $K$. 

Since $G$ is Borel, $G$ is analytic. So, by the Choquet capacitability theorem, $V(G)
=
\overline V(G)
=
\sup\limits_{\substack{F\subseteq G,\ F\ \mathrm{compact}}}
\overline V(F)
=
\sup\limits_{\substack{F\subseteq G,\ F\ \mathrm{compact}}}V(F)$. Choose compact sets $K_\ell\subseteq G$ satisfying $V(K_\ell)\geqslant V(G)-2^{-\ell}$,
and put $F_\ell
:=
\bigcup_{q\leqslant\ell}K_q$. Then $F_\ell\subseteq G$, $\,F_\ell\uparrow$, $\,V(F_\ell)\uparrow V(G)$. The original target problem is reduced to the compact-core
problems
\begin{equation}
V(F_\ell)
=
\sup_{\mathbb P\in\mathcal D}\mathbb P(F_\ell),
\qquad
\ell\geqslant1.
\label{eqn_compact_core_target_problems}
\end{equation}

\medskip
\noindent
\emph{Step 8: Finite whole-cell problems associated with a compact-core
target.}

Fix $\ell$. For $m\geqslant1$, let
\[
I_{m,k}
=
\begin{cases}
[k2^{-m},(k+1)2^{-m}),
&
k=0,\ldots,2^m-2,
\\[1mm]
[1-2^{-m},1],
&
k=2^m-1.
\end{cases}
\]
For $\mathcal K_m
:=
\{0,\ldots,2^m-1\}^d$, let $Q_{m,\mathbf k}
=
I_{m,k_1}\times\cdots\times I_{m,k_d},
\ 
\mathbf k=(k_1,\ldots,k_d)\in\mathcal K_m$, be the level-$m$ dyadic product cells. Define
\[
\mathcal D_m
:=
\left\{
\mathbb P\in\mathcal P(K):
\mathbb P\{\mathbf u:u_j\in I_{m,k}\}=2^{-m}
\quad
j=1,\ldots,d, \quad k=0,\ldots,2^m-1
\right\},
\]
and the corresponding cell-table polytope
\[
W_m
:=
\left\{
\mathbf p=(p_{\mathbf k})_{\mathbf k\in\mathcal K_m}\geqslant0:
\sum_{\mathbf k\in\mathcal K_m:k_j=k}p_{\mathbf k}=2^{-m},
\quad
j=1,\ldots,d, \quad k=0,\ldots,2^m-1
\right\}.
\]
As a shorthand, the marginal constraints defining $W_m$ can be written as $M_m\mathbf p=\mathbf r_m$, where $M_m$ is a $d2^m\times (2^m)^d$ indicator matrix, $\mathbf p$ is a $(2^m)^d\times 1$ vector of $p_{\mathbf k}$ probabilities, and $\mathbf r_m$ is a $d2^m\times 1$ vector whose entries are all $2^{-m}$.

Define the level-$m$ whole-cell cover of $F_\ell$ by $E_{\ell,m}
:=
\bigcup
\left\{
Q_{m,\mathbf k}:
Q_{m,\mathbf k}\cap F_\ell\neq\varnothing
\right\}$, and let $c_{\ell,m,\mathbf k}
:=
\mathbf 1_{\{Q_{m,\mathbf k}\subseteq E_{\ell,m}\}}$. The compact-core problem \eqref{eqn_compact_core_target_problems} is
therefore approximated by
\[
\phi_{\ell,m}^\ast
=
\max_{\mathbf p\in W_m}c_{\ell,m}^{\top}{\mathbf p}.
\]
The polytope $W_m$ is nonempty: for example, the cell table induced
by $\lambda^{\otimes d}$ belongs to $W_m$. Moreover, the marginal
equalities imply $\sum_{\mathbf k}p_{\mathbf k}=1$, so $W_m$ is closed and bounded, hence compact. Since
$\mathbf p\mapsto c_{\ell,m}^{\mathsf T}\mathbf p$ is continuous,
the maximum of this map over $W_m$ is attained.

\medskip
\noindent
\emph{Step 9: Lifting proves that the finite problem has no
stagewise relaxation gap.}

For $\mathbf p\in W_m$, define
\[
L_m(\mathbf p)
:=
\sum_{\mathbf k}
p_{\mathbf k}
\bigotimes_{j=1}^d
\left(
2^m\lambda|_{I_{m,k_j}}
\right).
\]
For a fixed coordinate $j$,
\[
\begin{aligned}
(\operatorname{pr}_j)_\#L_m(\mathbf p)
&=
\sum_{\mathbf k}
p_{\mathbf k}
\left(
2^m\lambda|_{I_{m,k_j}}
\right)\\
&=
\sum_{k=0}^{2^m-1}
\left(
\sum_{\mathbf k:k_j=k}p_{\mathbf k}
\right)
2^m\lambda|_{I_{m,k}}\\
&=
\sum_{k=0}^{2^m-1}
2^{-m}2^m\lambda|_{I_{m,k}}\\
&=
\lambda.
\end{aligned}
\]
Hence, $L_m(\mathbf p)\in\mathcal D$.
Moreover, $L_m(\mathbf p)(Q_{m,\mathbf k})
=
p_{\mathbf k}$. 

Every $\mathbb P\in\mathcal D_m$ induces the table $p_{\mathbf k}:=\mathbb P(Q_{m,\mathbf k})\in W_m$ and, since $E_{\ell,m}$ is a union of dyadic cells, $\mathbb P(E_{\ell,m})=c_{\ell,m}^{\top}{\mathbf p}$. Thus $\sup_{\mathbb P\in\mathcal D_m}\mathbb P(E_{\ell,m})
\leqslant
\max_{\mathbf p\in W_m}c_{\ell,m}^{\top}{\mathbf p}$. Furthermore, as we have just shown, every $\mathbf p\in W_m$ has a lift $L_m(\mathbf p)\in\mathcal D$
satisfying $L_m(\mathbf p)(Q_{m,\mathbf k})=p_{\mathbf k}$
for every $\mathbf k$, and therefore $L_m(\mathbf p)(E_{\ell,m})=c_{\ell,m}^{\top}{\mathbf p}$. Thus $\max_{\mathbf p\in W_m}c_{\ell,m}^{\top}{\mathbf p}
\leqslant
\sup_{\mathbb P\in\mathcal D}\mathbb P(E_{\ell,m})$.
Finally, $\mathcal D\subseteq\mathcal D_m$, thus $\sup_{\mathbb P\in\mathcal D}\mathbb P(E_{\ell,m})
\leqslant
\sup_{\mathbb P\in\mathcal D_m}\mathbb P(E_{\ell,m})$. These imply the following sandwich:
\begin{equation}
\sup_{\mathbb P\in\mathcal D_m}\mathbb P(E_{\ell,m})
\leqslant
\max_{\mathbf p\in W_m}c_{\ell,m}^{\top}{\mathbf p}
\leqslant
\sup_{\mathbb P\in\mathcal D}\mathbb P(E_{\ell,m})
\leqslant
\sup_{\mathbb P\in\mathcal D_m}\mathbb P(E_{\ell,m}).
\label{eqn_stagewise_no_gap_proof}
\end{equation}
So, these suprema are equal. The corresponding equality for infima is proved analogously.


\medskip
\noindent
\emph{Step 10: Each finite problem is itself an exchange fixed-point
problem.}
For $\mathbf p\in W_m$, define
\[
\mathfrak X_m(\mathbf p)
:=
\left\{
(\mathbf a,\mathbf b):
0\leqslant \mathbf a\leqslant \mathbf p,
\quad \mathbf b\geqslant 0,\quad
M_m\mathbf a=M_m\mathbf b
\right\}.
\]
If $(\mathbf a,\mathbf b)\in\mathfrak X_m(\mathbf p)$, then $\mathbf p-\mathbf a+\mathbf b\geqslant0\,$
and $\,M_m(\mathbf p-\mathbf a+\mathbf b)
=
\mathbf r_m-M_m\mathbf a+M_m\mathbf b
=
\mathbf r_m$. Thus $\,\mathbf p-\mathbf a+\mathbf b\in W_m$.

Conversely, if $\mathbf q\in W_m$, define $\mathbf a=(\mathbf p-\mathbf q)^+$ and $\mathbf b=(\mathbf q-\mathbf p)^+$. Then $\mathbf q=\mathbf p-\mathbf a+\mathbf b$,
$0\leqslant \mathbf a\leqslant \mathbf p$ and, since $M_m(\mathbf q-\mathbf p)=0$,
one finds $M_m\mathbf a=M_m\mathbf b$. 

This means every feasible $\mathbf q\in W_m$ is obtained from $\mathbf p$ by a
finite admissible exchange. Hence, $\,W_m
=
\left\{
\mathbf p-\mathbf a+\mathbf b:
(\mathbf a,\mathbf b)\in\mathfrak X_m(\mathbf p)
\right\}$. It follows, exactly as in Steps~2 and~3, that
$(\mathbf p_{\ell,m}^\ast,\phi_{\ell,m}^\ast)$ solves the finite target problem if
and only if $\,(0,0)
\in\allowbreak
\operatorname*{arg\,max}_{(\mathbf a,\mathbf b)\in\mathfrak X_m(\mathbf p_{\ell,m}^\ast)}
\allowbreak
\bigl\{
c_{\ell,m}^{\top}
\bigl(\mathbf p_{\ell,m}^\ast-\mathbf a+\mathbf b\bigr)
-\allowbreak
\phi_{\ell,m}^\ast
\bigr\}$
and $\phi_{\ell,m}^\ast
=
c_{\ell,m}^{\top}{\mathbf p}_{\ell,m}^\ast$. Moreover, the finite full residual is exactly the finite objective gap: for arbitrary $\mathbf p\in W_m$, $\,\mathscr R_{\ell,m}(\mathbf p):=
\max_{(\mathbf a,\mathbf b)\in\mathfrak X_m(\mathbf p)}
c_{\ell,m}^{\top}(\mathbf b-\mathbf a)=
\max_{\mathbf q\in W_m}
c_{\ell,m}^{\top}\mathbf q
-
c_{\ell,m}^{\top}{\mathbf p}$.

For every $0<\mathbf a\leqslant \mathbf p^\ast$ where $\mathbf p^\ast$ is extremal, define $\mathcal C_m(\mathbf a)
:=
\{\mathbf b\geqslant0:M_m\mathbf b=M_m\mathbf a\}$. Applying the fibre argument from Step~4 gives $c_{\ell,m}^{\top}\mathbf a
=
\max_{\mathbf b\in\mathcal C_m(\mathbf a)}
c_{\ell,m}^{\top}\mathbf b$. Equivalently, with $\rho_{\mathbf a}
=
\frac{c_{\ell,m}^{\top}\mathbf a}{\mathbf 1^{\top}\mathbf a}$, one has $\mathbf a
\in
\operatorname*{arg\,max}_{\mathbf b\in\mathcal C_m(\mathbf a)}
\bigl\{
c_{\ell,m}^{\top}\mathbf b
-
\rho_{\mathbf a}\mathbf 1^{\top}\mathbf b
\bigr\}$.

Writing $\mathbf d=\mathbf b-\mathbf a$,
finite feasibility is equivalently $M_m\mathbf d=0$ and $\mathbf p^\ast + \mathbf d\geqslant0$. Therefore, finite extremality is equivalent to $c_{\ell,m}^{\top}\mathbf d\leqslant0$ for every feasible $\text{ker}(M_m)$ direction. Directions satisfying $c_{\ell,m}^{\top}\mathbf d=0$ are tight and generate the complete finite extremal face.

\medskip
\noindent
\emph{Step 11: Same-stage exchange repair and lifting of the repaired
solution.}

Let $\widehat{\mathbb P}_m\in\mathcal D_m$ be any relaxed stage-$m$ prior and let $\mathbf p$, with $p_{\mathbf k}
=
\widehat{\mathbb P}_m(Q_{m,\mathbf k})$, be its cell table. Let $\mathbf p^\ast$ solve the finite target problem and let $\mathbf a=(\mathbf p-\mathbf p^\ast)^+$, $\mathbf b=(\mathbf p^\ast-\mathbf p)^+$. Then $\mathbf p^\ast=\mathbf p-\mathbf a+\mathbf b$, $0\leqslant\mathbf a\leqslant\mathbf p$ and $M_m\mathbf a=M_m\mathbf b$. Hence, $\mathbf p^\ast$ is obtained from $\mathbf p$ by one complete stage-$m$
marginal-preserving exchange.

For every cell with $p_{\mathbf k}>0$, define $\widehat\alpha|_{Q_{m,\mathbf k}}
=
\frac{a_{\mathbf k}}{p_{\mathbf k}}
\widehat{\mathbb P}_m|_{Q_{m,\mathbf k}}$,
and define this restriction to be zero when $p_{\mathbf k}=0$. Then $0\leqslant\widehat\alpha\leqslant\widehat{\mathbb P}_m$,
and the cell table of $\widehat\alpha$ is $\mathbf a$. Choose any probability measure $\kappa_{\mathbf k}$ supported in
$Q_{m,\mathbf k}$ and define $\widehat\beta
=
\sum_{\mathbf k}b_{\mathbf k}\kappa_{\mathbf k}$. The cell table of $\widehat\beta$ is $\mathbf b$. Put $\widehat{\mathbb P}_m^\ast
=
\widehat{\mathbb P}_m-\widehat\alpha+\widehat\beta$. Its cell table is $\mathbf p^\ast=\mathbf p-\mathbf a+\mathbf b$. Since $M_m\mathbf a=M_m\mathbf b$, the stage-$m$ strip marginals are unchanged, so $\widehat{\mathbb P}_m^\ast\in\mathcal D_m$. Moreover, $\widehat{\mathbb P}_m^\ast(E_{\ell,m})
=
c_{\ell,m}^{\top}{\mathbf p}^\ast$. The lift $\mathbb P_{\ell,m}^\ast
:=
L_m(\mathbf p^\ast)$ belongs to $\mathcal D$ and satisfies $\mathbb P_{\ell,m}^\ast(E_{\ell,m})
= c_{\ell,m}^{\top}{\mathbf p}^\ast$. 


\medskip
\noindent
\textbf{IV. Passage to limits and application to robust VaR aggregation.}

\noindent
\emph{Step 12: Finite exchange fixed-point pairs yield a fixed-point pair
for each compact-core target.}

Fix $\ell$ and set $\overline E_{\ell,m}
:=
\overline{E_{\ell,m}}$. The sets $\overline E_{\ell,m}$ are compact and decreasing, and $\bigcap_{m\geqslant1}\overline E_{\ell,m}
=
F_\ell$. Indeed, if $\mathbf x\notin F_\ell$, compactness of $F_\ell$ gives $\operatorname{dist}(\mathbf x,F_\ell)>0$, and all sufficiently fine dyadic cells near $\mathbf x$ are disjoint from
$F_\ell$.

The difference $\overline E_{\ell,m}\setminus E_{\ell,m}$ is contained in a finite union of dyadic coordinate hyperplanes. For every
$\mathbb P\in\mathcal D$, $\mathbb P\{\mathbf u:u_j=k2^{-m}\}
=
\lambda(\{k2^{-m}\})
=
0$. Hence, $\mathbb P(E_{\ell,m})
=
\mathbb P(\overline E_{\ell,m})$ for all 
$\mathbb P\in\mathcal D$,
and therefore $V(E_{\ell,m})
=
V(\overline E_{\ell,m})$. Since $\overline E_{\ell,m}\downarrow F_\ell$, the compact-set continuity
proved in Step~7 gives $V(E_{\ell,m})
=
V(\overline E_{\ell,m})
\downarrow
V(F_\ell)$. Using the no-gap identity \eqref{eqn_stagewise_no_gap_proof}, $\phi_{\ell,m}^\ast
=
V(E_{\ell,m})$ and, hence, $\phi_{\ell,m}^\ast
\downarrow
V(F_\ell)$.

For each $m$, let $\bigl(\mathbf p_{\ell,m}^{\ast},\phi_{\ell,m}^{\ast}\bigr)$ be a finite exchange fixed-point pair, and let $\mathbb P_{\ell,m}^{\ast}
=
L_m(\mathbf p_{\ell,m}^{\ast})$ be the related lifted prior in $\mathcal D$. Then, $\mathbb P_{\ell,m}^\ast(\overline E_{\ell,m})=\mathbb P_{\ell,m}^\ast(E_{\ell,m})
=
V(E_{\ell,m})
=
V(\overline E_{\ell,m})$. 
We now relate these priors to discrete extremal priors. For each dyadic
cell $Q_{m,\mathbf k}$, choose a point $\mathbf x_{\ell,m,\mathbf k}\in Q_{m,\mathbf k}$,
and define
\begin{equation}
\widehat{\mathbb P}_{\ell,m}^\ast
:=
\sum_{\mathbf k}
p_{\ell,m,\mathbf k}^\ast
\delta_{\mathbf x_{\ell,m,\mathbf k}}.
\label{eqn_discrete_extremal_prior}
\end{equation}
The measure $\widehat{\mathbb P}_{\ell,m}^\ast$ has cell table
$\mathbf p_{\ell,m}^\ast$. Consequently, $\widehat{\mathbb P}_{\ell,m}^\ast\in\mathcal D_m$ and, because $E_{\ell,m}$ is a union of whole dyadic cells, $\widehat{\mathbb P}_{\ell,m}^\ast(E_{\ell,m})
=
c_{\ell,m}^{\top}{\mathbf p}_{\ell,m}^\ast
=
\phi_{\ell,m}^\ast$. That is, $\widehat{\mathbb P}_{\ell,m}^\ast$ is a discrete extremal prior for
the $m$-th discretised problem.

The two measures
$\widehat{\mathbb P}_{\ell,m}^\ast$ and
$\mathbb P_{\ell,m}^\ast$ assign the same mass to every level-$m$ dyadic
cell. Hence, for every bounded Lipschitz function $f:K\to\mathbb R$,
\begin{align}
\left|
\int_K f\,d\widehat{\mathbb P}_{\ell,m}^\ast
-
\int_K f\,d\mathbb P_{\ell,m}^\ast
\right|
&\leqslant
\sum_{\mathbf k}
p_{\ell,m,\mathbf k}^\ast
\operatorname{osc}_{Q_{m,\mathbf k}}(f)
\nonumber\\
&\leqslant
\operatorname{Lip}(f)\sqrt d\,2^{-m},
\label{eqn_dyadic_bounded_Lipschitz_error}
\end{align}
for $\operatorname{osc}_{A}(f)
:=
\sup_{\mathbf x,\mathbf y\in A}|f(\mathbf x)-f(\mathbf y)|\,$
and $\,\operatorname{Lip}(f)
:=
\sup_{\substack{\mathbf x,\mathbf y\in K\\\mathbf x\neq \mathbf y}}
\frac{|f(\mathbf x)-f(\mathbf y)|}{\|\mathbf x-\mathbf y\|}$ where $A\subseteq K$. Indeed, every level-$m$ dyadic cell has diameter at most
$\sqrt d\,2^{-m}$.

Let $\{f_r:r\geqslant1\}$ be a countable convergence-determining family of bounded Lipschitz functions on $K$. Recall that by weak compactness of $\mathcal D$, there exist discretisation levels $m_k\uparrow\infty$ and a measure $\mathbb P_\ell^\ast\in\mathcal D$ such that $\mathbb P_{\ell,m_k}^\ast
\Rightarrow
\mathbb P_\ell^\ast$. For every fixed $r$, $\int_K f_r\,d\mathbb P_{\ell,m_k}^\ast
\longrightarrow
\int_K f_r\,d\mathbb P_\ell^\ast$, while \eqref{eqn_dyadic_bounded_Lipschitz_error} gives $\left|
\int_K f_r\,d\widehat{\mathbb P}_{\ell,m_k}^\ast
-
\int_K f_r\,d\mathbb P_{\ell,m_k}^\ast
\right|
\longrightarrow0$. Therefore,  by the triangle inequality, for all $r\geqslant1$, $\int_K f_r\,d\widehat{\mathbb P}_{\ell,m_k}^\ast
\longrightarrow
\int_K f_r\,d\mathbb P_\ell^\ast\ \text{ as }k\longrightarrow\infty$. Since $\{f_r:r\geqslant1\}$ is convergence determining,
\begin{equation}
\widehat{\mathbb P}_{\ell,m_k}^\ast
\Rightarrow
\mathbb P_\ell^\ast.
\label{eqn_discrete_extremal_prior_weak_convergence}
\end{equation}
This shows $\widehat{\mathbb P}_{\ell,m_k}^\ast$ and $\mathbb P_{\ell,m_k}^\ast$ have the same weak limit $\mathbb P_\ell^\ast$.

It remains to identify this limit as a solution of the compact-core target
problem. Fix $q$. For all sufficiently large $k$, $F_{\ell}\subseteq\overline E_{\ell,m_k}
\subseteq
\overline E_{\ell,q}$. Therefore, $V(F_{\ell})\leqslant V(\overline E_{\ell,m_k})=\mathbb P_{\ell,m_k}^\ast(\overline E_{\ell,m_k})
\leqslant
\mathbb P_{\ell,m_k}^\ast(\overline E_{\ell,q})$. Since $\overline E_{\ell,q}$ is closed, $\mathbb P_{\ell,m_k}^\ast
\Rightarrow
\mathbb P_\ell^\ast$ implies $V(F_\ell)
\leqslant
\operatorname{limsup}_{k\to\infty}\mathbb P_{\ell,m_k}^\ast(\overline E_{\ell,q})
\leqslant
\mathbb P_\ell^\ast(\overline E_{\ell,q})$ by the Portmanteau theorem. Letting $q\to\infty$ and using continuity from above for the fixed measure
$\mathbb P_\ell^\ast$ yields $V(F_\ell)
\leqslant
\mathbb P_\ell^\ast(F_\ell)
\leqslant
V(F_\ell)$. Hence $\mathbb P_\ell^\ast(F_\ell)
=
V(F_\ell)$.

We conclude that the discrete extremal priors admit a weakly convergent sequence whose
limit solves the compact-core target problem
\eqref{eqn_compact_core_target_problems}. By part~\textup{(a)}, $\bigl(\mathbb P_\ell^\ast,V(F_\ell)\bigr)$ is an exchange fixed-point pair for the $F_\ell$-problem.

\medskip
\noindent
\emph{Step 13: Passage from the compact-core problems to the Borel-event
value.}

By Step~7, $V(F_\ell)\uparrow V(G)$ as $\ell\to\infty$,
while Step~12 gives, for every fixed $\ell$, $\phi_{\ell,m}^\ast
\downarrow
V(F_\ell)$ as $m\to\infty$. Consequently, $V(G)
=
\lim_{\ell\to\infty}V(F_\ell)
=
\lim_{\ell\to\infty}
\lim_{m\to\infty}
\phi_{\ell,m}^\ast$. Together with $\widehat{\mathbb P}_{\ell,m_k}^\ast
\Rightarrow
\mathbb P_\ell^\ast$ and $\mathbb P_\ell^\ast(F_\ell)=V(F_\ell)$, proved in Step~12, this establishes part \textup{(c)}.

\medskip
\noindent
\emph{Step 14: fixed-point pair characterisations of the robust VaR
bounds.}

For each $j$, the quantile versions $Q^-_j$ and $Q^+_j$ in
\eqref{eqn_VaR_quantile_versions} differ only on a
$\lambda$-null set: they differ only at probability levels corresponding to countably many nondegenerate, disjoint intervals on which $H_j$ is constant. Since every $\mathbb P\in\mathcal D$ has uniform
coordinate marginals,
\begin{equation}
S^-=S^+
\quad
\mathbb P\text{-a.s. for every }\mathbb P\in\mathcal D.
\label{eqn_VaR_aggregate_versions_same_law_proof}
\end{equation}
For an arbitrary real-valued random variable $X$, $\inf\{t:\mathbb P(X\leqslant t)>\eta\}
=
\sup\{t:\mathbb P(X\geqslant t)\geqslant1-\eta\}$: this follows from the monotonicity of $F(t)=\mathbb P(X\leqslant t)$ and continuity of $\mathbb P$ from below. Applying this identity together with
\eqref{eqn_VaR_aggregate_versions_same_law_proof} proves
\eqref{eqn_symmetric_VaR_representation}.

The function $S^-$ is lower semicontinuous and $S^+$ is upper
semicontinuous. Hence, the sets $C_s$ and $G_s$ in
\eqref{eqn_VaR_threshold_events} are compact for every $s\in\mathbb R$.
Their definitions immediately give
\[
G_r^+(s)=V(C_s),
\qquad
G_r^-(s)=1-V(G_s),
\]
which proves \eqref{eqn_reduced_envelopes}.

By \eqref{eqn_symmetric_VaR_representation},
\begin{align}
\operatorname{VaR}_{\eta}^{\mathrm{lower}}
&=
\inf\left\{
s:
\exists\,\mathbb P\in\mathcal D
\text{ with }\mathbb P(C_s)>\eta
\right\}
\nonumber\\
&=
\inf\{s:G_r^+(s)>\eta\}.
\label{eqn_lower_VaR_inversion_proof}
\end{align}
Similarly,
\begin{align}
\operatorname{VaR}_{\eta}^{\mathrm{upper}}
&=
\sup\left\{
s:
\exists\,\mathbb P\in\mathcal D
\text{ with }\mathbb P(G_s)\geqslant1-\eta
\right\}.
\label{eqn_upper_VaR_feasibility_proof}
\end{align}
Since $G_s$ is compact, Step~7 shows that $V(G_s)$ is attained.
Consequently,
\[
\exists\,\mathbb P\in\mathcal D:
\mathbb P(G_s)\geqslant1-\eta
\quad\Longleftrightarrow\quad
V(G_s)\geqslant1-\eta.
\]
Thus, \eqref{eqn_upper_VaR_feasibility_proof} becomes
\begin{equation}
\operatorname{VaR}_{\eta}^{\mathrm{upper}}
=
\sup\{s:V(G_s)\geqslant1-\eta\}
=
\sup\{s:G_r^-(s)\leqslant\eta\}.
\label{eqn_upper_VaR_sup_proof}
\end{equation}
The following shows $\{s:G_r^-(s)\leqslant\eta\}$ is nonempty and bounded above. Let $a_j\leqslant Q^+_j(u_j)\leqslant b_j$ for all $u_j$. Put $a:=\sum_{j=1}^n a_j$ and 
$b:=\sum_{j=1}^n b_j$. Then, $a\leqslant S^+(\mathbf u)\leqslant b$ for every $\mathbf u\in K$. For $s\leqslant a$, one has $G_s=\{\mathbf u\in K:S^+(\mathbf u)\geqslant a\geqslant s\}=K$, so $G_r^-(s)=1-V(K)=0\leqslant\eta$: this shows non-emptiness. For $s>b$, $G_s=\{\mathbf u\in K:S^+(\mathbf u)\geqslant s> b\}=\varnothing$ and $G_r^-(s)=1-V(\varnothing)=1>\eta$: this shows boundedness from above. 

Let $s_\eta^\ast=\sup\{s:G_r^-(s)\leqslant\eta\}$, and
choose $s_k\uparrow s_\eta^\ast$, $V(G_{s_k})\geqslant1-\eta$.
Since $G_{s_k}\downarrow G_{s_\eta^\ast}$, the compact-set continuity proved in Step~7 gives $V(G_{s_k})
\downarrow
V(G_{s_\eta^\ast})$. Therefore, $V(G_{s_\eta^\ast})\geqslant1-\eta$, so the supremum in \eqref{eqn_upper_VaR_sup_proof} is attained. This proves
\eqref{eqn_reduced_lower_VaR} and
\eqref{eqn_upper_VaR_extremal_threshold}.

$C_s$ and $G_s$ are compact, so their extremal probabilities are
attained. Applying part \textup{(a)} to these two events gives extremal priors $\mathbb P_s^{+,\ast},
\mathbb P_s^{-,\ast}\in\mathcal D$ satisfying $\,\mathbb P_s^{+,\ast}(C_s)=V(C_s)\,$ and 
$\,\mathbb P_s^{-,\ast}(G_s)=V(G_s)$. The residual identity \eqref{eqn_exchange_residual_gap} gives
\eqref{eqn_VaR_positive_residual} and
\eqref{eqn_VaR_negative_residual}; the global fixed-point condition gives
\eqref{eqn_VaR_positive_fixed_point} and
\eqref{eqn_VaR_negative_fixed_point}; and part \textup{(b)} gives
\eqref{eqn_VaR_positive_hereditary_support},
\eqref{eqn_VaR_negative_hereditary_support}, and the tight-exchange description of all other extremal priors.

For $A_s^+:=C_s$ and $A_s^-:=G_s$, let $E_{s,m}^{\pm}
:=
\bigcup
\left\{
Q_{m,\mathbf k}:
Q_{m,\mathbf k}\cap A_s^\pm\neq\varnothing
\right\}$ and define
$\phi_{s,m}^\pm
:=
V(E_{s,m}^\pm)$. The dyadic partitions refine coherently, and the two threshold families
are nested in $s$. Thus the sets $E_{s,m}^\pm$ may be chosen coherently
both in $m$ and in $s$.

Applying Steps~8--12 with the compact target $A_s^\pm$ gives $\phi_{s,m}^+
\downarrow
V(C_s)
=
G_r^+(s)$
and $\phi_{s,m}^-
\downarrow
V(G_s)
=
1-G_r^-(s)$. By Step~12, the corresponding extremal priors admit weakly convergent
subsequences whose limits attain $V(C_s)$ and $V(G_s)$, respectively.
This proves \eqref{eqn_VaR_upper_envelope_limit} and
\eqref{eqn_VaR_lower_envelope_limit}, together with the asserted
convergence of the associated extremal priors.

Define $q^+(\gamma)
:=
\inf\{s:G_r^+(s)>\gamma\}$. 
Since $\phi_{s,m}^+\downarrow G_r^+(s)$ for every $s$,
one has $\phi_{s,m}^+\geqslant G_r^+(s)$ and hence $q_m^+(\gamma)\leqslant q^+(\gamma)$. Moreover, $q_m^+(\gamma)$ is nondecreasing in $m$. Fix
$\gamma'<\gamma$ and $x<q^+(\gamma')$. Then
$G_r^+(x)\leqslant\gamma'$, so
$\phi_{x,m}^+\leqslant\gamma$ for all sufficiently large $m$.
Since $s\mapsto\phi_{s,m}^+$ is nondecreasing, this implies
$q_m^+(\gamma)\geqslant x$ eventually. Letting
$x\uparrow q^+(\gamma')$ gives $\liminf_{m\to\infty}q_m^+(\gamma)
\geqslant q^+(\gamma')$. Letting $\gamma'\uparrow\gamma$, one obtains
\[
q^+(\gamma-)
\leqslant
\liminf_{m\to\infty}q_m^+(\gamma)
\leqslant
\limsup_{m\to\infty}q_m^+(\gamma)
\leqslant
q^+(\gamma).
\]
Hence, at every continuity point $\gamma$ of the generalized inverse
$q^+$,
\begin{equation}
q_m^+(\gamma)
\longrightarrow
q^+(\gamma).
\label{eqn_lower_finite_inverse_convergence_proof}
\end{equation}
The nondecreasing function $q^+$ has at most countably many
discontinuities. Therefore, for every $\eta\in(0,1)$, continuity points
$\eta_r\downarrow\eta$ may be chosen. Using
\eqref{eqn_lower_VaR_inversion_proof},
\eqref{eqn_lower_finite_inverse_convergence_proof}, and the
right-continuity of $q^+$ gives $q_m^+(\eta_r)
\longrightarrow
\operatorname{VaR}_{\eta_r}^{\mathrm{lower}}$ (for $m\to\infty$),
and $\operatorname{VaR}_{\eta_r}^{\mathrm{lower}}
\downarrow
\operatorname{VaR}_{\eta}^{\mathrm{lower}}$ (for $r\to\infty$).
This proves
\eqref{eqn_stable_level_lower_VaR_convergence}.

For the upper bound, put $W(s):=V(G_s)$, $W_m(s):=\phi_{s,m}^-$. Both $W$ and $W_m$ are nonincreasing in $s$, while $W_m(s)\downarrow W(s)$ for every $s$. Since $E_{s,m}^-\supseteq G_s$, $q_m^-(\eta)\geqslant s_\eta^\ast$,
and the sequence $\{q_m^-(\eta)\}_{m\geqslant1}$ is nonincreasing.

Let $t>s_\eta^\ast$. By
\eqref{eqn_upper_VaR_extremal_threshold}, $W(t)<1-\eta$. Since $W_m(t)\downarrow W(t)$, one has $W_m(t)<1-\eta$ for all sufficiently large $m$. The monotonicity of $W_m$ in $s$ then
implies $q_m^-(\eta)\leqslant t$ for all sufficiently large $m$. Letting $t\downarrow s_\eta^\ast$
yields $q_m^-(\eta)
\downarrow
s_\eta^\ast
=
\operatorname{VaR}_{\eta}^{\mathrm{upper}}$, which proves \eqref{eqn_direct_upper_VaR_convergence}.

Finally, let $\mathbb P_{s_\eta^\ast}^{-,\ast}$ be any extremal
prior for the $G_{s_\eta^\ast}$-problem. Then $\bigl(
\mathbb P_{s_\eta^\ast}^{-,\ast},
V(G_{s_\eta^\ast})
\bigr)$ is an exchange fixed-point pair by Theorem~\ref{thm_VaRExample}. So, $\mathbb P_{s_\eta^\ast}^{-,\ast}(G_{s_\eta^\ast})
=
V(G_{s_\eta^\ast})
\geqslant1-\eta$. By \eqref{eqn_symmetric_VaR_representation}, $\operatorname{VaR}_{\eta}
\bigl(S;\mathbb P_{s_\eta^\ast}^{-,\ast}\bigr)
\geqslant s_\eta^\ast$. The reverse inequality follows from the definition of
$\operatorname{VaR}_{\eta}^{\mathrm{upper}}=s_\eta^\ast$.
Therefore $\operatorname{VaR}_{\eta}
\bigl(S;\mathbb P_{s_\eta^\ast}^{-,\ast}\bigr)
=
\operatorname{VaR}_{\eta}^{\mathrm{upper}}$, which proves \eqref{eqn_upper_VaR_attainment} and completes the proof of
Corollary~\ref{Cor_VaRExample}.
\end{proof}
\newpage
\section{Proof of Theorem~\ref{thm:canonical-strike-recursive}}
\begin{lemma}[Anti-Monge structure]
\label{lem:canonical-antimonge}
For every $m\in\mathbb{N}$, the matrix $F^{(m)}:=(F_{ij}^{(m)})_{1\leqslant i,j\leqslant N_m}$
is anti-Monge (see \cite{burkard1996perspectives}):
\[
F_{ij}^{(m)}+F_{k\ell}^{(m)}
\geqslant
F_{i\ell}^{(m)}+F_{kj}^{(m)}
\qquad
(i<k,\ j<\ell).
\]
Consequently, the restriction of $F^{(m)}$ to every dyadic block
$B\in\mathscr{B}_{m,r}$ is also anti-Monge.
\end{lemma}

\begin{proof}
Set
$\alpha_i^{(m)}:=
Q_1\!\left(\pi_1(\mathbf{x}_{ij}^{(m)})\right)$, \ $\beta_j^{(m)}
:=
Q_2\!\left(\pi_2(\mathbf{x}_{ij}^{(m)})\right)$ \ 
and \ $\varphi(t):=(t-k)^+$.
Then
\[
F_{ij}^{(m)}=\varphi\!\bigl(\alpha_i^{(m)}+\beta_j^{(m)}\bigr).
\]
Because $Q_1,Q_2$ are increasing, $\alpha_i^{(m)}\leqslant \alpha_k^{(m)}$ whenever $i<k$, $\beta_j^{(m)}\leqslant \beta_\ell^{(m)}$ whenever $j<\ell$. Fix $i<k$ and $j<\ell$. Put
\[
x:=\alpha_i^{(m)}+\beta_j^{(m)},
\quad
y:=\alpha_k^{(m)}+\beta_\ell^{(m)},
\quad
x':=\alpha_i^{(m)}+\beta_\ell^{(m)},
\quad
y':=\alpha_k^{(m)}+\beta_j^{(m)}.
\]
Then $x+y=x'+y'$,
and the pair $(x,y)$ is more spread out than $(x',y')$. Since $\varphi$ is convex, $\varphi(x)+\varphi(y)\geqslant \varphi(x')+\varphi(y')$. Substituting back gives
\[
F_{ij}^{(m)}+F_{k\ell}^{(m)}
\geqslant
F_{i\ell}^{(m)}+F_{kj}^{(m)}.
\]
Any dyadic block is a contiguous submatrix, and every contiguous submatrix of an
anti-Monge matrix is anti-Monge.
\end{proof}

\begin{lemma}[Local reverse-permutation reduction on a dyadic block]
\label{lem:block-reverse-permutation}
Fix $m\in\mathbb{N}$ and let $B=I\times J\in\mathscr{B}_{m,r}$ for some $r\in\{0,\dots,m-1\}$, where $I=\{i_1<\cdots<i_n\}$, $J=\{j_1<\cdots<j_n\}$, and 
$n=2^{m-r}$. Define
\[
{\mathcal W}(B)
:=
\left\{
\textnormal{\textbf{x}}=(x_{ij})_{(i,j)\in B}\in[0,\infty)^B:\;\;\begin{array}{l}
\sum_{j\in J}x_{ij}=\frac{1}{N_m}\ \forall i\in I,\\[3pt]
\sum_{i\in I}x_{ij}=\frac{1}{N_m}\ \forall j\in J
\end{array}
\right\}.
\]
If the restriction of $F^{(m)}$ to $B$ is anti-Monge, then
\begin{equation}
\inf_{\textnormal{\textbf{x}}\in{\mathcal W}(B)}\sum_{(i,j)\in B}x_{ij}F_{ij}^{(m)}
=
\frac{1}{N_m}\sum_{t=1}^{n}F_{i_t,j_{n+1-t}}^{(m)}.
\label{eqn_antiMongeconsequence_1}
\end{equation}
In particular,
\begin{equation}
\inf_{\textnormal{\textbf{x}}\in{\mathcal W}(B)}\sum_{(i,j)\in B}{x}_{ij}F_{ij}^{(m)}
=
\inf_{\substack{\textnormal{\textbf{x}}\in{\mathcal W}(B)\\
\operatorname{supp}(\textnormal{\textbf{x}})\subseteq B^{\mathrm{NW}}\cup B^{\mathrm{SE}}}}
\sum_{(i,j)\in B}x_{ij}F_{ij}^{(m)}.
\label{eqn_antiMongeconsequence_2}
\end{equation}
\end{lemma}

\begin{proof}
Since $N_m \mathbf{x}$ is doubly stochastic on $I\times J$, the Birkhoff--von Neumann theorem (see \cite{Birkhoff1946Tres,vonNeumann1953Assignment})
gives
\[
\mathbf{x}=\frac{1}{N_m}\sum_{\pi}\lambda_\pi \mathbb P_\pi,
\qquad
\lambda_\pi\geqslant 0,\quad \sum_\pi\lambda_\pi=1,
\]
where $\pi$ runs over permutations of $\{1,\dots,n\}$ and $\mathbb P_\pi$ is the corresponding
permutation matrix on $I\times J$. Hence,
\[
\sum_{(i,j)\in B}x_{ij}F_{ij}^{(m)}
=
\frac{1}{N_m}\sum_\pi \lambda_\pi
\sum_{t=1}^{n}F_{i_t,j_{\pi(t)}}^{(m)}.
\]
So it suffices to minimise over permutation assignments.

Let $\pi$ be any permutation. If there exist $s<t$ with $\pi(s)<\pi(t)$, then the two
assignments $(i_s,j_{\pi(s)})$ and $(i_t,j_{\pi(t)})$ form a crossing. Since the
restriction of $F^{(m)}$ to $B$ is anti-Monge,
\[
F_{i_s,j_{\pi(s)}}^{(m)}+F_{i_t,j_{\pi(t)}}^{(m)}
\geqslant
F_{i_s,j_{\pi(t)}}^{(m)}+F_{i_t,j_{\pi(s)}}^{(m)}.
\]
Thus, swapping the two assignments does not increase the cost. Repeating finitely many
times yields the reverse permutation $\pi_{\mathrm{rev}}(t):=n+1-t$, without increasing the total cost. Therefore, $\sum_{t=1}^{n}F_{i_t,j_{n+1-t}}^{(m)}
\leqslant
\sum_{t=1}^{n}F_{i_t,j_{\pi(t)}}^{(m)}$ for every permutation $\pi$.
This proves \eqref{eqn_antiMongeconsequence_1}.

Finally, the graph of the reverse permutation lies in the local anti-diagonal of $B$,
hence in $B^{\mathrm{NW}}\cup B^{\mathrm{SE}}$, proving \eqref{eqn_antiMongeconsequence_2}.
\end{proof}
Lemmas~\ref{lem:canonical-antimonge} and \ref{lem:block-reverse-permutation} are used in the following five-step proof of Theorem~\ref{thm:canonical-strike-recursive}.
\begin{proof}\textbf{Step 1: infima from extreme approximants converge to objective infimum.}
For the original basket payoff $f$ and $g\equiv 1$, Proposition~\ref{prop_gen_multiple_marginals_piecewise_refined} gives the stage-$m$ product-cell problem over the polytope ${\mathcal W}_m$, and Corollary~\ref{cor_finite_marginal_approx}
gives convergence of the corresponding stage-$m$ extrema to $\phi^\ast$, because $f$
is continuous on the compact square $K$. More precisely, by construction, the stage-$m$ infimum is $\underline{\phi}_m$ and $\underline{\phi}_m\to \phi^\ast$. Moreover, $0\leqslant b_{ij}^{(m)}-a_{ij}^{(m)}
\leqslant
\omega_f(m)$, where
\[
\omega_f(m)
:=
\sup\{|f(\textnormal{\bf x})-f(\textnormal{\bf y})|:
\textnormal{\bf x},\textnormal{\bf y}\in C_{ij}^{(m)}
\text{ for some }i,j\}.
\]
Hence, $\underline{\phi}_m
\leqslant
\overline{\phi}_m
\leqslant
\underline{\phi}_m+\omega_f(m)$. Since $f$ is uniformly continuous on $K$, we have $\omega_f(m)\to 0$ so ``squeeze'' implies:
\begin{equation}
\overline{\phi}_m\to \phi^\ast.
\label{eqn_BasketPayoff_5}
\end{equation}
Consequently, since $\underline{\phi}_m
\leqslant
\phi_m^{\mathrm{full}}
\leqslant
\overline{\phi}_m$, ``squeeze'' again implies:
\begin{equation}
\phi_m^{\mathrm{full}}\to \phi^\ast.
\label{eqn_BasketPayoff_10}
\end{equation}

\paragraph{Step 2: recursive uncrossing on each dyadic block.}
Fix $m\in\mathbb{N}$. We prove \eqref{eqn_BasketPayoff_1} by induction on $r$.

For $r=0$, the statement is trivial because $ \bigcup\limits_{B\in\mathscr{B}_{m,0}}\!\!\!\!\!\!\!B=\{1,\ldots,N_m\}^2$. Assume \eqref{eqn_BasketPayoff_1} has been proved for $r-1$, and let $\textbf{w}\in{\mathcal W}_m$ with $\operatorname{supp}(\textbf{w})\subseteq  \bigcup\limits_{B\in\mathscr{B}_{m,r-1}}\!\!\!\!\!\!\!\!\!\!B$. The set $ \bigcup\limits_{B\in\mathscr{B}_{m,r-1}}\!\!\!\!\!\!\!\!\!\!B$ is a disjoint union of the dyadic blocks $B=I\times J\in \mathscr{B}_{m,r-1}$, and these blocks have disjoint row and column sets. Since the support of $\mathbf{w}$ is
contained in the union of these blocks, the restriction $\mathbf{w}|_B$ belongs to the polytope
\[
{\mathcal W}(B)
:=
\left\{
\textnormal{\textbf{x}}=(x_{ij})_{(i,j)\in B}\in[0,\infty)^B:
\;\;\begin{array}{l}\sum_{j\in J}x_{ij}=\frac{1}{N_m}\ \forall i\in I,\\[3pt]
\sum_{i\in I}x_{ij}=\frac{1}{N_m}\ \forall j\in J
\end{array}
\right\}.
\]

By Lemma~\ref{lem:canonical-antimonge}, the restriction of $F^{(m)}$ to $B$ is
anti-Monge. Therefore, Lemma~\ref{lem:block-reverse-permutation} applies inside $B$: if a local
feasible mass pattern has a crossing, then swapping the crossing assignments does not
increase the local cost. Iterating finitely many times yields a local minimiser whose
support is contained in the two local anti-diagonal children $B^{\mathrm{NW}}\cup B^{\mathrm{SE}}$. Thus,
\begin{equation}
\inf_{\textnormal{\textbf{x}}\in{\mathcal W}(B)}\sum_{(i,j)\in B}x_{ij}F_{ij}^{(m)}
=
\inf_{\substack{\textnormal{\textbf{x}}\in{\mathcal W}(B)\\
\operatorname{supp}(\textnormal{\textbf{x}})\subseteq B^{\mathrm{NW}}\cup B^{\mathrm{SE}}}}
\sum_{(i,j)\in B}x_{ij}F_{ij}^{(m)}.
\label{eqn_BasketPayoff_6}
\end{equation}

Apply \eqref{eqn_BasketPayoff_6} independently on each block in $\mathscr{B}_{m,r-1}$. Since those
blocks have disjoint row and column sets, the blockwise replacements assemble into a new
matrix $\textbf{w}'\in{\mathcal W}_m$ with $\operatorname{supp}(\mathbf{w}')\subseteq \bigcup\limits_{B\in\mathscr{B}_{m,r}}\!\!\!\!\!\!\!B$ and $\sum_{i,j=1}^{N_m}w'_{ij}F_{ij}^{(m)}
\leqslant
\sum_{i,j=1}^{N_m}w_{ij}F_{ij}^{(m)}$. Therefore,
\[
\phi_m^{\mathrm{full}}=\inf_{\substack{\textbf{w}\in{\mathcal W}_m\\ \operatorname{supp}(\textbf{w})\subseteq \bigcup\limits_{B\in\mathscr{B}_{m,r-1}}\!\!\!\!\!\!\!\!\!\!B}}\,
\sum_{i,j=1}^{N_m}w_{ij}F_{ij}^{(m)}
=
\inf_{\substack{\textbf{w}\in{\mathcal W}_m\\ \operatorname{supp}(\textbf{w})\subseteq  \bigcup\limits_{B\in\mathscr{B}_{m,r}}\!\!\!\!\!\!\!\!B}}\,
\sum_{i,j=1}^{N_m}w_{ij}F_{ij}^{(m)}.
\]
This proves \eqref{eqn_BasketPayoff_1}.

\paragraph{Step 3: reduction to the anti-diagonal solution branch.}
Taking $r=m$ in \eqref{eqn_BasketPayoff_1},
\begin{equation}
\phi_m^{\mathrm{full}}
=
\inf_{\substack{\textbf{w}\in{\mathcal W}_m\\ \operatorname{supp}(\textbf{w})\subseteq  \bigcup\limits_{B\in\mathscr{B}_{m,m}}\!\!\!\!\!\!\!\!B}}\,
\sum_{i,j=1}^{N_m}w_{ij}F_{ij}^{(m)}.
\label{eqn_BasketPayoff_7}
\end{equation}
But for each row index $i$, there is exactly one retained cell in $S_{m,m}$, namely $C_{i,N_m+1-i}^{(m)}$,
and similarly for each column. Therefore there is only one feasible mass allocation in
${\mathcal W}_m$ supported on $S_{m,m}$:
\begin{equation}
w_{ij}^{\mathrm{ad},(m)}
=
\frac{1}{N_m}\mathbf{1}\{j=N_m+1-i\}.
\label{eqn_BasketPayoff_8}
\end{equation}
Hence,
\begin{equation}
\phi_m^{\mathrm{full}}
=
\frac{1}{N_m}\sum_{i=1}^{N_m}F_{i,N_m+1-i}^{(m)}.
\label{eqn_BasketPayoff_9}
\end{equation}


\paragraph{Step 4: identify the limiting value as the strike-payoff formula.}
From \eqref{eqn_BasketPayoff_9}, choose the midpoint payoff for every cell, so
\[
\phi_m^{\mathrm{full}}
=
\frac{1}{N_m}\sum_{i=1}^{N_m}
f\!\left(
\frac{i-\frac12}{N_m},
1-\frac{i-\frac12}{N_m}
\right).
\]
Define
\[
h(u):=f(u,1-u)=\bigl(Q_1(u)+Q_2(1-u)-k\bigr)^+.
\]
Since $Q_1,Q_2$ are continuous, $h$ is continuous on $[0,1]$, and the above sum is
the midpoint Riemann sum for $h$. Therefore
\[
\phi_m^{\mathrm{full}}
\to
\int_0^1 h(u)\,du
=
\int_0^1 \bigl(Q_1(u)+Q_2(1-u)-k\bigr)^+\,du.
\]
Thus, \eqref{eqn_BasketPayoff_10} implies \eqref{eqn_BasketPayoff_3}.

\paragraph{Step 5: canonical anti-diagonal representatives and Wasserstein closeness.}
For each $m\in\mathbb N$, define
\[
\mathbb P_m^{\mathrm{ad}}
=
\frac1{N_m}\sum_{i=1}^{N_m}\delta_{\textnormal{\textbf{x}}_i^{(m)}},
\qquad
\textnormal{\textbf{x}}_i^{(m)}
=
\left(\frac{i-\frac12}{N_m},\,1-\frac{i-\frac12}{N_m}\right).
\]
Since $\textnormal{\textbf{x}}_i^{(m)}\in C_{i,N_m+1-i}^{(m)}$ for every $i$, we have $\operatorname{supp}(\mathbb P_m^{\mathrm{ad}})
\subseteq
S_{m,m}$. Moreover, if $\textnormal{\textbf{x}}=(u,v)\in C_{i,N_m+1-i}^{(m)}$, then $\frac{i-1}{N_m}\leqslant u\leqslant \frac{i}{N_m}$ and $\frac{N_m-i}{N_m}\leqslant v\leqslant \frac{N_m+1-i}{N_m}$. Hence, $|u+v-1|\leqslant \frac1{N_m}=2^{-m}$. Therefore, $S_{m,m}
\subseteq
\{(u,v)\in K:\ |u+v-1|\leqslant 2^{-m}\}$, which proves \eqref{eqn_BasketPayoff_4}.

Next, because $f$ is continuous on the compact set $K$, for every $i,j$ the
restriction of $f$ to the compact closure $\overline C_{ij}^{(m)}$ attains its minimum.
Choose points $\textnormal{\textbf{z}}_{ij}^{(m)}\in \overline C_{ij}^{(m)}$ such that $f(\textnormal{\textbf{z}}_{ij}^{(m)})=a_{ij}^{(m)}$,
and define a valid stage-$m$ cellwise approximant that equals $a_{ij}^{(m)}$ on cell $\overline{C}_{ij}^{(m)}$. Then the infimum from this approximant over $\mathcal W_m$ is $\underline{\phi}_m$: by construction, this is also the stage-$m$ infimum for the general problem in terms of $f$. And, by the recursive mass re-allocation arguments in Steps 2 and 3, $\underline{\phi}_m$ can be obtained using the anti-diagonal supported extremal distribution 
\[
\mathbb P_{m,\star}^{\mathrm{ad}}
=
\frac1{N_m}\sum_{i=1}^{N_m}\delta_{\textnormal{\textbf{x}}_{i,\star}^{(m)}},
\qquad
\textnormal{\textbf{x}}_{i,\star}^{(m)}\in \overline C_{i,N_m+1-i}^{(m)}
\]
upon re-allocating probability masses from the $\textnormal{\textbf{z}}_{ij}^{(m)}$ to some $\textnormal{\textbf{x}}_{i,\star}^{(m)}$.

Now $\textnormal{\textbf{x}}_i^{(m)}\in C_{i,N_m+1-i}^{(m)}\subseteq
\overline C_{i,N_m+1-i}^{(m)}$ and
$\textnormal{\textbf{x}}_{i,\star}^{(m)}\in \overline C_{i,N_m+1-i}^{(m)}$, so both points
lie in the same closed anti-diagonal cell. Therefore,
\[
\|\textnormal{\textbf{x}}_i^{(m)}-\textnormal{\textbf{x}}_{i,\star}^{(m)}\|
\leqslant
\operatorname{diam}\!\bigl(\overline C_{i,N_m+1-i}^{(m)}\bigr)
=
\operatorname{diam}\!\bigl(C_{i,N_m+1-i}^{(m)}\bigr).
\]
Taking the maximum over $i$ gives
\[
\max_{1\leqslant i\leqslant N_m}
\|\textnormal{\textbf{x}}_i^{(m)}-\textnormal{\textbf{x}}_{i,\star}^{(m)}\|
\leqslant
\max_{1\leqslant i\leqslant N_m}
\operatorname{diam}\!\bigl(C_{i,N_m+1-i}^{(m)}\bigr).
\]
Since each $C_{i,N_m+1-i}^{(m)}$ is a dyadic square of side length $1/N_m$, its
diameter tends to $0$ uniformly in $i$ as $m\to\infty$. Hence,
\[
\max_{1\leqslant i\leqslant N_m}
\|\textnormal{\textbf{x}}_i^{(m)}-\textnormal{\textbf{x}}_{i,\star}^{(m)}\|
\leqslant
\max_{1\leqslant i\leqslant N_m}\operatorname{diam}(C_{i,N_m+1-i}^{(m)})\to 0.
\]

Finally, let $W_1$ denote the Wasserstein--1 distance associated with the same norm
$\|\cdot\|$. Couple the $i$th atom of $\mathbb P_m^{\mathrm{ad}}$ with the $i$th atom of
$\mathbb P_{m,\star}^{\mathrm{ad}}$. This yields
\[
W_1(\mathbb P_m^{\mathrm{ad}},\mathbb P_{m,\star}^{\mathrm{ad}})
\leqslant
\frac1{N_m}\sum_{i=1}^{N_m}
\|\textnormal{\textbf{x}}_i^{(m)}-\textnormal{\textbf{x}}_{i,\star}^{(m)}\|
\leqslant
\max_{1\leqslant i\leqslant N_m}
\|\textnormal{\textbf{x}}_i^{(m)}-\textnormal{\textbf{x}}_{i,\star}^{(m)}\|.
\]
Therefore, $W_1(\mathbb P_m^{\mathrm{ad}},\mathbb P_{m,\star}^{\mathrm{ad}})\to 0$,
as stated in \eqref{eqn_Wass1Dist}.
\end{proof}
\newpage
\section{Proof of Theorem~\ref{thm_Gen2stateMarkovChain}}
\begin{proof}
We proceed in six steps.
\paragraph{Step 1: applying Proposition~\ref{prop_genprob1_v1} to the original fractional problem.}
Since every ${\mathbb P}\in\mathcal D$ satisfies ${\mathbb P}(\Theta_{\mathrm{feas}})=1$, we may regard the effective state space as $K=\Theta_{\mathrm{feas}}$. On this restricted space, define
\[
B_i^-:=\{\textnormal{\textbf{q}}\in\Theta_{\mathrm{feas}}:t\leqslant L_i(\textnormal{\textbf{q}})<\epsilon\},
\qquad
B_i^+:=\{\textnormal{\textbf{q}}\in\Theta_{\mathrm{feas}}:\epsilon\leqslant L_i(\textnormal{\textbf{q}})\leqslant 1\}.
\]
For each $i$, $B_i^-\cap B_i^+=\varnothing$ and $B_i^-\cup B_i^+=\Theta_{\mathrm{feas}}$. The constraints defining $\mathcal D$ are precisely ${\mathbb P}(B_i^-)=\theta$ and ${\mathbb P}(B_i^+)=1-\theta$ ($i=1,\ldots,m$). 

Apply Proposition~\ref{prop_genprob1_v1} to the $2m$ measurable sets $B_1^-,B_1^+,\ldots,B_m^-,B_m^+$. Because, for each $i$, the pair $B_i^-,B_i^+$ is disjoint and exhaustive, every nonempty Boolean atom chooses exactly one of the two sets in each pair. Hence, the nonempty atoms are indexed by subsets $S\subseteq[m]$ and are exactly
\[
A_S
=
\bigcap_{i\in S}B_i^-
\cap
\bigcap_{i\notin S}B_i^+.
\]
Equivalently, $A_S
=
\{\textnormal{\textbf{q}}\in\Theta_{\mathrm{feas}}:
L_i(\textnormal{\textbf{q}})\in[t,\epsilon)\text{ if }i\in S,\
L_i(\textnormal{\textbf{q}})\in[\epsilon,1]\text{ if }i\notin S\}$.

If ${\mathbb P}\in\mathcal D$ and $p_S={\mathbb P}(A_S)$,
then $p_S=0$ for $S\in\mathcal M_{\mathrm{miss}}$,
$p_S\geqslant0$ otherwise,
and $\sum_{S\subseteq[m]}p_S=1$. Moreover, $\sum_{S\ni i}p_S
=
{\mathbb P}(B_i^-)
=
\theta$. Thus $C_{\mathrm{full}}p=b$, so $p\in\mathcal P$.

Conversely, if $p\in\mathcal P$ and, for every active nonmissing atom, ${\textnormal{\textbf{q}}}_S\in A_S$, then ${\mathbb P}
=
\sum_{S:p_S>0}p_S\delta_{{\textnormal{\textbf{q}}}_S}$ belongs to $\mathcal D$, because ${\mathbb P}(B_i^-)
=
\sum_{S\ni i}p_S
=
\theta$ and ${\mathbb P}(B_i^+)
=
\sum_{S\not\ni i}p_S
=
1-\sum_{S\ni i}p_S
=
1-\theta$.

Therefore, Proposition~\ref{prop_genprob1_v1} yields the atomised fractional representation
\begin{equation}
\Phi^\ast
=
\inf_{p\in\mathcal P}
\inf_{\substack{{\textnormal{\textbf{q}}}_S\in A_S\\ S\in\mathcal S_p}}
\frac{
\sum_{S\in\mathcal S_p}p_Sf({\textnormal{\textbf{q}}}_S)
}{
\sum_{S\in\mathcal S_p}p_Sg({\textnormal{\textbf{q}}}_S)
},
\qquad
\mathcal S_p:=\{S\subseteq[m]:p_S>0\}.
\label{eq:GenMC_atomised_fractional}
\end{equation}
The infimum is written rather than a minimum because the atoms $A_S$ are half-open and the corresponding support-point infima need not be attained. Atoms with $p_S=0$ are omitted because their support locations are irrelevant.

\paragraph{Step 2: closure minimisers for the atomwise Dinkelbach problems.}
Since each $L_i=L_{\sigma_i}$ is polynomial in the coordinates of $\textnormal{\textbf{q}}$, the functions
$L_i$ are continuous. Hence $\Theta_{\mathrm{feas}}$ is closed in the compact set
$\Theta$, and therefore $\Theta_{\mathrm{feas}}$ is compact. For each nonmissing $S$, the relative closure $\overline A_S=\operatorname{cl}_{\Theta_{\mathrm{feas}}}(A_S)$ is compact.

For each $\phi\geqslant0$, the Dinkelbach difference $h_\phi$ is continuous on $\Theta_{\mathrm{feas}}$ (because $f$ and $g$ are). Therefore, for every nonmissing $S$, $\inf_{\textnormal{\textbf{q}}\in A_S}h_\phi(\textnormal{\textbf{q}})
=
\inf_{\textnormal{\textbf{q}}\in\overline A_S}h_\phi(\textnormal{\textbf{q}})
=
\min_{\textnormal{\textbf{q}}\in\overline A_S}h_\phi(\textnormal{\textbf{q}})$. Thus, $\overline N_S(\phi)
=
\arg\min_{\textnormal{\textbf{q}}\in\overline A_S}h_\phi(\textnormal{\textbf{q}})$ is nonempty. These closure minimisers describe limiting support locations: they may not belong to their respective half-open atoms $A_S$.

\paragraph{Step 3: separability of weighted Dinkelbach sum over active atoms.}
Define $F(\phi)
=
\inf_{{\mathbb P}\in\mathcal D}
{\mathbb E}_{\mathbb P}[h_\phi(\textnormal{\textbf{Q}})]$. Using the fractional problem \eqref{eq:GenMC_atomised_fractional} from Step~1, 
\begin{equation}
F(\phi)
=
\inf_{p\in\mathcal P}
\inf_{\substack{{\textnormal{\textbf{q}}}_S\in A_S\\ S\in\mathcal S_p}}
\sum_{S\in\mathcal S_p}p_Sh_\phi({\textnormal{\textbf{q}}}_S).
\label{eqn_2stateMarkovGenDinkelbach}
\end{equation}
For fixed $p\in\mathcal P$, the variables ${\textnormal{\textbf{q}}}_S$, $S\in\mathcal S_p$, are independent across atoms and the objective in \eqref{eqn_2stateMarkovGenDinkelbach} is an additive finite sum. Hence,
\[
\inf_{\substack{{\textnormal{\textbf{q}}}_S\in A_S\\ S\in\mathcal S_p}}
\sum_{S\in\mathcal S_p}p_Sh_\phi({\textnormal{\textbf{q}}}_S)
=
\sum_{S\in\mathcal S_p}p_S\inf_{\textnormal{\textbf{q}}\in A_S}h_\phi(\textnormal{\textbf{q}}).
\]
Indeed, for every choice ${\textnormal{\textbf{q}}}_S\in A_S$, $h_\phi({\textnormal{\textbf{q}}}_S)\geqslant m_S(\phi)$, so $\sum_{S\in\mathcal S_p}p_Sh_\phi({\textnormal{\textbf{q}}}_S)
\geqslant
\sum_{S\in\mathcal S_p}p_Sm_S(\phi)$. Conversely, for every $\eta>0$ and every $S\in\mathcal S_p$, choose ${\textnormal{\textbf{q}}}_{S,\eta}\in A_S$ such that $h_\phi({\textnormal{\textbf{q}}}_{S,\eta})\leqslant m_S(\phi)+\eta$. Then,
\[
\sum_{S\in\mathcal S_p}p_Sh_\phi({\textnormal{\textbf{q}}}_{S,\eta})
\leqslant
\sum_{S\in\mathcal S_p}p_Sm_S(\phi)
+
\eta\sum_{S\in\mathcal S_p}p_S
\leqslant
\sum_{S\in\mathcal S_p}p_Sm_S(\phi)+\eta.
\]
Letting $\eta\downarrow0$ gives the reverse inequality.

Therefore,  $F(\phi)=\inf_{p\in\mathcal P}\sum_{S\in\mathcal S_p}p_Sm_S(\phi)$. Since $p_S=0$ for $S\notin\mathcal S_p$, 
\[
F(\phi)
=
\inf_{p\in\mathcal P}
\sum_{S\subseteq[m]}p_Sm_S(\phi)
=
\min_{p\in\mathcal P}p^\top m(\phi).
\]
The final infimum is a minimum because $\mathcal P$ is a compact polytope and the map $p\mapsto p^\top m(\phi)$ is linear.


\paragraph{Step 4: the Dinkelbach residual equation.}
The preceding step explains how to evaluate $F(\phi)$ for fixed $\phi$. It does not by itself determine the infimum value $\Phi^\ast$. The following shows that the unique root of $F(\phi)$ is the infimum.

Note that $F$ is monotonic and continuous. If $\phi_2>\phi_1$, then
for every ${\mathbb P}\in\mathcal D$,
\[
{\mathbb E}_{\mathbb P}[h_{\phi_2}(\textnormal{\textbf{Q}})]
=
{\mathbb E}_{\mathbb P}[h_{\phi_1}(\textnormal{\textbf{Q}})]
-
(\phi_2-\phi_1){\mathbb E}_{\mathbb P}[g(\textnormal{\textbf{Q}})].
\]
Since $0<t\leqslant {\mathbb E}_{\mathbb P}[g(\textnormal{\textbf{Q}})]\leqslant 1$, taking infima over $\mathcal D$ gives
\[
F(\phi_1)-(\phi_2-\phi_1)
\leqslant
F(\phi_2)
\leqslant
F(\phi_1)-t(\phi_2-\phi_1).
\]
Hence, $F$ is Lipschitz continuous and strictly decreasing. 

Now, for every ${\mathbb P}\in\mathcal D$,
\begin{equation}
{\mathbb E}_{\mathbb P}[h_\phi(\textnormal{\textbf{Q}})]
=
{\mathbb E}_{\mathbb P}[g(\textnormal{\textbf{Q}})]
\left\{
\frac{{\mathbb E}_{\mathbb P}[f(\textnormal{\textbf{Q}})]}
{{\mathbb E}_{\mathbb P}[g(\textnormal{\textbf{Q}})]}
-\phi
\right\}.
\label{eqn_MarkovDinkelbachRootEquivalence}
\end{equation}
Since $g\geqslant t>0$ on $\Theta_{\mathrm{feas}}$, every feasible denominator is uniformly positive: ${\mathbb E}_{\mathbb P}[g(\textnormal{\textbf{Q}})]\geqslant t$. If $\phi<\Phi^\ast$ then, for every ${\mathbb P}\in\mathcal D$, \eqref{eqn_MarkovDinkelbachRootEquivalence} implies ${\mathbb E}_{\mathbb P}[h_\phi(\textnormal{\textbf{Q}})]
\geqslant
t(\Phi^\ast-\phi)>0$, so $F(\phi)>0$. However, if $\phi>\Phi^\ast$ instead, by the definition of $\Phi^\ast$ there exists ${\mathbb P}\in\mathcal D$ such that $\frac{{\mathbb E}_{\mathbb P}[f(\textnormal{\textbf{Q}})]}
{{\mathbb E}_{\mathbb P}[g(\textnormal{\textbf{Q}})]}
<\phi$, hence ${\mathbb E}_{\mathbb P}[h_\phi(\textnormal{\textbf{Q}})]<0$. Thus, $F(\phi)<0$. 

This sign change of $F$, together with the monotonicity and continuity of $F$, proves that $F(\phi)=0$ has a unique solution. Any bracketed scalar root-finding method for $F$ (e.g. bisection)
converges to this root.

Finally, if $\phi=\Phi^\ast$, every feasible prior has nonnegative Dinkelbach difference, so $F(\Phi^\ast)\geqslant0$. Also, from \eqref{eqn_MarkovDinkelbachRootEquivalence}, ${\mathbb E}_{\mathbb P}[h_{\Phi^\ast}(\textnormal{\textbf{Q}})]
\leqslant
\left\{
\frac{{\mathbb E}_{\mathbb P}[f(\textnormal{\textbf{Q}})]}
{{\mathbb E}_{\mathbb P}[g(\textnormal{\textbf{Q}})]}
-{\Phi^\ast}
\right\}$ since $g$ is bounded above by $1$ on $\Theta_{\mathrm{feas}}$. So, taking the infimum over $\mathcal D$, $F(\Phi^\ast)\leqslant0$. Hence, $F(\Phi^\ast)=0$; i.e. $\Phi^\ast$ is the unique root of $F$.

\paragraph{Step 5: canonical exchange basis for the outer mass problem.}
We now characterise the minimisers of $F(\phi)=\min_{p\in\mathcal P}p^\top m(\phi)$. 

For $|S|\geqslant2$, recall $d_S
:=
e_S+(|S|-1)e_\varnothing-\sum_{i\in S}e_{\{i\}}$. Since $c_S
=
(1-|S|)c_\varnothing+\sum_{i\in S}c_{\{i\}}$, we have $C_{\mathrm{full}}d_S=0$. The columns $c_\varnothing,c_{\{1\}},\ldots,c_{\{m\}}$ are linearly independent. Hence, $\operatorname{rank}(C_{\mathrm{full}})=m+1$ and $\dim\ker C_{\mathrm{full}}=2^m-(m+1)$. There are exactly $2^m-(m+1)$ vectors $d_S$ with $|S|\geqslant2$, and they are linearly independent because $d_S$ is the only one with a nonzero coordinate in the $S$-coordinate among the non-anchor coordinates. Therefore, the columns of $\underline D$ form a basis of $\ker C_{\mathrm{full}}$.

Let $p^\ast\in\mathcal P$. A vector $v$ is a feasible tangent direction at $p^\ast$ for the polytope $\mathcal P$ if and only if \textbf{i)} $C_{\mathrm{full}}v=0$, \textbf{ii)} $v_S=0$ for $S\in\mathcal M_{\mathrm{miss}}$, and
\textbf{iii)} $v_S\geqslant0$ whenever $S\notin\mathcal M_{\mathrm{miss}}$ and $p_S^\ast=0$. Since the columns of $\underline D$ form a basis of $\ker C_{\mathrm{full}}$, every such direction can be written as $v={\underline D}\alpha$ for some $\alpha\in\mathcal A(p^\ast)$. Thus, $p^\ast\in\arg\underset{p\in\mathcal P}{\min}\;\; p^\top m(\phi)$ if and only if $\alpha^\top {\underline D}^\top m(\phi) =m(\phi)^\top {\underline D}\alpha\geqslant0$ 
for every $\alpha\in\mathcal A(p^\ast)$.

The canonical probability-mass exchange gradient is ${\underline D}^\top m(\phi)$:
\[
\Delta_S(\phi)
=
d_S^\top m(\phi)
=
m_S(\phi)
+
(|S|-1)m_\varnothing(\phi)
-
\sum_{i\in S}m_{\{i\}}(\phi),
\qquad |S|\geqslant2.
\]

\paragraph{Step 6: fixed-point conditions and extremal solution convergence.}
Assume first that $\phi^\ast=\Phi^\ast$. By Step~4, $F(\phi^\ast)=0$. By Step~3, choose $p^\ast\in\arg\min_{p\in\mathcal P}p^\top m(\phi^\ast)$. Then $(p^\ast)^\top m(\phi^\ast)=F(\phi^\ast)=0$. For every active atom $S\in\mathcal S_{p^\ast}$, choose $\bar {\textnormal{\textbf{q}}}_S^\ast\in\overline N_S(\phi^\ast)$. Then $h_{\phi^\ast}(\bar {\textnormal{\textbf{q}}}_S^\ast)=m_S(\phi^\ast)$,
so $\sum_{S\in\mathcal S_{p^\ast}}
p_S^\ast h_{\phi^\ast}(\bar {\textnormal{\textbf{q}}}_S^\ast)
=
\sum_{S\subseteq[m]}p_S^\ast m_S(\phi^\ast)
=
0$. Therefore, $\sum_{S\in\mathcal S_{p^\ast}}
p_S^\ast\allowbreak
f(\bar{\mathbf q}_S^\ast)
\mathrel{\allowbreak=\allowbreak}
\phi^\ast\allowbreak
\sum_{S\in\mathcal S_{p^\ast}}
p_S^\ast\allowbreak
g(\bar{\mathbf q}_S^\ast)
$. Since $g\geqslant t>0$ on $\Theta_{\mathrm{feas}}$, $\sum_{S\in\mathcal S_{p^\ast}}
p_S^\ast g(\bar {\textnormal{\textbf{q}}}_S^\ast)>0$. Thus,
\[
\frac{
\sum_{S\in\mathcal S_{p^\ast}}
p_S^\ast f(\bar {\textnormal{\textbf{q}}}_S^\ast)
}{
\sum_{S\in\mathcal S_{p^\ast}}
p_S^\ast g(\bar {\textnormal{\textbf{q}}}_S^\ast)
}
=
\phi^\ast.
\]
This proves $(E1_q^{\mathrm{cl}})$ and $(E2_{\mathrm{obj}})$. Since $p^\ast$ minimises the
outer linear problem in \eqref{eqn_2stateMarkovGenDinkelbach}, Step~5 with $\phi=\phi^*$ proves $(E1_p)$.

Conversely, suppose that $p^\ast$ and $(\bar {\textnormal{\textbf{q}}}_S^\ast)_{S\in\mathcal S_{p^\ast}}$
satisfy $(E1_q^{\mathrm{cl}})$, $(E1_p)$, and $(E2_{\mathrm{obj}})$. By
$(E1_q^{\mathrm{cl}})$, $h_{\phi^\ast}(\bar {\textnormal{\textbf{q}}}_S^\ast)=m_S(\phi^\ast)$ for $S\in\mathcal S_{p^\ast}$.
By $(E2_{\mathrm{obj}})$, $\sum_{S\in\mathcal S_{p^\ast}}
p_S^\ast f(\bar {\textnormal{\textbf{q}}}_S^\ast)
=
\phi^\ast
\sum_{S\in\mathcal S_{p^\ast}}
p_S^\ast g(\bar {\textnormal{\textbf{q}}}_S^\ast)$. Rearranging,  $\sum_{S\in\mathcal S_{p^\ast}}
p_S^\ast h_{\phi^\ast}(\bar {\textnormal{\textbf{q}}}_S^\ast)=0$.
Using $(E1_q^{\mathrm{cl}})$ again, this becomes $(p^\ast)^\top m(\phi^\ast)=0$. By $(E1_p)$ and Step~5, $p^\ast\in\arg\underset{p\in\mathcal P}{\min}\; p^\top m(\phi^\ast)$. Hence, $F(\phi^\ast)
=
\min_{p\in\mathcal P}p^\top m(\phi^\ast)
=
(p^\ast)^\top m(\phi^\ast)
=
0$. By Step~4, $\phi^\ast=\Phi^\ast$.

For convergence, suppose $\phi_r\to\Phi^\ast$ and $p^{(r)}\in\arg\min_{p\in\mathcal P}p^\top m(\phi_r)$. Since $\mathcal P$ is compact, every sequence $\{p^{(r)}\}$ has an accumulation point. Along a subsequence, let $p^{(r_k)}\to p^\ast$. For each nonmissing atom $S$, the map
\[
\phi\mapsto m_S(\phi)
=
\inf_{\textnormal{\textbf{q}}\in A_S}\{f(\textnormal{\textbf{q}})-\phi g(\textnormal{\textbf{q}})\}
\]
is Lipschitz, because it is the infimum of affine functions of $\phi$ with slopes in $[-1,0]$. Therefore, $m(\phi_{r_k})\to m(\Phi^\ast)$ componentwise. Passing to the limit in the inequality $(p^{(r_k)})^\top m(\phi_{r_k})
\leqslant
p^\top m(\phi_{r_k})$ for every $p\in\mathcal P$,
gives $(p^\ast)^\top m(\Phi^\ast)
\leqslant
p^\top m(\Phi^\ast)$ for every $p\in\mathcal P$. Thus, $p^\ast\in\arg\underset{p\in\mathcal P}{\min}\; p^\top m(\Phi^\ast)$. Since $F(\Phi^\ast)=0$, $(p^\ast)^\top m(\Phi^\ast)=0$.
Therefore, after choosing closure minimisers on active atoms, $p^\ast$ can be
completed to a tuple satisfying $(E1_q^{\mathrm{cl}})$, $(E1_p)$, $(E2_{\mathrm{obj}})$ at $\Phi^\ast$. If the outer minimiser at $\Phi^\ast$ is unique, every accumulation point is the same, and hence $p^{(r)}$ converges to it.

Finally, let $p^\ast$ be a fixed-point mass vector and choose, for every active atom $S\in\mathcal S_{p^\ast}$, a closure minimiser $\bar {\textnormal{\textbf{q}}}_S^\ast\in\overline N_S(\Phi^\ast)$. Since $\bar {\textnormal{\textbf{q}}}_S^\ast\in\overline A_S$, for every $\eta>0$ there exists ${\textnormal{\textbf{q}}}_{S,\eta}\in A_S$ such that $\|{\textnormal{\textbf{q}}}_{S,\eta}-\bar {\textnormal{\textbf{q}}}_S^\ast\|\leqslant \eta$. By continuity of $h_{\Phi^\ast}$ and the fact that $h_{\Phi^\ast}(\bar {\textnormal{\textbf{q}}}_S^\ast)=m_S(\Phi^\ast)$, we may also choose ${\textnormal{\textbf{q}}}_{S,\eta}$ so that $h_{\Phi^\ast}({\textnormal{\textbf{q}}}_{S,\eta})
\leqslant
m_S(\Phi^\ast)+\eta$. Define ${\mathbb P}_\eta
=
\sum_{S\in\mathcal S_{p^\ast}}p_S^\ast\delta_{{\textnormal{\textbf{q}}}_{S,\eta}}$. By Step~1, ${\mathbb P}_\eta\in\mathcal D$. Moreover, ${\mathbb E}_{{\mathbb P}_\eta}[h_{\Phi^\ast}(\textnormal{\textbf{Q}})]
=
\sum_{S\in\mathcal S_{p^\ast}}
p_S^\ast h_{\Phi^\ast}({\textnormal{\textbf{q}}}_{S,\eta})
\leqslant
\sum_{S\subseteq[m]}p_S^\ast m_S(\Phi^\ast)+\eta
=
\eta$. Since $F(\Phi^\ast)=0$, every feasible prior has nonnegative Dinkelbach difference at $\Phi^\ast$. Hence, $0
\leqslant
{\mathbb E}_{{\mathbb P}_\eta}[h_{\Phi^\ast}(\textnormal{\textbf{Q}})]
\leqslant
\eta$. Using ${\mathbb E}_{{\mathbb P}_\eta}[h_{\Phi^\ast}(\textnormal{\textbf{Q}})]
=
{\mathbb E}_{{\mathbb P}_\eta}[f(\textnormal{\textbf{Q}})]
-
\Phi^\ast{\mathbb E}_{{\mathbb P}_\eta}[g(\textnormal{\textbf{Q}})]$,
and ${\mathbb E}_{{\mathbb P}_\eta}[g(\textnormal{\textbf{Q}})]\geqslant t$,
we obtain $0
\leqslant
\frac{{\mathbb E}_{{\mathbb P}_\eta}[f(\textnormal{\textbf{Q}})]}
{{\mathbb E}_{{\mathbb P}_\eta}[g(\textnormal{\textbf{Q}})]}
-
\Phi^\ast
\leqslant
\frac{\eta}{t}$.
Therefore, 
$\frac{{\mathbb E}_{{\mathbb P}_\eta}[f(\textnormal{\textbf{Q}})]}
{{\mathbb E}_{{\mathbb P}_\eta}[g(\textnormal{\textbf{Q}})]}
\to
\Phi^\ast$ as $\eta\downarrow0$.
\end{proof}
\newpage
\section{Proof of Proposition~\ref{prop_convergence_to_ordinary_BayesianInference}}
\begin{proof}
For fixed $m$ and each cell $K_i^{(m)}$, define $p_i^{(m)}:=\mu\!\left(K_i^{(m)}\right)$ and
\[
f_{i,m}^-:=\inf_{\mathbf{x}\in K_i^{(m)}} f(\mathbf{x}),\quad
f_{i,m}^+:=\sup_{\mathbf{x}\in K_i^{(m)}} f(\mathbf{x}),\quad
g_{i,m}^-:=\inf_{\mathbf{x}\in K_i^{(m)}} g(\mathbf{x}),\quad
g_{i,m}^+:=\sup_{\mathbf{x}\in K_i^{(m)}} g(\mathbf{x}).
\]
Set
\[
A_m^-:=\sum_{i=1}^{n_m} f_{i,m}^-\,p_i^{(m)},\;
A_m^+:=\sum_{i=1}^{n_m} f_{i,m}^+\,p_i^{(m)},\;
B_m^-:=\sum_{i=1}^{n_m} g_{i,m}^-\,p_i^{(m)},\; 
B_m^+:=\sum_{i=1}^{n_m} g_{i,m}^+\,p_i^{(m)}.
\]

Take any $\mathbb P\in\mathcal D_m$. Note that $\mathbb P(K_i^{(m)})=p_i^{(m)}$, $\mathbb E[f]=\sum_{i=1}^{n_m}\mathbb E[f{\mathbf 1}_{K_i^{(m)}}]$ and $\mathbb E[g]=\sum_{i=1}^{n_m}\mathbb E[g{\mathbf 1}_{K_i^{(m)}}]$. By definition of infimum and supremum on each cell,
\[
f_{i,m}^-\,p_i^{(m)}
\leqslant
\mathbb E[f{\mathbf 1}_{K_i^{(m)}}]
\leqslant
f_{i,m}^+\,p_i^{(m)},\quad g_{i,m}^-\,p_i^{(m)}
\leqslant
\mathbb E[g{\mathbf 1}_{K_i^{(m)}}]
\leqslant
g_{i,m}^+\,p_i^{(m)}.
\]
Summing over $i$ then gives $A_m^- \leqslant \mathbb E[f]\leqslant A_m^+$ and $B_m^- \leqslant \mathbb E[g]\leqslant B_m^+$. Hence, whenever $B_m^->0$,
\[
\frac{A_m^-}{B_m^+}
\leqslant
\frac{\mathbb E[f]}{\mathbb E[g]}
\leqslant
\frac{A_m^+}{B_m^-}.
\]
Therefore, for such $m$,
\begin{equation}
\frac{A_m^-}{B_m^+}\leqslant L_m\leqslant U_m\leqslant \frac{A_m^+}{B_m^-}.
\label{eqn_LmlessthanUm}
\end{equation}

Since $f$ is continuous on compact $K$, it is uniformly continuous. Let $\omega_f(m):=
\allowbreak
\max_{1\leqslant i\leqslant n_m}
\allowbreak
\sup_{\mathbf{x},\mathbf{y}\in K_i^{(m)}}
\allowbreak
\lvert f(\mathbf{x})-f(\mathbf{y})\rvert$. 
Since the mesh tends to $0$, $\omega_f(m)\to 0$. Thus, 
$0\leqslant A_m^+-A_m^- \leqslant \omega_f(m)\sum_{i=1}^{n_m}p_i^{(m)}=\omega_f(m)\to 0$. Also, for each $i$, $f_{i,m}^-\,p_i^{(m)}
\leqslant
\mathbb E_\mu[f{\mathbf 1}_{K_i^{(m)}}]
\leqslant
f_{i,m}^+\,p_i^{(m)}$, 
so summing over $i$, $A_m^- \leqslant \mathbb E_\mu[f] \leqslant A_m^+$. Since the gap $A_m^+-A_m^-\to 0$, it follows that
$A_m^-\to \mathbb E_\mu[f]$ and $A_m^+\to \mathbb E_\mu[f]$. The same argument for $g$ yields $B_m^-\to \mathbb E_\mu[g]$ and $B_m^+ \to \mathbb E_\mu[g]$.

By assumption, $\mathbb E_\mu[g]>0$. Therefore, for all sufficiently large $m$, $B_m^- > \mathbb E_\mu[g]/2>0$. So, for all sufficiently large $m$, every $\mathbb P\in\mathcal D_m$ satisfies $\mathbb E[g]\geqslant B_m^->0$, and the bounds in \eqref{eqn_LmlessthanUm} are valid.

Finally,
\[
\frac{A_m^-}{B_m^+}\to \frac{\mathbb E_\mu[f]}{\mathbb E_\mu[g]},
\qquad
\frac{A_m^+}{B_m^-}\to \frac{\mathbb E_\mu[f]}{\mathbb E_\mu[g]},
\]
since numerator and denominator in each ratio converge and $\mathbb E_\mu[g]>0$. By the squeeze theorem applied to \eqref{eqn_LmlessthanUm},
\[
L_m\to \frac{\mathbb E_\mu[f]}{\mathbb E_\mu[g]},
\qquad
U_m\to \frac{\mathbb E_\mu[f]}{\mathbb E_\mu[g]}.\qedhere
\]\end{proof}

\end{document}